\documentclass{article}
\usepackage[a4paper]{geometry}
\usepackage{amsmath}
\usepackage{amssymb}
\usepackage{amsthm}
\usepackage{mathrsfs}
\usepackage{mathtools}

\usepackage{tikz}
\usetikzlibrary{calc} 
\usepackage{hyperref}
\usepackage{cleveref}
\usepackage{tikz-cd}
\usepackage{mathtools}

\usepackage{placeins}

\usepackage{float}

\usepackage{comment}

\usepackage{enumerate}

\newtheorem{thm}{Theorem}[section]
\crefname{thm}{Theorem}{Theorems}
\newtheorem{cor}[thm]{Corollary}
\newtheorem{prop}[thm]{Proposition}
\crefname{prop}{Proposition}{Propositions}
\newtheorem{lem}[thm]{Lemma}
\crefname{lem}{Lemma}{Lemmas}
\newtheorem{clm}[thm]{Claim}
\newtheorem{subclm}[thm]{Subclaim}

\newtheorem{defn}[thm]{Definition}

\crefname{defn}{Definition}{Definitions}

\newtheorem{rmk}[thm]{Remark}

\newtheorem{obs}[thm]{Observation}

\newtheorem*{ack*}{Acknowledgements}

\numberwithin{equation}{section}

\newcommand{\Ent}{\mathrm{Ent}}
\newcommand{\Hess}{\mathrm{Hess}}
\newcommand{\dom}{\mathrm{dom}}

\newcommand{\co}{\operatorname{co}}

\newcommand{\M}{M_{\lambda, p}}
\newcommand{\R}{\mathbb{R}}

\usepackage{titling}
\title{Sharp Quantitative Stability for the\\Pr\'ekopa-Leindler and
Borell-Brascamp-Lieb Inequalities
}
\author{Alessio Figalli
\thanks{Department of Mathematics, ETH Z\"urich, Switzerland. email: alessio.figalli@math.ethz.ch},
Peter van Hintum
\thanks{Department of Mathematics, ETH Z\"urich, Switzerland. email: peter.vanhintum@math.ethz.ch }, 
Marius Tiba
\thanks{Ecole Polytechnique Fédérale Lausanne, Switzerland UK. email: marius.tiba@epfl.ch}}

\begin{document}

\maketitle

\begin{abstract}
The Borell-Brascamp-Lieb inequality is a classical extension of the Prékopa-Leindler inequality, which in turn is a functional counterpart of the Brunn-Minkowski inequality. The stability of these inequalities has received significant attention in recent years. Despite substantial progress in the geometric setting, 
a sharp quantitative stability result for the Prékopa-Leindler inequality has remained elusive, even in the special case of log-concave functions.
In this work, we provide a unified and definitive stability framework for these foundational inequalities. By establishing the optimal quantitative stability for the Borell-Brascamp-Lieb inequality in full generality, we resolve the conjectured sharp stability for the Prékopa-Leindler inequality as a particular case. Our approach builds on the recent sharp stability results for the Brunn-Minkowski inequality obtained by the authors in \cite{BMStab, OTBMStab}.
\end{abstract}

\setcounter{tocdepth}{2}
\tableofcontents

\section{Introduction}

The stability of geometric inequalities has received a lot of attention over recent years. This study goes back to the isoperimetric inequality stating that among all bodies of a given volume, the Euclidean ball minimizes the surface, a classical result attributed to Queen Dido of Carthage. In the landmark result \cite{FMP08} (resp. \cite{Figalli10amass}), the authors settled the quantitative stability of that inequality (resp. its anisotropic version) by showing that bodies with nearly minimal surface area must be close to a ball (resp. a Wulff shape) in a sharp quantitative sense. 
However, the isoperimetric is merely the base of a hierarchy of geometric inequalities (see e.g. \cite{gardner2002brunn}).

\subsection{The Brunn-Minkowski inequality}
The Brunn-Minkowski inequality is a fundamental result on which much of convex geometry is built \cite{schneider2013convex}. It asserts that for sets $A,B\subset\mathbb{R}^n$ of equal volume and a parameter $\lambda\in(0,1)$, we have $$|\lambda A+(1-\lambda)B|\geq |A|,$$
with equality if and only if $A$ and $B$ are the same convex body less a measure zero set, up to translation. Note that, up to exchanging the roles of $A$ and $B$, one can always assume that $\lambda \in (0,1/2].$

The stability of the Brunn-Minkowski inequality has a rich history with a series of results obtained in the last years \cite{ruzsa1991diameter,ruzsa1997brunn,ruzsa2006additive,Figalli09,figalli2010mass,christ2012planar,christ2012near,eldan2014dimensionality,figalli2015quantitative,figalli2015stability,figalli2017quantitative,Barchiesi,carlen2017stability,figalli2021quantitative,van2021sharp,van2023locality,van2020sharp,planarBM,SharpDelta}, see the introduction of \cite{BMStab} for a description of these contributions.
Recently, the current authors managed to conclude this line of inquiry \cite{BMStab,OTBMStab} showing the following result.\footnote{Here and in the sequel, given $m,a>0$ we shall write $O_m(a)$ (resp. $\Omega_m(a)$) to denote a quantity that is bounded from above (resp. from below) by $C_ma$, where the constant $C_m>0$ may depend on $m$.}

\begin{thm}\label{BMStab}
    Let $|A|=|B|$ and $\lambda \in (0,1/2]$, and assume that $|\lambda A+(1-\lambda)B|\leq (1+\delta)|A|$ for some $\delta>0$ sufficiently small in terms of $n$ and $\lambda$. Then there exists a convex set $K$ so that, up to a translation,
$$K\supset A\cup B,\qquad |K\setminus A|+|K\setminus B|= O_{n,\lambda}\left(\sqrt{\delta}\right)|A|.$$
\end{thm}
In other words, the theorem above states that $A$ and $B$ are $\sqrt{\delta}$-close to the same convex set.
Moreover, if we do not insist on the same convex set, then a stronger linear stability holds (see \cite{BMStab}):
$$|\co(A)\setminus A|+|\co(B)\setminus B|=O_{n,\lambda}(\delta)|A|,$$
where $\co(A)$ is the convex hull of $A$, namely, the smallest convex set containing $A$. Both the square root and the linear dependence are optimal, in the sense that the powers $1/2$ and $1$ cannot be replaced by anything bigger.

\subsection{The Pr\'ekopa-Leindler inequality}
The Brunn-Minkowski inequality has a functional extension called the Pr\'ekopa-Leindler inequality. Given integrable functions $f,g,h\colon\mathbb{R}^n\to\mathbb{R}_{\geq0}$ and a parameter $\lambda\in (0,1/2]$, assume that $
\int f\,dx=\int g\,dx$ and that, for all $x,y\in \mathbb{R}^n$, we have $$h(\lambda x +(1-\lambda)y)\geq f(x)^{\lambda}g(y)^{1-\lambda}$$
(here $f,g,h$ mirror $A,B,\lambda A+(1-\lambda)B$ in the Brunn-Minkowski inequality). Then, the Pr\'ekopa-Leindler inequality states that $\int h\,dx\geq \int f\,dx$, with equality if and only if $f$ and $g$ are (up to a translation) the same \emph{log-concave} function almost everywhere.\footnote{Recall that a function $f$ is log-concave if for all $x,y\in\mathbb{R}^n$ and for all $\lambda\in[0,1]$, we have $f(\lambda x+ (1-\lambda)y)\geq f(x)^{\lambda}f(y)^{1-\lambda}$.} 

When restricted to indicator functions of sets, the Pr\'ekopa-Leindler inequality directly implies the Brunn-Minkowski inequality. In the reverse direction, the Pr\'ekopa-Leindler inequality can be deduced from the Brunn-Minkowski inequality, by associating to a function $f:\R^n\to \R$ the set $A_f\subset \mathbb{R}^{n+k}$ defined by
$$A_f:=\bigcup_{x\in\mathbb{R}^n} \{x\}\times \underbrace{\left[0,f(x)^{1/k}\right]\times \ldots\times \left[0,f(x)^{1/k}\right]}_{k\text{ times}},$$
so that $|A_f|=\int f\,dx$.
Applying Brunn-Minkowski to the sets $A_f$ and $A_g$, and noticing that $|\lambda A_f+(1-\lambda) A_g|\downarrow \int h\,dx$ as $k \to \infty$,\footnote{More precisely, assuming that $f,g$ are bounded and with bounded support, then $|\lambda A_f+(1-\lambda) A_g|\downarrow \int h\,dx$ as $k \to \infty$ by dominated convergence. This implies the Pr\'ekopa-Leindler inequality in this case, and the general case follows by approximation.} the Pr\'ekopa-Leindler inequality follows.

Despite the interest in the stability of Pr\'ekopa-Leindler \cite{bucur2014lower,ball2010stability,ball2011stability,boroczky2021stability,boroczky2023quantitative}, sharp stability results have been harder to obtain. Indeed, all previous results provide quantitative estimates but with non-sharp exponents. Furthermore, except for the very special case where one considers one-dimensional log-concave functions \cite{figalli_ramos_PL24}, even for special classes of functions (like log-concave functions in $\mathbb R^n$ or one-dimensional arbitrary functions), no optimal quantitative results are known.

As for the Brunn-Minkowski inequality, it was believed that a sharp $\sqrt{\delta}$-stability should persist in the functional context. More precisely, let $f,g,h$, and $\lambda \in (0,1/2]$ as before. If $\int h\,dx\leq (1+\delta)\int f\,dx$ for some $\delta \geq 0$, then there should exist a log-concave function $\ell\colon\mathbb{R}^n\to\mathbb{R}_{\geq 0}$ so that, up to a translation,
$$\int \big(|\ell-f|+|g-\ell|\big)\,dx= O_{\lambda,n}\left(\sqrt{\delta}\right)\int f\,dx.$$
Despite being a very natural statement, one major challenge here is to identify the ``correct'' log-concave function $\ell$. Indeed, even when $f=g$, $\ell$ {\em cannot} be chosen as the log-concave hull of $f$ (see Remark~\ref{rem:no pconcave hull} below).
In this paper, we settle this line of investigation by proving the bound above as a special case of our main theorem.

\subsection{The Borell-Brascamp-Lieb inequality}
Although the Pr\'ekopa-Leindler condition $h(\lambda x +(1-\lambda)y)\geq f(x)^{\lambda}g(y)^{1-\lambda}$ seems natural from the perspective of the previously mentioned proof using limits of sets in $\mathbb{R}^n$, it is, in fact, stronger than what is needed, as noticed by Borell \cite{Borell} and independently by Brascamp and Lieb \cite{brascamp-lieb}. To state their result, consider the following definition.

\begin{defn}
Given parameters $\lambda\in[0,1]$ and $p\in \mathbb{R}$, and real numbers $x,y\geq 0$, let
$$M_{\lambda,p}(x,y)=\begin{cases}0 &\text{ if } xy=0\\
\left(\lambda x^{p}+(1-\lambda)y^{p}\right)^{1/p}&\text{ if } p\neq 0 \text{ and }xy\neq0\\
x^{\lambda}y^{1-\lambda}&\text{ if $p=0$} \text{ and }xy\neq0\end{cases}.$$
\end{defn}
For $p=0$, this definition corresponds to the Pr\'ekopa-Leindler condition, as $h(\lambda x +(1-\lambda)y)\geq f(x)^\lambda g(y)^{1-\lambda}= M_{\lambda,0}(f(x),g(y))$. Also, by the monotonicity of means, we have $M_{\lambda,p}(x,y)\leq M_{\lambda,q}(x,y)$  for $p<q$, so the condition $h(\lambda x +(1-\lambda)y)\geq M_{\lambda,p}(f(x),g(y))$ gets weaker as $p$ decreases.

The following definition is also important:
\begin{defn}
A function $\ell:\mathbb{R}^n\to\mathbb{R}_{\geq0}$ is \emph{$p$-concave} if for all $x,y\in\mathbb{R}^n$ and $\lambda\in[0,1]$, we have that $\ell(\lambda x +(1-\lambda)y)\geq \M(\ell(x),\ell(y))$.
\end{defn}
In these terms, $0$-concave is the same as log-concave. Perhaps more intuitively, a function $f$ is $p$-concave for $p>0$ (resp. $p<0$) if its support is convex and $f^p$ is concave (resp. convex) inside its support.
 The Borell-Brascamp-Lieb inequality asserts that we can decrease the parameter $p$ in the Pr\'ekopa-Leindler condition all the way down to $p=-1/n$.

\begin{thm}[Borell-Brascamp-Lieb Inequality]
    Let $n \in \mathbb{N}$, $\lambda\in (0,1/2]$, $p \in [-1/n, \infty)$, and $f,g,h \colon \mathbb{R}^n \rightarrow \mathbb{R}_{\geq0}$ integrable functions. Assume that $\int f\,dx = \int g\,dx$  and 
    $$h(\lambda x+ (1-\lambda)y) \geq \M(f(x),g(y))$$ for all $x, y \in \mathbb{R}^n$.  Then
    $$\int h\,dx \geq \int f\,dx.$$
   In addition, if $p>-1/n$, then equality holds
    if and only if $f$ and $g$ are, up to a translation, the same {$p$-concave} function a.e. 
\end{thm}
\begin{rmk}
\label{rmk:p-1n}
The family of equality cases for $p=-1/n$ is richer than the one for $p>-1/n$, e.g., equality holds for the triple
$$
f=\textbf{1}_{[0,1]^n}, \quad g=t^{-n}\textbf{1}_{[0,t]^n},
\quad \text{and}\quad h=M_{\lambda,-1/n}(1,t^{-n})\textbf{1}_{[0,\lambda+(1-\lambda)t]^n}=(\lambda+(1-\lambda)t)^{-n}\textbf{1}_{[0,\lambda+(1-\lambda)t]^n}.
$$ 
Characterizing the equality cases for $p=-1/n$ is a non-trivial task done by Dubuc \cite{dubuc1977criteres}. 
\end{rmk}

 Although the equality case for the Borell-Brascamp-Lieb inequality is well understood (see \cite{BaloghKristaly2018} for a recent proof that works also on curved spaces), very little is known at the level of stability. In particular, the papers \cite{ghilli2017quantitative,rossi2017stability} deal with the case $p>0$, which is weaker than 
Pr\'ekopa-Leindler.

\subsection{Main results}

The main result of this paper is a sharp stability result for the Borell-Brascamp-Lieb inequality for $p>-1/n$.

\begin{thm}\label{main}
    Let $n \in \mathbb{N}$, $\lambda \in (0,1/2]$, and $p \in (-1/n, \infty)$. Let $f,g, h\colon \mathbb{R}^n\rightarrow \mathbb{R}_{\geq0}$ be integrable functions such that
    \begin{itemize}
        \item $\int_{\mathbb{R}^n} f\,dx=  \int_{\mathbb{R}^n} g\,dx$,
        \item $h(\lambda x +(1-\lambda)y) \geq M_{\lambda,p}(f(x),g(y))$ for all $x,y \in \mathbb{R}^n$, and
        \item $\int_{\mathbb{R}^n}h\,dx = (1+\delta)\int_{\mathbb{R}^n}f\,dx$ for some $\delta \geq 0$.
    \end{itemize} 
    Then there exists a $p$-concave function $\ell \colon \mathbb{R}^n\rightarrow \mathbb{R}_{\geq0}$ such that, up to a translation\footnote{That is, there exists a $v\in\mathbb{R}^n$, so that $\int_{\mathbb{R}^n} \big(|f(x)-\ell(x)| + |g(x-v)-\ell(x)| \big)\, dx=O_{n,\lambda, p} \left(\sqrt{\delta}\right)\int_{\mathbb{R}^n}f(x)dx.$ }, $$\int_{\mathbb{R}^n} \big(|f-\ell| + |g-\ell|\big)\,dx= O_{n,\lambda, p} \left(\sqrt{\delta}\right)\int_{\mathbb{R}^n}f\,dx.$$ 
\end{thm}
In the special case $p=0$, we obtain the optimal quantitative stability of the Pr\'ekopa-Leindler inequality, resolving the main conjecture in this direction and improving the results from \cite{bucur2014lower,ball2010stability,ball2011stability,boroczky2021stability,boroczky2023quantitative,figalli_ramos_PL24}.
\begin{cor}\label{PLmain}
    Let $n \in \mathbb{N}$ and $\lambda \in (0,1/2]$. Let $f,g, h\colon \mathbb{R}^n\rightarrow \mathbb{R}_{\geq0}$ be integrable functions such that
    \begin{itemize}
        \item $\int_{\mathbb{R}^n} f\,dx=  \int_{\mathbb{R}^n} g\,dx$,
        \item $h(\lambda x +(1-\lambda)y) \geq f(x)^{\lambda}g(y)^{1-\lambda}$ for all $x,y \in \mathbb{R}^n$, and
        \item $\int_{\mathbb{R}^n}h\,dx = (1+\delta)\int_{\mathbb{R}^n}f\,dx$ for some $\delta \geq 0$.
    \end{itemize} 
    Then there exists a log-concave function $\ell \colon \mathbb{R}^n\rightarrow \mathbb{R}_{\geq0}$ such that, up to a translation,
    $$\int_{\mathbb{R}^n} \big(|f-\ell| + |g-\ell| \big)\,dx=O_{n,\lambda} \left(\sqrt{\delta}\right)\int_{\mathbb{R}^n}f\,dx.$$ 
\end{cor}

The assumption that $\int f\,dx=\int g\,dx$ in these theorems and all previous discussions is a matter of normalization, and the general case can be obtained by rescaling (which simply changes the value of $\lambda$).

While the sharp stability of the Brunn-Minkowski inequality (\Cref{BMStab}) serves as a fundamental tool throughout the proofs, the functional nature of these inequalities demands a highly refined analysis and the introduction of several new ideas.
More precisely, as detailed in Sections \ref{sect:structure thm 1} and \ref{sect:structure thm2} below, our results rely critically on the stability of the Brunn-Minkowski inequality, yet achieving these results requires a completely novel approach.

\begin{rmk}
As mentioned before, the square root dependence is optimal, as can be seen by considering the one-dimensional example
$$f=\frac{1}{1+\sqrt{\delta}}\textbf{1}_{\left[0,1+\sqrt{\delta}\right]},\quad
g=\left(1+\sqrt{\delta}\right)\textbf{1}_{\left[0,\frac{1}{1+\sqrt{\delta}}\right]},\quad\text{and} \quad h=\textbf{1}_{\left[0,\frac{1+\sqrt{\delta}+\frac{1}{1+\sqrt{\delta}}}{2}\right]}.$$
\end{rmk}




Regarding the dependence on $\lambda$, based on corresponding results for the Brunn-Minkowski inequality (see, in particular, \cite{Figalli09, Barchiesi,Euclidean}), it is natural to expect that the expression $O_{n,\lambda,p}(\sqrt{\delta})$ can be replaced by $ O_{n,p}(\sqrt{\delta/\lambda})$. In the next theorem we prove this improved bound in the case that one of the two functions is $p$-concave.

For $p>-1/n$, write $\zeta_{n,p}:=\min\{1,1+np\}.$

\begin{thm}\label{sharplambda}
    Let $n \in \mathbb{N}$, $\lambda \in (0,1/2]$, and $p \in (-1/n, \infty)$. Let $\ell,g, h\colon \mathbb{R}^n\rightarrow \mathbb{R}_{\geq0}$ be integrable functions such that
    \begin{itemize}
        \item $\ell$ is $p$-concave,
        \item $\int_{\mathbb{R}^n} \ell\,dx=  \int_{\mathbb{R}^n} g\,dx$,
        \item $h(\lambda x +(1-\lambda)y) \geq M_{\lambda,p}(\ell(x),g(y))$ for all $x,y \in \mathbb{R}^n$, and
        \item $\int_{\mathbb{R}^n}h\,dx = (1+\delta)\int_{\mathbb{R}^n}\ell\,dx$ for some $\delta \geq 0$.
    \end{itemize} 
    Then
    \[
    \inf_{v\in\mathbb R^n}
    \int_{\mathbb R^n}|g(x+v)-\ell(x)|\,dx
    \leq \frac{n^{O(n)}}{\zeta_{n,p}}
    \sqrt{\frac{\delta}{\lambda}}\int_{\mathbb R^n}\ell\,dx.
    \]
\end{thm}

Since the family of equality cases for $p = -1/n$ is richer than the class of $(-1/n)$-concave functions (see \Cref{rmk:p-1n}), \Cref{main} and \Cref{sharplambda} do not hold for $p = -1/n$. In particular, the constant multiplying $\sqrt{\delta}$ must blow up as $p \to -1/n^+$, as \Cref{sharplambda} also shows (since $\zeta_{n,p} \to 0$ as $p\to -1/n^+$). Indeed, fix
$\lambda=1/2$ and, for $-1/n<p<0$, consider
\[
\ell=\textbf{1}_{[0,1]^n},\qquad
g=2^{-n}\textbf{1}_{[0,2]^n},\qquad
h=M_{1/2,p}(1,2^{-n})\textbf{1}_{[0,3/2]^n}.
\]
For this triple, the deficit $\delta_p$ satisfies $\delta_p=\left(\frac32\right)^n
\left(\frac{1+2^{-np}}2\right)^{1/p}-1\leq O_n(1+np)$.
If $C_{n,p}$ denotes the smallest admissible constant in
\Cref{sharplambda}, the example gives the lower bound and the theorem gives
the upper bound
\[
\Omega_n\left((1+np)^{-1/2}\right)\leq C_{n,p}\leq
n^{O(n)}(1+np)^{-1}.
\]

\begin{rmk}
In \Cref{sharplambda}, we assume $\lambda\in(0,1/2]$. The same result for $\lambda\in (1/2,1)$ with $\lambda$ in the conclusion replaced by $1-\lambda$ follows quickly from \Cref{sharplambda} at $\lambda=1/2$ by a trick due to B\"or\"oczky and De, cf. \cite[Lemma 7.4]{boroczky2021stability}. It's the functional analogue of the observation that if $K\subset\mathbb{R}^n$ is convex, then $(2\lambda-1)K+2(1-\lambda)K=K$, so that by Brunn-Minkowski for $B\subset\mathbb{R}^d$ with $|K|=|B|$, $$(2\lambda-1)|K|^{1/n}+2(1-\lambda)\left|\frac{K+B}{2}\right|^{1/n}\leq |\lambda K+(1-\lambda)B|^{1/n}.$$
Hence, if $|\lambda K+(1-\lambda)B|\leq (1+\delta)|K|$, then  $\left|\frac{K+B}{2}\right|=O_n\left(\frac{\delta}{1-\lambda}\right)|K|$.
\end{rmk}

\begin{rmk}
The conclusions of \Cref{main}, \Cref{PLmain}, and \Cref{sharplambda} are formulated to control the functions $f$ (or $\ell$) and $g$. However, these conclusions immediately yield the same result for $h$. For illustration consider $f,g,h,\ell$ as in \Cref{main} and assume there is no translation needed in the conclusion. Let $f'=\min\{f,\ell\}$, $g'=\min\{g,\ell\}$, and $h'=\min\{h,\ell\}$. By the conclusion of \Cref{main}, we have
$$\int f' dx, \int g' dx \geq \left(1-O_{n,\lambda,p}\left(\sqrt{\delta}\right)\right)\int f dx.$$
Hence, by Borell-Brascamp-Lieb applied to $f'$ and $g'$, we find that
$$\int h' dx\geq \left(1-O_{n,\lambda,p}\left(\sqrt{\delta}\right)\right)\int f dx,$$
and thus
$$\int |h-\ell| dx\leq 2 \int (h- h') dx \leq O_{n,\lambda,p}\left(\sqrt{\delta}\right)\int f dx.$$
\end{rmk}

\subsection{Sharp stability for log-Sobolev-type inequalities }

Given $\mu_\varphi=e^{-\varphi}$ a log-concave probability density, we denote its Fenchel--Legendre transform by $\varphi^*:\R^n\to \R\cup\{+\infty\}$,  where
\[
\varphi^*(y):=\sup_{x\in\R^n}\{y\cdot x-\varphi(x)\}\in(-\infty,+\infty].
\] 
Also, given a nonnegative function $G \in L^1(\R^n)$, we define
\[
Z_G:=\int_{\R^n} G\,dx,\qquad \Ent(G):=\int_{\R^n} G\log\!\Big(\frac{G}{Z_G}\Big)\,dx
= \int_{\R^n} G\log G\,dx - Z_G\log Z_G
\]

We can now state a sharp stability result for a large class of log-Sobolev-type inequalities.

\begin{thm}\label{thm:LSI}
Let $\mu_\varphi=e^{-\varphi}$ be a log-concave probability density on $\R^n$.
Then, for every $\psi:\R^n\to \R$ such that $e^{-\psi}\in L^1(\R^n)\cap W^{1,1}_{\rm loc}(\R^n)$,
\begin{equation}\label{eq:29}
\int_{\R^n}\varphi^*(\nabla\psi(x))e^{-\psi(x)}\,dx-nZ_{e^{-\psi}}-\Ent(e^{-\psi})
\geq c_n \,Z_{e^{-\psi}} \inf_{x_0 \in \R^n} \left\|\frac{e^{-\psi}}{Z_{e^{-\psi}}}-e^{-\varphi(\cdot-x_0)}\right\|_{L^1(\R^n)}^2,
\end{equation}
where $c_n>0$ is a computable dimensional constant, and we use the convention
$\nabla \psi:=-\frac{\nabla e^{-\psi}}{e^{-\psi}},$
which is well-defined $e^{-\psi}$-a.e.
\end{thm}

A direct consequence of the previous theorem is the following sharp stability result for the anisotropic $p$-log-Sobolev inequality.

\begin{cor}\label{cor:gauge_power_signed}
Let $\|\cdot\|:\R^n\to[0,\infty)$ be convex, positively $1$--homogeneous, with
$\|x\|=0$ only if $x=0$, and define the polar gauge
$\|\xi\|_*:=\sup\{\xi\cdot x:\ \|x\|\le 1\}.$
Fix $q>1$ and let $p=\frac{q}{q-1}$. Set
\[
L_q:=\int_{\R^n}\exp\!\Big(-\frac{\|x\|^q}{q}\Big)\,dx\in(0,\infty),
\]
and define the families of functions 
\[
\mathcal E_q:=\left\{
\rho_{a,x_0}(x):=\frac{a^{n/q}}{L_q}\exp\!\Big(-\frac{a\|x-x_0\|^q}{q}\Big)
:\ a>0,\ x_0\in\R^n
\right\},
\]
\[
\check{\mathcal E}_q:=\left\{
\check\rho_{a,x_0}(x):=\frac{a^{n/q}}{L_q}\exp\!\Big(-\frac{a\|x_0-x\|^q}{q}\Big)
:\ a>0,\ x_0\in\R^n
\right\},
\]
and $\mathcal S_q:=\mathcal E_q\cup\{-\check{\mathcal E}_q\}$
 (when $\|\cdot\|$ is symmetric, then $\mathcal S_q=\{\pm\rho_{a,x_0}\}$).
Then, for every $u\in W^{1,p}(\R^n)$ with
$\|u\|_p>0$ and $\Ent(|u|^p)<\infty$,
\begin{equation}\label{eq:gauge_power_signed}
\frac{n}{p}\,\|u\|_p^p\,
\log\!\left(
\frac{p^p}{n e^{p-1} L_q^{p/n}}\,
\frac{\int_{\R^n}\|\nabla u(x)\|_*^p\,dx}{\|u\|_p^p}
\right)
-\Ent(|u|^p)
\ \ge\
\frac{c_n}4\,\|u\|_p^p\,
\inf_{\rho\in \mathcal S_q}
\left\|
\frac{|u|^{p-1}u}{\|u\|_p^p}-\rho
\right\|_{L^1(\R^n)}^{\,2},
\end{equation}
with $c_n$ as in 
Theorem~\ref{thm:LSI}.
\end{cor}

Taking the limit as $p\to 1^+$ in the previous result, we obtain the sharp stability for the anisotropic $1$-log-Sobolev inequality, recovering the main result of \cite{FMP}.

\begin{cor}\label{cor:wulff_signed_sets}
Let $\|\cdot\|:\R^n\to[0,\infty)$ be convex, positively $1$--homogeneous, and such that
$\|x\|=0$ if and only if $x=0$. Define
\[
\|\xi\|_*:=\sup\{\xi\cdot x:\ \|x\|\le 1\},
\qquad
K:=\{x\in\R^n:\ \|x\|<1\}.
\]
Let
\[
\mathcal W:=\left\{\frac{\chi_{x_0+rK}}{|rK|}:r>0,\ x_0\in\R^n\right\},
\qquad
\check{\mathcal W}:=\left\{\frac{\chi_{x_0-rK}}{|rK|}:r>0,\ x_0\in\R^n\right\},
\]
$\mathcal S_\infty:=\mathcal W\cup(-\check{\mathcal W}),$
and let $c_n$ be as in Theorem~\ref{thm:LSI}
Then, for every $u\in W^{1,1}(\R^n)$ with
$\|u\|_1>0$ and $\Ent(|u|)<\infty$,
\begin{equation}\label{eq:wulff_signed}
n\|u\|_1
\log\!\left(
\frac{1}{n|K|^{1/n}}\,
\frac{\int_{\R^n}\|\nabla u(x)\|_*\,dx}{\|u\|_1}
\right)
-\Ent(|u|)
\ \ge\
\frac{c_n}4\,\|u\|_1\,
\inf_{\rho\in \mathcal S_\infty}
\left\|
\frac{u}{\|u\|_1}-\rho
\right\|_{L^1(\R^n)}^{\,2}.
\end{equation}
In particular, for every measurable set $E\subset\R^n$ with $0<|E|<\infty$ and finite anisotropic perimeter
\[
P_{\|\cdot\|}(E):=\sup\left\{\int_E \mathrm{div}\,\eta\,dx:\ \eta\in C_c^1(\R^n;\R^n),\ \|\eta(x)\|\le 1\ \forall x\in\R^n\right\},
\]
one has
\begin{equation}\label{eq:wulff_sets}
n\log\!\left(
\frac{P_{\|\cdot\|}(E)}{n|K|^{1/n}|E|^{(n-1)/n}}
\right)
\ \ge\
\frac{c_n}4\,
\inf_{\rho\in \mathcal W}
\left\|
\frac{\chi_E}{|E|}-\rho
\right\|_{L^1(\R^n)}^{\,2}.
\end{equation}
\end{cor}

\begin{rmk}\label{rem:cor_lit}
Quantitative stability for logarithmic Sobolev inequalities has been extensively studied in the Gaussian and Euclidean settings; see, for instance,
\cite{Carlen91,CarlenDuality,BobkovGozlanRobertoSamson,FathiIndreiLedoux,FeoPosteraroRoberto,EldanLehecShenfeld,IndreiMarcon,IndreiKim,DolbeaultToscani,BrigatiDolbeaultSimonov,DEFFL,BDSreview}.
For related stability results concerning Sobolev inequalities, see also
\cite{CianchiFuscoMaggiPratelli,FigalliZhang,DEFFL}.
For $p=1$, sharp stability for the anisotropic Sobolev inequality on $BV(\R^n)$, as well as for the associated anisotropic $1$-log-Sobolev family, was established in \cite{FMP}.
In contrast, the finite-$p$ anisotropic gauge-power stability result for signed functions stated in Corollary~\ref{cor:gauge_power_signed} does not seem to appear in the literature in this sharp form.
\end{rmk}

\subsection{Proof of \Cref{main}}

We will prove \Cref{main} in two steps. First, as stated in the following theorem, we prove that the two functions are close to each other (this is the equivalent of finding an upper bound on $|A\triangle B|$ in the context of the Brunn-Minkowski inequality). 

\begin{thm}\label{symmetric_diff}
    Let $n \in \mathbb{N}$, $\lambda \in (0,1/2]$, and $p \in (-1/n, \infty)$. Let $f,g, h\colon \mathbb{R}^n\rightarrow \mathbb{R}_{\geq0}$ be integrable functions such that \begin{itemize}
        \item $\int_{\mathbb{R}^n} f\,dx=  \int_{\mathbb{R}^n} g\,dx$,
        \item $h(\lambda x +(1-\lambda)y) \geq M_{\lambda,p}(f(x),g(y))$ for all $x,y \in \mathbb{R}^n$, and
        \item $\int_{\mathbb{R}^n}h\,dx = (1+\delta)\int_{\mathbb{R}^n} f\,dx$  for some $\delta \geq 0$.
    \end{itemize}  
    Then, up to a translation, 
    $$\int_{\mathbb{R}^n} |f-g|\,dx=O_{n,\lambda, p} \left(\sqrt\delta\right)\int_{\mathbb{R}^n} f\,dx.$$ 
\end{thm}
Second, we prove the following linear stability result in the case where $f=g$, showing that $f$ must be linearly close to a $p$-concave function.
\begin{thm}\label{linear}
    Let $n \in \mathbb{N}$, $\lambda \in (0,1/2]$, and $p \in (-1/n, \infty)$. Let $f, h\colon \mathbb{R}^n\rightarrow \mathbb{R}_{\geq0}$ be integrable functions such that 
    \begin{itemize}
        \item $h(\lambda x +(1-\lambda)y) \geq M_{\lambda,p}(f(x),f(y))$ for all $x,y \in \mathbb{R}^n$, and
        \item $\int_{\mathbb{R}^n}h\,dx = (1+\delta)\int_{\mathbb{R}^n} f\,dx$  for some $\delta \geq 0$.
    \end{itemize}
    Then there exists a $p$-concave function $\ell \colon \mathbb{R}^n\rightarrow \mathbb{R}_{\geq 0}$ such that
    $$\int_{\mathbb{R}^n} |f-\ell|\,dx=O_{n,\lambda, p} (\delta)\int_{\mathbb{R}^n} f\,dx.$$
\end{thm}
\begin{rmk}\label{rem:no pconcave hull}
This second result can be seen as a functional extension of the linear results in  \cite{figalli2015stability,figalli2017quantitative,van2020sharp}, which were proved in the run-up to the recently established stability of the Brunn-Minkowski inequality. More precisely, as shown in \cite{van2020sharp}, given $A\subset\mathbb{R}^n$ with $|\frac{A+A}{2}\setminus A|$ small, we have
$$|\co(A)\setminus A|=O_{n}\left(\left|\frac{A+A}{2}\setminus A\right|\right).$$
However, contrary to the case of sets (where $\co(A)$ contains $A$), in 
 \Cref{linear} we cannot additionally demand that $\ell\geq f$ everywhere. This can be seen by considering $f=\textbf{1}_{[0,1]}+\delta^{100/\lambda}\textbf{1}_{[v,v+1]}$ for some large number $v\in\mathbb{R}$. One can note that, for any $p$-concave function $\ell$ with $\ell\geq f$ a.e., $\int \ell\to\infty$ as $v\to\infty$. On the other hand, if we let $h=\textbf{1}_{[0,1]}+\delta^{100(1-\lambda)/\lambda}\textbf{1}_{[(1-\lambda)v,(1-\lambda)v+1]}+\delta^{100}\textbf{1}_{[\lambda v,\lambda v+1]}+\delta^{100/\lambda}\textbf{1}_{[v,v+1]}$, then $$h(\lambda x+(1-\lambda)y)\geq M_{\lambda,0}(f(x),f(y))\geq M_{\lambda,p}(f(x),f(y))\qquad \forall\,p\leq 0$$ and $\int h\,dx\leq (1+\delta)\int f\,dx$, so our theorem applies.
\end{rmk}

We now show how Theorems~\ref{symmetric_diff} and ~\ref{linear} directly imply our main result.
\begin{proof}[Proof of \Cref{main}]
By \Cref{symmetric_diff} we have that, up to a translation, $\int |f-g|\,dx\leq O_{n,\lambda,p}\left(\sqrt{\delta}\right)\int f\,dx$. Thus, if we define $k:=\min\{f,g\}$, it follows that
$$
\int k\,dx\geq \Big(1- O_{n,\lambda,p}\left(\sqrt{\delta}\right) \Big)\int f\,dx \qquad \text{and}\qquad h(\lambda x+(1-\lambda)y)\geq \M(k(x),k(y)).
$$
In particular, we have $\int h\, dx\leq \left(1+O_{n,\lambda,p}\left(\sqrt{\delta}\right)\right) \int k\, dx$
Hence, applying \Cref{linear} to $k$ and $h$, we find a $p$-concave function $\ell$ such that $\int |\ell-k|\,dx=O_{n,\lambda,p}\left(\sqrt{\delta}\right)\int k\,dx$. Combining these two results yield
$$\int \big(|f-\ell|+|g-\ell|\big)\,dx\leq \int \big(|f-k|+|k-\ell|+|g-k|+|k-\ell|\big)\,dx=O_{n,\lambda,p}\left(\sqrt{\delta}\right)\int f\,dx,$$
as desired.
\end{proof}
Because of this discussion, the main goal of this paper will be to prove 
Theorems~\ref{symmetric_diff} and ~\ref{linear}.

\subsection{Structure of the paper}
After introducing some notation, Sections~\ref{sect:structure thm 1}--\ref{sect:proof thm1} are devoted to the proof of \Cref{symmetric_diff}. More precisely, after describing the structure of its proof, we first prove the result in one dimension, then in two dimensions, and finally in the general case. Sections~\ref{sect:structure thm2}--\ref{sect:proof thm2} are devoted to the proof of \Cref{linear}. Section \ref{sect:sob} is devoted to proving \Cref{thm:LSI}. The final Sections \ref{sharplambdaoverviewsec}--\ref{sharplambdaproofsec} are devoted to the proof of \Cref{sharplambda}. To better navigate through the paper, we invite the reader to look at the table of contents on the first page.

\subsection{AI Declaration}
All mathematical ideas, arguments, and proofs in this paper were developed entirely by the authors, without AI assistance. We used AI only to help us track the dependence of the constant on $n$ and $p$ in \Cref{sharplambda}, and to assist with the prose and exposition in Sections~\ref{sect:sob}--\ref{sharplambdaproofsec}.

\section{Notation}
Here and in the following, $p \in (-1/n,\infty)$ and $\lambda \in (0,1/2].$
\begin{defn}
    For a function $f \colon \mathbb{R}^n \rightarrow \mathbb{R}_{\geq 0}$, we denote by $\co_p(f) \colon \mathbb{R}^n \rightarrow \mathbb{R}_{\geq 0}$ the minimal $p$-concave function such that $\co_p(f)\geq f$ everywhere in $\mathbb{R}^n$.
\end{defn}
Note that a priori such a function might not exist for a general $f$. We shall only consider $\co_p(f)$ with $f$ for which it does.
\begin{defn}\label{defn_supconv}
Given functions $f,g\colon\mathbb{R}^n\to\mathbb{R}_{\geq0}$ define their $(\lambda,p)-\sup$-convolution $\M^*(f,g)$ by
$$\M^*(f,g):\mathbb{R}^n\to\mathbb{R}_{\geq 0}, \quad z\mapsto \sup_{x,y:\lambda x+(1-\lambda)y=z}\M(f(x),g(y)).$$
\end{defn}

\begin{defn}\label{defn_cone}
A convex set $C\subset \mathbb{R}^n$ is called a \emph{cone} if there exists a hyperplane $H$ not containing the origin and a bounded convex set $P \subset H$ such that $$C=\bigsqcup_{t \geq 0} tP.$$
\end{defn}

Note that a cone always has its vertex at the origin, which we will denote by $o$.

\begin{defn}
Let $S_n$ denote a regular simplex in $\mathbb{R}^n$. Denote by $F_n^0, \dots, F_n^n$ the faces of $S_n$. Assume the distance from $o$ to  $F_n^0, \dots, F_n^n$ is $1$. Denote $H_n^0, \dots, H_n^n$ the supporting hyperplanes. Finally, construct the half-spaces $H_n^{i,-}, H_n^{i,+}$, where the former contains $o$.
\end{defn}

We shall assume that the simplex $S_n$ and the basis $e_1, \dots, e_n$ of $\mathbb{R}^n$ are such that $e_1\perp H_n^1$. 

\begin{defn}\label{Cnidefn}
Construct the cones $C_n^i=\sqcup_{t \geq 0} t F_n^i$ and let $\mathfrak{C}_n=\{C_n^0,\dots, C_n^n\}$.
\end{defn}

Note that $\mathfrak{C}_n$ forms an essential partition of $\mathbb{R}^n$.

\section{\Cref{symmetric_diff}: structure of the proof}
\label{sect:structure thm 1}

The proof of \Cref{symmetric_diff} consists of three parts. More precisely, after collecting in the next section a series of preliminary results, we first prove the 1-dimensional version. Then we use this result to prove the 2-dimensional case, and finally we show that the 2-dimensional result implies the $n$-dimensional case.
\subsection{Overview of the 1-dimensional proof}
The initial reduction for this part of the proof is included in \Cref{ndimensionalsymdiffsect}. The specific reduction is then in  \Cref{1Dsymdifsection}.

\begin{enumerate}
\item We define level sets $F_s:=\{x\in\mathbb{R}: f(x)>s\}$ (and $G_s$ and $H_s$ analogously).
\item Then $F_0$ is the support of $f$ and $\int_\mathbb{R} f(x)\,dx=\int_0^\infty |F_s|\,ds$ (analogously for $g$ and $h$)
\item We use the inclusion $\lambda F_s+(1-\lambda)G_t\subset H_{u}$ for $M_{\lambda,p}(s,t)=u$ with properly chosen $s$ and $t$.
\item Using this lower bound on $|H_u|$, we find that typically $|F_s|\approx|G_t|$ and $\lambda|F_s|+(1-\lambda)|G_t|\approx |H_{u}|$ for these $s$ and $t$.
\item Combining this information with the stability of the 1-dimensional Brunn-Minkowski inequality, we deduce that $|\co(F_s)\setminus F_s|$ and $|\co(G_t)\setminus G_t|$ are typically small.
\item Up to changing $f$ and $g$ little, we can assume that all the level sets are nearly intervals.
\item Considering the transport map $T$ from $f$ to $g$ and using the bound $h(\lambda x+(1-\lambda)y)\geq \M(f(x),g(y))$, we find that typically $\frac{dT}{dx}(x)=\frac{f(x)}{g(T(x))}$ is close to $1$. 
\item In particular, changing $f$ and $g$ only little, we find new functions $f_1$ and $g_1$ so that $\frac{dT}{dx}$ is almost everywhere close to one. 
\item Since all of these steps may have caused gaps in the support of $f_1$ and $g_1$, we next push the support of these functions together to remove the gaps to create $f_2$ and $g_2$.
\item Using that the level sets were near intervals, we show that $f_1$ and $g_1$ differ little from $f_2$ and $g_2$. 
\item Next, we fill in the level sets so that they become intervals, to find function $f_3$ and $g_3$ changing little from $f_2$ and $g_2$.
\item Finally we define the auxiliary function $G(x):=g_3(T_2(x))$, where $T_2$ is the transport map pushing $f_2$ onto $g_2$, and we control the differences $|f_3-G|$ and $|g_3-G|$.
\end{enumerate}

\subsection{Overview of the 2-dimensional proof}
The initial reduction for this part of the proof is included in \Cref{ndimensionalsymdiffsect}. The specific reduction is then in \Cref{2DSymDifSec}.
First we establish the result on well-behaved tubes.
More precisely, assume first that $f$ and $g$ are defined on parallel tubes that are infinite in one direction, and that $f$ and $g$ are large near the base. Then the result follows as described in the next six steps.
\begin{enumerate}
    \item Let $f:[0,r]\times[0,\infty)\to[0,1]$ and $g:[0,r']\times[0,\infty)\to[0,1]$,
    and assume that for all $x\in[0,r]\times[0,1]$ we have $f(x)\geq 0.1$, and similarly for $x\in[0,r']\times[0,1]$ we have $g(x)\geq 0.1$. 
    \item Consider the map $T:[0,r]\to[0,r']$ so that $T(0)=0$ and $\frac{dT}{dx}(x)=\frac{\int f(x,y)\,dy}{\int g(T(x),y)\,dy}$.
    \item Using the 1-dimensional Borell-Brascamp-Lieb inequality, we find that 
    $$\int h(\lambda x+(1-\lambda)T(x),y)\,dy\geq M_{\lambda,q}\left(\int f(x,y)\,dy,\int g(T(x),y)\,dy\right),$$
    for some $q>-1$ depending on $p>-1/2$. Integrating over $x$, this implies
    $$\int h(x,y)\,dxdy\geq \int f(x,y)\,dxdy.$$
    \item Using the sharp stability in the one-dimensional case twice (in both of the two inequalities above), we find that for all $x\in[0,r]$ there is some translation $v_x$ so that
    $$\int |f(x,y)-g(T(x),y-v_x)|\,dy$$
    is small. Moreover, $\frac{dT}{dx}$ is typically close to 1, and thus $T(x)$ is close to $x$ (recall that $T(0)=0$).
    \item Since the functions are ``big'' (say $f,g \geq 0.1$) close to the base $x=0$, we find that $v_x$ must be small for all $x$.
    \item Finally, we use that the level sets of $f$ and $g$ are nearly convex, so that shifting the functions by the small translations $T(x)-x$ and $v_x$ doesn't affect the symmetric difference much, so that indeed $\int |f-g|\,dx$ is small.
\end{enumerate}

The next step is to use the previous result for suitable restrictions of $f$ and $g$ to tubes. To this aim, we need to make sure that the corresponding tubes are close together. We can do this well enough to find small symmetric difference in the region where $f$ and $g$ are not smaller than say $0.0001$.
\begin{enumerate}
    \item After affinely transforming the domain, we find that in a disk  $D$ of constant radius around the origin both $f$ and $g$ are mostly at least $0.1$. 
    \item We can partition $\mathbb{R}^2$ into 3 cones $C$ at the origin so that, inside each of them, $\int_C f\,dx=\int_C g\,dx\geq 0.1\int f\,dx$.
    \item Consider a tube  (i.e., an affine image of $[0,1]\times[0,\infty)$) whose base (i.e., the image of $[0,1]\times\{0\}$) is contained in $C\cap\frac12D$.
    \item There is a corresponding parallel tube $R'$ so that $\int_R f\,dx=\int_{R'}g\,dx$ and $\int_{\lambda R+(1-\lambda)R'} h\,dx\leq (1+O(\delta))\int_R f\,dx$. Moreover, $R$ and $R'$ are close together, so that $\int_{R\triangle R'}(f+g)\,dx$ is small.
    \item Using the above result in tubes, we find that $\int_R |f-g|\,dx$ is small.
    \item As we can choose the width of the tube $R$ on the scale of $D$, we can cover $100D$ with say $1000$ different tubes $R$, so that $\int_{100D}|f-g|\,dx$ is small.
\end{enumerate}

Finally, we want to repeatedly trim the functions $f$ and $g$ so that the different instances of $100D$ cover everything.
\begin{enumerate}
    \item Consider level set $F_t:=\{x\in\mathbb{R}^2: f(x)>t\}$ (and the analogous $G_t$ and $H_t$). By a transport approach leveraging the stability of the Brunn-Minkowski inequality \cite{BMStab} we find that these level sets are nearly convex and of comparable size and shape.
    \item We additionally find that they cannot change size to rapidly, so that $F_{0.1t}\subset 100F_t$.
    \item One consequence is that not too much mass can be in low level sets, i.e., $\int_{\mathbb{R}^2\setminus F_{\delta^{100}}}f\,dx$ is small, so we can restrict our attention to $f\cdot \textbf{1}_{F_{\delta^{100}}}$.
    \item Consider the truncated functions $f^i:=\min\{f, 2^{-i}\}$ (and corresponding $g^i$) for $i=1,\dots, \log_2(\delta^{100})$.
    \item By the previous result controlling the symmetric difference for not too small $f$ and $g$, we find that there exist translates $v_i$ so that $\int_{F_{0.001\cdot 2^{-i}}}|f_i(x)-g_i(x+v_i)|\,dx$ is small for all $i$.
    \item Since there is a serious overlap between these integrals for consecutive $i$'s, we find that the difference between $v_i$ and $v_{i+1}$ must be small.
    \item By a careful analysis of the interaction between the different $i$'s we find that, in fact, $\sum_i \int_{F_{0.001\cdot 2^{-i}}}|f_i(x)-g_i(x)|\,dx$ is small.
    \item This sum provides an upper bound for $\int |f-g|\,dx$ so that we can conclude.
\end{enumerate}

\subsection{Overview of the $n$-dimensional proof}

The initial reduction for this part of the proof is included in \Cref{ndimensionalsymdiffsect}. The specific reduction is then in \Cref{sect:proof thm1}.

For the reduction from $n$ dimensions to 2 dimensions, we use the following steps.
\begin{enumerate}
    \item We first use a transport approach through the level sets, leveraging the sharp stability of the Brunn-Minkowski inequality \cite{BMStab}, to show that the level sets $F_t:=\{x\in\mathbb{R}^n: f(x)>t\}$ (and the analogous $G_t$ and $H_t$) are nearly convex and of similar size.
    \item After affinely transforming the domain, we find that in a ball around the origin both $f$ and $g$ are mostly at least $0.1$ (cf. \Cref{symmetric_diff_technical}).
    \item We then partition $\mathbb{R}^n$ into $n+1$ reasonable looking cones $K$ at the origin, so that $\int_K f\,dx=\int_K g\,dx\geq \Omega_n\left(\int f\,dx\right)$ (cf. \Cref{symmetric_diff_reformulation_technical}).
    \item Since $f$ and $g$ are $\geq 0.1$ on {most of} a ball around the origin, if we prove the result in each of these cones, the translation in each of the cones can be taken to be zero.
    \item Using a technical result from \cite{BMStab}, we can further partition $\mathbb{R}^n$ into smaller (very very small) cones (maintaining the previous properties) with the additional property that in all but one direction the cone is so narrow that the function is nearly constant on fibers in those directions (cf. \Cref{symmetric_diff_narrow_n_cone_alt}).
    \item Given a small cone $C$ provided by the previous step, let $C^{z,w}$ be those $(n-2)$-dimensional fibers so that $\max_{x,y\in C^{z,w}}|f(x)-f(y)|+|g(x)-g(y)|$ is small (here the smallness can be chosen as small as desired). 
    \item If we let $F(z,w)= |C^{z,w}|\cdot \min_{x\in C^{z,w}}f(x)$ (and $G(z,w), H(z,w)$ analogously), then $$\int_{\mathbb{R}^2} F(z,w)\,dzdw\approx \int_C f\,dx  \text{ and }\int_{\mathbb{R}^2} |F(z,w)-G(z,w)|\,dzdw\approx \int_C |f-g|\,dx$$ (cf. \Cref{symmetric_diff_constant_on_fibers_n_cone}).
    \item Since the cone is convex (so $(z,w)\mapsto |C^{z,w}|$ is $1/(n-2)$-concave), we can use H\"older's inequality to find the following: For every $p>-1/n$ there is a $q>-1/2$ so that if $\M(f(x),g(y))\geq h(\lambda x+(1-\lambda)y)$ for some $x\in C^{z,w}$ and $y\in C^{z',w'}$, then $M_{\lambda, q}(F(z,w),G(z',w'))\geq H(\lambda z+(1-\lambda)z',\lambda w+(1-\lambda)w')$ (cf. \Cref{symmetric_diff_2D}).
    \item We finally apply the two-dimensional theorem to the functions $F,G,H$ to prove the result in each of the cones, and thus in general.
\end{enumerate}

\section{Preliminary results for the proof of \Cref{symmetric_diff}}\label{ndimensionalsymdiffsect}
In this section we show how \Cref{symmetric_diff} is implied by the following proposition. Most of the weight of that reduction is in \Cref{almostconvexlevelsets} and its more easily applicable \Cref{aclcor}.

\begin{prop}\label{symmetric_diff_technical}
    Given $n \in \mathbb{N}$, $\lambda \in (0,1/2]$, $p \in (-1/n, \infty)$ and $r\in (0, \infty)$ there exists $d=d_{n, \lambda, p, r}$ such that the following holds. Let $f,g, h\colon \mathbb{R}^n\rightarrow [0,1]$ be continuous functions with bounded support and $\beta\in (0,1)$ such that  
    \begin{itemize}
        \item $\int_{\mathbb{R}^n} f\,dx=  \int_{\mathbb{R}^n} g\,dx=1$,
        \item for every $i =0,\dots,n$,  $\int_{C_n^i} f\,dx=  \int_{C_n^i} g\,dx$, 
        \item $|\{x\in rS_n: f(x)<\beta \text{ or } g(x)<\beta\}|=O_{n,\lambda,p,r}(\delta)$
        and 
        \item for all $x,y \in \mathbb{R}^n$ we have $h(\lambda x +(1-\lambda)y) \geq M_{\lambda,p}(f(x),g(y))$.
    \end{itemize}
    If $\int_{\mathbb{R}^n}h\,dx = 1+\delta$ for some $0\leq \delta \leq d$, then $\int_{\mathbb{R}^n} |f-g|\,dx=O_{n,\lambda, p, r,\beta} \left(\sqrt{\delta}\right)$. 
\end{prop}

First, we show that we may assume that $f,g,h$ are continuous.

\begin{lem}\label{continuousmaker}
Given $f,g,h\colon\mathbb{R}^n\to\mathbb{R}_{\geq0}$ integrable functions with $\int_{\mathbb{R}^n} f\,dx=\int_{\mathbb{R}^n} g\,dx$ and $h(\lambda x+(1-\lambda)y)\geq \M(f(x),g(y))$ for all $x,y\in\mathbb{R}^n$, then for any $\eta>0$ there exist functions $f',g',h'\colon\mathbb{R}^n\to\mathbb{R}_{\geq0}$ continuous functions with $h'(\lambda x+(1-\lambda)y)\geq \M(f'(x),g'(y))$ and $\int_{\mathbb{R}^n} \left(|f-f'|+|g-g'|+|h-h'|\right)\,dx< \eta$.
\end{lem}

Next we show that $f,g,h$ take values in $[0,1]$ and have bounded domains.

\begin{lem}\label{BigFirstCleanup}
Let $f,g, h\colon \mathbb{R}^n\rightarrow \mathbb{R}_{\geq0}$ be continuous functions such that  \begin{itemize}
        \item $\int_{\mathbb{R}^n} f\,dx=  \int_{\mathbb{R}^n} g\,dx=1$,
        \item for all $x,y \in \mathbb{R}^n$ we have $h(\lambda x +(1-\lambda)y) \geq M_{\lambda,p}(f(x),g(y))$, and 
        \item $\int_{\mathbb{R}^n}h\,dx =1+\delta$ for some $0\leq \delta \leq d$.
    \end{itemize} 
    Then there exist continuous functions $f',g',h'\colon\mathbb{R}^n\rightarrow [0,1]$ and an affine transformation $A:\mathbb{R}^n\to\mathbb{R}^n$ with the following properties. 
    \begin{itemize}
    \item the support of $f',g'$, and $h'$ is bounded,
    \item $\int_{\mathbb{R}^n} f'\,dx=  \int_{\mathbb{R}^n} g'\,dx=1$, 
    \item for all $x,y \in \mathbb{R}^n$ we have $h'(\lambda x +(1-\lambda)y) \geq M_{\lambda,p}(f'(x),g'(y))$,
    \item $\int_{\mathbb{R}^n}\left(|f(x)-|\det(A)|^{-1}f'(Ax)|+|g(x)-|\det(A)|^{-1}g'(Ax)|\right)\,dx\leq \delta^{100}$.
    \end{itemize}
\end{lem}

We can find translates of the cones from \Cref{Cnidefn} which evenly partion $f$ and $g$.
\begin{lem}\label{parallelconeslem}
Let $f, g \colon \mathbb{R}^n\rightarrow [0,1]$ be continuous functions with bounded support and let $\alpha_0,\dots,\alpha_n\geq0$ so that $\sum \alpha_i=1$.

Then, there exist $a, b \in \mathbb{R}^n$ such that for every $i \in\{0,\dots,n\}$  
$$\int_{a+C_n^i} f\,dx=  \int_{b+C_n^i} g\,dx=\alpha_i\int_{\mathbb{R}^n} f\,dx.$$
In particular, this result holds for $\alpha_i=\frac{1}{n+1}$.
\end{lem}

Using the stability results for the Brunn-Minkowski inequality from \cite{BMStab,OTBMStab}, we derive in \Cref{almostconvexlevelsets} that the level sets of $f$ and $g$ must be almost convex and almost the same. Moreover, we find that most of $f$ and $g$ is contained in a level set that is not too low. We use the following fact about the averages $\M$.

\begin{lem}\label{Holderderivative}
For $p\in (-1,0),\lambda\in (0,1), t\in \mathbb{R}_{\geq 0}$ and differentiable map $T\colon\mathbb{R}_{\geq0}\to\mathbb{R}_{\geq0}$, we have
$$\frac{d}{dt} M_{\lambda,p}(t,T(t))\geq \frac{1}{M_{\lambda,-p}\left(1,\frac{1}{\frac{dT}{dt}}\right)}$$
\end{lem}

\begin{lem}\label{almostconvexlevelsets}
 Given $n \in \mathbb{N}$, $\lambda \in (0,1/2]$, $p \in (-1/n, \infty)$, $\alpha,\alpha'>0$ there exists $d=d_{n, \lambda, p,\alpha'}>0$ and $\beta=\beta_{n,\lambda,p,\alpha'}>0$ such that the following holds. Let $f,g, h\colon \mathbb{R}^n\rightarrow [0,1]$ be continuous functions with bounded support such that 
 \begin{itemize}
     \item $\int_{\mathbb{R}^n} f\,dx=  \int_{\mathbb{R}^n} g\,dx=1$,
     \item for all $x,y \in \mathbb{R}^n$ we have $h(\lambda x +(1-\lambda)y) \geq M_{\lambda,p}(f(x),g(y))$,
     \item $\int_{\mathbb{R}^n}h\,dx = (1+\delta)$ for some $\delta\leq d$.
 \end{itemize}
 Let $F_t:=\{x\in\mathbb{R}^n: f(x)>t\}$, $G_t:=\{x\in\mathbb{R}^n: g(x)>t\}$, and $H_t:=\{x\in\mathbb{R}^n: h(x)>t\}$. Let $T:[0,1]\to[0,1]$ be the transport map with $\frac{dT}{dt}(t)=\frac{|F_t|}{|G_{T(t)}|}$, i.e., so that $\int_{0}^t|F_s|\,ds=\int_{0}^{T(t)}|G_s|\,ds$.
Define the following subsets of $[0,1]$ depending on parameter $\alpha>0$;
 \begin{align*}
 I_1&:=\left\{t\in[0,1]: \frac{dT}{dt}(t)\not\in[1-\alpha,1+\alpha]\right\}\\
 I_2&:=\{t\in[0,1]: |\co(F_t)\setminus F_t|\geq \alpha |F_t| \text{ or }|\co(G_{T(t)})\setminus G_{T(t)}|\geq \alpha |G_{T(t)}|\}\\
 I_3&:=\{t\in[0,1]:|\lambda F_t+(1-\lambda)G_{T(t)}|\geq (1+\alpha) M_{\lambda,1/n}(|F_t|,|G_{T(t)}|)\}\\
 I_4&:=\{t\in[0,1]:\forall x\in\mathbb{R}^n, |(x+F_t)\triangle H_{M_{\lambda,p}(t,T(t))}|\geq \alpha |F_t|\}\\
 I_5&:=\{t\in[0,1]:\forall x\in\mathbb{R}^n, |(x+\co(F_t))\triangle \co(G_{T(t)})|\geq \alpha |F_t|\}
 \end{align*}
 and let $I:=\bigcup I_i$.
 Then
 $$\int_{I}|F_t|dt=\int_{T(I)}|G_t|\,dt=O_{n,p,\lambda,\alpha}(\delta),$$ 
 $$\int_{[0,1]\setminus I} \left(|\co(F_t)\setminus F_t|+|\co(G_{T(t)})\setminus G_{T(t)}|\right)\,dt=O_{n,p,\lambda}(\delta).$$
 Moreover, 
 $$\int_{F_{\beta\sup f}}f\,dx,\int_{G_{\beta\sup g}}g\,dx\geq 1-\alpha'.$$
 
\end{lem}

\begin{cor}\label{aclcor}
Given $f,g,h$ as in \Cref{almostconvexlevelsets}, then for every $\alpha>0$ we can find $f',g'\colon\mathbb{R}^n\to[0,1]$ so that 
\begin{itemize}
    \item $f'\leq f$ and $g'\leq g$,
    \item $\int f'=\int g'$,
    \item $\int \left(|f-f'|+|g-g'|\right)\leq O_{n,\lambda,p,\alpha}(\delta)$, and
    \item $\int h\leq (1+O_{n,\lambda,p,\alpha}(\delta))\int f'$.
\end{itemize}    
Let $F'_t:=\{x\in\mathbb{R}^n: f'(x)>t\}$, $G'_t:=\{x\in\mathbb{R}^n: g'(x)>t\}$, and $H_t:=\{x\in\mathbb{R}^n: h(x)>t\}$, the level sets of $f',g',h'$ respectively. Let $T:[0,1]\to[0,1]$ be the transport map with $\frac{d}{d t}T(t)=\frac{|F'_t|}{|G'_{T(t)}|}$, i.e., so that $\int_{0}^t|F'_s|\,ds=\int_{0}^{T(t)}|G'_s|\,ds$. Then for all $t\in[0,1]$ we have
\begin{itemize}
\item $ \frac{d}{dt}T(t)\in[1-\alpha,1+\alpha]$,
\item $|\co(F'_t)\setminus F'_t|\leq \alpha |F'_t|$ and $|\co(G'_{t})\setminus G'_{t}|\leq \alpha |G'_{t}|$,
\item $|\lambda F'_t+(1-\lambda)G'_{T(t)}|\leq (1+\alpha) M_{\lambda,1/n}(|F'_t|,|G'_{T(t)}|)$, 
\item $\exists x_t\in\mathbb{R}^n, |(x_t+F'_t)\triangle H'_{M_{\lambda,p}(t,T(t))}|\leq \alpha |F'_t|$, and
\item $\exists x_t\in\mathbb{R}^n, |\co(x_t+F'_t)\triangle \co(G'_{T(t)})|\leq \alpha |F'_t|$.
\item $\int_{F'_{\beta\sup f}}f\,dx,\int_{G'_{\beta\sup g}}g\,dx\geq 1-\alpha'.$
\end{itemize}

\end{cor}

One example of a quick consequence of \Cref{almostconvexlevelsets} that we will use in the proof in two dimensions is the following result for functions which are rarely small. As in \Cref{almostconvexlevelsets}, $\beta$ can be thought of as $\Omega_{n\lambda,p}(1)$.

\begin{lem}\label{SimilarSupportLem}
Given $f,g,h:\mathbb{R}^n\to \{0\}\cup[\beta,1]$ and $\delta>0$ sufficiently small with respect to $n$, $p$, and $\lambda$ with 
\begin{itemize}
    \item $h=\M^*(f,g)$
    \item $\int f\,dx=\int g\,dx=1$
    \item $\int h\,dx\leq 1+\delta$
\end{itemize}
Then (up to tranlastion) $|supp(f)\triangle supp(g)|\leq O_{n,\lambda,p,\beta}\left(\sqrt{\delta}\right)$.

Moreover,if functions $f,g,h:\mathbb{R}^n\to[0,1]$ don't just take values in $\{0\}\cup[\beta,1]$, but we do have  $\int_{\{x: f(x)<\beta \text{ or } g(x)<\beta\}}(f+g) dx\leq \gamma$, then (up to tranlastion) $|supp(f)\triangle supp(g)|\leq O_{n,\lambda,p,\beta}\left(\sqrt{\delta+\gamma}\right)$.
\end{lem}

Finally, we record that for functions as in \Cref{aclcor}, slight translations don't result in big symmetric difference.

\begin{lem}\label{smalltranslationsmallsymdiflem}
Let $f\colon\mathbb{R}^n\to[0,1]$ with $\int f\,dx=1$ and level sets $F_t:=\{x\colon f(x)>t\}$, so that $\int_{0}^1 |\co(F_t)\setminus F_t|\,dt\leq \delta$. If for some $t_0, \eta>0$, we have $o,v\in \eta\co(F_{t_0})$ and $|\co(F_{t_0})|\leq 2|F_{t_0}|$, then
$$\int (f(x)-f(x+v))\,dx= O_{n, t_0}(\delta+\eta).$$
\end{lem}

\subsection{Proof of Lemmas}

\begin{proof}[Proof of \Cref{BigFirstCleanup}]
Without loss of generality, we may assume $\sup f\geq \sup g$. Let $\ell:=(\max\{\sup f,\sup g\})^{-1/n}$ and let $$f_1:\mathbb{R}^n\to [0,1], x\mapsto \ell^nf(\ell^{-1}x), g_1:\mathbb{R}^n\to [0,1], x\mapsto \ell^ng(\ell^{-1}x)\text{ and }h_1:\mathbb{R}^n\to [0,1], x\mapsto \ell^nh(\ell^{-1}x).$$ 
To reduce to sets with bounded support, note that for sufficiently large $r=r_{f,\delta}$ in terms of $f$ and $\delta$, we have that $\int_{[-r,r]^n}f_1\,dx= 1-\delta^{1000}$. Similarly, we can find $r'=r'_{g_1,\delta}$, so that $\int_{[-r',r']^n}g\,dx= 1-\delta^{1000}$. Hence, let $f':=f_1\cdot \textbf{1}[-r,r]^n$ and $g':=g_1\cdot \textbf{1}[-r',r']^n$. Let $h'=h_1$, so that
$$\int h'\,dx=\int h_1\,dx=\int h\,dx\leq (1+\delta)\int f\,dx\leq \frac{1+\delta}{1-\delta^{1000}}\int f'\,dx\leq (1+2\delta)\int f'\,dx.$$
\end{proof}

Before we prove \Cref{parallelconeslem}, recall the Knaster-Kuratowski-Mazurkiewicz lemma \cite{knaster1929beweis}, which is a consequence of Sperner's lemma.
\begin{lem}[KKM Lemma]
Consider a simplex $\co(x_0,\dots x_n)\subset\mathbb{R}^n$ and closed sets $Y_0,\dots, Y_n$ so that for all subsets $I\subset\{0,\dots,n\}$ we have $\co(x_i: i\in I)\subset\bigcup_{i\in I}Y_i$. Then $\bigcap_{i=0}^n Y_i\neq \emptyset.$
\end{lem}

\begin{proof}[Proof of \Cref{parallelconeslem}]
We show that a translate $a$ exists for $f$, the result for $b$ follows analogously. Define the closed sets $Y_i\subset\mathbb{R}^n$ as follows, let $y\in Y_i$ if
$$\int_{C^i_n +y}f\,dx\geq \alpha_i\int_{\mathbb{R}^{n}}f\,dx.$$
Note that for a given $y\in\mathbb{R}^n$ we can't have $\int_{C^i_n +y}f\,dx<\alpha_i\int_{\mathbb{R}^{n}}f\,dx$ for all $i=0,\dots, n$ as that would imply $\int_{\mathbb{R}^n +y}f\,dx<\int_{\mathbb{R}^n}f\,dx$. Hence, $\bigcup_{i=0}^n Y_i=\mathbb{R}^n$.

To apply the KKM lemma, we use the simplex $\ell S_n$ for a sufficiently large $\ell$. Indeed, since $f$ has bounded support, we can choose $\ell$ so large that $supp (f)\subset \ell S_n$. For that $\ell$ we find that if $y\in \ell F^i_n$, then $\int_{C_n^i+y}g\,dx=0$ as $C_n^i+y$ falls completely outside of the support of $g$. Hence, $Y_i\cap \ell F^i_n=\emptyset$, which can be seen to imply the condition in the KKM lemma. We can thus find a point $a\in \bigcap_{i=0}^n Y_i$, i.e., so that 
$$\int_{C^i_n +a}f\,dx\geq \alpha_i\int_{\mathbb{R}^{n}}f\,dx,$$
for all $i=0,\dots, n$. This in fact implies
$$\int_{C^i_n +a}f\,dx=\alpha_i\int_{\mathbb{R}^{n}}f\,dx,$$
for all $i=0,\dots,n$, which concludes the lemma.
\end{proof}

\begin{proof}[Proof of \Cref{Holderderivative}]
 Simply expanding we find
 $$\frac{d}{dt} M_{\lambda,p}(t,T(t))=\frac{d}{dt} \left(\lambda t^p+(1-\lambda) T(t)^p\right)^{1/p}=\left(\lambda t^p+(1-\lambda) T(t)^p\right)^{1/p-1}\left(\lambda t^{p-1}+(1-\lambda) T(t)^{p-1}\frac{dT}{dt}(t)\right).$$
On the other hand, we have
$$\frac{1}{M_{\lambda,-p}\left(1,\frac{1}{\frac{dT}{dt}(t)}\right)}=M_{\lambda,p}\left(1,\frac{dT}{dt}(t)\right)=\left(\lambda+(1-\lambda) \left(\frac{dT}{dt}(t)\right)^p\right)^{1/p}.$$
Hence, it suffices to show that 
$$\left(\lambda t^{p-1}+(1-\lambda) T(t)^{p-1}\frac{dT}{dt}(t)\right)\left(\lambda+(1-\lambda) \left(\frac{dT}{dt}(t)\right)^p\right)^{-1/p}\geq \left(\lambda t^p+(1-\lambda) T(t)^p\right)^{1-1/p},$$
or, taking both sides to the power $1/(1-1/p)=p/(p-1)$
$$\left(\lambda t^{p-1}+(1-\lambda) T(t)^{p-1}\frac{dT}{dt}(t)\right)^{p/(p-1)}\left(\lambda+(1-\lambda) \left(\frac{dT}{dt}(t)\right)^p\right)^{-1/(p-1)}\geq \left(\lambda t^p+(1-\lambda) T(t)^p\right).$$
This final inequality follows immediately by H\"olders inequality with exponents $(p-1)/p$ and $1-p$ satisfying $\frac{1}{(p-1)/p}+\frac{1}{1-p}=1$.
\end{proof}

\begin{proof}[Proof of \Cref{almostconvexlevelsets}]
For notational convenience assume $\sup g\leq \sup f=1$. Other cases follow analogously. 
Consider the level sets 
$$F_t:=\{x: f(x)>t\},\,G_t=\{x: g(x)>t\},\,H_t=\{x:h(x)>t\},$$
so that e.g. $F_0=supp(f)$. By the construction of $h$, we have inclusion
$$\lambda F_t+(1-\lambda)G_s\subset H_{\M(t,s)}.$$
Consider the transport map $T\colon[0,1]\to[0,1]$ defined by the differential equation $\frac{dT}{dt}(t)=\frac{|F_t|}{|G_{T(t)}|}$, i.e., so that for all $t\in[0,1]$, we have
$$\int_{0}^t |F_s|\,ds =\int_{0}^{T(t)} |G_s|\,ds.$$
With all of this machinery set up, we consider the following lower bound on $\int hdx$;
\begin{align*}
  \int_{\mathbb{R}^n} h\,dx&=\int_0^1 |H_s|\,ds\geq \int_0^1 \left|H_{\M(t,T(t))}\right|\,d\M(t,T(t))\geq \int_0^1 |\lambda F_t+(1-\lambda) G_{T(t)}|\frac{dt}{M_{\lambda,-p}\left(1,\frac{1}{\frac{dT}{dt}(t)}\right)}\\
  &\geq \int_0^1 \frac{|\lambda F_t+(1-\lambda) G_{T(t)}|-M_{\lambda,1/n}(|F_t|,|G_{T(t)}|)}{M_{\lambda,-p}\left(1,\frac{1}{\frac{dT}{dt}(t)}\right)}\,dt + \int_0^1 M_{\lambda,1/n}\left(|F_t|,|G_{T(t)}|\right)\frac{dt}{M_{\lambda,-p}\left(1,\frac{1}{\frac{dT}{dt}(t)}\right)},\\
\end{align*}
where in the second inequality we use \Cref{Holderderivative}. Note that by the Brunn-Minkowski inequality, we have $|\lambda F_t+(1-\lambda) G_{T(t)}|-M_{\lambda,1/n}(|F_t|,|G_{T(t)}|)\geq 0$.
Using that $p\geq -1/n$ (and thus $-p\leq 1/n$) combined with $\frac{1}{\frac{dT}{dt}(t)}=\frac{|G_{T(t)}|}{|F_t|}$, we find that
$$\int_0^1 M_{\lambda,1/n}(|F_t|,|G_{T(t)}|)\frac{dt}{M_{\lambda,-p}\left(1,\frac{1}{\frac{dT}{dt}(t)}\right)}=\int_0^1 |F_t|\frac{M_{\lambda,1/n}(1,\frac{|G_{T(t)}|}{|F_t|})}{M_{\lambda,-p}\left(1,\frac{1}{\frac{dT}{dt}(t)}\right)}\,dt\geq\int_0^1 |F_t|\,dt=1.$$
In order to show that we typically have that $\frac{dT}{dt}(t)$ is close to 1, note that for  $t\in I_1$, we have
$$\frac{M_{\lambda,1/n}(1,t)}{M_{\lambda,-p}\left(1,t\right)}\geq 1+\Omega_{\lambda,n,p,\alpha}(1),$$
so that we find 
$$\int_0^1 |F_t|\frac{M_{\lambda,1/n}(1,\frac{|G_{T(t)}|}{|F_t|})}{M_{\lambda,-p}\left(1,\frac{1}{\frac{dT}{dt}(t)}\right)}\,dt\geq 1+\Omega_{\lambda,n,p,\alpha}\left(\int_{I_1}|F_t|dt\right),$$
and using the bound on $\int_{\mathbb{R}^n} h\,dx$, this implies that little mass is in the levels in $I_1$, in particular,
$$\int_{I_1}|F_t|\,dt=O_{\lambda,n,p,\alpha}(\delta).$$
Outside of $I_1$, we have $\frac{|F_t|}{|G_{T(t)}|}\in[1-\alpha,1+\alpha]$, so that we can use the stability of Brunn-Minkowski \cite{BMStab} to find
$$\left|\lambda F_t+(1-\lambda) G_{T(t)}\right|\geq M_{\lambda,1/n}(|F_t|,|G_{T(t)}|)+\Omega_{n,\lambda,\alpha}(\min\{|\co(F_t)\setminus F_t|, |F_t|\}).$$
Hence, the previous lower bound on $\int_{\mathbb{R}^n}h\,dx$ combined with the note that $M_{\lambda,-p}\left(1,\frac{1}{\frac{dT}{dt}(t)}\right)\in[1-\alpha,1+\alpha]$ implies
$$\delta\geq \int_{[0,1]\setminus I_1} \frac{|\lambda F_t+(1-\lambda) G_{T(t)}|-M_{\lambda,1/n}(|F_t|,|G_{T(t)}|)}{M_{\lambda,-p}\left(1,\frac{1}{\frac{dT}{dt}(t)}\right)}\,dt\geq \Omega_{n,\lambda,\alpha}\left(\int_{[0,1]\setminus I_1}\min\{|\co(F_t)\setminus F_t|, |F_t|\}\,dt \right).$$
Hence, we find
$$\int_{[0,1]\setminus I_1}\min\{|\co(F_t)\setminus F_t|, |F_t|\}\,dt\leq O_{\lambda,n,p}(\delta).$$
The same bound for $G_t$ follows analogously. Hence, recalling the definition of $I_2$ this implies 
\begin{align*}
&\int_{ I_1\cup I_2}|F_t|\,dt\leq O_{\lambda,n,p,\alpha}(\delta)\\
&\int_{[0,1]\setminus (I_1\cup I_2)}\left(\left|\co(F_t)\setminus F_t\right|+\left|\co(G_{T(t)})\setminus G_{T(t)}\right|\right)\,dt\leq O_{\lambda,n,p,\alpha}(\delta).
\end{align*}

Similarly for $I_3$, we have 
\begin{align*}
\delta&\geq \int_{[0,1]} \frac{|\lambda F_t+(1-\lambda) G_{T(t)}|-M_{\lambda,1/n}(|F_t|,|G_{T(t)}|)}{M_{\lambda,-p}\left(1,\frac{1}{\frac{dT}{dt}(t)}\right)}\,dt\geq \alpha \int_{I_3}\frac{M_{\lambda,1/n}(|F_t|,|G_{T(t)}|)}{M_{\lambda,-p}\left(1,\frac{1}{\frac{dT}{dt}(t)}\right)}\,dt\\
&\geq \alpha \int_{I_3}|F_t|\frac{M_{\lambda,1/n}(1,\frac{1}{\frac{dT}{dt}(t)})}{M_{\lambda,-p}\left(1,\frac{1}{\frac{dT}{dt}(t)}\right)}\,dt\geq \alpha \int_{I_3}|F_t|\,dt.
\end{align*}

For $I_4$, note that $I_4$ for a given $\alpha$ is contained in $I_1\cup I_3$ for a different smaller $\alpha'$. Indeed if $F_t$ and $G_{T(t)}$ are sufficiently similar in size (at most a factor $(1+\alpha')$ away from each other) and $|\lambda F_t+(1-\lambda)G_{T(t)}|\leq (1+\alpha'')M_{\lambda,1/n}(|F_t|,|G_{T(t)}|)$, then by \Cref{BMStab}, there exists a set convex set $K$ tightly containing scaled copies of $F_t$ and $G_{T(t)}$ and thus of $\lambda F_t+(1-\lambda)G_{T(t)}$. To be precise, consider $\alpha'''\leq \alpha'$ so that $|(1+\alpha''')F_t|=\left|\frac{G_{T(t)}}{1+\alpha'''}\right|$ (assume  that $|F_t|\geq |G_{T(t)}|$, the other case follows analogously). Then (up to translation) there exists a convex set $K$ so that $(1+\alpha''')F_t,\frac{G_{T(t)}}{1+\alpha'''}\subset K$, and $|K|\leq \left(1+O_{n,\lambda}\left(\sqrt{\alpha''}\right)\right)|(1+\alpha''')F_t|$. The former also implies  that 
\begin{align*}
\lambda F_t+(1-\lambda)G_{T(t)}&= \frac{\lambda}{(1+\alpha''')} (1+\alpha''')F_t+(1-\lambda)(1+\alpha''')\frac{G_{T(t)}}{(1+\alpha''')}\\
&\subset \frac{\lambda}{(1+\alpha''')} K+(1-\lambda)(1+\alpha''')K\subset (1+\alpha''')K.
\end{align*}
Hence, 
$$|(\lambda F_t+(1-\lambda)G_{T(t)})\triangle F_t|\leq |(1+\alpha''')K\setminus (\lambda F_t+(1-\lambda)G_{T(t)})|+|(1+\alpha''')K\setminus F_t|\leq O_{n,\lambda}\left(\sqrt{\alpha''+\alpha'''}\right)|F_t|,$$
so that the bound for $I_4$ follows from the bound on $I_1\cup I_3$. The bound for $I_5$ follows analogously.

For the final conclusion, we will show the result for $g$, the result for $f$ follows analogously. We first consider the following more modest claim.

\begin{clm}
There exists a $t_0=\Omega_{n,\lambda,p}(1)$ so that $\int_{\mathbb{R}^n\setminus G_{t_0}}g\left(x\right)\,dx\leq 1-\Omega_{n,\lambda,p}(1)$.
\end{clm}
\begin{proof}[Proof of Claim]
$$\int_{\mathbb{R}^n}h\,dx\geq \int_{\mathbb{R}^n\setminus(1-\lambda)G_t}h(x)\,dx\geq \int_{\mathbb{R}^n\setminus (1-\lambda)G_t}\M\left(f(o),g\left(\frac{x}{1-\lambda}\right)\right)\,dx\geq\int_{\mathbb{R}^n\setminus(1-\lambda)G_t}\frac{\M(1,t)}{t
}g\left(\frac{x}{1-\lambda}\right)\,dx,$$
where in the last inequality, we use that $\frac{\M(t,1)}{t
}$ is decreasing in $t$. We can rewrite to
$$\int_{\mathbb{R}^n\setminus(1-\lambda)G_t}\frac{\M(1,t)}{t
}\,g\left(\frac{x}{1-\lambda}\right)\,dx=\frac{\M(1,t)}{t
}(1-\lambda)^{n}\int_{\mathbb{R}^n\setminus G_t}\,g\left(x\right)\,dx.$$
Noting that as $t\to 0$, we have $\frac{\M(1,t)}{t
}\to (1-\lambda)^{1/p}=(1-\lambda)^{-n}(1+\Omega_{n,\lambda,p}(1))$, so that for some 
$t_0=\Omega_{n,\lambda, p}(1)$, we have $\frac{\M(1,t_0)}{t_0
}\geq(1-\lambda)^{-n}(1+\Omega_{n,\lambda,p}(1))$, so that
$$\int_{\mathbb{R}^n}h(x)\,dx\geq \left(1+\Omega_{n,\lambda,p}(1)\right)\int_{\mathbb{R}^n\setminus G_{t_0}}g\left(x\right)\,dx.$$

Since $\int h\,dx<1+\delta$, this implies 
$$\int_{\mathbb{R}^n\setminus G_{t_0}}g\left(x\right)\,dx\leq \frac{1+\delta}{1+\Omega_{n,\lambda,p}(1)}\leq 1-\Omega_{n,\lambda,p}(1),$$
which concludes the claim.
\end{proof}

To get from some to most of the mass, we iterate as follows until we find an appropriate function. Given the $t_0$ from the claim, consider functions $g'=\min\{g,t_0/2\}$, and $f':=\min\{f,T^{-1}(t_0/2)\}$, so that $\int f'\,dx=\int g'\,dx$. Also, let $h'=\min\{h,\M(T^{-1}(t_0/2),t_0/2)\}$, so that $h'\geq \M^*(f',g')$. If $\int g'\,dx\leq \alpha' \int g\,dx$, then in particular $\int_{\mathbb{R}^n\setminus G_{t_0/2}}g\,dx\leq \int g'\,dx\leq \alpha' \int g\,dx$, so we can set $t_0/2=\beta$ and the result follows.

If, on the other hand, we have $\int g'\,dx> \alpha' \int g\,dx$, then we will see that in fact $\int h'\,dx\leq \delta+ \int f'\,dx\leq (1+\delta/\alpha')\int f'\,dx $. Indeed, similar to the earlier computation, we find
\begin{align}
\int (h-h')\,dx&=\int_{\M(T^{-1}(t_0/2),t_0/2)}^1 |H_s|\,ds\geq \int_{\M(T^{-1}(t_0/2),t_0/2)}^1 |H_{\M(t,T(t))}|\,d\M(t,T(t))\nonumber\\
&\geq \int_{T^{-1}(t_0/2)}^1 \frac{|\lambda F_t+(1-\lambda) G_{T(t)}|}{M_{\lambda,-p}\left(1,\frac{1}{\frac{dT}{dt}(t)}\right)}\,dt\geq \int_{T^{-1}(t_0/2)}^1|F_t|\,dt=\int (f-f')\,dx=\int (g-g')\,dx.\label{horizontalcutdecreasingdoubling}
\end{align}
Rearranging shows that
$$\int h'\,dx\leq \left(\int h\,dx-\int f\,dx\right)+ \int f'\,dx\leq \delta+\int f'\,dx\leq (1+\delta/\alpha')\int f'\,dx. $$
Finally, we also have removed a significant amount from $g$; indeed 
$$\int_{\mathbb{R}^n} (g-g')\,dx\geq \int_{G_{t_0}} (g-g')\,dx \geq \frac12 \int_{G_{t_0}} g\,dx=\Omega_{n,\lambda,p}(1).$$

Hence, when we iterate, we find after at most $\log_{1-\Omega_{n,\lambda,p}(1)}(\alpha')=O_{n,\lambda,p,\alpha'}(1)$ steps, we find a $\beta=O_{n,\lambda,p,\alpha'}(1)$ so that if we let $g''=\min\{g,\beta\}$, then $\int g''\,dx\leq \alpha' \int g\,dx$. Hence, we find $\int_{\mathbb{R}^n\setminus G_{\beta}}g\,dx\leq \int g''\,dx\leq \alpha' \int g\,dx$. This concludes the proposition.
\end{proof}

\begin{proof}[Proof of \Cref{SimilarSupportLem}]
Consider the level sets 
$$F_t:=\{x: f(x)>t\},\, G_t=\{x: g(x)>t\},\, H_t=\{x:h(x)>t\},$$
so that $F_0=supp(f)$. As we have seen before, we have inclusion
$$\lambda F_t+(1-\lambda)G_s\subset H_{\M(t,s)}.$$
Consider the transport map $T\colon[0,1]\to[0,1]$ defined by the differential equation $\frac{dT}{dt}(t)=\frac{|F_t|}{|G_{T(t)}|}$, i.e., so that for all $t\in[0,1]$, we have
$$\int_{0}^t |F_s|\,ds =\int_{0}^{T(t)} |G_s|\,ds.$$
With all of this machinery set up, we consider the following lower bound on $\int hdx$;
\begin{align*}
  \int_{\mathbb{R}^n} h\,dx&=\int_0^1 |H_s|\,ds\geq \int_0^1 \left|H_{\M(t,T(t))}\right|\,d\M(t,T(t))\geq \int_0^1 \left|\lambda F_t+(1-\lambda) G_{T(t)}\right|\frac{dt}{M_{\lambda,-p}\left(1,\frac{1}{\frac{dT}{dt}(t)}\right)}\\
  &\geq \int_0^1 \frac{|\lambda F_t+(1-\lambda) G_{T(t)}|-M_{\lambda,1/n}(|F_t|,|G_{T(t)}|)}{M_{\lambda,-p}\left(1,\frac{1}{\frac{dT}{dt}(t)}\right)}\,dt+ \int_0^1 M_{\lambda,1/n}(|F_t|,|G_{T(t)}|)\frac{dt}{M_{\lambda,-p}\left(1,\frac{1}{\frac{dT}{dt}(t)}\right)},\\
\end{align*}
where in the second inequality we use \Cref{Holderderivative}. Note that by Brunn-Minkowski, $|\lambda F_t+(1-\lambda) G_{T(t)}|-M_{\lambda,1/n}(|F_t|,|G_{T(t)}|)\geq 0$.
Using that $p\geq -1/n$ (and thus $-p\leq 1/n$) combined with $\frac{1}{\frac{dT}{dt}(t)}=\frac{|G_{T(t)}|}{|F_t|}$, we find that
$$\int_0^1 M_{\lambda,1/n}(|F_t|,|G_{T(t)}|)\frac{dt}{M_{\lambda,-p}\left(1,\frac{1}{\frac{dT}{dt}(t)}\right)}=\int_0^1 |F_t|\frac{M_{\lambda,1/n}\left(1,\frac{|G_{T(t)}|}{|F_t|}\right)}{M_{\lambda,-p}\left(1,\frac{1}{\frac{dT}{dt}(t)}\right)}\,dt=\int_0^1 |F_t|\,dt=1.$$
Distinguish two cases; either $T(\beta)\leq \beta$ or $T(\beta)>\beta$. We assume the former and the latter follows analogously interchanging the roles of $f$ and $g$. Under that assumption we find $F_t=supp(f)$, $G_{T(t)}=supp(g)$, and $\frac{dT}{dt}(t)=\frac{|supp(f)|}{|supp(g)|}$ for all $t\in[0,\beta]$.
By the above computation, we find
\begin{align*}
   \delta&\geq \int_0^{\beta} |F_t|\Bigg(\frac{M_{\lambda,1/n}\left(1,\frac{|G_{T(t)}|}{|F_t|}\right)}{M_{\lambda,-p}\left(1,\frac{1}{\frac{dT}{dt}(t)}\right)}-1\Bigg)\,dt\\
   &\geq \int_0^{\beta} |F_t|\Bigg(\frac{M_{\lambda,1/n}\left(1,\frac{|supp(g)|}{|supp(f)|}\right)}{M_{\lambda,-p}\left(1,\frac{|supp(g)|}{|supp(f)|}\right)}-1\Bigg)\,dt\geq \beta\Bigg(\frac{M_{\lambda,1/n}\left(1,\frac{|supp(g)|}{|supp(f)|}\right)}{M_{\lambda,-p}\left(1,\frac{|supp(g)|}{|supp(f)|}\right)}-1\Bigg).
\end{align*}
Hence, using that $p>-1/n$, this implies that 
$$\frac{dT}{dt}(t)=\frac{|supp(g)|}{|supp(f)|}=1\pm O_{n,\lambda,p,\beta}\left(\sqrt{\delta}\right).$$
By the above, we also know that 
$$\beta\frac{\left(|\lambda F_0+(1-\lambda)G_0|-M_{\lambda,1/n}(|F_0|,|G_0|)\right)}{1+ O_{n,\lambda,p}\left(\sqrt{\delta}\right)}\leq \int_0^{\beta} \frac{|\lambda F_t +(1-\lambda) G_{T(t)}|-M_{\lambda,1/n}(|F_t|,|G_{T(t)}|)}{M_{\lambda,-p}\left(1,\frac{1}{\frac{dT}{dt}(t)}\right)}\,dt\leq \delta,$$
so that
$$|\lambda F_0+(1-\lambda)G_0|\leq (1+2\delta/\beta)M_{\lambda,1/n}(|F_0|,|G_0|).$$
Hence, by \Cref{BMStab} (the main result of \cite{BMStab}) we find that for $\alpha:=\left(\frac{|F_0|}{|G_0|}\right)^{1/n}=1\pm O_{n,\lambda,p,\beta}\left(\sqrt{\delta}\right)$, we have
$$\min_{x\in\mathbb{R}^n}|(x+F_0)\triangle \alpha G_0|=O_{n,\lambda,\beta}\left(\sqrt{\delta}\right).$$
Since $\alpha$ is so close to $0$, we find that
$$\min_{x\in\mathbb{R}^n}|(x+F_0)\triangle G_0|\leq \min_{x\in\mathbb{R}^n}|(x+F_0)\triangle \alpha G_0|+\min_{x\in\mathbb{R}^n}|(x+G_0)\triangle \alpha G_0|=O_{n,\lambda,\beta}\left(\sqrt{\delta}\right)+\left|1-\alpha^n\right|\cdot |G_0|=O_{n,\lambda,p,\beta}\left(\sqrt{\delta}\right).$$
This concludes the proof of the main part of the Lemma.

We now extend to functions which do take values below $\beta$, but only rarely, i.e., with $\int_{\{x: f(x)<\beta \text{ or } g(x)<\beta\}}(f+g) dx\leq \gamma$. Simply consider the functions $f'=f\cdot \textbf{1}\{x:f(x)\geq\beta\}$ and $g'=g\cdot \textbf{1}\{x:g(x)\geq\beta\}$, so that $\int f'\,dx,\int g'\,dx\geq 1-\gamma$ and thus $\int h\,dx\leq 1+\delta\leq (1+O_{n,\lambda,p,\beta}(\delta+\gamma))\int f'\,dx$. Hence, applying the above result for $f'$ and $g'$ gives (up to translation) $|supp(f')\triangle supp(g')|=O_{n,\lambda,p,\beta}(\sqrt{\delta+\gamma})$. Combining with the obvious observation that $supp(f)\triangle supp(f')=\{x:f(x)<\beta\}$, we find $|supp(f)\triangle supp(g)|=O_{n,\lambda,p,\beta}(\sqrt{\delta+\gamma})$.
\end{proof}

\begin{proof}[Proof of \Cref{smalltranslationsmallsymdiflem}]
We first note that it suffices to consider the level sets as follows
$$\int (f(x)-f(x+v))\,dx=\int_{0}^1 |F_t\triangle (F_t+v)|\,dt\leq \int_{0}^1 |\co(F_t)\triangle (\co(F_t)+v)|\,dt+ 2\delta.$$
We split this final integral in $t\geq t_0$ and $t<t_0$. First note that for $t_0$, we have $|\co(F_{t_0})\triangle (\co(F_{t_0})+v)|\leq O_n(\eta)|\co(F_{t_0})|$. For all $t\leq t_0$, we have $\co(F_{t})\supset \co(F_{t_0})$, so that again $|\co(F_{t})\triangle (\co(F_{t})+v)|\leq O_n(\eta)|\co(F_{t})|$. Integrating over those $t$, gives 
$$\int_{0}^{t_0} |\co(F_t)\triangle (\co(F_t)+v)|\,dt\leq O_n(\eta) \int_{0}^{t_0} |\co(F_t)|\,dt\leq O_n(\eta)\int f\,dx.$$
For $t>t_0$, we find that as $\co(F_{t})\subset \co(F_{t_0})$, we have
$$|\co(F_t)\triangle (\co(F_t)+v)|\leq |\co(F_{t_0})\triangle (\co(F_{t_0})+v)|\leq O_n(\eta)|\co(F_{t_0})|.$$
Recall that $|\co(F_{t_0})|\leq 2|F_{t_0}|\leq \frac{2}{t_0}\int_{F_{t_0}} f\,dx\leq \frac{2}{t_0}$, so that
$$\int_{t_0}^1 |\co(F_t)\triangle (\co(F_t)+v)|\,dt\leq (1-t_0)|\co(F_{t_0})\triangle (\co(F_{t_0})+v)|\leq \frac{2(1-t_0)}{t_0}O_n(\eta).$$
Combining these upper bounds proves the lemma.
\end{proof}

\subsection{Proof of the reduction}\label{initialcleanupreductionsection}

\begin{proof}[Proof that \Cref{symmetric_diff_technical} implies \Cref{symmetric_diff}]
We may assume $\delta$ is sufficiently small in terms of $n,\lambda$, and $p$, since otherwise the conclusion is trivially true by virtue of the trivial upper bound
$$\int_{\mathbb{R}^n}|f-g|\,dx\leq \int_{\mathbb{R}^n}(f+g)\,dx=2\int_{\mathbb{R}^n}f\,dx.$$
We will repeatedly use this assumption throughout this proof.

Apply \Cref{continuousmaker} with $\eta=\delta^{1000}$ to find continuous $f_1,g_1,h_1$. Then apply \Cref{BigFirstCleanup} to find $f_2,g_2,h_2$ taking values in $[0,1]$. Then apply \Cref{aclcor} with some $\alpha=\alpha'>0$ small in terms of $n,\lambda,p$ to make various statements in the following true, and find $f_3,g_3,h_3$ and $\beta=\Omega_{n,\lambda, p}(1)$ for which it suffices to prove the result and renormalize so that $\sup f_3=1$. Consider level sets $F^3,G^3,H^3$ and transport map $T:[0,1]\to[0,\sup g_3]$ as in \Cref{aclcor}. Note that $\sup g\in[1-\alpha,1+\alpha]$ since $\frac{d}{dt}T(t)\in[1-\alpha, 1+\alpha]$ for all $t\in[0,1]$ and $T(0)=0$. Assume $\sup g_3\leq \sup f_3$, the other case follows analogously.

\begin{clm}
Up to translation, we have $\left|\frac{\lambda}{2}\co(F^3_{\beta})\setminus F^3_{\beta}\right|=O_{n,\lambda,p}(\delta)$
\end{clm}
\begin{proof}[Proof of claim]
Find $x\in\mathbb{R}^n$, so that $g_3(x)\geq 0.9$. Then (up to translation) $H_{\M(\beta,0.9)}\supset \lambda F_{\beta}$. 
Find the $t_0$ so that $\M(t_0,T(t_0))=\M(\beta,0.9)$ and note that $t_0-\beta\geq \Omega_{\lambda,p}(1)$.
Since $\left|H^3_{\M(t_0,T(t_0))}\triangle F^3_{t_0}\right|\leq \alpha |F^3_{t_0}|$, and $\left|\co(F^3_{\beta})\setminus F^3_{\beta}\right|\leq \alpha|F^3_{\beta}|$, we find that (up to translation) 
$$\left|\co(\lambda F^3_{\beta})\setminus F^3_{t_0}\right|\leq \left|\co(\lambda F^3_{\beta})\setminus H^3_{\M(t_0,T(t_0))}\right|+\left|H^3_{\M(t_0,T(t_0))}\triangle F^3_{t_0}\right|\leq O_{\lambda,n}(\alpha)\left|F^3_{\beta}\right|,$$
so that 
$$\left|\co(\lambda F^3_{\beta}\cap F^3_{t_0})\right|\geq (1-O_{\lambda,n}(\alpha))\left|\co(\lambda F^3_{\beta})\right|,$$
so that (up to translation and for sufficiently small $\alpha$ in $\lambda$ and $n$) 
$$\frac\lambda2\co\left( F^3_{\beta}\right)\subset \co\left(\lambda F^3_{\beta}\cap F^3_{t_0}\right)\subset \co(F^3_{t_0}).$$

Hence, we find that $\left|\frac{\lambda}{2}\co(F^3_{\beta})\setminus F^3_{\beta}\right|$ is a lower bound on $\left|\co(F^3_t)\setminus F^3_t\right|$ for all $t\in[\beta,t_0]$, so that 
$$(t_0-\beta)\left|\frac{\lambda}{2}\co(F^3_{\beta})\setminus F^3_{\beta}\right|\leq \int_{0}^1 |\co(F^3_t)\setminus F^3_t|\,dt\leq O_{n,\lambda,p}(\delta).$$
The claim follows.
\end{proof}

\begin{clm}
$\left|\co(F^3_\beta)\setminus F^3_{\beta/2}\right|\leq O_{n,\lambda,p}(\delta)$
\end{clm}
\begin{proof}
For all $t\in(\beta/2,\beta)$, we have $\left|\co(F^3_t)\setminus F^3_t\right|\geq \left|\co(F^3_\beta)\setminus F^3_{\beta/2}\right|$. Hence, we have
$$(\beta-\beta/2)\left|\co(F^3_\beta)\setminus F^3_{\beta/2}\right|\leq \int_{0}^1 |\co(F^3_t)\setminus F^3_t|\,dt\leq O_{n,\lambda,p}(\delta).$$
Since $\beta=\Omega_{n,\lambda,p}(1)$, the claim follows.
\end{proof}

Analogously we have $\left|\co(G^3_{T(\beta)})\setminus G^3_{\beta/2}\right|=O_{n,\lambda,p}(\delta)$. Recalling from \Cref{aclcor}, we have that (up to translation) we have
$\left|\co(F^3_{\beta})\triangle \co(G^3_{T(\beta)})\right|\leq \alpha |F^3_{\beta}|$.
Hence, we can consider an affine transformation so that $R'S_n\subset\co(F^3_{\beta}), \co(G^3_{T(\beta)})\subset RS_n$ and $R/R'\leq O_n(1)$ using John's Theorem.

Apply \Cref{parallelconeslem} to find translations $a,b\in\mathbb{R}^n$ so that
$$\int_{a+C_i^n}f\,dx=\int_{b+C_i^n}g\,dx=\frac{1}{n+1}\int f\,dx.$$
We'll see that $a$ and $b$ are not too far from the origin. Recall that $\int_{F^3_\beta}f\,dx\geq 1-\alpha$, so $$\int_{(a+C_i^n)\cap F^3_\beta}f\,dx=\left(\frac{1}{n+1}-\alpha\right)\int f\,dx\geq \frac{1}{2n}.$$
On the other hand, since $\sup f\leq 1$, we have $\int_{(a+C_i^n)\cap F^3_\beta}f\,dx\leq |(a+C_i^n)\cap \co(F^3_\beta)|$.
Since $\beta=\Omega_{n,\lambda,p}(1)$, we find that $\beta |\co(F^3_\beta)|\leq 2\beta |F^3_\beta|\leq \int f\,dx\leq 1$
implies $|\co(F^3_\beta)|=O_{n,\lambda,p}(1)$, and $$|(a+C_i^n)\cap \co(F^3_\beta)|\geq \Omega_{n\lambda,p}(|\co(F^3_\beta)|),$$
for all $i=0,\dots,n$. Combined with the above normalization that $R'S_n\subset\co(F^3_{\beta})\subset RS_n$, this implies that there is some $r=\Omega_{n,\lambda,p}(1)$ so that $a+rS_n\subset \co(F^3_{\beta})$. Analogously $b+rS_n\subset \co(G^3_{T(\beta}))$. Translating $f$ and $g$ so that $a=b=o$, this implies 
$$\left|rS_n\setminus F^3_{\beta/2}\right|+\left|rS_n\setminus G^3_{\beta/2}\right|=O_{n,\lambda,p}(\delta).$$
We can now apply \Cref{symmetric_diff_technical} to find that $\int_{\mathbb{R}^n}|f_3-g_3|\,dx=O_{n,\lambda,p,r,\beta}\left(\sqrt{\delta}\right)=O_{n,\lambda,p}\left(\sqrt{\delta}\right)$ as $r,\beta=\Omega_{n,\lambda,p}(1)$. Pulling this back, we find that for some translates $v,v_1,v_2\in\mathbb{R}^n$, we have
\begin{align*}
    \int_{\mathbb{R}^n} |f(v+x)-g(x)|\,dx&\leq \eta+\int_{\mathbb{R}^n} |f_1(v_1+x)-g_1(x)|\,dx\\
    &\leq O_{n,\lambda,p}(\delta)+\int_{\mathbb{R}^n} |f_2(v_2+x)-g_2(x)|\,dx\\
    &\leq O_{n,\lambda,p}(\delta)+\int_{\mathbb{R}^n} |f_3(v_3+x)-g_3(x)|\,dx=O_{n,\lambda,p}\left(\sqrt{\delta}\right).
\end{align*}
Thus, \Cref{symmetric_diff} follows.
\end{proof}

\section{Proof of \Cref{symmetric_diff_1D}, i.e., \Cref{symmetric_diff} in $\mathbb{R}$}\label{1Dsymdifsection}
We aim to prove the following theorem which is the one dimensional instance of \Cref{symmetric_diff}.
\begin{thm}\label{symmetric_diff_1D}
    Given $\lambda \in (0,1/2]$ and $p \in (-1/2, \infty)$ there exists $d=d_{1, \lambda, p}>0$ such that the following holds. Let $f,g, h\colon \mathbb{R}\rightarrow \mathbb{R}_{\geq0}$ be integrable functions such that $\int_{\mathbb{R}} f\,dx=  \int_{\mathbb{R}} g\,dx=1$ and for all $x,y \in \mathbb{R}$ we have $h(\lambda x +(1-\lambda)y) \geq M_{\lambda,p}(f(x),g(y))$. If $\int_{\mathbb{R}}h\,dx = 1+\delta$ for some $0\leq \delta \leq d$, then, up to translation, $\int_{\mathbb{R}^2} |f-g|\,dx=O_{\lambda, p} \left(\sqrt{\delta}\right)$. 
\end{thm}

Given the information given by \Cref{almostconvexlevelsets} it suffices to prove the following proposition.

\begin{prop}\label{reducedSymDif1D}
For all $\lambda\in (0,1),p>-1$, there exists a $d=d_{\lambda,p}>0$, so that the following holds.
Let $f,g,h\colon\mathbb{R}\to\mathbb{R}_{\geq0}$ so that for all $x,y\in\mathbb{R}$ we have $h(\lambda x +(1-\lambda)y)\geq \M(f(x),g(y))$ and $\int h\,dx\leq (1+\delta)\int f\,dx$ with $\delta<d$. Additionally assume that for level sets $F_t:=\{x\in\mathbb{R}: f(x)>t\}$ and $G_t$ we have
$$
\int_{t\in \mathbb{R}_{\geq0}}\left(\left|\co(F_t)\setminus F_t\right|+\left|\co(G_{t})\setminus G_{t}\right|\right)\,dt\leq \delta.
$$
Then there exists some $v\in\mathbb{R}$ so that
$$\int_\mathbb{R} |f(x)-g(x+v)|\,dx=O\left(\sqrt{\delta}\right)\int_{\mathbb{R}}f(x)\,dx$$
\end{prop}
\begin{proof}[Proof of \Cref{reducedSymDif1D}]
Let $T\colon F_0\to G_0$ be the map pushing $f$ into $g$, so that $\int_{-\infty}^x f(y)\,dy=\int_{-\infty}^{T(x)}g(y)\,dy$, i.e., $\frac{dT}{dx}(x)=\frac{f(x)}{g(T(x))}$ almost everywhere. We use the bound $h\left(\lambda x+(1-\lambda)T(x)\right)\geq \M(f(x),g(T(x)))$, to find
$$\int h\left(\lambda x+ (1-\lambda) T(x)\right) \,d\left(\lambda x+ (1-\lambda) T(x)\right)\geq \int \M(f(x),g(T(x)))\left(\frac{\lambda+(1-\lambda)\frac{dT}{dx}(x)}{2}\right)\,dx$$
By definition $\frac{dT}{dx}(x)=\frac{f(x)}{g(T(x))}$, so that, we can rewrite 
\begin{align*}
&\int h(x)\,dx\\
&\geq \int \frac{1}{\frac{\lambda}{f(x)}+\frac{1-\lambda}{g(T(x))}}\left(\lambda+(1-\lambda)\frac{dT}{dx}(x)\right)\,dx +\int \left(\M(f(x),g(T(x)))-\frac{1}{\frac{\lambda}{f(x)}+\frac{1-\lambda}{g(T(x))}}\right)\left(\lambda+(1-\lambda)\frac{dT}{dx}(x)\right)\,dx\\
&\geq \int f(x)\,dx +\int f(x)\left(\M\left(\sqrt{\frac{dT}{dx}(x)},\frac{1}{\sqrt{\frac{dT}{dx}(x)}}\right)-\frac{1}{\frac{\lambda}{\sqrt{\frac{dT}{dx}(x)}}+(1-\lambda)\sqrt{\frac{dT}{dx}(x)}}\right)
\left(\frac{\lambda}{\sqrt{\frac{dT}{dx}(x)}}+(1-\lambda)\sqrt{\frac{dT}{dx}(x)}\right)\,dx\\
&\geq \int f(x)\,dx +\int f(x)\cdot\Omega_{\lambda,p}\left(\min\left\{1,\left|1-\sqrt{\frac{dT}{dx}(x)}\right|^2\right\}\right)\cdot\Omega_{\lambda}(1)\,dx\\
&\geq \int f(x)\,dx +\Omega_{\lambda,p}\left(\int f(x) \min\left\{1,\left|1-\sqrt{\frac{dT}{dx}(x)}\right|^2\right\}\, dx\right)
\end{align*}
Hence, as $\int h\,dx\leq (1+\delta)\int f\,dx$, we find 
$$\int f(x) \min\left\{1,\left|1-\sqrt{\frac{dT}{dx}(x)}\right|^2\right\}\,dx =O_{\lambda,p}(\delta)\int f\,dx,$$
so that $\frac{dT}{dx}(x)$ is close to 1 most of the time. To formalize that, consider the following set
$$I:=(\co(F_0)\setminus F_0)\cup \left\{x\in\co(F_0): \frac{dT}{dx}(x)\not\in (1/2,3/2)\right\}.$$
We find that 
$$\int_I f(x) \,dx\leq 4 \int_I f(x) \min\left\{1,\left|1-\sqrt{\frac{dT}{dx}(x)}\right|^2\right\}\,dx =O_{\lambda, p}(\delta),$$
and 
$$\int_{T(I)} g(x)\,dx=\int_I f(x)\,dx=O_{\lambda, p}(\delta),$$
so little mass of $f$ is in $I$. Hence, define
$$f_1(x):=\begin{cases}
f(x)&\text{ if } x\not\in I\\
0&\text{ if } x\in I.
\end{cases}\qquad\text{ and }\qquad g_1(x):=\begin{cases}
g(x)&\text{ if } x\not\in T(I)\\
0&\text{ if } x\in T(I),
\end{cases}$$
so that 
$$\int \left(|f-f_1|+|g-g_1|\right)\,dx=O_{\lambda, p}(\delta).$$
Defining $F^1_t:=\{x: f_1(x)>t\}\subset F_t$ and $G^1_t:=\{x: g_1(x)>t\}\subset G_t$, we find
\begin{align*}
\int_{t\in \mathbb{R}_{\geq0}}\left(\left|\co(F^1_t)\setminus F^1_t\right|+\left|\co(G^1_{t})\setminus G^1_{t}\right|\right)\,dt&\leq
\int_{t\in \mathbb{R}_{\geq0}}\left(\left|\co(F_t)\setminus F_t\right|+\left|\co(G_{t})\setminus G_{t}\right|\right)\,dt\\
&\,\,\,\,+\int_{\mathbb{R}} \left(|f(x)-f_1(x)|+|g(x)-g_1(x)|\right) \,dx\\
&=O_{\lambda, p}(\delta).
\end{align*}

We next compress the functions to skip the points in $I$. To this end consider the increasing surjective map $S_f\colon \mathbb{R}\to (\mathbb{R}\setminus I)$ with $\frac{dS_f}{dx}=1$ almost everywhere and $S_f(x)=x$ for some $x$ with $f_1(x)=\sup f_1$. Similarly define the increasing surjective map $S_g\colon \mathbb{R}\to (\mathbb{R}\setminus T(I))$ with $S_g'=1$ almost everywhere and $S_g(y)=y$ for some $y$ with $g_1(y)=\sup g_1$.

Using this map, define 
$$f_2(x):=f_1(S_f(x))\text{ and } g_2( x):= g_{1}(S_g(x)),$$
so that
$$\int f_2\,dx=\int g_2\,dx=\int f_1\,dx=\int g_1\,dx.$$
As ever define level sets  $F^2_t:=\{x: f_2(x)>t\}$ and $G^2_t:=\{x: g_2(x)>t\}$. Note that $|F^1_t|=|F^2_t|$ and $|G^1_t|=|G^2_t|$, while $|\co(F^1_t)|\geq|\co(F^2_t)|$ and $|\co(G^1_t)|\geq|\co(G^2_t)|$, so that
$$
\int_{t\in \mathbb{R}_{\geq0}}\left(\left|\co(F^2_t)\setminus F^2_t\right|+\left|\co(G^2_{t})\setminus G^2_{t}\right|\right)\,dt\leq\int_{t\in \mathbb{R}_{\geq0}}\left(\left|\co(F^1_t)\setminus F^1_t\right|+\left|\co(G^1_{t})\setminus G^1_{t}\right|\right) \,dt=O_{\lambda, p}(\delta).
$$
Considering the corresponding transport map $T_2\colon\mathbb{R}\to\mathbb{R}; x\mapsto S_g^{-1}(T(S_f(x)))$,
so that 
$$\frac{dT_2}{dx}(x)=\frac{dT}{dS_f(x)}(S_f(x))=\frac{f(S_f(x))}{g(T(S_f(x)))}=\frac{f_2(x)}{g(S_g(T_2(x)))}=\frac{f_2(x)}{g_2(T_2(x))}.$$ Since $S_f(x)\not \in I$, we find that $\frac{dT_2}{dx}(x)\in[1/2,3/2]$ for all $x\in \mathbb{R}$ and moreover
\begin{multline*}
\int_{\mathbb{R}} f_2(x) \left|1-\frac{dT_2}{dx}(x)\right|^2 dx=\int_{\mathbb{R}} f(S_f(x)) \left|1-\frac{dT}{dS_f(x)}(S_f(x))\right|^2 dx \\
=\int_{\mathbb{R}\setminus I} f(x) \left|1-\frac{dT}{dx}(x)\right|^2 dx
\leq \int_{\mathbb{R}} f(x) \min\Bigg\{1,\left|1-\sqrt{\frac{dT}{dx}(x)}\right|^2\Bigg\}\,dx=O_{\lambda, p}(\delta).
\end{multline*}
Finally, we evaluate
$$\int_\mathbb{R}\left( |f_2-f_1|+|g_2-g_1|\right) \,dx=\int_{\mathbb{R}_{\geq0}}\left( \left|F^2_t\triangle F^1_t\right|+\left|G^2_t\triangle G^1_t\right|\right)\,dt.$$
Since $S_f$ has a fixed point where $f_1$ attains its supremum, we have $\co(F^2_t)\subset\co(F_t^1)$. As moreover $\left|F^2_t|=| F^1_t\right|$, we have $\left|F^2_t\triangle F^1_t\right|\leq 2|\co(F^1_t)\setminus F^1_t|$, so that
$$\int_{\mathbb{R}_{\geq0}} \left|F^2_t\triangle F^1_t\right|\,dt\leq \int_{\mathbb{R}_{\geq0}} 2\left|\co(F^1_t)\setminus F^1_t\right|\,dt=O_{\lambda, p}(\delta) .$$
Analogously, we find 
$$\int_{\mathbb{R}_{\geq0}} \left|G^2_t\triangle G^1_t\right|\,dt=O_{\lambda, p}(\delta).$$

 Hence, 
$$\int_\mathbb{R}\left( |f_2-f_1|+|g_2-g_1|\right) \,dx=\int_{\mathbb{R}_{\geq0}}\left( \left|F^2_t\triangle F^1_t\right|+\left|G^2_t\triangle G^1_t\right|\right)\,dt=O_{\lambda, p}(\delta).$$

Next, we consider the functions given by filling in the level sets:
$$f_3:=\sup\{t:x\in \co(F^2_t)\}\text{ and }g_3:=\sup\{t:x\in \co(G^2_t)\},$$
so that letting (as always)
$F^3_t:=\{x: f_3(x)>t\}$ and $G^3_t:=\{x: g_3(x)>t\}$, then $F^3_t=\co(F^2_t)$ and $G^3_t=\co(G^2_t)$. Hence, we immediately find
$$\int \left(|f_2-f_3|+|g_2-g_3|\right)\,dx=\int_{\mathbb{R}_{\geq0}}\left(\left|F^2_t\triangle F^3_t\right|+\left|G^2_t\triangle G^3_t\right|\right) \,dt=\int_{t\in \mathbb{R}_{\geq0}}\left(\left|\co(F^2_t)\setminus F^2_t\right|+\left|\co(G^2_{t})\setminus G^2_{t}\right|\right)\,dt=O_{\lambda, p}(\delta).$$

Recall that $\int_{\mathbb{R}} f_2(x) \left|1-\frac{dT_2}{dx}(x)\right|^2 dx=O_{\lambda, p}(\delta)$, so that $\int_{\mathbb{R}} f_2(x) \left|1-\frac{dT_2}{dx}(x)\right| dx=O_{\lambda,p}\left(\sqrt{\delta}\right)$. Since $\frac{dT_2}{dx}(x)\in[1/2,3/2]$, we in fact also have 
$$\int_{\mathbb{R}}  \left|f_2(x)-g_2(T_2(x))\right| \,dx=\int_{\mathbb{R}} f_2(x) \left|1-\frac{1}{\frac{dT_2}{dx}(x)}\right| dx=O_{\lambda,p}\left(\sqrt{\delta}\right).$$
Translating this to $f_3$ and $g_3$, we find
\begin{align*}
\int_{\mathbb{R}}  \left|f_3(x)-g_3(T_2(x))\right|\,dx&\leq \int_{\mathbb{R}}\left( |f_2(x)-f_3(x)|+\left|f_2(x)-g_2(T_2(x))\right| +|g_2(T_2(x))-g_3(T_2(x))|\right)\,dx\\
&\leq O_{\lambda, p}(\delta)+O_{\lambda,p}\left(\sqrt{\delta}\right)+2\int |g_2(x)-g_3(x)|\,dx=O_{\lambda,p}\left(\sqrt{\delta}\right)
\end{align*}
In other words, defining auxiliary function $k(x):=g_3(T_2(x))$, we have
$$\int_{\mathbb{R}}  \left|f_3(x)-k(x)\right|\, dx=O_{\lambda,p}\left(\sqrt{\delta}\right).$$
Extending that $\int_{\mathbb{R}} f_2(x) \left|1-\frac{dT_2}{dx}(x)\right|\,dx=O_{\lambda,p}\left(\sqrt{\delta}\right)$ in a slightly different way, we also find
\begin{align*}
\int_{\mathbb{R}} k(x) \left|1-\frac{dT_2}{dx}(x)\right| \,dx&\leq    \int_{\mathbb{R}}  \left|k(x)-f_3(x)\right|\cdot \left|1-\frac{dT_2}{dx}(x)\right|\,dx +\int_{\mathbb{R}}  \left|f_3(x)-f_2(x)\right|\cdot\left|1-\frac{dT_2}{dx}(x)\right|\,dx\\
&\,\,\,\,+\int_{\mathbb{R}} f_2(x) \left|1-\frac{dT_2}{dx}(x)\right|\,dx\\
&\leq    2\int_{\mathbb{R}}  \left|k(x)-f_3(x)\right|\,dx +2\int_{\mathbb{R}}  \left|f_3(x)-f_2(x)\right|\,dx+\int_{\mathbb{R}} f_2(x) \left|1-\frac{dT_2}{dx}(x)\right| \,dx\\
&=O_{\lambda,p}\left(\sqrt{\delta}\right),
\end{align*}
and since $\frac{dT_2}{dx}(x)\in[1/2,3/2]$ for almost all $x$, we also have
$$\int_{\mathbb{R}} k(x) \left|1-\frac{d(T^{-1}_2)}{dx}(x)\right|\,dx=O_{\lambda,p}\left(\sqrt{\delta}\right).$$
Our remaining aim is to show that
$$\int_{\mathbb{R}} |k(x)-g_3(x)|\,dx=O_{\lambda,p}\left(\sqrt{\delta}\right).$$

To this end define a final collection of level sets $K_t:=\{x:k(x)>t\}$. Note that $K_t=T_2^{-1}(G^3_t)$. In terms of the level sets, we find
$$\int_{\mathbb{R}} |k(x)-g_3(x)|\,dx=\int_{\mathbb{R}_{\geq0}}\left|K_t\triangle G^3_t\right|\,dt.$$
Write $k^{-1}(t)$ for $\max K_t$ (a slight abuse of notation) and $g_3^{-1}(t)$ for $\max G^3_t$, so that $T_2(k^{-1}(t))=g_3^{-1}(t)$. Note that by translating $f$ and $g$ if necessary, we can guarantee that $k^{-1}(\sup g_1)=g_3^{-1}(\sup g_1)=0=T_2(0)$ and that $\left|K_t\triangle G^3_t\right|$ is upper bounded by $|k^{-1}(t)-g_3^{-1}(t)|$ plus the equivalent on the other end of the interval. Hence, it suffices to show
$\int_{\mathbb{R}_{\geq0}} |k^{-1}(t)-g_3^{-1}(t)|\,dt=O_{\lambda,p}\left(\sqrt{\delta}\right)$ as the other end of the interval follows analogously. We compute as follows
\begin{multline*}
\int_{\mathbb{R}_{\geq0}} \left|k^{-1}(t)-g_3^{-1}(t)\right|\,dt=\int_{\mathbb{R}_{\geq0}} \left|k^{-1}(t))-T_2(k^{-1}(t))\right|\,dt
=\int_{\mathbb{R}_{\geq0}}\left( \int_0^{k^{-1}(t)}\left|1-\frac{dT_2}{dx}(x)\right|\,dx\right)\,dt\\
=\int_{\mathbb{R}}\left( \int_0^{k(x)}\left|1-\frac{dT_2}{dx}(x)\right|\,dt\right)\,dx
=\int_{\mathbb{R}}k(x)\left|1-\frac{dT_2}{dx}(x)\right|\,dx
=O_{\lambda,p}\left(\sqrt{\delta}\right).
\end{multline*}

Hence, we find
$$\int_{\mathbb{R}} |k(x)-g_3(x)|\,dx=\int_{\mathbb{R}_{\geq0}}\left|K_t\triangle G^3_t\right|\,dt= O_{\lambda,p}\left(\sqrt{\delta}\right).$$
Putting it all together, we find
\begin{multline*}
\int |f-g|\,dx\leq \int \big( |f-f_1|+|f_1-f_2|+|f_2-f_3|+|f_3-k|+|k-g_3|+|g_3-g_2|+|g_2-g_1|+|g_1-g|\big)\,dx\\
=O_{\lambda, p}(\delta)+O_{\lambda, p}(\delta)+O_{\lambda, p}(\delta)+O_{\lambda,p}\left(\sqrt{\delta}\right)+O_{\lambda,p}\left(\sqrt{\delta}\right)+O_{\lambda, p}(\delta)+O_{\lambda, p}(\delta)+O_{\lambda, p}(\delta)=O_{\lambda,p}\left(\sqrt{\delta}\right),
\end{multline*}
which concludes the proof of the proposition.
\end{proof}

\section{Proof of \Cref{symmetric_diff_2D}, i.e., \Cref{symmetric_diff}  in $\mathbb{R}^2$}\label{2DSymDifSec}
We will now prove the following theorem, which is the two-dimensional instance of \Cref{symmetric_diff}.
\begin{thm}\label{symmetric_diff_2D}
    Given $\lambda \in (0,1/2]$ and $p \in (-1/2, \infty)$ there exists $d=d_{2, \lambda, p}>0$ such that the following holds. Let $f,g, h\colon \mathbb{R}^2\rightarrow [0,1]$ be integrable functions such that $\int_{\mathbb{R}^2} f\,dx=  \int_{\mathbb{R}^2} g\,dx=1$ and for all $x,y \in \mathbb{R}^2$ we have $h(\lambda x +(1-\lambda)y) \geq M_{\lambda,p}(f(x),g(y))$. If $\int_{\mathbb{R}^2}h\,dx = 1+\delta$ for some $0\leq \delta \leq d$, then, up to translation, $\int_{\mathbb{R}^2} |f-g|\,dx=O_{2,\lambda, p} \left(\sqrt{\delta}\right)$. 
\end{thm}

The first step is to leverage the 1-dimensional result to resolve the context of well behaved tubes as in the following proposition.

\begin{prop}\label{symmetric_diff_2_tube}
        Given $\lambda \in (0,1/2]$, $p \in (-1/2, \infty)$, $r, \ell, s_0,s_1 \in (0, \infty)$  there exists $d=d_{ \lambda, p, \ell, s_0, s_1}>0$ such that the following holds.  Let 
        $f\colon [0,1]\times [0,\infty)\rightarrow [0,s_1]$, $g\colon [0,r]\times [0,\infty)\rightarrow [0,s_1]$, and $f,g,h\colon [0,\lambda+ (1-\lambda)r]\times [0,\infty)\rightarrow [0,s_1]$   be continuous functions with bounded support such that
        \begin{itemize}
            \item $\int f\,dx=\int g\,dx$, 
            \item for every $x \in [0,1]$,$\int f(x,y)\,dy \leq  \ell s_1$,
            \item for every $x \in [0,r]$, $\int g(x,y)\,dy \leq  \ell s_1$, 
            \item $|\{x\in[0,1]^2: f(x)<s_0\}|+|\{x\in[0,r]\times[0,1]: g(x)<s_0\}|\leq \eta $
            \item for all $x\in[0,1]\times [0,\infty),y\in[0,r]\times [0,\infty)$ we have that $h(\lambda x +(1-\lambda)y) \geq M_{\lambda,p}(f(x),g(y))$, and
            \item $\int h\,dx \leq (1+\delta) \int f\,dx$ for some $0 \leq \delta \leq d$.
        \end{itemize}
Then $\int |f-g|\,dx =O_{\lambda,p,\ell,s_0,s_1}(\sqrt{\delta}+\eta)\int f\,dx$
\end{prop}

Using the result for tubes we can find the result for all points where the functions are not too small by showing that that region is not too big and hence can be covered effectively with tubes.
\begin{prop}\label{SymDiff_BigLevelSets}
        Given $\lambda \in (0,1/2]$, $p \in (-1/2, \infty)$,  there exists $d=d_{ \lambda, p}$ such that the following holds.  Let $f,g,h\colon \mathbb{R}^2\rightarrow [0,1]$   be continuous functions with bounded support such that
        \begin{itemize}
            \item $\int f\,dx=\int g\,dx$, 
            \item for all $x,y\in\mathbb{R}^2$ we have that $h(\lambda x +(1-\lambda)y) \geq M_{\lambda,p}(f(x),g(y))$, and
            \item $\int h\,dx \leq (1+\delta) \int f\,dx$ for some $0 \leq \delta \leq d$.
        \end{itemize}
Let $F_{0.1}:=\{x\in\mathbb{R}^2: f(x)>0.1\}$, then there is some $v\in\mathbb{R}^2$ so that $\int_{F_{0.1}} |f(x)-g(x+v)|\,dx =O_{\lambda,p}\left(\sqrt{\delta}\right)\int f\,dx$
\end{prop}

\subsection{Tubes: Proof of \Cref{symmetric_diff_2_tube}}
We use the following quick consequence of the 1-dimensional result
\begin{lem}\label{symmetric_diff_1_half}
     Given $\lambda \in (0,1/2]$, $p \in (-1/2, \infty)$, $\ell, s_0,s_1 \in (0, \infty)$  there exists $d=d_{ \lambda, p, \ell, s_0, s_1}$ such that the following holds.  Let $f\colon [0,\infty)\rightarrow [0,s_1]$, $g\colon [0,\infty)\rightarrow [0,s_1]$ and $h\colon  [0,\infty)\rightarrow [0,s_1]$ be continuous functions with bounded support such that 
     \begin{itemize}
         \item $\int f\,dx=\int g\,dx$,
         \item  $\int f(x)\, dx \leq  \ell \int_{[0,1]}f(x)\,dx $, 
         \item $\int g(x)\, dx \leq  \ell \int_{[0,1]}g(x)\,dx $, and
         \item  $|\{x\in[0,1]: \min\{g(x),f(x)\}<s_0\}|\leq \xi$
         \item for all $x,y$ we have that $h(\lambda x +(1-\lambda)y) \geq M_{\lambda,p}(f(x),g(y))$, and
         \item $\int h\,dx \leq (1+\delta)\int f\,dx$ for some $0 \leq \delta \leq d$.
     \end{itemize}
     Then $\int |f-g|\,dx=O_{\lambda, p, \ell, s_0, s_1}(\delta^{1/2}+\xi)\int f\,dx$.
\end{lem}
\begin{proof}[Proof of \Cref{symmetric_diff_1_half}]
By \Cref{symmetric_diff_1D}, we find that there exists some $v\in\mathbb{R}$ so that 
$$\int_\mathbb{R} |f(x)-g(x-v)|\,dx=O_{\lambda, p}\left(\sqrt{\delta}\right)\int f\,dx.$$
We aim to show that not much changes replacing $v$ by 0. Assuming $v> 0$ (the case $v<0$ follows analogously), we find that
$$\int_\mathbb{R} |f(x)-g(x-v)|\,dx\geq \int_{[0,v]} |f(x)-g(x-v)|\,dx=\int_{[0,v]} f(x)\,dx\geq \min\left\{\frac{\int f}{\ell}, (v-\xi)\cdot s_0\right\},$$
where the last minimum depends on whether $v\geq1$ or $v<1$. We can exclude the former case by choosing $d=d_{\lambda,p,\ell,s_0,s_1}$ sufficiently small and noting that it would imply
$$\frac{1}{\ell}\leq O_{\lambda,p}\left(\sqrt{\delta}\right)\leq O_{\lambda,p}(\sqrt{d}),$$
a contradiction.
Hence, we find
$$v\leq \xi+ O_{\lambda,p}\left(\frac{\sqrt{\delta}}{s_0}\right)\int f\,dx,$$
so the translate is small. To show that $g$ doesn't change much when translated, consider level sets $G_t:=\{x\in\mathbb{R}: g(x)>0\}$. By \Cref{almostconvexlevelsets} we may assume that they're mostly convex, viz
$$\int_{0}^{s_1}|\co(G_t)\setminus G_t|\,dt = O_{\lambda,p}(\delta)\int f\,dx.$$
Hence, we find
\begin{align*}
\int |g(x)-g(x-v)|\,dx&= \int_{0}^{s_1}|G_t\triangle (G_t+v)|\,dt\\
&\leq \int_{0}^{s_1}\left(|G_t\triangle \co(G_t)|+|\co(G_t)\triangle (\co(G_t)+v)|+|(\co(G_t)+v)\triangle (G_t+v)|\right)\,dt\\
&\leq \int_{0}^{s_1}\left(|\co(G_t)\setminus G_t|+v+|\co(G_t)\setminus G_t|\right)\,dt\\
&\leq v\cdot s_1+O_{\lambda,p}(\delta)\int f\,dx \leq \xi s_1+ O_{\lambda,p}\left(\frac{s_1}{s_0}\sqrt{\delta}\right)\int f\,dx\leq O_{\lambda,p,s_0,s_1}\left(\sqrt{\delta}+\xi\right)\int f\,dx
\end{align*}
Combining these two results we find
$$\int |f(x)-g(x)|\,dx\leq \int \left(|f(x)-g(x-v)|+|g(x-v)-g(x)|\right)\,dx\leq O_{\lambda,p,s_0,s_1}\left(\sqrt{\delta}+\xi\right)\int f\,dx,$$
which proves the lemma.
\end{proof}

With this lemma in place, we can prove \Cref{symmetric_diff_2_tube}.
\begin{proof}[Proof of \Cref{symmetric_diff_2_tube} using \Cref{symmetric_diff_1_half}]
Consider the maps $T\colon[0,1]\to[0,r]$ defined by $\frac{dT}{dx}(x)=\frac{\int f(x,y)\,dy}{\int g(T(x),y)\,dy}$, so that for all $t\in[0,1]$, we have
$$\int_{[0,t]\times[0,\infty)} f(x,y)\,dxdy=\int_{[0,T(t)]\times[0,\infty)} g(x,y)\,dxdy $$

For fixed $x \in [0,1]$ and $y_1,y_2 \in [0,\infty)$, we have $$h(\lambda x+ (1-\lambda) T(x),\lambda y_1 +(1-\lambda) y_2 ) \geq M_{\lambda,p}(f(x,y_1),g(T(x),y_2)) .$$

By Borell-Brascamp-Lieb inequality in dimension $1$, we deduce
$$\int_{\mathbb{R}_{\geq 0}}h(\lambda x+ (1-\lambda) T(x),y)\, dy\geq M_{\lambda,p/(1+p)}\left(1,\frac{1}{\frac{dT}{dx}(x)}\right) \int_{\mathbb{R}_{\geq 0}}f(x,y)\,dy.$$

Hence, we can find $\delta_x \geq0$ for every $x\in[0,1]$ such that 
$$\int_{\mathbb{R}_{\geq 0}}h(\lambda x+ (1-\lambda) T(x),y)\, dy= (1+\delta_x)M_{\lambda,p/(1+p)}\left(1,\frac{1}{\frac{dT}{dx}(x)}\right) \int_{\mathbb{R}_{\geq 0}}f(x,y)\,dy.$$

Integrating over $x$, we get
\begin{align*}
    \begin{split}
        (1+\delta)\int_{[0,1]} \int_{\mathbb{R}_{\geq 0}}f(x,y)\,dydx &\geq \int_{[0,\lambda+(1-\lambda)r]} \int_{\mathbb{R}_{\geq 0}}h(x,y)\,dydx\\
        &= \int_{[0,1]} \left(\lambda +(1-\lambda)\frac{dT}{dx}(x)\right)\left[\int_{\mathbb{R}_{\geq 0}}h(\lambda x+(1-\lambda)T(x),y)\,dy\right]\,dx\\
        &= \int_{[0,1]} \left(\lambda +(1-\lambda)\frac{dT}{dx}(x)\right)(1+\delta_x)M_{\lambda,p/(1+p)}\left(1,\frac{1}{\frac{dT}{dx}(x)}\right) \left[\int_{\mathbb{R}_{\geq 0}}f(x,y)\,dy\right]\,dx\\
        &= \int_{[0,1]} (1+\delta_x)\frac{M_{\lambda,p/(1+p)}\left(1,\frac{1}{\frac{dT}{dx}(x)}\right)}{M_{\lambda,-1}\left(1,\frac{1}{\frac{dT}{dx}(x)}\right)}\left[\int_{\mathbb{R}_{\geq 0}}f(x,y)\,dy\right]\,dx
    \end{split}
\end{align*}
Hence,
\begin{equation}\label{twoconclusioneq}\delta  \int_{[0,1]} \int_{\mathbb{R}_{\geq 0}}f(x,y)\,dydx \geq \int_{[0,1]} \left(-1+(1+\delta_x)\frac{M_{\lambda,p/(1+p)}\left(1,\frac{1}{\frac{dT}{dx}(x)}\right)}{M_{\lambda,-1}\left(1,\frac{1}{\frac{dT}{dx}(x)}\right)}\right)\left[\int_{\mathbb{R}_{\geq 0}}f(x,y)\,dy\right]\,dx\end{equation}

Recall $p>-1/2$ and hence $p/(1+p)>-1$ and therefore $M_{\lambda,p/(1+p)}\left(1,\frac{1}{\frac{dT}{dx}(x)}\right) \geq M_{\lambda,-1}\left(1,\frac{1}{\frac{dT}{dx}(x)}\right)$. Hence, we may conclude
\begin{equation}\label{smalldeltax}\delta  \int_{[0,1]} \int_{\mathbb{R}_{\geq 0}}f(x,y)\,dydx \geq \int_{[0,1]} \delta_x \left[\int_{\mathbb{R}_{\geq 0}}f(x,y)\,dy\right]\,dx.
\end{equation}

With this we can ``clean up" the interval $[0,1]$ to find that $\frac{dT}{dx}$ is mostly well-behaved and in most fibres a significant part of the mass of $f$ (and $g$) lies inside the box $[0,1]^2$ (and $[0,r]\times[0,1]$ respectively). Consider the following subset of $[0,1]$ which we'll show contain little mass of $f$;
\begin{align*}
    I&:=\big\{x\in[0,1]: |\{y\in[0,1]:f(x,y)<s_0\}|\leq1/2 \text{ and }|\{y\in[0,1]:g(T(x),y)<s_0\}|\leq 1/2 \big\}
\end{align*}
Indeed, by assumption, we find that 
$$|[0,1]\setminus I|\cdot \frac12\leq |\{x\in[0,1]^2: f(x)<s_0\}|+|\{x\in[0,r]\times[0,1]: g(x)<s_0\}|\leq  \eta.$$
Note that since for all $x\in[0,1]$, we have $\int f(x,y)\,dy\leq \ell s_1$, we find that
\begin{equation}\label{littleoutsideIeq}
    \int_{([0,1]\setminus I)\times \mathbb{R}^+}f(x,y)\,dxdy\leq 2\eta \ell s_1,
\end{equation}
which is negligible for us, so we'll focus on what happens in $I$. One thing we know about $x\in I$, is that
$$\int f(x,y)\,dy,\int g(T(x),y)\,dy \geq \frac12 s_0,$$
and by assumption $\int f(x,y)\,dy,\int g(T(x),y)\,dy\leq \ell s_1$, so
$\max\left\{\frac{dT}{dx}(x),\frac{1}{\frac{dT}{dx}(x)}\right\}\leq \frac{\ell s_1}{\frac12 s_0}=O_{s_0,s_1,\ell}(1)$

Hence, for $x\in I$, we have $\frac{M_{\lambda,p/(1+p)}\left(1,\frac{1}{\frac{dT}{dx}(x)}\right)}{M_{\lambda,-1}\left(1,\frac{1}{\frac{dT}{dx}(x)}\right)} = 1+\Theta_{\lambda, p,\ell,s_0,s_1}\left(\left|\frac{dT}{dx}(x)-1\right|^2\right)$, so that \Cref{twoconclusioneq} additionally implies
$$\delta  \int_{[0,1]} \int_{\mathbb{R}_{\geq 0}}f(x,y)\,dydx \geq \int_{I} \Theta_{\lambda, p,\ell,s_0,s_1}\left(\left|\frac{dT}{dx}(x)-1\right|^2\right)\left[ \int_{\mathbb{R}_{\geq 0}}f(x,y)\,dy\right]\,dx.$$
Given that for $x \in I$, we have $\int_{\mathbb{R}_{\geq 0}}f(x,y)\,dy=\Theta_{s_0,s_1, \ell}(1)$, we deduce
$$\Theta_{\lambda, p, s_0,s_1, \ell}(\delta) \geq \int_{I}\left|\frac{dT}{dx}(x)-1\right|^2\,dx.$$
By the Cauchy-Schwarz inequality we get
\begin{equation}\label{eqmike}
   \Theta_{\lambda, p, s_0,s_1, \ell}\left(\sqrt{\delta}\right) \geq \int_{I}\left|\frac{dT}{dx}(x)-1\right|\,dx. 
\end{equation}

Given $T(0)=0$, by integrating, we get that for $x \in [0,1]$, $|T(x)-x|\leq \Theta_{\lambda, p, s_0,s_1, \ell}\left(\delta^{1/2}+\eta\right)$.

For fixed $x \in I$, the functions $h(\lambda x+ (1-\lambda) T(x),y)$, $f(x,y)$ and $g(T(x),y)$ satisfy the hypothesis of \Cref{symmetric_diff_1_half} for some $\xi_x\geq 1/2$ with the property that $\int_{[0,1]} \xi_x \,dx\leq \eta$, so that
\begin{equation}\label{eqjamesbond}
    \int_{\mathbb{R}_{\geq 0}}\left|f(x,y)-\frac{dT}{dx}(x)g(T(x),y)\right|\,dy = O_{\lambda,p, \ell, s_0,s_1, \frac{dT}{dx}(x)}\left(\delta_x^{1/2}+\xi_x\right)\int_{\mathbb{R}_{\geq 0}}f(x,y)\,dy=O_{\lambda,p, \ell, s_0,s_1}\left(\delta_x^{1/2}+\xi_x\right)\int_{\mathbb{R}_{\geq 0}}f(x,y)\,dy,
\end{equation}
where the last inequality follows from the fact that $\frac{dT}{dx}(x)=\Theta_{s_0,s_1, \ell}(1)$.
Combining \Cref{smalldeltax} with \Cref{eqjamesbond}, and using Cauchy-Schwarz inequality,  we get
\begin{align*}
         \int_{I}\int_{\mathbb{R}_{\geq 0}}\left|f(x,y)-\frac{dT}{dx}(x)g(T(x),y)\right|\,dydx&\leq \int_{I}O_{\lambda,p, \ell, s_0,s_1}\left(\delta_x^{1/2}+\xi_x\right)\left[\int_{\mathbb{R}_{\geq 0}}f(x,y)\,dy\right]\,dx\\
         &\leq O_{\lambda,p, \ell, s_0,s_1}\left(\delta^{1/2}+\eta\right) \int_{I} \int_{\mathbb{R}_{\geq 0}}f(x,y)\,dydx\\
         &\leq O_{\lambda,p, \ell, s_0,s_1}\left(\delta^{1/2}+\eta\right) \int_{[0,1]} \int_{\mathbb{R}_{\geq 0}}f(x,y)\,dydx,
\end{align*}
where in the last inequality we used that almost no mass of $f$ lies outside of $I$.

Combining the last inequality with \Cref{eqmike} and the fact that $\int_{\mathbb{R}_{\geq 0}}g(T(x),y)\,dy=\Theta_{\ell,s_0,s_1}(1)$, we conclude
\begin{align*}
\int_{I}\int_{\mathbb{R}_{\geq 0}}|f(x,y)-g(T(x),y)|\,dydx &\leq\int_{I}\int_{\mathbb{R}_{\geq 0}}\left(\left|f(x,y)-\frac{dT}{dx}(x)g(T(x),y)\right|+g(T(x),y)\left|\frac{dT}{dx}(x)-1\right|\right)\,dydx \\
&\leq O_{\lambda,p, \ell, s_0,s_1}\left(\delta^{1/2}+\eta\right) \int_{[0,1]} \int_{\mathbb{R}_{\geq 0}}f(x,y)\,dydx.
\end{align*}

To translate this information to get control over $\int |f-g|$, we consider the auxiliary function $k\colon[0,1]\times[0,\infty)\to\mathbb{R}_{\geq 0}, (x,y)\mapsto g(T(x),y)$, so that (recalling \Cref{littleoutsideIeq})
\begin{multline*}
 \int_{[0,1]}\int_{\mathbb{R}^+} |f-k|\,dydx\leq   \int_{[0,1]\setminus I}\int_{\mathbb{R}^+} (f+k)\,dydx+\int_I\int_{\mathbb{R}^+} |f-k|\,dydx\\
 \leq    4\eta\ell s_1+\int_{I}\int_{\mathbb{R}_{\geq 0}}|f(x,y)-g(T(x),y)|\,dydx =O_{\lambda,p, \ell, s_0,s_1}\left(\delta^{1/2}+\eta\right) \int_{[0,1]} \int_{\mathbb{R}_{\geq 0}}f(x,y)\,dydx.
\end{multline*}
We shall show that also $\int |g-k|\,dx$ is small. To that end, consider the level sets $G_t:=\{x\in\mathbb{R}^2: g(x)>t\}$ and $K_t:=\{x\in\mathbb{R}^2:k(x)>t\}$, so that
$$\int |g-k|\,dx=\int_{0}^{s_1}|G_t\triangle K_t|\,dt.$$
By \Cref{almostconvexlevelsets}, we may assume
$$\int_{0}^{s_1} |\co(G_t)\setminus G_t|\,dt\leq O_{\lambda,p}(\delta)\int f\,dx.$$
Additionally, consider the map $\mathcal{T}\colon [0,1]\times[0,\infty)\to [0,r]\times[0,\infty), (x,y)\mapsto(T(x),y)$ so that $\mathcal{T}(K_t)=G_t$. Since $\frac{dT}{dx}(x)=\Theta_{s_0,s_1,\ell}(1)$ for $x\in I$ and $|[0,1]\setminus I|\leq 2\eta$, we have that
$$\int_{0}^{s_1} |\mathcal{T}^{-1}(\co(G_t))\triangle K_t|\,dt=\int_{0}^{s_1} |\mathcal{T}^{-1}(\co(G_t)\setminus G_t)|\,dt\leq O_{\lambda,p,s_0,s_1,\ell}(\delta+\eta)\int f\,dx.$$
Hence, it suffices to prove that
\begin{equation}\label{movingslightlyeq}
\int_{0}^{s_1} |\mathcal{T}^{-1}(\co(G_t))\triangle \co(G_t)|\,dt\leq O_{\lambda, p, s_0,s_1, \ell}\left(\delta^{1/2}+\eta\right)\int f\,dx
\end{equation}
is small. We shall approach this fibre by fibre as follows. For a given $t$, we can express $|\mathcal{T}(\co(G_t))\triangle \co(G_t)|$ by considering the line segments $\co(G_t)_y:=\{x: (x,y)\in\co(G_t)\}$, so that
\begin{align*}
|\mathcal{T}(\co(G_t))\triangle \co(G_t)|&=\int_{0}^\infty |\co(G_t)_y\triangle T(\co(G_t)_y)|\,dy\\
&\leq \int_{0}^\infty \left(|\min \co(G_t)_y-T(\min \co(G_t)_y)|+|\max \co(G_t)_y-T(\max \co(G_t)_y)|\right)\,dy.
\end{align*}
Recall that (\Cref{eqmike}) for all $x\in[0,1]$ we have $|T(x)-x|\leq O_{\lambda, p, s_0,s_1, \ell}\left(\delta^{1/2}+\eta\right)$, so that 
\begin{align*}
|\mathcal{T}(\co(G_t))\triangle \co(G_t)|&\leq O_{\lambda, p, s_0,s_1, \ell}\left(\delta^{1/2}+\eta\right) \left|\{y:\co(G_t)_y\neq\emptyset\}\right|,
\end{align*}
and, integrating over $t$, that
$$\int_{0}^{s_1} |\mathcal{T}^{-1}(\co(G_t))\triangle \co(G_t)|\,dt\leq O_{\lambda, p, s_0,s_1, \ell}\left(\delta^{1/2}+\eta\right) \left|\{(t,y)\in\mathbb{R}^2_{\geq0}:\co(G_t)_y\neq\emptyset\}\right|.$$
Hence it suffices to show that $\left|\{(t,y):\co(G_t)_y\neq\emptyset\}\right|$ is not too big. Consider the function $m\colon[0,\infty)\to[0,s_1],y\mapsto \max_{x\in[0,r]}g(x,y)$. Per the following claim it suffices to control $\int_0^\infty m\,dy$. 

\begin{clm}
$\int_0^\infty m\,dy\geq \frac{1}{4} \left|\{(t,y)\in\mathbb{R}^2_{\geq0}:\co(G_t)_y\neq\emptyset\}\right|$ \end{clm}
\begin{proof}[Proof of Claim.]
By definition $\int_0^\infty m\,dy=\left|\{(t,y)\in\mathbb{R}^2_{\geq0}:(G_t)_y\neq\emptyset\}\right|$. Recall that $|G_t|\geq \frac12|\co(G_t)|$ for all $t\in\mathbb{R}_{\geq0}$. For fixed $t$, $\{y\in\mathbb{R}_{\geq0}:(G_t)_y\neq\emptyset\}$ is the projection of (fairly convex) $G_t\subset\mathbb{R}^2$ onto one of the coordinate axis. The projection preserves the approximate convexity as follows, so that for fixed $t$, we have
$$\frac{|\{y\in\mathbb{R}_{\geq0}:(G_t)_y\neq\emptyset\}|}{ |\{y\in\mathbb{R}_{\geq0}:\co(G_t)_y\neq\emptyset\}|}\geq 1-\sqrt{\frac{|G_t|}{|\co(G_t)|}}\geq 1-\sqrt{1/2}\geq 1/4.$$
Integrating over $t$ proves the claim.
\end{proof}

If for a given $y$, we have $m(y)\geq \alpha$, then there exists an $x_y$ so that $g(x_y,y)\geq \alpha$. By definition of $h$, we find that
for at least $\frac12\leq 1-\eta$ of the $z\in [0,1]^2$, we have $f(z)\geq s_0$. Hence, for all $(x',y')\in [0,\lambda+(1-\lambda)r]\times [(1-\lambda)y, \lambda+(1-\lambda)y]$, we have some $z\in[0,1]^2$ so that 
 $$h(x',y')\geq h(\lambda z+ (1-\lambda)(x_y,y))\geq \M(f(z),g(x_y,y))\geq \M(s_0,\alpha).$$

Hence, we find 

$$\int_{[0,\lambda+(1-\lambda)r]}\int_{[(1-\lambda)y, \lambda+(1-\lambda)y]}h(x',y')\,dy'dx'\geq \frac{\lambda^2}{2} \M(s_0,\alpha)\geq \frac{\lambda^2}{2}\min\{s_0,m(y)\}\geq \frac{\lambda^2 s_0}{2s_1} m(y).$$
Integrating this inequality over $y$, we find that
\begin{multline*}
\frac{\lambda^2 s_0}{2s_1}\int_{0}^\infty m(y)\,dy\leq \int_0^\infty \left(\int_{[0,\lambda+(1-\lambda)r]}\int_{[(1-\lambda)y, \lambda+(1-\lambda)y]}h(x',y')\,dy'dx'\right)\,dy\\
=\frac{\lambda}{1-\lambda}\int_{[0,\lambda+(1-\lambda)r]}\int_{[0,\infty)}h(x',y')\,dy'dx'.
\end{multline*}
Since $\int_{\mathbb{R}^2} h\,dx\leq 2\int_{\mathbb{R}^2} f\,dx$, this implies  that 
$$\int_0^\infty m\,dy=O_{\lambda,s_0,s_1}\left(\int_{\mathbb{R}^2} f\,dx\right).$$
Combining the last few steps, we find
\begin{multline*}
\int_{0}^{s_1} \left|\mathcal{T}^{-1}(\co(G_t))\triangle \co(G_t)\right|\,dt \leq O_{\lambda, p, s_0,s_1, \ell}\left(\sqrt{\delta}\right) \left|\{(t,y):\co(G_t)_y\neq\emptyset\}\right|\\
\leq O_{\lambda, p, s_0,s_1, \ell}\left(\delta^{1/2}+\eta\right)\int m\,dy\leq O_{\lambda, p, s_0,s_1, \ell}\left(\delta^{1/2}+\eta\right)\int f\,dx.
\end{multline*}
Putting everything together, we now find
\begin{align*}
&\int |f-g|\,dx\leq \int \left(|f-k|+|k-g|\right)\,dx\leq O_{\lambda,p, \ell, s_0,s_1}\left(\delta^{1/2}+\eta\right)\int f \,dx +\int_{0}^{s_1}|G_t\triangle K_t|\,dt\\
&\leq O_{\lambda,p, \ell, s_0,s_1}\left(\delta^{1/2}+\eta\right)\int f\,dx +\int_{0}^{s_1}\left(|G_t\triangle \co(G_t)|+|\co(G_t)\triangle \mathcal{T}^{-1}(\co(G_t))|+|\mathcal{T}^{-1}(\co(G_t)) \triangle \mathcal{T}^{-1}(G_t)|\right)\,dt\\
&\leq O_{\lambda,p, \ell, s_0,s_1}\left(\delta^{1/2}+\eta\right)\int f\,dx,
\end{align*}
which concludes the proof of the proposition.
\end{proof}

\subsection{Large level sets: Proof of \Cref{SymDiff_BigLevelSets}}

In order to proof this proposition, we will show that there exists a large ball in $F_{0.1}$ in which  $\co(F_{0.1})\setminus F_{0.1}$ is very small. Within this ball we can properly position the tubes needed to cover all of $F_{0.1}$ using the following lemma.

\begin{lem}\label{positioningwithinthe2dconelem}
    Given $\lambda \in (0,1/2]$, $p \in (-1/2, \infty)$, $r, \ell \in (0, \infty)$, where $r$ is sufficiently small,   there exists $d=d_{\lambda, p, r, \ell}>0$ such that for $0 \leq \delta < d$ the following holds. Let $w=(r/100,0)$. Let $C:=\{(x,y): x>0, |y|\leq x\}$. Let $B'$ be the ball of radius $r$ centred at the origin intersected with $C$. Let $f,g, h\colon C\rightarrow [0,1]$ be continuous functions with bounded support such that  
    \begin{itemize}
        \item $\int_{C} f\, dx= \int_{C} g\, dx $,
        \item $\int_{C}f\,dx \leq  \ell \int_{B' }f\,dx $,
        \item for all $x,y \in C$ we have $h(\lambda x +(1-\lambda)y) \geq M_{\lambda,p}(f(x),g(y))$,
        \item $\int_{C}h\,dx = (1+\delta) \int_{C} f\,dx$,
        \item $|B'\setminus F_{\beta}|+|B'\setminus G_{\beta}|\leq \delta$
    \end{itemize}
    Then for every unit vectors $f_1, f_2 \in\mathbb{R}^2$ with $\text{span}\{f_1, f_2\}= \text{span}\{e_1,e_2\}$, $|<f_1, e_1>| \leq 1/100$ and $|<f_1, f_2>| \geq 1/200$, there exists $v \in B(o,O_{ \lambda, p, r, \ell,\beta}\left(\sqrt{\delta}\right))$  and $r' \in [0, 2r]$so that 
    \begin{itemize}
        \item $\int_{w+R_{r/100}^{f_1,f_2}}f\,dx=\int_{w+v+R_{r'/100}^{f_1,f_2}}g\,dx$, and
        \item $\int_{w+v/2+R_{(r+r')/200}^{f_1,f_2}}h\,dx\leq (1+\delta)\int_{w+R_{r/100}^{f_1,f_2}}f\,dx$. 
    \end{itemize}
\end{lem}

To then cover the $F_{0.01}$ we use the following observation.

\begin{obs}\label{coveringaballwithtubes}
    For every $m \in (0, \infty)$ there exists $j =O_{r}(m)$ such that  the following holds. There exists a collection of pairs of unit vectors $\{f_1^i,f_2^i\}_{i=1}^j \subset \mathbb{R}^2$ such that  $|<f_1^i, e_1>| \leq 1/100$ and $|<f_1^i, f_2^i>| \geq 1/200$ and $\bigcup_i R_{r}^{f_1^i,f_2^i} \supset B(0, m)$.
\end{obs}

We first prove \Cref{positioningwithinthe2dconelem}

\begin{proof}[Proof of \Cref{positioningwithinthe2dconelem}]
We construct the appropriate $r'$ and $v$ by cutting with consecutive hyperplanes which we show to be close together. First consider the halfspace $G_1^{-}$ containing the origin with defining hyperplane $G_1:=w+\mathbb{R}\times \{0\}$, where $\mathbb{R}\times \{0\}$ is considered in the basis $f_1,f_2$, i.e., this is the hyperplane $w+\mathbb{R}f_1$ in either base. Let $G_1'$ be the parallel hyperplane with the property that
$$\int_{C\cap G_1^{-}} f\,dx =\int_{C\cap G_1'^{-}} g\,dx.$$
We shall show that the distance between these two planes is at most $O_{\lambda,p,\ell}\left(\sqrt{\delta}\right)$. 
First note that we have
$$\int_{C\cap (\lambda G_1^-+(1-\lambda)G_1'^-)}h\,dx-\int_{C\cap G_1^-}f\,dx\leq \delta\int_{C}f\,dx\leq O_{\ell}(\delta)\int_{C\cap G_1^-}f\,dx,$$
where in the last step we used that $C\cap G_1^-\subset C\cap B'$ and $|C\cap G_1^-|=\Omega(|C\cap B'|)$. Hence, by \Cref{SimilarSupportLem}, we find that 
$$|(C\cap G_1^-)\triangle(C\cap G_1'^-)|=|(supp (f\mid_{C\cap G_1^-}))\triangle (supp (g\mid_{C\cap G_1'^-}))|=O_{p,\lambda,\ell}\left(\sqrt{\delta}\right),$$
which of course implies the planes are distance at most $O_{p,\lambda,\ell}\left(\sqrt{\delta}\right)$ apart.

We proceed similarly for the two other hyperplanes defining $w+R_{r/100}^{f_1,f_2}$. 
Consider the hyperplane $G_2:=w+\{0\}\times\mathbb{R}$ in the basis $f_1,f_2$, i.e., the hyperplane $w+\mathbb{R}f_2$ in the basis $e_1,e_2$. As before Let $G_2'$ be the parallel hyperplane so that for the parallel halfspaces $G_2^+$ and $(G_2')^+$, we have
$$\int_{C\cap G_1^+\cap G_2^+}f\,dx=\int_{C\cap (G'_1)^+\cap (G'_2)^+}g\,dx.$$
In order to show that these hyperplanes are (again) close together, we would like to use (again) \Cref{SimilarSupportLem}, so we restrict our attention further to the bit of the domain where $f(x)\geq 0.1$ as follows.  Consider the hyperplane $G_3$ parallel to $G_1$ containing $2w$, i.e., $w+\mathbb{R}f_1$ and $G_3'$ parallel to these hyperplanes so that
$$\int_{C\cap G_1^+\cap G_2^+\cap G_3^-}f\,dx=\int_{C\cap (G'_1)^+\cap (G'_2)^+\cap (G_3')^-}g\,dx.$$
Since $C\cap G_3^-$ and $C\cap (G_3')^-$ are subsets of $C\cap B$, we still have that the functions $f$ and $g$ take values in $[0.1,1]$ on most of these domains. Since $|C\cap G_1^+\cap G_2^+\cap G_3^-|=\Omega_{r}(|C\cap B|)$, we have 
$$\int_{C\cap G_1^+\cap G_2^+\cap G_3^-}f\,dx\geq \Omega_{r}(|C\cap B|)\geq \Omega_{r}\left( \int_{C\cap B}fdx\right)\geq \Omega_{r,\ell}\left( \int_{ C}f\,dx\right),$$
so that 
$$\int_{\lambda(C\cap G_1^+\cap G_2^+\cap G_3^-)+(1-\lambda)(C\cap (G'_1)^+\cap (G'_2)^+\cap (G_3')^-)}h\,dx-\int_{C\cap G_1^+\cap G_2^+\cap G_3^-}f\,dx\leq \delta\int_C f\,dx=\Omega_{r,\ell}\left(\int_{C\cap G_1^+\cap G_2^+\cap G_3^-}f\,dx\right).$$
Analogous to the previous application, by \Cref{SimilarSupportLem}, we find that up to translation 
$$|(C\cap G_1^+\cap G_2^+\cap G_3^-)\triangle (C\cap (G'_1)^+\cap (G'_2)^+\cap (G_3')^-)|=O_{p,\lambda,\ell}\left(\sqrt{\delta}\right).$$
In particular, that implies that indeed $G_2$ and $G_2'$ are at most a distance $O_{p,\lambda,\ell}\left(\sqrt{\delta}\right)$ apart. Finally, let $G_4$ be the remaining hyperplane defining $w+R_{r/100}^{f_1,f_2}$, i.e., the plane parallel to $G_2$ so that $w+R_{r/100}^{f_1,f_2}=G_1^+\cap G_2^+\cap G_4^-$. Analogously to $G_2$ and $G_2'$, we find that $G_4$ and $G_4'$ are at distance at most $O_{p,\lambda,\ell}\left(\sqrt{\delta}\right)$ from one another, which allows us to conclude.

Let $v\in B\left(o,O_{\lambda,p,r,\ell}\left(\sqrt{\delta}\right)\right)$ so that $v+G_1=G_1'$ and $v+G_2=G_2'$. Let $r'$ be so that $w+v+R_{r'/100}^{f_1,f_2}=(G_1')^+\cap (G_2')^+\cap (G_4')^-$, which thus implies $r'=r\pm O_{\lambda,p,r,\ell}\left(\sqrt{\delta}\right)$
\end{proof}

\begin{proof}[Proof of \Cref{SymDiff_BigLevelSets}]
For notational convenience normalise so that $\int f\,dx=\int g\,dx=1$ by rescaling the domain. We may assume $\sup{f(x)}=1$.

Analogously to \Cref{initialcleanupreductionsection} where we show that \Cref{symmetric_diff} follows from \Cref{symmetric_diff_technical}, we may additionally assume that  for $i=1,2,3$ we have (for $C_i=C_n^i$ as defined in \Cref{Cnidefn})
$$\int_{C_i}f\,dx=\int_{C_i}g\,dx,$$
and for some $r,\beta=\Omega_{\lambda,p}(1)$, we have that if we let $B'=B(o,r)$ the ball of radius $r$ centred at the origin, then
$$|B'\setminus F_\beta|+|B'\setminus G_\beta|=O_{\lambda,p}(\delta)\text{ and }|\co(F_{\beta})|\leq 2|F_\beta|\leq 2\beta^{-1}=O_{\lambda,p}(1).$$
Note that by Borell-Brascamp-Lieb, we have that $\int_{C_i}h\,dx\geq \int_{C_i}f\,dx$ and thus
$$\int_{C_i}h\,dx-\int_{C_i}f\,dx\leq \delta \int f\,dx\leq O_{\lambda,p}(\delta)\int_{C_i}f\,dx,$$
where in the last inequality we used that each of the cones contains a $\Omega_{\lambda,p}(1)$ proportion of $\int f\,dx$ (even when restricted to $C_i\cap B'$.

We consider each of the cones separately; in particular, we consider $C_1=\{(x,y): x>0, |y|\leq x\}$ and the others will follow analogously. Let $w=(r/100,0)$. Since $|B'\cap C_1\cap F_{\beta}|\geq \Omega_{\lambda,p}(1)$ and $|\co(F_\beta)|=O_{\lambda,p}(1)$, we find that there is an $m=O_{\lambda,p}(1)$ so that $F_\beta\subset B(w, m)$.  Apply \Cref{coveringaballwithtubes} with parameters $r/100$ and $m$  to find $j=O_{\lambda,p}(1)$ a collection $\{f_1^i,f_2^i\}_{i=1}^j$ so that $B'\cap C_1\subset w+\bigcup_i R_{r/100}^{f_1^i,f_2^i}$. By \Cref{positioningwithinthe2dconelem}, we can find $v_i\in B\left(o,O_{\lambda,p}\left(\sqrt{\delta}\right)\right)$ so that 
\begin{itemize}    \item$\int_{w+R_{r/100}^{f^i_1,f^i_2}}f\,dx=\int_{w+v_i+R_{r'/100}^{f^i_1,f^i_2}}g\,dx$, and
    \item $\int_{w+v_i/2+R_{(r+r')/200}^{f^i_1,f^i_2}}h\,dx\leq (1+\delta)\int_{w+R_{r/100}^{f^i_1,f^i_2}}f\,dx$. 
\end{itemize}
Since $w+R_{r/100}^{f^i_1,f^i_2}$ intersects $B'$ in a big part, at least a $1/\ell=\Omega_{\lambda,p}(1)$ proportion of $f$ and $g$ is concentrated near the base of the tube $w+R_{r/100}^{f^i_1,f^i_2}$. Hence, we can apply \Cref{symmetric_diff_2_tube} with $s_1=1$, $s_0=\beta\geq \Omega_{\lambda,p}(1)$, and $\eta=O_{\lambda, p}(\delta)$, we find that 
$$\int_{w+R_{r/100}^{f^i_1,f^i_2}}|f(x)-g(x-v_i)|\,dx=O_{\lambda,p}\left(\sqrt{\delta}\right)\int f\,dx$$
 
Using that $v_i\in B\left(o,O_{\lambda,p}\left(\sqrt{\delta}\right)\right)$ we find by \Cref{smalltranslationsmallsymdiflem} that;
$$\int_{w+R_{r/100}^{f^i_1,f^i_2}}|f(x)-g(x)|\,dx=O_{\lambda,p}\left(\sqrt{\delta}\right)\int f\,dx$$
Summing this over all $i$ gives that
\begin{align*}
\int_{F_{\beta}}|f-g|\,dx&\leq \int_{B(m,w)}|f-g|\,dx\leq \int_{w+\bigcup_i R_{r/100}^{f_1^i,f_2^i}}|f-g|\,dx\leq \sum_i  \int_{w+ R_{r/100}^{f_1^i,f_2^i}}|f-g|\,dx\\
&\leq j\cdot O_{\lambda,p}\left(\sqrt{\delta}\right)\int f\,dx=O_{\lambda,p}\left(\sqrt{\delta}\right)\int f\,dx.
\end{align*}
Since $F_{0.1}\subset F_\beta$, this concludes the proof of the proposition.
\end{proof}

\subsection{Gluing together level sets: Proof of \Cref{symmetric_diff_2D}}
In this section, we'll see that applying \Cref{SymDiff_BigLevelSets} to the cut off functions $\min\{f,2^i\}$ shows that all level sets of $f$ have small symmetric difference to level sets of $g$ up to translation. Careful analysis allows us to conclude that those translations may all be assumed to be the same. The main idea is that \Cref{SymDiff_BigLevelSets} implies that level sets are close together 

\begin{proof}[Proof of \Cref{symmetric_diff_2D}]
As ever, consider the level sets $F_t,G_t,H_t$ and transport map $T\colon [0,1]\to[0,1]$ so that $\int_0^t |F_s|\,ds=\int_0^{T(t)} |G_s|\,ds$. Recall that by \Cref{almostconvexlevelsets}, we have that
$$\int \left(|\co(F_t)\setminus F_t|+|\co(G_t)\setminus G_t|\right)\,dt=O_{\lambda,p}(\delta)\int f\,dx$$
A particular consequence of \Cref{SymDiff_BigLevelSets} is that for some translate $v_0$, we have
$$\int_{0.1}^1 |F_t\triangle (v_0+G_t)|\,dt=O_{\lambda,p}\left(\sqrt{\delta}\right).$$
In order to extend this to all level sets, consider the functions
$$f_i:=\min\{f,2^{-i}\}, g_i:=\min\{g,T(2^{-i})\}, h_i:=\min\{h,\M(2^{-i},T(2^{-i}))\}, $$
so that $h_i\leq \M^*(f_i,g_i)$. Let $\delta_i:=\delta \frac{\int f\,dx}{\int f_i\,dx}$ (so that $\delta_i$ is non-decreasing in $i$) and note that\footnote{For more elaboration on the second inequality see \Cref{horizontalcutdecreasingdoubling}. }
$$\int h_i\,dx\leq \int h\,dx -\int (h-h_i)\,dx\leq \int h\,dx- \int (f-f_i)\,dx\leq \int f_i\,dx +\delta \int f\,dx\leq (1+\delta_i)\int f_i\,dx.$$

In the end we will only consider $f_i$ so that $\int f_i\,dx\geq \sqrt{\delta}\int f\,dx$ (which we'll see has $i=O_{\lambda,p}(-\log(\delta))$; beyond that it's trivial to control the symmetric difference. Hence, we may continue to assume $\delta_i$ is uniformly sufficiently small in terms of $\lambda$ and $p$ whenever we need it.

Since the level sets of $f_i$ are just the level sets of $f$ up to $2^{-i}$, we can again apply \Cref{SymDiff_BigLevelSets} to find a translate $v_i$
$$\int_{0.1\cdot 2^{-i}}^{2^{-i}} |F_t\triangle (v_i+G_t)|\,dt=O_{\lambda,p}\left(\sqrt{\delta_i}\right)\int f_i\,dx.$$
To get control over $v_i$ we note that consecutive integrals overlap significantly. Indeed, we have that
$$\int_{0.1\cdot 2^{-i}}^{0.5\cdot 2^{-i}} \left(|F_t\triangle (v_i+G_t)|+|F_t\triangle (v_{i+1}+G_t)|\right)\,dt\leq O_{\lambda,p}\left(\sqrt{\delta_{i+1}}\right)\int f_i\,dx,$$
so that  by the triangle inequality, we have
$$\int_{0.1\cdot 2^{-i}}^{0.5\cdot 2^{-i}} |F_t\triangle (v_i-v_{i+1})+F_t)|\,dt\leq O_{\lambda,p}\left(\sqrt{\delta_{i+1}}\right)\int f_i\,dx,$$
and recalling that $\int|\co(F_t)\setminus F_t|\,dt\leq O_{\lambda,p}(\delta)\int f\,dx$, this implies 
$$\int_{0.1\cdot 2^{-i}}^{0.5\cdot 2^{-i}} |\co(F_t)\triangle (v_i-v_{i+1})+\co(F_t))|\,dt\leq O_{\lambda,p}\left(\sqrt{\delta_{i+1}}\right)\int f_i\,dx+O_{\lambda,p}(\delta)\int f\,dx\leq O_{\lambda,p}\left(\sqrt{\delta_{i+1}}\right)\int f_i\,dx.$$
This implies that $v_i-v_{i+1}$ is small from the perspective of $F_{0.1\cdot 2^i}$. Indeed, if some vector $v$ is not in $\eta K$ for some (centered) convex set $K\subset\mathbb{R}^2$, then for any subset $X\subset K$, we have $|X\triangle (v+X)|\geq \Omega(\eta)|X|$. Hence, combining the above inequality with $\int_{0.1\cdot 2^{-i}}^{0.5\cdot 2^{-i}} |\co(F_t)|\,dt=\Omega_{\lambda,p}\left( \int f_i\,dx\right)$ and $\co(F_t)\subset\co(F_{0.1\cdot 2^{-i}})$ for all $t\in[0.1\cdot 2^{-i},0.5\cdot 2^{-i}]$, we derive that
$$v_i-v_{i+1}\in O_{\lambda,p}\left(\sqrt{\delta_i}\right)\co(F_{0.1\cdot 2^{-i}}).$$

By induction, we can derive that (aiming to have the same translate throughout)
$$v_j-v_0=\sum_{i=0}^{j-1} v_{i+1}-v_i\in \sum_{i=0}^{j-1}O_{\lambda,p}\left(\sqrt{\delta_i}\right)\co(F_{0.1\cdot 2^{-i}})\subset j\cdot O_{\lambda,p}\left(\sqrt{\delta_j}\right)\co(F_{0.1\cdot 2^{-j}}),$$
where in the last containment we use that $\delta_i$ is increasing in $i$ and $\co(F_{0.1\cdot 2^{-i}})$ form a nested sequence of sets.
Returning to the symmetric difference, this implies
$$\int_{2^{-i}}^{0.5\cdot 2^{-i}} |F_t\triangle (v_0+G_t)|\,dt\leq \int_{2^{-i}}^{0.5\cdot 2^{-i}} \left(|F_t\triangle (v_i-v_0)+F_t|+|F_t\triangle (v_i+G_t)|\right)\,dt,$$
where the latter is easily bounded and the former is controlled since (again) $|\co(F_t)\setminus F_t|$ is small. Indeed, we have 
$$|F_t\triangle (v_i-v_0)+F_t|\leq |\co(F_t)\triangle (v_i-v_0)+\co(F_t)|+2|\co(F_t)\setminus F_t|, $$
and for all $t\geq 0.1\cdot 2^{-i}$, since $v_i-v_0\in i\cdot O_{\lambda,p}\left(\sqrt{\delta_i}\right)\co(F_{0.1\cdot 2^{-i}})$ we have 
$$|\co(F_t)\triangle (v_i-v_0)+\co(F_t)|\leq i\cdot O_{\lambda,p}\left(\sqrt{\delta_i}\right)|\co(F_{0.1\cdot 2^{-i}})|\leq i\cdot O_{\lambda,p}\left(\sqrt{\delta_i}\right)|F_{0.1\cdot 2^{-i}}|.$$
Hence, we find
$$\int_{2^{-i}}^{0.5\cdot 2^{-i}} |F_t\triangle (v_i-v_0)+F_t|\,dt\leq 2^{-i}\cdot i\cdot O_{\lambda,p}\left(\sqrt{\delta_i}\right)|F_{0.1\cdot 2^{-i}}|\leq i\cdot O_{\lambda,p}\left(\sqrt{\delta_i}\right)\int f_i\,dx,$$
and thus
$$\int_{2^{-i}}^{0.5\cdot 2^{-i}} |F_t\triangle (v_0+G_t)|\,dt\leq  (i+1)\cdot O_{\lambda,p}\left(\sqrt{\delta_i}\right)\int f_i\,dx+O_{\lambda,p}(\delta)\int f\,dx\leq i\cdot O_{\lambda,p}\left(\sqrt{\delta_i}\right)\int f_i\,dx.$$
Recalling the definition of $\delta_i$ and summing over all $i\leq i_0$, where we pick $i_0$ minimal so that $\int f_i\leq\sqrt{\delta}\int f\,dx$, we find
\begin{multline*}
    \int_{0}^1 |F_t\triangle (G_t+v_0)|\,dt=\int f_{i_0}\,dx+\sum_{i=0}^{i_0} \int_{2^{-(i+1)}}^{ 2^{-i}} |F_t\triangle (v_0+G_t)|\,dt\\
    \leq \int f_{i_0}\,dx+\sum_{i=0}^{i_0} i\cdot O_{\lambda,p}\left(\sqrt{\delta_i}\right)\int f_i\,dx \leq O_{\lambda,p}\left(\sqrt{\delta}\right)\int f \,dx\cdot\left[ \sum_{i=0}^{i_0} i\cdot \sqrt{\frac{\int f_i\,dx}{\int f\,dx}}\right].
\end{multline*}
As we've used several times before, as long as $\int h_i\,dx\leq (1+\eta)\int f_i\,dx$ for some $\eta\ll_{\lambda,p}1$ (which is the case for $i\leq i_0$, we find that 
$$\int f_{i}\,dx-\int f_{i+1}\,dx=\int_{2^{-(i+1)}}^{2^{-i}}|F_t|\,dt\geq \Omega_{\lambda,p}\left(\int f_i\,dx\right).$$ Hence, the $\int f_{i}\,dx$ form a geometric sequence (as long as $i\leq i_0$) and thus $ \sum_{i=0}^{i_0} i\cdot \sqrt{\frac{\int f_i\,dx}{\int f\,dx}}=O_{\lambda,p}(1)$. This allows us to conclude
$$\int_{x\in\mathbb{R}^2} |f(x)-g(x-v_0)|\,dx=\int_{0}^1 |F_t\triangle (G_t+v_0)|\,dt=O_{\lambda,p}\left(\sqrt{\delta}\right)\int f\,dx .$$
\end{proof}

\section{Proof of \Cref{symmetric_diff} in $\mathbb R^n$}
\label{sect:proof thm1}

\subsection{Reduction from \Cref{symmetric_diff} to \Cref{symmetric_diff_reformulation_technical}}

Thanks to the arguments in \Cref{ndimensionalsymdiffsect}, to show the validity of \Cref{symmetric_diff} in $\mathbb R^n$ it suffices to prove \Cref{symmetric_diff_technical}.
We first show that the latter is implies by the following resut.

\begin{prop}\label{symmetric_diff_reformulation_technical}
    Given $n \in \mathbb{N}$, $i \in \{0,\dots,n\}$, $\lambda \in (0,1/2]$, $p \in (-1/n, \infty)$ and $r,\beta\in (0, \infty)$ there exists $d=d_{n, \lambda, p, r,\beta}>0$ such that the following holds. Let $f,g, h\colon C_n^i\rightarrow [0,1]$ be continuous functions with bounded support such that  
    \begin{itemize}
        \item $\int_{C_n^i} f\, dx=  \int_{C_n^i} g\,dx=1$,  
        \item $\int_{C_n^i}h\,dx = 1+\delta$ for some $0\leq \delta \leq d$,
        \item $|\{x\in rS_n\cap C_n^i: f(x)<\beta\text{ or }g(x)<\beta\}|\leq \delta$, and
        \item for all $x,y \in \mathbb{R}^n$ we have $h(\lambda x +(1-\lambda)y) \geq M_{\lambda,p}(f(x),g(y))$.
    \end{itemize}
    Then $\int_{C_n^i} |f-g|\,dx=O_{n,\lambda, p, r,\beta} \left(\sqrt{\delta}\right)$. 
\end{prop}

\begin{proof}[Proof that \Cref{symmetric_diff_reformulation_technical} implies \Cref{symmetric_diff_technical} ]
Consider $f$ and $g$ as in \Cref{symmetric_diff_technical}, and let $f_i$ (resp. $g_i$) be the function $f$ (resp $g$) restricted to $C_n^i$ renormalized, i.e., 
$$f_i(x):=\frac{f(x) \textbf{1}_{C_n^i}(x)}{\int_{C_n^i}f\,dy},$$
so that $\int_{C_n^i}f_i(x)\,dx=\int_{C_n^i}g_i(x)\,dx=1$.

Let $h_i:=\M^*(f_i,g_i)$. We'll show that $h$ has integral not much larger than $1$. Note that the support of $h_i$ falls in $C_n^i$, so that by Borell-Brascamp-Lieb $\int_{C_n^i}h_i \,dx\geq 1$. Moreover, since the supports of the $h_i$ are essentially disctinct, we have that $\sum_{i=0}^n\left(\int_{C_n^i}f\,dy\right)h_i(x)\leq h(x)$ for almost all $x$. Hence, integrating over $x\in \mathbb{R}^n$, we find
$$\sum_{i=0}^n\left(\int_{C_n^i}f\,dy\right)\int_{C_n^i}h_i(x)\,dx\leq \int_{\R^n} h(x)\,dx\leq 1+\delta.$$
Isolating $h_{i_0}$ for some fixed $i_0$ and using $\int_{C_n^i}h_i\, dx\geq 1$, we find
$$\int_{C_n^{i_0}}h_{i_0}(x)\,dx\leq\frac{ (1+\delta)-\sum_{i\neq i_0}\left(\int_{C_n^i}f\,dy\right)\int_{C_n^i}h_i(x)\,dx}{\int_{C_n^{i_0}}f\,dy}\leq \frac{ \left(\int_{\mathbb{R}^n}f\,dy+\delta\right)-\sum_{i\neq i_0}\left(\int_{C_n^i}f\,dy\right)}{\int_{C_n^{i_0}}f\,dy}=1+\frac{\delta}{\int_{C_n^{i_0}}f\,dy}.$$
Finally, we use that $f$ is big on $rS_n$ to find that $$\int_{C_n^{i_0}}f\,dy\geq \int_{C_n^{i_0}\cap rS_n}f\,dy\geq \frac{\beta}{2}|C_n^{i_0}\cap rS_n|=\Omega_{n,r,\beta}(1).$$
Hence, we can choose $d_{n,\lambda,p,r,\beta}$ in \Cref{symmetric_diff_technical} sufficiently small so that $\frac{2\delta}{\beta\left|C_n^{i_0}\cap rS_n\right|}$ is smaller than the $d$ in \Cref{symmetric_diff_reformulation_technical}.

Now that we've found that the $f_i,g_i,h_i$ satisfy the conditions of \Cref{symmetric_diff_reformulation_technical}, we can apply that proposition in each of the cones and conclude that $\int_{C_n^i} |f_i-g_i|\, dx=O_{n,\lambda, p, r} \left(\frac{2\delta}{\beta\left|C_n^{i_0}\cap rS_n\right|}\right)^{1/2}=O_{n,\lambda, p, r,\beta} \left(\sqrt{\delta}\right)$. Summing over the cones, we find
$$\int_{\mathbb{R}^n}|f-g|\,dx=\sum_{i=0}^n\int_{C_n^i}f_i\,dy \int_{C_n^i}|f_i-g_i|\,dx\leq \sum_{i=0}^n\int_{C_n^i}|f_i-g_i|\,dx\leq O_{n,\lambda, p, r,\beta} \left(\sqrt{\delta}\right),$$
so that \Cref{symmetric_diff_technical} follows.
\end{proof}

\subsection{Reduction from \Cref{symmetric_diff_reformulation_technical} to \Cref{symmetric_diff_narrow_n_cone_alt}}

For the next sections we shall assume that the simplex $S_n$ and the basis $e_1, \dots, e_n$ of $\mathbb{R}^n$ are such that $e_1\perp H_n^1$ where we recall $H_n^1$ is the supporting hyperplane of the face $F_n^1$ of the simplex $S_n$.

\begin{prop}\label{symmetric_diff_narrow_n_cone_alt}
    Given $n \in \mathbb{N}$, $\lambda \in (0,1/2]$, $p \in (-1/n, \infty)$, $r,s, \ell,\beta \in (0, \infty)$, there exists $d=d^{\ref{symmetric_diff_narrow_n_cone_alt}}_{n, \lambda, p, r, \ell,\beta}>0$ so that for all $\delta\in (0,d)$, there exists $e=e_{n,\lambda,p,r,s,\ell,\beta,\delta}>0$ so that for all $\epsilon\in(0,e)$ the following holds. Let $C \subset C_n^1$ be a cone. Let $f,g, h\colon C\rightarrow [0,1]$ be continuous functions such that
    \begin{itemize}
        \item $\left|\int_{C} (f - g)\, dx\right|\leq  \varepsilon \int_{C} f\,  dx$,
        \item $\int_{C}h\, dx = (1+\delta) \int_{C} f\, dx$,
        \item $\int_{C}f\, dx \leq  \ell \int_{C \cap rS_n}f\,dx $,
        \item $|\{x\in rS_n\cap C: f(x)<\beta\text{ or }g(x)<\beta\}|\leq \gamma|C\cap rS_n|$,
        \item for every $x \not\in s S_n$ we have $f(x)=g(x)=0$,
        \item for all $x,y \in C$ we have $h(\lambda x +(1-\lambda)y) \geq M_{\lambda,p}(f(x),g(y))$, and
        \item  for all $z,w \in \mathbb{R}$ and $x, y \in C^{z,w}$ we have $|f(x)-f(y)| \leq \varepsilon$ and $|g(x)-g(y)| \leq \varepsilon$.
    \end{itemize}
    Then $\int_C |f-g|\, dx=O_{n,\lambda, p, r,\beta, \ell}\left(\sqrt{\delta+\gamma}\right)\int_{C} f\, dx  $.
\end{prop}

\begin{lem}\label{coneconcentrationlemalt}
Let $n \in \mathbb{N}$, $\lambda \in (0,1/2]$, $p \in (-1/n, \infty)$ and $r,\beta\in (0, \infty)$, then there exist $\ell=\ell^{\ref{coneconcentrationlemalt}}_{n,\lambda,p,r,\beta}$ and $d=d^{\ref{coneconcentrationlemalt}}_{n,\lambda,p,r,\beta}>0$ so that the following hold.  Let $f,g, h\colon \mathbb{R}^n\rightarrow [0,1]$ be continuous functions with bounded support and a cone $C$ such that 
\begin{itemize}
    \item $\int_{\mathbb{R}^n} f\, dx=  \int_{\mathbb{R}^n} g \,dx=1$,
    \item $\int_{\mathbb{R}^n}h\, dx \leq 1+d$,
    \item $\int_{C} f\, dx=  \int_{C} g \,dx$,
    \item $\int_C h\,dx\leq (1+d)\int_C f\,dx$,
    \item $|\{x\in rS_n: f(x)<\beta\text{ or }g(x)<\beta\}|\leq 0.1|rS_n|$,
    \item $|\co(\{x\in \mathbb{R}^n: f(x)>t\})|\leq 2|\{x\in \mathbb{R}^n: f(x)>t\}|$ for all $t\in[0,1]$, and
    \item for all $x,y \in \mathbb{R}^n$ we have $h(\lambda x +(1-\lambda)y) \geq M_{\lambda,p}(f(x),g(y))$.
\end{itemize} 
Then we have $\int_{C}f\, dx \leq  \ell |C\cap rS_n|$.
\end{lem}

\begin{defn}
Given a subcone $C\subset C_n^1$ and a basis $f_1,\dots,f_n$ with $e_1=f_1$, define the \emph{slice} $C^{z,w}:=C\cap (z,w)\times\mathbb{R}^{n-2}$.
\end{defn}

\begin{lem}\label{partition_small_cones}
    Let $n \in \mathbb{N}$ and let $f,g\colon C_n^1\rightarrow [0,1]$ be continuous functions with bounded support such that  $\int_{C_n^1} f dx=  \int_{C_n^1} g dx$. Then for every $\varepsilon>0$ there exists a family $\mathcal{F}=\mathcal{F}_1 \sqcup \mathcal{F}_2$ of cones partitioning the cone $C_n^1$  such that
    \begin{itemize}
        \item $\sum_{C \in \mathcal{F}_1}\int_{C}f\,dx<\varepsilon$.
        \item For every $C \in \mathcal{F}_2$, $\left|\int_{C} (f  - g)\, dx\right| \leq \varepsilon \int_{C} f\, dx $.
        \item For every $C \in \mathcal{F}_2$, there exists a basis $f_1, \dots f_n$ depending on $C$ with $f_1=e_1$, such that for all $z,w \in \mathbb{R}$ we have $\text{diam}(C^{z,w}) \leq \varepsilon |z|$.
    \end{itemize}
    
\end{lem}

\subsubsection{Proof of Reduction}

\begin{proof}[Proof that \Cref{symmetric_diff_narrow_n_cone_alt} implies \Cref{symmetric_diff_reformulation_technical}]
We consider the cone $C_n^1$, the other $C_n^i$ follow analogously by rotation.

Given that $f$ and $g$ have bounded support, there exists an $s$ so that for $x\not \in sS_n$ we have $f(x)=g(x)=0$.

Apply \Cref{partition_small_cones} with parameter $\eta>0$ sufficiently small in terms of all of $f,g,n,\lambda,p,r,s,\ell,\beta,\epsilon$, to find essential partition $\mathcal{F}=\mathcal{F}_1\sqcup\mathcal{F}_2$ of $C_n^1$. Further subdivide $\mathcal{F}_2$ into two sets as follows. For every cone $C\in\mathcal{F}_2$ define 
$$\delta_C:=\frac{\int_C h\,dx}{\int_C f\,dx}-1,$$
so that 
$$\sum_{C\in \mathcal{F}_2}\delta_C\int_C f\,dx\leq \sum_{C\in \mathcal{F}_2}\int_C (h-f)\,dx\leq \int_{\mathbb{R}^n} (h-f)\,dx+2\eta\leq 2\delta,$$
where in the penultimate inequality we used that $\int_C (h-f)\,dx\geq -\eta\int_C f\,dx$ for $C\in \mathcal{F}$ and $\sum_{C\in\mathcal{F}_1}\int (h-f)\,dx\geq \sum_{C\in\mathcal{F}_1}\int (-f)\,dx\geq -\eta$.

Analogously define 
$$\gamma_C:=\frac{|\{x\in rS_n
\cap C: f(x)<\beta\text{ or }g(x)<\beta\}|}{|rS_n\cap C|},$$
so that
$$\sum_{C\in \mathcal{F}} \gamma_C|rS_n\cap C|=\sum_{C\in \mathcal{F}}|\{x\in rS_n
\cap C: f(x)<\beta\text{ or }g(x)<\beta\}|=|\{x\in rS_n: f(x)<\beta\text{ or }g(x)<\beta\}|\leq O_{n,\lambda,p,r}(\delta).$$

Partition $\mathcal{F}_2=\mathcal{F}_\gamma^+\sqcup\mathcal{F}_\delta^+\sqcup\mathcal{F}^-$ as follows.   If $\delta_C\geq\frac12\min\left\{d_{n,\lambda,p,r,\beta}^{\ref{coneconcentrationlemalt}},d^{\ref{symmetric_diff_narrow_n_cone_alt}}_{n, \lambda, p, r,\beta, \ell_{n,\lambda,p,r}^{\ref{coneconcentrationlemalt}}}\right\}$ let $C\in \mathcal{F}^+$. If $\gamma_C\geq \sqrt{\delta}$ (and $C\not\in \mathcal{F}_\delta^+$), let $C\in \mathcal{F}_\gamma^+$. 
 Finally, let $\mathcal{F}^-:=\mathcal{F}_2\setminus(\mathcal{F}_\gamma^+\sqcup\mathcal{F}_\delta^+)$

It's easy to address the cones in $\mathcal{F}_1\sqcup\mathcal{F}_{\gamma}^+\sqcup\mathcal{F}_{\delta}^+ $ as they contain little mass. First for $C\in\mathcal{F}_1$ by definition
$$\sum_{C\in\mathcal{F}_1}\int_C |f-g|\,dx\leq \sum_{C\in\mathcal{F}_1}\int_C (f+g)\,dx\leq \eta.$$

For the $C\in \mathcal{F}_\delta^+$, first note that for all $C'\in \mathcal{F}_2$, we have by the Borell-Brascamp-Lieb inequality that 
$$\int_{C'}h\,dx\geq \min\left\{\int_{C'}f\,dx,\int_{C'}g\,dx\right\}\geq \int_{C'}f\,dx-\left|\int_{C'} (f  - g)\, dx\right|\geq (1-\eta)\int_{C'}f\,dx,$$
Hence, we can find a bound on 
\begin{align*}\int_{\mathbb{R}^n}hdx&\geq \sum_{C\in\mathcal{F}_\delta^+}\int_{C}h\,dx +\sum_{C'\in\mathcal{F}_2\setminus \mathcal{F}_\delta^+}\int_{C'}h\,dx\\
&\geq \left(1+\frac12\min\left\{d_{n,\lambda,p,r,\beta}^{\ref{coneconcentrationlemalt}},d^{\ref{symmetric_diff_narrow_n_cone_alt}}_{n, \lambda, p, r, \beta,\ell_{n,\lambda,p,r}^{\ref{coneconcentrationlemalt}}}\right\}\right)\sum_{C\in\mathcal{F}_\delta^+}\int_{C}f\,dx+\sum_{C'\in\mathcal{F}_2\setminus\mathcal{F}_\delta^+}(1-\eta)\int_{C'}f\,dx\\
&\geq \left[\int_{\mathbb{R}^n}f\,dx-\sum_{C\in\mathcal{F}_1}\int_{C}f\,dx\right]+\frac12\min\left\{d_{n,\lambda,p,r,\beta}^{\ref{coneconcentrationlemalt}},d^{\ref{symmetric_diff_narrow_n_cone_alt}}_{n, \lambda, p, r, \beta,\ell_{n,\lambda,p,r}^{\ref{coneconcentrationlemalt}}}\right\}\sum_{C\in\mathcal{F}_\delta^+}\int_{C}f\,dx-\eta\sum_{C'\in\mathcal{F}_2\setminus\mathcal{F}_\delta^+}\int_{C'}f\,dx\\
&\geq \left[1-\eta\right]+\frac12\min\left\{d_{n,\lambda,p,r,\beta}^{\ref{coneconcentrationlemalt}},d^{\ref{symmetric_diff_narrow_n_cone_alt}}_{n, \lambda, p, r, \beta,\ell_{n,\lambda,p,r}^{\ref{coneconcentrationlemalt}}}\right\}\sum_{C\in\mathcal{F}_\delta^+}\int_{C}f\,dx-\eta
\end{align*}
Recalling that $\int_{\mathbb{R}^n}h\,dx\leq 1+\delta$ and noting that $\frac12\min\left\{d_{n,\lambda,p,r,\beta}^{\ref{coneconcentrationlemalt}},d^{\ref{symmetric_diff_narrow_n_cone_alt}}_{n, \lambda, p, r, \beta,\ell_{n,\lambda,p,r}^{\ref{coneconcentrationlemalt}}}\right\}=\Omega_{n, \lambda, p, r,\beta}(1)$, we can rearrange this inequality to find 
$$\sum_{C\in\mathcal{F}_\delta^+}\int_{C}|f-g|\,dx\leq 3\sum_{C\in\mathcal{F}_\delta^+}\int_{C}f\,dx\leq O_{n, \lambda, p, r,\beta}(\delta+2\eta)=O_{n, \lambda, p, r,\beta}(\delta).$$

For $C\in\mathcal{F}_\gamma^+$, we find that
$$\sqrt{\delta}\sum_{C\in\mathcal{F}_\gamma^+}|rS_n\cap C|\leq \sum_{C\in\mathcal{F}_\gamma^+}\gamma_C|rS_n\cap C|\leq O_{n,\lambda,p,r,\beta}(\delta),$$
so that $\sum_{C\in\mathcal{F}_\gamma^+}|rS_n\cap C|\leq O_{n,\lambda,p,r,\beta}(\sqrt{\delta})$. By \Cref{coneconcentrationlemalt}, we find that for these $C$, that $\int_C f\,dx\leq O_{n,\lambda,p,r,\beta}(|C\cap rS_n|)$. Hence,
$$\sum_{C\in\mathcal{F}_\gamma^+}\int|f-g|\,dx\leq 3\sum_{C\in\mathcal{F}_\gamma^+}\int f\, dx\leq 3\ell\sum_{C\in\mathcal{F}_\gamma^+}|C\cap rS_n|\leq O_{n,\lambda,p,r,\beta,\ell}(\sqrt{\delta}).$$

Finally, we turn our attention to cones $C\in\mathcal{F}^-$ to which we first apply \Cref{coneconcentrationlemalt} \footnote{Note that though $f$ and $g$ don't have the same integral in $C$, we can switch to $f'\leq f$ and $g'\leq g$ removing at most $\epsilon\int_C fdx$ to equalize the integral and apply \Cref{coneconcentrationlemalt}.}  to find that $\int_{C}f\, dx \leq \ell_{n,\lambda,p,r,\beta}^{\ref{coneconcentrationlemalt}}\int_{C \cap rS_n}f\,dx$.
 Since $f$ and $g$ are bounded continuous functions with bounded supports, they are uniformly continuous, i.e., for any $\epsilon>0$ we can choose $\eta$ sufficiently small in $f,g,$ and $\epsilon$, so that we have $|x-y|<\eta$ implies $|f(x)-f(y)|,|g(x)-g(y)|\leq \epsilon$. Moreover, the parameter $s$ in \Cref{symmetric_diff_narrow_n_cone_alt} depends only on $f$ and $g$, so we can pick $\epsilon<e_{n,\lambda,p,r,s,\ell,\beta,\delta}$ and still find an appropriate $\eta$. Hence, we can apply \Cref{symmetric_diff_narrow_n_cone_alt} with parameters $n,\lambda,p,r$, $s=s_{f,g}$, $\ell=\ell_{n,\lambda,p,r,\beta}^{\ref{coneconcentrationlemalt}}$, and $\delta=\delta_C<d^{\ref{symmetric_diff_narrow_n_cone_alt}}_{n, \lambda, p, r,\beta, \ell}$ and $\gamma=\gamma_C$ to the cones $C\in\mathcal{F}^-$, to find that $$\int_{C} |f-g|\,dx=O_{n,\lambda,p,r,\ell_{n,\lambda,p,r,\beta}^{\ref{coneconcentrationlemalt}}}\left(\sqrt{\delta_C+\gamma_C}\right)\int_{C}f\,dx=O_{n,\lambda,p,r,\beta}\left(\sqrt{\delta_C+\gamma_C}\right)\int_{C}f\,dx.$$
 Recall that $\sum_{C\in\mathcal{F}_2}\delta_C\int_C f\,dx\leq 2\delta$, and by \Cref{coneconcentrationlemalt},
 $$\sum_{C\in\mathcal{F}^-}\gamma_C\int_C f\,dx\leq O_{n,\lambda,p,r,\beta}\left(\sum_{C\in\mathcal{F}^-}\gamma_C|C\cap rS_n|\right)\leq O_{n,\lambda,p,r,\beta}(\delta).$$
 Thus
 \begin{multline*}
 \int_{\mathbb{R}^n}|f-g|\,dx=\sum_{C\in\mathcal{F}_1\cup\mathcal{F}^+}\int_{C} |f-g|\,dx+\sum_{C\in\mathcal{F}^-}\int_{C} |f-g|\,dx\\
 \leq O_{n, \lambda, p, r,\beta}(\delta)+\sum_{C\in\mathcal{F}^-}O_{n,\lambda,p,r,\beta}\left(\sqrt{\delta_C+\gamma_C}\right)\int_{C}f\,dx \leq O_{n, \lambda, p, r,\beta}\left(\sqrt{\delta}\right),
 \end{multline*}
 which finishes the proof of the reduction.
\end{proof}

\subsubsection{Proof of Lemmas}

\begin{proof}[Proof of \Cref{coneconcentrationlemalt}]
For a contradiction let $c_{n,\lambda,p,r,\beta}$ sufficiently large in its parameters to be determined later and $C$ a cone so that 
$$\int_{C}f \,dx> c_{n, \lambda, p, r,\beta} |C \cap rS_n| $$
Consider the level sets $F_t^C, G_t^C,H_t^C$ for the restrictions of the $f,g,h$ to $C$, i.e., $F_t^C:=\{x\in C: f(x)>t\}$. By \Cref{almostconvexlevelsets}, we find some $\beta_C=\Omega_{n,\lambda,p}(1)$ so that
$$\int_{ F_{\beta_C}^C}f\,dx\geq 0.9\int_C f\,dx\geq 0.9 c_{n,\lambda,p,r,\beta}|C\cap rS_n| .$$
 On the other hand, we have  $\int_{ F_{\beta_C}^C}f\,dx\leq \left|F_{\beta_C}^C\right|$. Since $F_{\beta_C}^C$ is so large, in particular it is not contained in $C\cap (0.9 c_{n,\lambda,p,r,\beta})^{1/n}rS_n$, so we can find a point $v\not\in (0.9 c_{n,\lambda,p,r,\beta})^{1/n}rS_n$ with $f(v)\geq \beta_C$. Consider the level set $F_{\beta_C}\supset F_{\beta_C}^C$, which by assumption has $|F_{\beta_C}|\geq \frac12|\co(F_{\beta_C})|$. We may assume that $\beta_C\leq \beta$, so that we also have $|F_{\beta_C}\cap rS_n|\geq 0.9|rS_n|$. This implies that that 
 $|\co(F_{\beta_C})|\geq |\co((F_{\beta_C}\cap rS_n)\cup\{v\})|\to \infty$
 as $c_{n,\lambda,p,r,\beta}\to \infty$. On the other hand, we have the trivial bound that $|\co(F_{\beta_C})|\leq 2|F_{\beta_C}|\leq \beta_C^{-1}\int f\,dx=O_{n,\lambda,p,r}(1)$, which for sufficiently large $c_{n,\lambda,p,r,\beta}$ yields a contradiction.
\end{proof}

For the proof of \Cref{partition_small_cones}, we use one of the technical results from \cite{BMStab}. First we need the following definitions from that paper.
\begin{defn}
Let $\mathcal{C}^n$ be the family of cones  in $\mathbb{R}^n$ and let $\mathcal{T}_k^n$ be the set of codimension $k$ subspaces of $\mathbb{R}^n$.
\end{defn}

\begin{defn}
Say a function $k \colon \mathcal{C}^{n} \times \mathcal{T}^{n}_2 \rightarrow \mathcal{T}^{n}_1$ is a \emph{respectful} function if $L\subset k(C,L)$. A respectful function $f$ induces functions $k^-, k^+ \colon \mathcal{C}^{n} \times \mathcal{T}^{n}_2 \rightarrow \mathcal{C}^{n}$, where $k^-(C,L), k^+(C,L)$ are the cones the hyperplane $k(C,L)$ partitions $C$ into.
\end{defn}

\begin{defn}
Given a respectful function $k \colon \mathcal{C}^{n} \times \mathcal{T}^{n}_2 \rightarrow \mathcal{T}^{n}_1$ and a cone $C$, we say $\mathcal{F}$ is a \emph{valid partition of $C$} into cones if there exists a sequence of families $\{C\}=\mathcal{G}_0, \dots, \mathcal{G}_j=\mathcal{F}$ such that if $\mathcal{G}_i=\{C_{1}, C_{2}, \dots, C_{k}\}$, then there exists codimension-two subspaces $L_1, \dots, L_{m}$ such that $\mathcal{G}_{i+1}=\bigcup_{j=1}^m R_j$, where $R_j=\{k^+(C_{j}, L_j), k^-(C_{j}, L_j)\}$ or $R_j=\{C_j\}$.
\end{defn}

Now recall the following slightly weakened version of theorem 4.11 from \cite{BMStab}.

\begin{thm}[Theorem 4.11 from \cite{BMStab} reformulated]\label{importedconesthm}
For every $\eta>0$ the following holds.  Given a respectful $f \colon \mathcal{C}^{n} \times \mathcal{T}^{n}_2 \rightarrow \mathcal{T}^{n}_1$, there exists a valid partition $\mathcal{F}$ of $C_n^1$ that can be written as $\mathcal{F}'=\mathcal{F}'_0 \sqcup \mathcal{F}'_1 \sqcup \mathcal{F}'_2$ such that 
\begin{enumerate}
    \item $\sum_{C \in \mathcal{F}'_0} |C\cap S_n| \leq \eta.$
    \item For every cone $C \in \mathcal{F}'_1$, we find that $C\cap\partial S_n$ has diameter at most $\epsilon$. 
    \item For every cone $C \in \mathcal{F}'_2$ there exists a sub-cone $C'$ of $C$ with $|C'\cap S_n| \geq (1-\eta)|C\cap S_n|$ and there exists a basis $f_1, \dots f_n$ depending on $C$ with $f_1=e_1$, such that for all $z,w \in \mathbb{R}$ we have $\text{diam}((C')^{z,w}) \leq \eta ||z||$. 
\end{enumerate}
\end{thm}

With these in place, let us prove the lemma.

\begin{proof}[Proof of \Cref{partition_small_cones}]
To apply \Cref{importedconesthm}, we need to define a respectful function $k\colon \mathcal{C}^n\times \mathcal{T}^n_2\to\mathcal{T}_1^n$. To this end consider a cone $C\in\mathcal{C}^n$ and a subspace $L\in\mathcal{T}^n_2$. First, if $\int_C f\,dx\neq \int_C g\,dx$, let $k(C,L)$ be any hyperplane containing $L$; this case will not occur for us.

Hence, assume $\int_C f\,dx= \int_C g\,dx$. Fix a hyperplane $H_0\in\mathcal{T}^n_1$ so that $L\subset H_0$. Note that all such hyperplanes can be obtained from $H_0$ by a rotation in the two dimensional subspace orthogonal to $L$. Hence, let $H_{\theta}$ be the hyperplane obtained by rotating by $\theta$ angle in that orhthogonal plane, so that e.g. $H_{\pi}=H_0$. Consider the induced $H_0^+$ and $H_0^-$ and their rotated counterparts, e.g. $H_{\pi}^+=H_{0}^-$.

\begin{clm}
There exists a $\theta\in[0,\pi]$ so that $\int_{C\cap H_\theta^+} f\,dx= \int_{C\cap H_\theta^+} g\,dx$ and $\int_{C\cap H_\theta^-} f\,dx= \int_{C\cap H_\theta^-} g\,dx$
\end{clm}
\begin{proof}[Proof of claim]
This follows easily by continuity of each of these integrals in $\theta$. Indeed, consider the discrepancy function $D\colon \mathbb{R}\to\mathbb{R}$ defined by $D(\theta):=\int_{C\cap H_\theta^+} f\,dx- \int_{C\cap H_\theta^+} g\,dx$, which is periodic and continuous. Note that
\begin{align*}
D(\pi)&=\int_{C\cap H_\pi^+} f\,dx- \int_{C\cap H_\pi^+} g\,dx=\int_{C\cap H_0^-} f\,dx- \int_{C\cap H_0^-} g\,dx\\
&=\left(\int_C f\,dx-\int_{C\cap H_0^+} f\,dx\right)-\left(\int_C g\,dx-\int_{C\cap H_0^+} g\,dx\right)\\
&=\int_{C\cap H_0^+} g\,dx-\int_{C\cap H_0^+} f\,dx= -D(0).
\end{align*}
Hence, by continuity there exists some $\theta\in[0,\pi]$ with $D(\theta)=0$, i.e., $\int_{C\cap H_\theta^+} f\,dx= \int_{C\cap H_\theta^+} g\,dx$. This immediately implies that also $\int_{C\cap H_\theta^-} f\,dx= \int_{C\cap H_\theta^-} g\,dx$.
\end{proof}
For every pair $(C,L)$, let $k(C,L)$ be the hyperplane $H_\theta$ given by this claim.

With this respectful function defined, we can apply \Cref{importedconesthm} with $\eta=\eta_{f,g,\epsilon}$ chosen sufficiently small depending on $f$, $g$ and $\epsilon$, to find valid partition $\mathcal{F}'=\mathcal{F}_0'\sqcup \mathcal{F}'_1 \sqcup \mathcal{F}'_2$ and for every $C\in \mathcal{F}'_2$ let $C'\subset C$ be the subcone as specified by the theorem. From this partition, we create partition $\mathcal{F}=\mathcal{F}_1\sqcup\mathcal{F}_2$ as follows.
Let $\mathcal{F}_2:=\mathcal{F}_1'\sqcup\{C':C\in\mathcal{F}'_2\}$ and let $\mathcal{F}_1:=\mathcal{F}_0'\sqcup\{C\setminus C': C\in\mathcal{F}_2'\}$.

To check the first condition, consider the contribution from $\mathcal{F}_0'$. Since $f$ and $g$ have bounded support, there exists an $R=R_{f,g}$, so that $supp(f),supp(g)\subset RS_n$, i.e., $\int_{RS_n}f+gdx=\int_{\R^n}f+gdx$. Hence, we find
$$\sum_{C\in \mathcal{F}_0'}\int_{C}fdx\leq \sum_{C\in \mathcal{F}_0'}|C\cap RS_n|\leq \eta R^n\leq \epsilon/2,$$
for $\eta$ sufficiently small in $R$, $n$, and $\epsilon$
 For the contribution from $\{C\setminus C': C\in\mathcal{F}_2'\}$, we use
$$\sum_{(C\setminus C')\in\{C\setminus C': C\in\mathcal{F}_2'\}}|(C\setminus C')\cap S_n|\leq \sum_{ C\in\mathcal{F}_2'}\eta |C\cap S_n|\leq \eta |S_n|<\epsilon/2,$$
as long as $\eta$ is small in terms of $n$ and $\epsilon$. Hence, the first condition is satisfied.

The second condition is the most tricky one. For $C\in\mathcal{F}_1'$ we have $\int_Cf\,dx=\int_C g\,dx$ so the condition is immediately satisfied. For the other cones we need the following crude observations. Since $f,g\leq 1$, for any cone $C\in\mathcal{F}_2'$, we have
\begin{align*}
\int_{C\setminus C'} f\,dx\leq |(C\setminus C')\cap RS_n|=\left(\frac{R}{r}\right)^n|(C\setminus C')\cap rS_n|\leq \frac{\eta}{1-\eta} \left(\frac{R}{r}\right)^n|C'\cap rS_n|\leq \frac{\epsilon}{2}0.1|C'\cap rS_n|\leq \frac{\epsilon}{2} \int_{ C'} f\,dx,  
\end{align*}
where in the penultimate inequality we used that $\eta$ is sufficiently small in terms of $R$ (i.e., $f,g$) and $\epsilon$. To conclude, we find that as $\int_Cf\,dx=\int_C g\,dx$, we have
$$\left|\int_{C'} (f-g)\,dx\right|=\left|\int_{C\setminus C'} (f-g)\,dx\right|\leq \int_{C\setminus C'}f\,dx+\int_{C\setminus C'}g\,dx\leq \epsilon \int_{ C'} f\,dx.$$
The third condition is immediate for the $C'\in \{C':C\in\mathcal{F}'_2\}$ and also easy to verify for $C\in \mathcal{F}_1'$, since $C^{z,w}\subset C\cap (|z|\partial S_n)=|z|(C\cap \partial S_n)$, so $diam(C^{z,w})\leq |z|diam(C\cap \partial S_n)\leq \epsilon|z|$. Hence, the lemma follows.
\end{proof}

\subsection{Reduction from \Cref{symmetric_diff_narrow_n_cone_alt} to \Cref{symmetric_diff_constant_on_fibers_n_cone}}

\begin{prop}\label{symmetric_diff_constant_on_fibers_n_cone}
    Given $n \in \mathbb{N}$, $\lambda \in (0,1/2]$, $p \in (-1/n, \infty)$, $r,\beta, \ell \in (0, \infty)$,   there exists $d=d_{n, \lambda, p, r,\beta, \ell}>0$ such that the following holds. Let $C \subset C_n^1$ be a cone. Let $f,g, h\colon C\rightarrow [0,1]$ be continuous functions with bounded support such that  
    \begin{itemize}
        \item $\int_{C} f\,  dx= \int_{C} g\,  dx $,
        \item $\int_{C}f \,dx \leq  \ell \int_{C \cap rS_n}f\,dx $,
        \item $|\{x\in rS_n\cap C: f(x)<\beta\text{ or }g(x)<\beta\}|\leq \gamma|C\cap rS_n|$, 
        \item  for all $x,y \in C$ we have $h(\lambda x +(1-\lambda)y) \geq M_{\lambda,p}(f(x),g(y))$,
        \item for all $z,w \in \mathbb{R}$ and $x, y \in C^{z,w}$ we have $f(x)=f(y)$ and $g(x)=g(y)$, and
        \item $\int_{C}h \,dx = (1+\delta) \int_{C} f \,dx$ for some $0\leq \delta \leq d$
    \end{itemize}
    Then $\int_C |f-g|\, dx=O_{n,\lambda, p, r, \beta,\ell}\left(\sqrt{\delta+\gamma}\right)\int_{C} f\, dx$.
\end{prop}

\begin{proof}[Proof that \Cref{symmetric_diff_constant_on_fibers_n_cone} implies \Cref{symmetric_diff_narrow_n_cone_alt}]
We construct $f',g'$ to which we can apply \Cref{symmetric_diff_constant_on_fibers_n_cone} by shrinking these functions a tiny bit (proportional to $\epsilon$). First, construct $f_1\leq f$ and $g_1\leq g$ by taking the infimum of $f$ in every slice $C^{z,w}$
$$f_1(x):=\inf\left\{f(y):\exists z,w\in \mathbb{R}: x,y\in C^{z,w}\right\}  $$
and $g_1$ analogously. Since $f(x)-f_1(x)\leq \epsilon$, we find that, 
$$\int_{C}f_1\,dx\geq \int_{C}f\,dx- \epsilon |supp(f)|\geq \int_{C}\,fdx- \epsilon |sS_n|=\int_{C}f\,dx-O_{n,s}(\epsilon),$$
and analogously for $g_1$.
Combining these, we find
$$\left|\int_C (f_1-g_1)\,dx\right|\leq \left|\int_C (f-g)\,dx\right|+\left|\int_C (g-g_1)\,dx\right|+\left|\int_C (f-f_1)\,dx\right|=O_{n,\ell,r}(\epsilon)+2\cdot O_{n,s}(\epsilon)=O_{n,\ell,r,s}(\epsilon).$$
Thus we can decrease one of the two functions while keeping the other the same to create $f'\leq f_1,g'\leq g_1$ with $\int_C f'dx=\int_C g'dx$ while maintaining all the properties required to apply \Cref{symmetric_diff_constant_on_fibers_n_cone} (e.g. only reduce all points in a $C^{z,w}$ simultaneously and only reduce outside of $rS_n/2$), we can do this with
$$\int_C (f_1-f')\,dx+\int_C (g_1-g')\,dx=\left|\int_C (f_1-g_1)\,dx\right|=O_{n,\ell,r,s}(\epsilon).$$
Since, $f'\leq f$ and $g'\leq g$, we still have for all $x,y \in C$ that $h(\lambda x +(1-\lambda)y) \geq M_{\lambda,p}(f'(x),g'(y))$. We can evaluate
$$\int_C h\,dx\leq (1+\delta)\int_C f\,dx\leq \frac{1+\delta}{1-O_{n,\ell,r,s}(\epsilon)}\int_C f'\,dx\leq (1+2\delta)\int_C f'\,dx,$$
where the last inequality follows from choosing $e_{n,\lambda,p,r,s,\ell,\delta}$ sufficiently small. Since $e_{n,\lambda,p,r,s,\ell,\delta}$ is definitely smaller than $\beta$, we find that $|\{x\in rS_n: f'(x)<\beta/2\text{ or }g'(x)<\beta/2\}|\leq |\{x\in rS_n: f(x)<\beta\text{ or }g(x)<\beta\}| \leq \gamma$.

Now we can apply \Cref{symmetric_diff_constant_on_fibers_n_cone} to functions $f',g',h$, and parameters $n,\lambda,p,\ell$ as in \Cref{symmetric_diff_narrow_n_cone_alt} and $\beta/2, r/2$ to give
$$\int_C|f'-g'|\,dx=O_{n,\lambda, p,\beta/2, r/2, \ell}\left((2\delta)^{1/2}+\gamma\right)\int_{C} f'\, dx=O_{n,\lambda, p, r,\beta, \ell}\left(\sqrt{\delta}+\gamma\right)\int_{C} f\,dx.$$
Bringing this back to $f$ and $g$, we find
$$\int_C|f-g|\,dx\leq \int_C|f'-g'|\,dx+\int_C|f-f'|\,dx+\int_C|g-g'|\,dx=O_{n,\lambda, p, r,\beta, \ell}\left(\sqrt{\delta}+\gamma\right)\int_{C} f\,dx,$$
which concludes the proof of the reduction.
\end{proof}

\subsection{Reduction from \Cref{symmetric_diff_constant_on_fibers_n_cone} to \Cref{symmetric_diff_2D}}

\begin{defn}
Consider the three cones  $K_1,K_2,K_3$ partitioning $\mathbb{R}^2$; 
$$K_1:=\{(x,y): x<0,|y|\leq |x|\}, K_2=\{(x,y): y\geq \max\{0,-x\}\}, K_3=\{(x,y): y\leq \min\{0,x\}\},$$
and corresponding cones $K_1^n,K_2^n,K_3^n$ partitioning $\mathbb{R}^n$:
$$K_i^n=K_i\times \mathbb{R}^{n-2}.$$
\end{defn}

The following lemma shows that if partition $f$ inside cone $C$ somewhat evenly into parallel to $K_i^n$, then the corresponding partition of $g$ almost coincides. The proof is very similar to the proof of \Cref{positioningwithinthe2dconelem}.

\begin{lem}\label{2dconeswithinthecone}
    Given $n \in \mathbb{N}$, $\lambda \in (0,1/2]$, $p \in (-1/n, \infty)$, $r, \beta,\ell \in (0, \infty)$ where $r$ is sufficiently small in terms of $n$, there exists $d=d_{n, \lambda, p, r,\beta, \ell}>0$ such that for $0 \leq \delta < d$ the following holds. Let $w=(r/2,0,\dots, 0),$. Let $C$ be a cone such that $ B(e_1, 1) \cap H_n^1 \subset C \cap H_n^1 \subset  B(e_1, n) \cap H_n^1$. Let $f,g, h\colon C\rightarrow [0,1]$ be continuous functions with bounded support such that  
    \begin{itemize}
        \item $\int_{C} f \, dx= \int_{C} g  \,dx $,
        \item $\int_{C}f \,dx \leq  \ell \int_{C \cap r(H_n^1)^- }f\,dx $,
        \item $|\{x\in (H_n^1)^-\cap C: f(x)<\beta\text{ or }g(x)<\beta\}|\leq \gamma|C\cap (H_n^1)^-|$,
        \item for all $x,y \in C$ we have $h(\lambda x +(1-\lambda)y) \geq M_{\lambda,p}(f(x),g(y))$,
        \item $\int_{C}h \,dx = (1+\delta) \int_{C} f \,dx$
    \end{itemize}
    there exists $v \in B(o,O_{n, \lambda, p, r,\beta, \ell}\left(\sqrt{\delta+\gamma}\right))$ so that for all $i=1,2,3$
    \begin{itemize}
        \item $\int_{w+K_i^n}f\,dx=\int_{w+v+K_i^n}g\,dx$, and
        \item $\int_{w+(1-\lambda)v+K_i^n}h\,dx\leq (1+O_{n, \lambda, p, r,\beta, \ell}(\delta))\int_{w+K_i^n}f\,dx$.
    \end{itemize}
\end{lem}

Recall $e_1 \perp H_n^1$ and actually $e_1 \in H_n^1$ where $H_n^1$ is the supporting hyperplane of the face $F_n^1$ of $S_n$. 

\begin{lem}\label{conessandwichlem}
Given $n \in \mathbb{N}$ and a cone $C \subset C_n^1$, there exists a bijective linear map $\alpha: \mathbb{R}^n \rightarrow \mathbb{R}^n$ such that 
\begin{itemize}
    \item $\alpha(e_1)=e_1$, 
    \item $\alpha(e_2)=\xi e_2$ for some $\xi \in \mathbb{R}^*$, and 
    \item $ B(e_1, 1) \cap H_n^1 \subset \alpha(C) \cap H_n^1 \subset  B(e_1, n) \cap H_n^1$.
\end{itemize} Note that for all $z,w \in \mathbb{R}$, $(\alpha C)^{z, \xi w}= \alpha( C^{z, w})$.
\end{lem}

\begin{lem}\label{bigfincentralball}
    Given $n \in \mathbb{N}$ and  $r,\beta \in (0, 1)$,   there exist $s_1 , s_0 \in (0, \infty)$ such that the following holds. Let $C$ be a cone such that $ B(e_1, 1) \cap H_n^1 \subset C \cap H_n^1 \subset  B(e_1, n) \cap H_n^1$. Let $f,g\colon C\rightarrow [0,1]$ be continuous functions with bounded support such that
    \begin{itemize}
        \item  $|\{x \in C \cap r(H_n^1)^- :f(x)< \beta\text{ or }g(x)< \beta\}|\leq \gamma |C \cap r(H_n^1)^-|$, and
        \item for all $(z,w)\in \mathbb{R}^2$, and all $x,y\in C^{z,w}$ we have $f(x)=f(y)$ and $g(x)=g(y)$.
    \end{itemize}  Let $C' = \pi_{1,2}(C)=\{(z,w)\in\mathbb{R}^2: C^{z,w}\neq\emptyset\}$ and  $f',g', h'\colon C'\rightarrow \mathbb{R}_{\geq0}$ be defined by 
     $$f'(z,w)= f(x)|C^{z,w}|,\text{ and } \,g'(z,w)= g(x)|C^{z,w}|,$$ 
     for some $x=x_{z,w}\in C^{z,w}$
     Then for every $(z,w) \in B((r/2,0), r/10n)\subset\mathbb{R}^2$ we have 
     $$|\{(z,w) \in B((r/2,0), r/10n):f'(z,w)
     ,g'(z,w) \in[s_0,s_1]\}|\geq (1-O_{n,r,\beta}(\gamma))|B((r/2,0), r/10n)|.$$
\end{lem}

\begin{lem}\label{Holderapplicationlem}
     Given $n \in \mathbb{N}$, $\lambda \in (0,1/2]$, $p \in (-1/n, \infty)$ and $q=\frac{p}{1+np}$ the following holds. Let $C \subset \mathbb{R}^n$ be a cone. Let $f,g, h\colon C\rightarrow [0,1]$ be continuous functions with bounded support such that 
     \begin{itemize}
    \item for all $x,y\in C$, we have $h(\lambda x +(1-\lambda)y) \geq M_{\lambda,p}(f(x),g(y))$, and
    \item  for all $z,w \in \mathbb{R}$ and $x, y \in C^{z,w}$ we have $f(x)=f(y)$ and $g(x)=g(y)$.
    \end{itemize} Let $C' = \pi_{1,2}(C)$ $=\{(z,w)\in\mathbb{R}^2: C^{z,w}\neq\emptyset\}$ be a cone in $ \mathbb{R}^2$. Let  $f',g', h'\colon C'\rightarrow \mathbb{R}_{\geq0}$ be defined by 
     $$f'(z,w)= f(x)|C^{z,w}|, \,g'(z,w)= g(x)|C^{z,w}|,\text{ and }h'(z,w)=h'(x)|C^{z,w}|$$
     for some $x=x_{z,w} \in C^{z,w}$. 
     Then, for $x, y \in C'$
     \begin{itemize}
         \item $h'(\lambda x +(1-\lambda)y) \geq M_{\lambda,q}(f'(x),g'(y))$,
         \item $\int_{C'} f'\, dx = \int_C f\, dx  $,
         \item $\int_{C'} g'\, dx = \int_C g\, dx  $, 
         \item $\int_{C'} h' \,dx = \int_C h \,dx$, and
         \item $\int_{C'}|f'-g'|\,dx=\int_C |f-g|\,dx$.
     \end{itemize} 
     Furthermore, for any vector $w'\in\mathbb{R}^2$ and corresponding vector $w=(w_1',w_2',0,\dots,0)\in\mathbb{R}^n$ and any $i=1,2,3$, we have
     \begin{itemize}
         \item $\int_{w'+K_i} f' \,dx = \int_{w+K_i^n} f \,dx  $,
         \item $\int_{w'+K_i} g' \,dx = \int_{w+K_i^n} g \,dx  $, and
          \item$\int_{w'+K_i} h' \,dx = \int_{w+K_i^n} h \,dx$.
     \end{itemize} 
\end{lem}

\begin{lem}\label{pqaverageswitchinglem}
    Given $n \in \mathbb{N}^*$, $\lambda\in (0,1/2]$, $p \in (-1/(n+2), 0)$ and $q= \frac{p}{1+pn} \in (-1/2, 0)$ the following holds. If $a,b,c, u, v \in (0, \infty)$ satisfy $ b^{1/n} \geq \lambda a^{1/n}+ (1-\lambda)c^{1/n} $, then $$ b \bigg(\lambda u^p +(1-\lambda)v^p\bigg)^{1/p} \geq  \bigg(\lambda (au)^q +(1-\lambda)(cv)^q\bigg)^{1/q},  $$  
    i.e., $b \M(u,v)\geq M_{\lambda,q}(au,cv).$
\end{lem}

\subsubsection{Proof of reduction}

\begin{proof}[Proof that \Cref{symmetric_diff_2D} implies \Cref{symmetric_diff_constant_on_fibers_n_cone}]

By \Cref{conessandwichlem}, there exists a bijective linear map $\alpha: \mathbb{R}^n \rightarrow \mathbb{R}^n$ such that if we denote $C_1= \alpha(C)$ the cone in $\mathbb{R}^n$ , then
\begin{enumerate}
    \item $\alpha(e_1)=e_1$, 
    \item $\alpha(e_2)=\xi e_2$ for some $\xi \in \mathbb{R}^*$, and 
    \item $ B(e_1, 1) \cap H_n^1 \subset C_1 \cap H_n^1 \subset  B(e_1, n) \cap H_n^1$.
  \item for all $z,w \in \mathbb{R}$, $(C_1)^{z,  w}= \alpha( C^{z, \xi^{-1}w})$.
 \end{enumerate}
Let and let $f_1,g_1,h_1 \colon \mathbb{R}^n \rightarrow [0,1]$ be continuous functions with bounded support defined by $f_1=f\circ \alpha$, $g_1=g\circ \alpha$ and $h_1=h \circ \alpha$. As we shall explain below, it is easy to check that the following properties hold:
    \begin{enumerate}
        \item $\int_{C_1} f_1  \,dx=  \int_{C_1} g_1  \,dx $,
        \item $\int_{C_1} f_1  \,dx \leq  \ell \int_{C_1 \cap r(H_n^1)^-}f_1\,dx  $,
        \item for every $|\{x \in C_1 \cap r(H_n^1)^- :f_1(x) < \beta\text{ or }g_1(x)< \beta\}|\leq \gamma |C_1 \cap r(H_n^1)^-|$, 
        \item  for all $x,y \in C_1$ we have $h_1(\lambda x +(1-\lambda)y) \geq M_{\lambda,p}(f_1(x),g_1(y))$,
        \item for all $z,w \in \mathbb{R}$ and $x, y \in C_1^{z,w}$ we have $f_1(x)=f_1(y)$ and $g_1(x)=g_1(y)$, and
        \item $\int_{C_1}h_1 \,dx = (1+\delta) \int_{C_1} f_1 \,dx$ for some $0\leq \delta \leq d$,
        \item $\int_C|f-g|\,dx / \int_C f \,dx=\int_{C_1}|f_1-g_1|\,dx/ \int_{C_1} f_1 \,dx$.
    \end{enumerate}
Indeed, the first point follows from the fact $\int_{C_1} f_1  \,dx= |\det(\alpha)| \int_{C} f  \,dx$ and $\int_{C_1} g_1  \,dx= |\det(\alpha)| \int_{C} g  \,dx$.
    
    Recall $H_n^1$ is the supporting hyperplane of the face $F_n^1$ of the simplex $S_n$ and $e_1$ is orthogonal to $H_n^1$. Because $\alpha(e_1)=e_1$, it follows that $r(H_n^1)= \alpha(r(H_n^1))$ and $\alpha(C \cap rS_n)= \alpha(C \cap r(H_n^1)^{-}) =  C_1 \cap r(H_n^1)^-$.   
    Thus, the 
    second point follows from the fact that  $|\det(\alpha)| \int_{C \cap rS_n}f\,dx =  \int_{C_1 \cap r(H_n^1)^-}f_1\,dx$.

    The third point follows from the fact that all $\alpha(C \cap rS_n)= C_1 \cap r(H_n^1)^-$. 

    The forth point follows from the fact that all $x, y \in C_1$ satisfy $x= \alpha (z)$ and $y = \alpha(w)$ for some $z, w\in C $. Note  $h_1(\lambda x +(1-\lambda)y)= h(\lambda z +(1-\lambda)w)  \geq M_{\lambda,p}(f(z),g(w))= M_{\lambda,p}(f_1(x),g_1(y))$.

    The fifth point follows from the fact that for all $z,w \in \mathbb{R}$ and all $x, y \in C_1^{z,w}$ there exist $a,b \in C^{z,\xi^{-1}w}$ such that $x= \alpha (a)$ and $y = \alpha(b)$. Note  $f_1(x)= f(\alpha(a))= f(\alpha (b))=f_1(y) $.

    Finally, the last two points follow in a similar way to the first point.

 Now, let $C_2 = \pi_{1,2}(C_1)=\{(z,w)\in\mathbb{R}^2: C_1^{z,w}\neq\emptyset\}$ be a cone in $ \mathbb{R}^2$. Let  $f_2,g_2, h_2\colon C_2\rightarrow \mathbb{R}_{\geq0}$ be defined by 
     $$f_2(z,w)= f_1(x)|C_1^{z,w}|, \,g_2(z,w)= g_1(x)|C_1^{z,w}|\text{ and }h_2(z,w)=h_1(x)|C_1^{z,w}|$$
for some $x=x_{z,w} \in C^{z,w}$. Note that
$$\frac{\int |f_2-g_2|\,dx }{\int f_2 \,dx} = \frac{\int |f_1-g_1|\,dx}{ \int f_1\,dx}= \frac{ \int |f-g|\,dx}{ \int f\,dx}.$$

By \Cref{bigfincentralball} combined with the third and fifth point, there exist $s_0=s_0(n,r,\beta)>0$ and $s_1=s_1(n,r,\beta)>0$ such that
$$|\{(z,w) \in B((r/2,0), r/10n):f_2(z,w)
     ,g_2(z,w) \in[s_0,s_1]\}|\geq (1-O_{n,r,\beta}(\gamma))|B((r/2,0), r/10n)|.$$

Let $w_0=(r/2,0,\dots,0)$ as in \Cref{2dconeswithinthecone}.
By \Cref{2dconeswithinthecone} combined with  \Cref{Holderapplicationlem} there exists $v \in B\left(o,O_{n,\lambda,p,r,\beta,\ell}\left(\sqrt{\delta+\gamma}\right)\right) \cap \left(\mathbb{R}^2\times \{0\}^{n-2}\right) $   so that for all $i=1,2,3$, we have 
    \begin{itemize}
        \item $\int_{w_0+K^n_i}f_1\,dx=\int_{w_0+v+K^n_i}g_1\,dx$, and
        \item $\int_{w_0+(1-\lambda)v+K^n_i}h_1\,dx\leq (1+O_{n,\lambda,p,r,\beta,\ell}(\delta))\int_{w_0+K^n_i}f_1\,dx$.
    \end{itemize}
    In a slight abuse of notation, since $w_0$ and $v$ are both in $\mathbb{R}^2\times\{0\}^{n-2}$, we also use $w_0$ for $(r/2,0)$ and $v$ for the projection of $v$ onto its first two coordinates.
Hence, let $f^i_2,g^i_2,h^i_2$ be defined by
$$f^i_2=f_2\cdot \textbf{1}_{w_0+K_i}, g^i_2=g_2\cdot \textbf{1}_{w_0+v+K_i},h^i_2=h_2\cdot \textbf{1}_{w_0+(1-\lambda)v+K_i}.$$
By \Cref{symmetric_diff_2D}, we find that there exist translations $v_i$ (for $i=1,2,3$) so that 
$$\int_{\mathbb{R}^2}|f_2^i(x)-g_2^i(x+v_i)|\,dx=O_{n,\lambda,p,r,\beta,\ell}\left(\sqrt{\delta+\gamma}\right).$$
We'll show that the same is true without the translation $v_i$, by showing that $v_i$ is not much different from $-v$.
For simplicity, consider $K_1$ and $v_1$ the others will follow analogously. Consider the translation $v_1+v$. Since the defining lines $e_2-e_1$ and $-e_2-e_1$ of the cone $K_1$ are at a right angle, one of the two makes a big angle with $v_1+v$, so that
$$|(w_0+K_1 \triangle (w_0+v_1+v+K_1))\cap B(w_0,r/10n)|\geq \Omega_{n,r}(||v_1+v||).$$
Since in this region the functions are large (except for in the small set $|\{(z,w) \in B((r/2,0), r/10n):f_2(z,w)
     ,g_2(z,w) <s_0\}|\leq O_{n,\lambda,p,r,\beta}(\gamma)$), this gives a lower bound on the previously established symmetric difference;
\begin{multline*}
s_0\cdot \Big(|(w_0+K_1 \triangle (w_0+v_1+v+K_1))\cap B(w_0,r/10n)|-O_{n,\lambda,p,r,\beta}(\gamma)\Big)\\
\leq \int_{(w_0+K_1 \triangle (w_0+v_1+v+K_1))\cap B(w_0,r/10n)}\left|f_2^1(x)-g_2^1(x+v_1)\right|\,dx\leq  \int_{\mathbb{R}^2}\left|f_2^1(x)-g_2^1(x+v_1)\right|\,dx,
\end{multline*}
so that recalling that $s_0=\Omega_{n,r}(1)$, we have
$$||v_1||\leq ||v||+||v_1+v||=O_{n,\lambda,p,r,\beta,\ell}\left(\sqrt{\delta+\gamma}\right).$$
Since this translate is small, we find by \Cref{smalltranslationsmallsymdiflem} that shifting the function by this amount doesn't change the symmetric difference much, i.e.
$$\int_{\mathbb{R}^2}\left|f_2^i(x)-g_2^i(x)\right|\,dx=O_{n,\lambda,p,r,\beta,\ell}\left(\sqrt{\delta+\gamma}\right).$$

 Combining these, we conclude
 $$\frac{\int_{\mathbb{R}^n} |f-g|\, dx}{ \int_{\mathbb{R}^n} f \,dx} = \frac{\int_{\mathbb{R}^2} |f_2-g_2|\, dx}{ \int_{\mathbb{R}^2} f_2\, dx}\leq \frac{\sum_{i=1}^3 \int_{\mathbb{R}^2}\left|f_2^i-g_2^i\right|\,dx}{\int_{\mathbb{R}^2} f_2\,dx}= O_{n, \lambda, p, r,\beta,\ell, s_0, s_1}\left(\sqrt{\delta+\gamma}\right).$$
The conclusion follows.
\end{proof}

\subsubsection{Proof of Lemmas}

\begin{proof}[Proof of \Cref{2dconeswithinthecone}]
By \Cref{parallelconeslem}, there exists a $v$ so that for $i=1,2,3$, we have
$$\int_{w+K_i^n}f\,dx=\int_{w+v+K_i^n}g\,dx.$$
Our task will be to show that $v\in B(o,O_{n,\lambda,p,r,\beta,\ell}\left(\sqrt{\delta+\gamma}\right))$.

First note that we have
$$\int_{C\cap (w+ (1-\lambda)v K_1)}h\,dx-\int_{C\cap (w+K_1)}f\,dx\leq \delta\int_{C}fdx\leq O_{n,\ell}(\delta)\int_{C\cap K_1}f\,dx,$$
where in the last step we used that $C\cap (w+K_1)\subset C\cap (H_n^1)^-$ and $|C\cap (w+K_1)|=\Omega_{n}(|C\cap (H_n^1)^-|)$. Hence, by \Cref{SimilarSupportLem}, we find that 
$$|(C\cap w+K_1)\triangle(C\cap w+v+K_1)|=|(supp (f\mid_{C\cap (w+K_1)}))\triangle (supp (g\mid_{C\cap (w+v+K_1)}))|=O_{n,p,\lambda,\beta,\ell}\left(\sqrt{\delta+\gamma}\right),$$
which immediately implies that $v\in B(o,O_{n,\lambda,p,r,\beta,\ell}\left(\sqrt{\delta+\gamma}\right))$.
\end{proof}

\begin{proof}[Proof of \Cref{conessandwichlem}]
Consider the convex set $K:=H_{n}^1\cap C$ and let $K':=K-e_1\subset \{0\}\times\mathbb{R}^{n-1}$. By John's theorem there exists an ellipsoid $E\subset \{0\}\times\mathbb{R}^{n-1}$ so that $E\subset K'\subset (n-1)E$. There exists a bijective linear map $\zeta\colon\{0\}\times\mathbb{R}^{n-1}\to \{0\}\times\mathbb{R}^{n-1}$ so that $\zeta(E)=\{0\}\times B^{n-1}(o,1)$, and thus $\zeta(K')\subset \zeta((n-1)E)=B(o,n-1)$. Now consider $\zeta(e_2)$ and let $\gamma\colon\{0\}\times\mathbb{R}^{n-1}\to \{0\}\times\mathbb{R}^{n-1}$ be the rotation (a bijective linear map) bringing $\zeta(e_2)$ parallel to $e_2$. Let $\alpha':\{0\}\times\mathbb{R}^{n-1}\to \{0\}\times\mathbb{R}^{n-1}$ be the bijective linear map $\alpha= \gamma \circ \zeta$. Finally let 
$$\alpha\colon\mathbb{R}^{n}\to \mathbb{R}^{n}; (x_1,\dots,x_n)\mapsto (x_1, \alpha'(x_2,\dots,x_n)),$$
be the bijective linear map applying $\alpha'$ to the last $n-1$ coordinates. This map clearly satisfies the conditions in the lemma.
\end{proof}

\begin{proof}[Proof of \Cref{bigfincentralball}]
First note that $B((r/2,0),r/10n)\subset C\cap r(H_n^1)^-$ and conversely $|B((r/2,0),r/10n)|\geq \Omega_{n,r}(|C\cap r(H_n^1)^-|)$.

  Consider $(z,w)\in B((r/2,0),r/10n)$ so that $f(x),g(x)\geq \beta$ for all $x\in C^{z,w}$, so all we need to do is control $|C^{z,w}|$.  For the upper bound, simply use that since $(z,w)\in B((r/2,0),r/10n)$ is on the same side of the hyperplane $H_n^1$ as the origin, we have $C^{z,w}\subset \co((B(e_1,n)\cap H_n^1)\cup \{o\})$. Since this last set only depends on $n$, we immediately get $|C^{z,w}|=O_n(1)$.
  
  For the lower bound, we similarly use that $\co((B(e_1,1)\cap H_n^1)\cup \{o\})\subset C$. By convexity, this implies for $z\in r/2+[-r/10n,r/10n]$, we have 
  $$ B(ze_1,z)\cap zH_n^1\subset C\cap d\{z\}\times\mathbb{R}^{n-1},$$
  so that we have
  \begin{multline*}|C^{z,w}|\geq \left|B(ze_1,z)\cap (z,w)\times\mathbb{R}^{n-2}\right|\geq \left(\sqrt{z^2-w^2}\right)^{n-2}b_{n-2}\\
  \geq \left(\sqrt{(r/2-r/10n)^2-(r/10n)^2}\right)^{n-2}b_{n-2}=\Omega_{n,r}(1),
  \end{multline*}
  where $b_{n-2}$ is the volume of the $n-2$ dimensional ball of radius 1. Hence, choices $s_0=\Omega_{n,r,\beta}(1)$ and $s_1=O_{n,r,\beta}(1)$ as in the lemma exist.
\end{proof}

\begin{proof}[Proof of \Cref{Holderapplicationlem}]
The equalities for the integrals follow from the simple observation that if we fix for every $z,w\in\mathbb{R}$, some point $x_{z,w}$  in $C^{z,w}$, then
$$\int_{C'}f'(z,w)\,d(z,w)=\int_{C'}f(x_{z,w})|C^{z,w}|\,d(z,w)=\int_{C'}\int_{C^{z,w}}f(x)\,dxd(z,w)=\int_C f(x)\,dx.$$
Indeed the equalities for $g$ and $h$ follow analogously. By the definitions, we have that $|f'(z,w)-g'(z,w)|=|f(x_{z,w})-g(x_{z,w})|\cdot |C^{z,w}|$, so that integral follows analogously as well. As do the integrals involving $R_{r}^{f_1,f_2}$ with the note that $R_{r}^{f_1,f_2}=\pi_{1,2}(R_{r}^{f_1,f_2})$, where the former is a subset of $\mathbb{R}^2$ and the latter of $\mathbb{R}^n$.

Finally, we turn our attention to the lower bound on $h'$. Consider two points $(z,w),(z',w')\in C'$. The cone $C$ is a convex set, so that $C^{\lambda (z,w)+(1-\lambda)(z',w')}\supset \lambda C^{(z,w)} +(1-\lambda)C^{(z',w')}$ and thus by the Brunn-Minkowski inequality, we have 
$$\left|C^{\lambda (z,w)+(1-\lambda)(z',w')}\right|^{1/n}\geq \lambda \left|C^{z,w}\right|^{1/n}+(1-\lambda)\left|C^{z',w'}\right|^{1/n}.$$ Hence, we can apply \Cref{pqaverageswitchinglem} with $b=\left|C^{\lambda (z,w)+(1-\lambda)(z',w')}\right|$, $a=\left|C^{z,w}\right|$, $c=|C^{z',w'}|$, $u=f(x_{z,w})$, and $v=g(y)$, to find that
\begin{align*}
h'(\lambda(z,w)+(1-\lambda)(z',w'))&=\left|C^{\lambda (z,w)+(1-\lambda)(z',w')}\right|\cdot h\left(\lambda(z,w)+(1-\lambda)(z',w')\right)\\
&\geq \left|C^{\lambda (z,w)+(1-\lambda)(z',w')}\right|\cdot M_{\lambda,p}\left(f(x_{z,w}),g(x_{z',w'})\right)\\
&\geq M_{\lambda, q}\left(\left|C^{z,w}\right|\cdot f(x_{z,w}),\left|C^{z',w'}\right|\cdot g(x_{z',w'})\right)\\
&=M_{\lambda, q}\left(f'(z,w),g'(z',w')\right).
\end{align*}
This concludes the proof of the lemma.
\end{proof}

\begin{proof}[Proof of \Cref{pqaverageswitchinglem}]
We need to show that $$ \bigg(\lambda a^{1/n}+(1-\lambda) c^{1/n}\bigg)^n \bigg(\lambda u^p +(1-\lambda)v^p\bigg)^{1/p} \geq  \bigg(\lambda (au)^q +(1-\lambda)(cv)^q\bigg)^{1/q}. $$
After rearranging, this is equivalent to 
$$ \bigg(\lambda a^{1/n}+(1-\lambda) c^{1/n}\bigg)^n    \bigg(\lambda (au)^q +(1-\lambda)(cv)^q\bigg)^{-1/q} \geq \bigg(\lambda u^p +(1-\lambda)v^p\bigg)^{-1/p}.$$
After more rearranging, this is equivalent to
$$ \bigg(\lambda a^{1/n}+(1-\lambda) c^{1/n}\bigg)^{-pn}    \bigg(\lambda (au)^q +(1-\lambda)(cv)^q\bigg)^{p/q} \geq \lambda u^p +(1-\lambda)v^p,$$
and also to
$$ \bigg(\lambda (a^{-p})^{-1/pn}+(1-\lambda) (c^{-p})^{-1/pn}\bigg)^{-pn}    \bigg(\lambda ((au)^p)^{q/p} +(1-\lambda)((cv)^p)^{q/p}\bigg)^{p/q} \geq \lambda u^p +(1-\lambda)v^p.$$
Note that $p/q$ and $-pn$ are positive numbers with $-pn+p/q=-pn+1+pn=1$. By H\"older's inequality we get 

$$ \bigg(\lambda (a^{-p})^{-1/pn}+(1-\lambda) (c^{-p})^{-1/pn}\bigg)^{-pn}    \bigg(\lambda ((au)^p)^{q/p} +(1-\lambda)((cv)^p)^{q/p}\bigg)^{p/q} \geq \lambda a^{-p}(au)^p+(1-\lambda) c^{-p}(cv)^p.$$

The conclusion now follows.
\end{proof}

\section{\Cref{linear}: structure of the proof}
\label{sect:structure thm2}

To prove \Cref{linear}, the main difficulty is finding the correct $p$-concave function. To that end consider a maximizer $f'\colon\mathbb{R}^n\to\mathbb{R}_{\geq0}$ of the following quantity:
\begin{equation}\label{maxquantity}\int \big(\M^*(f,f)-\M^*(f',f')\big)\,dx -(1+c) \int (f-f')\,dx,\end{equation}
where $c=c_{n,\lambda,p}>0$ is some small constant and with the constraint that $f'\leq f$ pointwise. Intuitively, this can be seen as removing any parts of $f$ whose removal decreases $\M^*(f,f)$ more than $f$. 

The bulk of the work is to show that if we let $\ell$ be the minimal $p$-concave function with $\ell\geq f'$ pointwise (i.e., the $p$-concave hull of $f'$), then $\int (\ell-f')\,dx=O_{n,\lambda,p}(\delta)\int f\,dx$.
Indeed, note that we can rewrite the above expression as 
\begin{align*}
\int \left[(\M^*(f,f)-f)-(\M^*(f',f')-f')\right]\,dx - c \int (f-f')\,dx&\leq \int (\M^*(f,f)-f)\,dx- c \int (f-f')\,dx\\
&\leq \delta\int f\,dx- c \int (f-f')\,dx.
\end{align*}
Since we maximize the left-hand side and we can always choose $f'=f$ as a competitor, it follows that the right-hand side is at least zero. In particular
$$\int (f-f')\,dx\leq c^{-1}\delta\int f\,dx.$$
Hence, it suffice to show that $\int (\ell-f')\,dx=O_{n,\lambda,p}(\delta)\int f\,dx$.

 To use that $f'$ is a maximizer of \ref{maxquantity}, we use that $f'$ is of the form $f'=\min\{f,\ell\}$. For any $p$-concave function $\ell'$ (in particular small alterations of $\ell$), the induced function $f''=\min\{f,\ell'\}$ is in some sense worse than $f''$. To create an interesting $\ell'$ to compare to, we take the following approach to ``shaving off'' near the faces (slightly moving  one of the faces of $\ell$).

\begin{enumerate}
\item Consider a non-degenerate (i.e., full-dimensional) face $F$ of  $\ell$, i.e., a subset of the support of $\ell$ on which $\ell^p$ is linear. Then, $\ell\mid_{F}=L\mid_{F}$ for some linear function $L^p$ on $\mathbb{R}^n$.
    \item  
    The fact that $f'$ is a maximizer of \Cref{maxquantity} implies
    the minimality of $\ell$ in the following sense: If we consider $f'':=\min\{f', L-\epsilon\}$ (or equivalently $\ell'=\min\{\ell,L-\epsilon\}$ and $f''=\min\{f,\ell'\}$) for some very small $\epsilon>0$ (i.e., $f''$ is obtained from $f'$ by ``shaving off'' the points near the face), then
    $$\int \left(\M^*(f',f')-\M^*(f'',f'')\right)\,dx \leq (1+c) \int (f'-f'')\,dx.$$
    \item Let $X:=\{x\in F: f'(x)=\ell(x)\}\subset \mathbb{R}^n$ and note that (assuming that $f$ is well-behaved)
    $$\int (f'-f'')\,dx=\epsilon(1+o_{\epsilon}(1))|X|,$$
    as $\epsilon\to0$ (where $o_{\epsilon}(1)\to 0$ as $\epsilon\to0$). 
    \item On the other hand, a closer look shows that
    $$\int \left(\M^*(f',f')-\M^*(f'',f'')\right)\,dx=\epsilon(1+o(1))|\lambda X+(1-\lambda)X|.$$
    \item Combining these, we find that
    $$|\lambda X+(1-\lambda) X|\leq (1+2c)|X|,$$
    with $c$ sufficiently small in terms of $n,\lambda,p$, so we can apply the stability of the Brunn-Minkowski inequality to deduce\footnote{Here and in the sequel, given positive constants $a$ and $b$, by $a \ll_{n,\lambda,p}b$ we mean that $a \leq cb$ for some constant $c>0$ that depends on $n,\lambda,p$ and that can be made arbitrary small.} $$|\co(X)\setminus X|\ll_{n,\lambda,p}|X|.$$
    In addition, one can note that $\co(X)=F$.
\end{enumerate} 
Given the steps above, we fist establish the result in the case when $L$ varies by at most a constant factor (say $100$) in the face $F$. To this end, we use the stability of the Brunn-Minkowski inequality in each of the level sets combined with the information that $|\co(X)\setminus X|\ll_{n,\lambda,p}|X|$. We will use a refined version of the above proof to reduce to this case. Indeed we find a triangulation of the support of $f'$ into simplices in which $\ell^p$ is essentially linear, $\ell$ varies by at most a factor 100, and with vertices $v$ on which $f'(v)=\ell(v)$. We achieve this as follows.

\begin{enumerate}
    \item Consider a one-dimensional face $F=[0,1]$ (i.e., an edge) of $\ell$, so that by an argument akin to the one above we have $|\co(X)\setminus X|\ll_{n,\lambda,p}|X|$, where $\co(X)$ is the entire interval $[0,1]$.
    \item Consider a one-dimensional face $F=[0,1]\times\{0\}^{n-1}$ (an edge) of $\ell$, i.e. so that there exists a function $L\geq \ell$ so that $L^p$ is linear on $\mathbb{R}^n$, $L=\ell$ on $F$, $L>\ell$ outside $F$, and $L=\ell=f'$ at the two points $\{0,1\}\times\{0\}^{n-1}$.
    \item By an argument akin to the one above, we have $|\co(X)\setminus X|\ll_{n,\lambda,p}|X|$ (in the 1-dimensional Hausdorff measure), where $\co(X)=F$.
    \item Our analysis will look at arbitrarily small shavings so that we may restrict our attention to what happens in this (one-dimensional!) face.
    \item For a contradiction, assume there is some big gap in $X$, i.e., there are $a,b\in [0,1]$ so that $f'(a\times\{0\}^{n-1})>100 \,f'(b\times\{0\}^{n-1})$ and $X\cap ([a,b]\times\{0\}^{n-1})=\{a,b\}\times\{0\}^{n-1}.$
    \item Rather than pushing down $L$ everywhere, we shave by rotating $L$ around $a$ as per the following definition: Consider the function $L_\epsilon$ so that
    \begin{itemize}
        \item $L_\epsilon=L=\ell$ on $(-\infty,a]\times\mathbb{R}^{n-1}$,
        \item $L_\epsilon^p$ is linear on $[a,\infty)\times\mathbb{R}^{n-1}$,
        \item $L_\epsilon(b)=\ell-\epsilon$.
    \end{itemize}
    Note that $L_{\epsilon}\to L$ pointwise as $\epsilon\to0$.
    \item Similarly to before, consider the function $f''=\min\{f',L_\epsilon\}$ (or equivalently $\ell'=\min\{\ell,L_\epsilon\}$ and $f''=\min\{f,\ell'\}$), and use the minimality of $f'$ to find that we cannot remove much of $\M^*(f',f')$ compared to $f'$.
    \item Note that the removed mass $\int (f'-f'')\,dx$ all lies in $[b,1]\times\mathbb{R}^{n-1}$, where $\ell-L$ is small, while the mass of $$\int \left(\M^*(f',f')-\M^*(f'',f'')\right)\,dx$$ also lies in $[a,b]\times\mathbb{R}^{n-1}$, where $\ell-L$ is (relatively) large. A careful inspection of this discrepancy leads to a contradiction.
\end{enumerate}
Hence, for every edge $F$ of $\ell$ we find that the set $\ell(X\cap F)$ does not leave big gaps in $\ell(F)$. We can thus find a triangulation of the domain using elements of $X$ on the edges so that in each of the simplices, the $\ell$ varies by at most a factor 100. Using the previously described approach, this allows us to conclude.

\section{Proof of \Cref{linear}}
\label{sect:proof thm2}

\subsection{Reduction from \Cref{linear} to \Cref{linear_continuous}}
\begin{prop}\label{linear_continuous}
    Given $n \in \mathbb{N}$, $\lambda \in (0,1/2]$, $p \in (-1/n, \infty)$ there exists $d=d_{n, \lambda, p}>0$ such that the following holds. Let $f\colon \mathbb{R}^n\rightarrow \mathbb{R}_{\geq 0}$ be a continuous function with bounded support such that  
    \begin{itemize}
        \item $\int_{\mathbb{R}^n}\M^*(f,f) \,dx \leq (1+\delta) \int f \,dx$ for some $0 \leq \delta \leq d$.
    \end{itemize}
    Then there exists a $p$-concave function $\ell \colon \mathbb{R}^n\rightarrow \mathbb{R}_{\geq 0}$ such that
        \begin{itemize}
        \item $\int_{\mathbb{R}^n} |f-\ell|\, dx=O_{n,\lambda, p} (\delta)\int_{\mathbb{R}^n} f\,dx $.
    \end{itemize}
\end{prop}

\begin{proof}[Proof \Cref{linear_continuous} implies \Cref{linear}]
    We can find a continuous functions with bounded support $f'\colon \mathbb{R}^n\rightarrow \mathbb{R}_{\geq0}$ such that $\int_{\mathbb{R}^n} |f-f'|\, dx\leq \delta \int_{\mathbb{R}^n} f\,dx$ and $\int_{\mathbb{R}^n} \left( \M^*(f',f') - \M^*(f,f)\right)\, dx\leq \delta \int_{\mathbb{R}^n} f\,dx$. It follows that $\int_{\mathbb{R}^n} \M^*(f',f')\, dx\leq (1+5\delta) \int_{\mathbb{R}^n} f'\,dx$. By \Cref{linear_continuous}, we get that there exists a $p$-concave function $\ell \colon \mathbb{R}^n\rightarrow \mathbb{R}_{\geq 0}$ such that $\int_{\mathbb{R}^n} |f'-\ell| \,dx=O_{n,\lambda, p} (5\delta) \int_{\mathbb{R}^n} f'\, dx$. We conclude $\int_{\mathbb{R}^n} |f-\ell|\, dx=O_{n,\lambda, p} \left(\delta\right)\int_{\mathbb{R}^n} f\,dx$.
    
\end{proof}

\subsection{Reduction from \Cref{linear_continuous} to \Cref{linear_continuous_level_weak}, \ref{linear_shaving_process}, \ref{linear_already_shaved}}
Before we can state our results we collect some definitions in the subsection below.
\subsubsection{Definitions}
For the rest of this section we use the following definitions.

\begin{defn}
    Given a compact convex set (possibly degenerate) $K \subset \mathbb{R}^n$ we denote by $\text{int}(K)$ the relative interior of $K$ i.e.  letting $H$ be the affine subspace spanned by $K$, we define $\text{int}(K)=\{ x \in K \colon \exists r>0, B(x,r) \cap H \subset K\}$.
\end{defn}

\begin{defn}
    Given a compact convex set (possibly degenerate) $K \subset \mathbb{R}^n$, we say that $v \in K$ is a \emph{vertex} of $K$ if whenever $v=tx+(1-t)y$ for some $t\in (0,1)$ and $x,y \in K$, we have $v=x=y$.  
\end{defn}

\begin{defn}
    Given two points $a,b \in \mathbb{R}^n$, we denote by $[a,b]:=\{ta+(1-t)b: t\in[0,1]\}$ the line segment from $a$ to $b$.
\end{defn}

\begin{defn}
    Given a compact convex set (possibly degenerate) $K \subset \mathbb{R}^n$, we say that $[a,b] \subset K$ is an \emph{edge} of $K$ if whenever $c \in [a,b]$, $t\in (0,1)$ and $d,e \in K$ are such that $c= td+(1-t)e$, we have $d,e \in [a,b]$.   
\end{defn}

\begin{defn}
     Given $n \in \mathbb{N}$, $p \in (-1/n,0)$, $y \in \mathbb{R}^n$, and $d \in \mathbb{R}$,   a \emph{$p$-plane with parameters $y$ and $d$} is a function $h_{y,d} \colon \mathbb{R}^n \rightarrow \mathbb{R}_{\geq 0} \cup \{\infty\}$ defined by $$h_{y,d}(x)=\begin{cases}0 &\text{ if } p>0 \text{ and } <x,y>+d \leq 0\\
\infty &\text{ if } p< 0 \text{ and }<x,y>+d \leq 0\\
     (<x,y>+d)^{1/p}   &\text{ if } p\neq 0 \text{ and } <x,y>+d > 0\\
\exp(<x,y>+d)&\text{ if $p=0$} \end{cases}.$$
\end{defn}

\begin{rmk}\label{p-plane_p-concave}\label{concave_function}
    A $p$-plane $h_{y,d} \colon \mathbb{R}^n  \rightarrow \mathbb{R}_{\geq 0}\cup \{\infty\}$ is a continuous $p$-concave function. Moreover, if we denote by $W$ the $(n-1)$-dimensional subspace that is the orthogonal complement of $y$, then for any $x \in \mathbb{R}^n$, we have $h_{y,d}(x)= h_{y,d}(x+W)$ i.e. $h_{y,d}$ is constant on $x+W$.
\end{rmk}

\begin{defn}\label{tangent_plane}
    Let $n \in \mathbb{N}$ and $p \in (-1/n,\infty)$. Let $g \colon \mathbb{R}^n \rightarrow \mathbb{R}_{\geq 0}$ be a continuous $p$-concave function with bounded support. We say that a compact convex set $F \subset \mathbb{R}^n$ with $\text{int}(F) \subset \text{supp}(g)$ is a \emph{ $p$-face of $g$} if there exists a $p$-plane $h \colon \mathbb{R}^n \rightarrow \mathbb{R}_{\geq 0}$ such that $h(x)=g(x)$ for all $x \in F$ and $h(x)>g(x)$ for all $ x \in \text{supp}(g) \setminus F$.  In this case, we say $h$ is a \emph{tangent} $p$-plane.
\end{defn}

\subsubsection{Main results}
We are now ready to state the main results in this subsection.

\begin{prop}\label{linear_continuous_level_weak}
    Given $n \in \mathbb{N}$, $\lambda \in (0,1/2]$, $p \in (-1/n, \infty)$ and $ m \in (1, \infty)$ there exists $d=d_{n, \lambda, p,m}>0$ such that the following holds. Let $f\colon \mathbb{R}^n\rightarrow \mathbb{R}_{\geq 0}$ be a continuous function with bounded support such that  
    \begin{itemize}
        \item $\int_{\mathbb{R}^n}\M^*(f,f) \,dx \leq (1+\delta) \int f \,dx$ for some $0 \leq \delta \leq d$
        \item for any $x \in \co(\text{supp}(f))$, there exists a $p$-face $F \subset \overline{\co(\text{supp}(f))}$ of $\co_p(f)$ and a simplex $T \subset F$ (not necessarily  full dimensional) such that $x \in \text{int}(T)$  and for any $y \in V(T)$ we have $f(y)=\co_p(f)(y)$ and $\co_p(f)(x)/\co_p(f)(y) \in [1/m,m]$. 
    \end{itemize}
    Then 
        \begin{itemize}
        \item $\int_{\mathbb{R}^n}\left( \co_p(f)-f\right)\, dx=O_{n,\lambda, p,m} (\delta)\int_{\mathbb{R}^n} f\,dx $.
    \end{itemize}
\end{prop}

\begin{prop}\label{linear_shaving_process}
    Given $n \in \mathbb{N}$, $\lambda \in (0,1/2]$, $p \in (-1/n, \infty)$, $c\in (0,1)$ there exists $d=d_{n, \lambda, p, c}>0$ such that the following holds. Let $f\colon \mathbb{R}^n\rightarrow \mathbb{R}_{\geq 0}$ be continuous functions with bounded support such that  
    \begin{itemize}
        \item $\int_{\mathbb{R}^n}\M^*(f,f)\, dx = (1+\delta)\int_{\mathbb{R}^n}f\,dx$ for some $0\leq \delta \leq d$.
    \end{itemize}
    There exists a $p$-concave function   $\ell' \colon \mathbb{R}^n\rightarrow \mathbb{R}_{\geq 0}$  such that $f'\colon \mathbb{R}^n \rightarrow \mathbb{R}_{\geq 0}$, $f'=\min(f,\ell')$ satisfies
        \begin{itemize}
        \item $\int_{\mathbb{R}^n} \left[-c(f-f')+ (\M^*(f,f)- f)-(\M^*(f',f')-f') \right]\,dx \geq 0$.
    \end{itemize}
    Moreover,
    \begin{itemize}
        \item $\int |f-f'|\,dx \leq O_c(\delta) \int f\,dx$
        \item  $\int_{\mathbb{R}^n}\M^*(f',f')\, dx = (1+3\delta)\int_{\mathbb{R}^n}f'\,dx$.
    \end{itemize}
     Furthermore, for any $p$-concave function $\ell'' \colon \mathbb{R}^n\rightarrow \mathbb{R}_{\geq 0}$, the function $f''\colon \mathbb{R}^n \rightarrow \mathbb{R}_{\geq 0}$, $f''=\min(f',\ell'')$ satisfies
        \begin{itemize} 
        \item either $f''=f'$
        \item or $\int_{\mathbb{R}^n} \left[-c(f'-f'')+ (\M^*(f',f')- f')-(\M^*(f'',f'')-f'') \right]\,dx < 0$.
    \end{itemize}
\end{prop}

\begin{prop}\label{linear_already_shaved}
Given $n \in \mathbb{N}$, $\lambda \in (0,1/2]$, $p \in (-1/n, \infty)$, there exists $c=c_{n, \lambda, p}>0$ and $m=m_{n, \lambda, p}>0$ such that the following holds. Let $f\colon \mathbb{R}^n\rightarrow \mathbb{R}_{\geq 0}$ be a continuous functions with bounded support such that  for any $p$-concave function $\ell' \colon \mathbb{R}^n\rightarrow \mathbb{R}_{\geq 0}$, the function $f'\colon \mathbb{R}^n \rightarrow \mathbb{R}_{\geq 0}$, $f'=\min(f,\ell')$ satisfies
        \begin{itemize}
        \item either $f'=f$,
        \item or $\int_{\mathbb{R}^n} \left[-c(f-f')+ (\M^*(f,f)- f)-(\M^*(f',f')-f')\right] dx < 0$.
    \end{itemize}
Then for any $p$-face $F$ of $\co_p(f)$ and any edge $[p,q] \subset F$, with $\co_p(f)(p)>\co_p(f)(q)$ there exists a partition of the segment $[p,q)=[x_1,x_2)\sqcup  [x_2, x_{3}) \sqcup \dots$ such that
$f(x_i)=\co_p(f)(x_i)$ and $f(x_i)/f(x_{i+1}) \in [1/m,m]$ for all $i$.
\end{prop}

\begin{proof}[Proof  \Cref{linear_continuous_level_weak}, \ref{linear_shaving_process}, \ref{linear_already_shaved} imply \Cref{linear_continuous}]

Let $f$ be a function satisfying the hypothesis of \Cref{linear_continuous}. Fix $c,m$ as in \Cref{linear_already_shaved}. By \Cref{linear_shaving_process}, there exists a $p$-concave function    $\ell' \colon \mathbb{R}^n\rightarrow \mathbb{R}_{\geq 0}$  such that $f'\colon \mathbb{R}^n \rightarrow \mathbb{R}_{\geq 0}$, $f'=\min(f,\ell')$ satisfies
    \begin{itemize}
        \item $\int |f-f'|\,dx \leq O_c(\delta) \int f\,dx$
        \item  $\int_{\mathbb{R}^n}\M^*(f',f')\, dx = (1+3\delta)\int_{\mathbb{R}^n}f'\,dx$.
    \end{itemize}
     Furthermore, for any $p$-concave function $\ell'' \colon \mathbb{R}^n\rightarrow \mathbb{R}_{\geq 0}$, the function $f''\colon \mathbb{R}^n \rightarrow \mathbb{R}_{\geq 0}$, $f''=\min(f',\ell'')$ satisfies
        \begin{itemize} 
        \item either $f''=f'$
        \item or $\int_{\mathbb{R}^n} \left[-c(f'-f'')+ (\M^*(f',f')- f')-(\M^*(f'',f'')-f'') \right]\,dx < 0$.
    \end{itemize}

By \Cref{linear_already_shaved}, for any $p$-face $F$ of $\co_p(f')$ and any edge $[p,q] \subset F$, with $\co_p(f')(p)>\co_p(f')(q)$ there exists a partition of the segment $[p,q)=[x_1,x_2)\sqcup  [x_2, x_{3}) \sqcup \dots$ such that
$f'(x_i)=\co_p(f')(x_i)$ and $f'(x_i)/f'(x_{i+1}) \in [1/m,m]$ for all $i$. 

This immediately implies for any $x \in \co(\text{supp}(f'))$, there exists a face $F \subset \overline{\co(\text{supp}(f'))}$ of $\co_p(f')$ and a simplex $T \subset F$ (not necessarily  full dimensional) such that $x \in \text{int}(T)$ and for any $y \in V(T)$ we have $f'(y)=\co_p(f')(y)$ and $\co_p(f')(x)/\co_p(f')(y) \in [1/m,m]$.

By \Cref{linear_continuous_level_weak} combined with the second bullet point, we deduce that
 $$\int_{\mathbb{R}^n}\left( \co_p(f')-f'\right)\, dx=O_{n,\lambda, p,m} (\delta)\int_{\mathbb{R}^n} f'\,dx $$

 By the first bullet point, we conclude that 
 $$\int_{\mathbb{R}^n}\left( |\co_p(f')-f|\right)\, dx=O_{n,\lambda, p,m} (\delta)\int_{\mathbb{R}^n} f\,dx $$

\end{proof}

\subsection{Reduction from \Cref{linear_continuous_level_weak} to \Cref{linear_continuous_level_medium}}

\begin{prop}\label{linear_continuous_level_medium}
    Given $n \in \mathbb{N}$, $\lambda \in (0,1/2]$, $p \in (-1/n, \infty)$, $ m \in (1, \infty)$ and $k \in (1,\infty)$ there exist $d=d_{n, \lambda, p,m,k}>0$ and $s=s_{n,\lambda, p,m,k}\geq 2k$ such that the following holds. Let $f\colon \mathbb{R}^n\rightarrow \mathbb{R}_{\geq 0}$ be a function with bounded support such that
    \begin{itemize}
        \item $f$ is continuous on $\overline{\co(\text{supp}(f))}$
        \item there exists a compact set $Y\subset\mathbb{R}^n$ such that $\co_p(f)=\co_p(1_Y \times f)$
        and for $y \in Y$ we have $f(y) \in [1/m,m]$
        \item $\int_{\mathbb{R}^n}f\, dx > k^{-1}\int_{\mathbb{R}^n}\co_p(f)\, dx  $
        \item letting $S=\{x \colon f(x) \geq 1/ms\}$, we have $\int_{\mathbb{R}^n}\M^*(f \times 1_S,f\times 1_S) \,dx \leq (1+\delta) \int_{\mathbb{R}^n} (f\times 1_{\lambda S+(1-\lambda) S} )\,dx$ for some $0 \leq \delta \leq d$.
 
    \end{itemize}
    Then 
        \begin{itemize}
        \item $\int_{\mathbb{R}^n}\left( \co_p(f)-f\right)\, dx=O_{n,\lambda, p,k,m} (\delta)\int_{\mathbb{R}^n} f\,dx \leq \int_{\mathbb{R}^n} f\,dx $.
    \end{itemize}.
\end{prop}

We need the following lemma.
\begin{lem}\label{p-concave_tail}
    Let $n \in \mathbb{N}$, $p \in (-1/n,\infty)$. Let $ g \colon \mathbb{R}^n \rightarrow \mathbb{R}_{\geq 0}$ be a $p$-concave function. Assume without loss of generality that $g(o)=1$ and $g(x) \leq 1$ for every $x \in \mathbb{R}^n$.  Let $S= \{x \colon g(x) \geq 1/2\}$. Then $\int_{S} g \,dx =\Theta_{n,p}(1)\int_{\mathbb{R}^n} g \,dx$. 
\end{lem}

\begin{proof}[Proof of \Cref{p-concave_tail}]
    Note that $S$ is convex and contains $o$. Because $g$ is $p$-concave, it follows that for $\lambda \geq 1$ and $x \in \partial(\lambda S)$ (so that $\lambda^{-1} x \in \partial(S)$), we have $g(\lambda^{-1} x) \geq M_{\lambda^{-1},p}(g(x), g(o))$. Assume $g(x)\neq 0$.

    First assume $p \in (-1/n,0)$. We get $ g(x)\leq \lambda^{1/p}(1/2^p-1+\lambda^{-1})^{1/p} \leq \lambda^{1/p}$. Therefore, $\int_{S^c} g \,dx \leq \int_{\lambda \geq 1} \lambda^{1/p} n \lambda^{n-1}|S| \,d\lambda = \Theta_{n, \lambda, p}(|S|)$.

    Second assume $p = 0$.  We get $g(x) \leq 1/2^{\lambda}$.
    Hence, $\int_{S^c} g \,dx \leq \int_{\lambda \geq 1} 2^{-\lambda} n \lambda^{n-1}|S| \,d\lambda = \Theta_{n, \lambda, p}(|S|)$. 

    Third assume $p>0$. We get $ g(x)\leq \lambda^{1/p}(1/2^p-1+\lambda^{-1})^{1/p}$, so $\lambda \leq (1- 1/2^p)^{-1}$. Hence, $\int_{S^c} g \,dx |(1-1/2^p)^{-1} S|=(1-1/2^p)^{-n}|S| $.

    We have $ \int_S g \,dx \geq 2^{-1}|S| $. We conclude
    $\int_{S} g \,dx =\Theta_{n,p}(1)\int_{\mathbb{R}^n} g \,dx$. 
    
\end{proof}

\begin{proof}[Proof \Cref{linear_continuous_level_medium} implies \Cref{linear_continuous_level_weak}]
     Let $f \colon \mathbb{R}^n \rightarrow \mathbb{R}_{\geq 0}$ be a  function  that satisfies the hypothesis in \Cref{linear_continuous_level_weak}. Assume without loss of generality that $f(o)=1$ and $f(x) \leq 1$ for every $x \in \mathbb{R}^n$. As $f$ is continuous with bounded support, it follows that $\co_p(f)$ is continuous with bounded support. Construct the level sets  $K_i= \{x \colon \co_p(f) \geq m^{-i}\}$. Because $\co_p(f)$ is continuous with bounded support, we deduce that $K_i$ are compact. Note that the level sets $K_i$ of $\co_p(f)$ are convex and nested $K_i \subset K_{i+1}$. Note that for $i\geq 1$, $\text{int}(K_i)=\{x \colon \co_p(f)>m^{-i}\}$ is an open set and $K_{i-1} \subset \text{int}(K_i)$.

    Define the compact subset $Y_i  \subset K_{i+1} \setminus \text{int}(K_{i-2})$, $Y_i := \{ y \in K_{i+1} \setminus \text{int}(K_{i-2}) \colon f(y)=\co_p(f)(y) \}$. By the second hypothesis, it follows that for $x \in K_{i} \setminus \text{int}(K_{i-1})$, $\co_p(f)(x) = \co_p(f \times 1_{Y_i})(x)$.

    Define $f_i= \min(f, \co_p(1_{Y_i} \times f))$. Note that for $y \in Y_i$ we have $f_i(y)=f(y) = \co_p(f)(y) \in [m^{-1-i}, m^{2-i}]$. Moreover $\co_p(f_i)=\co_p(1_{Y_i} \times f_i)$. Indeed, on the one hand we trivially have $\co_p(f_i) \geq \co_p(1_{Y_i} \times f_i)$. On the other hand we have $\co_p(f_i) \leq \co_p(1_{Y_i} \times f) = \co_p(1_{Y_i} \times f_i)$. Furthermore, as  $f$ is continuous and as $\co_p(1_{Y_i} \times f)$ is continuous on its support which is convex and compact, we deduce that $f_i$ is continuous on $\overline{\co(\text{supp}(f_i))}$. Additionally, for $x \in K_i \setminus \text{int}(K_{i-1})$, $\co_p(f_i)(x)=\co_p(f)(x)$. Indeed, $\co_p(f_i)(x)= \co_p(1_{Y_i} \times f_i)(x)=  \co_p(1_{Y_i} \times f)(x)=\co_p(f)(x)$. 

    \begin{clm}
        If $\int_{\mathbb{R}^n}f_i \,dx \geq 2^{-1}\int_{\mathbb{R}^n} \co_p(f_i) \,dx$, then  $\int_{\mathbb{R}^n}f_{i+1} \,dx \geq  k^{-1}\int_{\mathbb{R}^n} \co_p(f_{i+1}) \,dx$ for some $2\leq k=O_{n,p,m}(1)$.
    \end{clm}


    \begin{proof}[Proof of Claim]

        First, by construction, for $x \in K_{i-1}$, $\co_p(f_i)(x) \in [m^{1-i},m^{2-i}]$, for $x \in K_i \setminus K_{i-1}$, $\co_p(f_i)=\co_p(f)(x) \in [m^{-i}, m^{1-i}]$, for $x \in K_{i+1} \setminus K_i$, $\co_p(f_i)(x) \leq m^{-i}$ and for $x \in K_{i+1}^c$, $\co_p(f_i)(x)=0$. Hence, $\co_p(f_{i+1})= \Omega_m(1)\co_p(f_i)$  and for $x \in K_i$, $\co_p(f_{i+1})(x) = O_m(1)\co_p(f_i)$. In particular, $$\int_{K_i} \co_p(f_{i+1}) \, dx\leq O_m(1) \int_{K_i}\co_p(f_i) \,dx.$$

        By construction, for $x \in \mathbb{R}^n$, $f_i(x)= \min(f(x), \co_p(f_i)(x))$. By the first paragraph, we get 
        
\begin{equation}\label{f_i_f_{i+1}}
            \int_{\mathbb{R}^n}f_{i+1} \,dx = \Omega_m(1) \int_{\mathbb{R}^n} f_i\,dx. 
\end{equation}

    Second, $\int_{K_i^c}\co_p(f_{i+1})\,dx = O_{n,p,m} \int_{K_i}\co_p(f_i)\,dx$. To see this, consider the function $g_i \colon \mathbb{R}^n \rightarrow \mathbb{R}_{\geq 0}$ defined as follows: for $x \in K_i$, $g_i(x)= \co_p(f_i)(x)$ and for $x \in K_i^c$, $g_i=\co_p(f_{i+1})(x)$. As observed above, for $x \in K_i \setminus K_{i-1}$, we have $g_i(x)=\co_p(f_i)(x)= \co_p(f)(x)$ and for $x \in K_{i+1} \setminus K_i$, we have $g_{i}(x)= \co_p(f_{i+1})(x)= \co_p(f)(x)$. As $\co_p(f_i)$, $\co_p(f_{i+1})$ and $\co_p(f)$ are $p$-concave, and as being $p$-concave is a local property and $g_i$ is obtained by gluing together these functions, we deduce that $g_i$ is also $p$-concave. By the first paragraph, for $x \in K_{i-1}$, we have $g_i(x) \in [m^{1-i},m^{2-i}]$ and for $x \in K_i^c$, we have $g_i(x) \leq m^{-i}$.  By \Cref{p-concave_tail} (assuming wlog $m\geq 2$), we deduce $\int_{K_i}g_i(x) \, dx= \Omega_{n,p,m}(1) \int_{K_i^c}g_i(x) \,dx$. In other words,  $$\int_{K_i^c}\co_p(f_{i+1}) \,dx = O_{n,p,m} \int_{K_i}\co_p(f_i) \,dx.$$

    Putting everything together, we conclude
    $$\int_{\mathbb{R}^n} f_{i+1} \,dx \geq \Omega_m(1) \int_{\mathbb{R}^n} f_i\,dx \geq \Omega_m(1) \int_{\mathbb{R}^n} \co_p(f_i)\,dx \geq \Omega_{n,p,m}(1) \int_{\mathbb{R}^n} \co_p(f_{i+1})\,dx.$$
        
    \end{proof}

    Let $d=d_{n, \lambda, p,m^2,k}>0$ and $s=s_{n,\lambda, p,m^2,k}>0$ be the output of \Cref{linear_continuous_level_medium} with input $n,\lambda,p$ and $m^2$ (instead of $m$) and $k$.

    For $i \in \mathbb{N}$, define the set $S_{i}=\{x \colon f_{i}(x) \geq m^{-i-2}s^{-1}\}$ and define $\delta_i \geq 0$ such that $$\int_{\mathbb{R}^n}\M^*(f_i \times 1_{S_i},f_i\times 1_{S_i}) \,dx = (1+\delta_i) \int_{\mathbb{R}^n} (f_i\times 1_{\lambda S_i+(1-\lambda) S_i} )\,dx.$$ 
    
    Finally, let $i_0= \inf \{i \colon \delta_i > d\}$.

    \begin{clm}
        For $1 \leq i <i_0$, we have $\int_{\mathbb{R}^n} \co_p(f_i)\,dx \leq (1+O_{n,\lambda, p,k,m} (\delta_i))\int_{\mathbb{R}^n} f_i\,dx \leq 2\int_{\mathbb{R}^n} f_i\,dx$.
    \end{clm}

    \begin{proof}[Proof of Claim]
        We first consider the case $i=1$. By a weak stability result (see e.g. \Cref{almostconvexlevelsets}, \Cref{aclcor}), we have $ \int_{\mathbb{R}^n}f_1 \,dx \geq 2^{-1}\int_{\mathbb{R}^n}\co_p(f_1) \,dx$. Recalling the properties of $f_1$, we see that $mf_1$ satisfies the hypothesis of \Cref{linear_continuous_level_medium} with parameters $n,\lambda,p$ and $m^2$ (instead of $m$) and $k$ and thus we can apply \Cref{linear_continuous_level_medium} to conclude the case $i=1$.

        Now assume that the claim holds for $i-1$. Then $m^if_i$ satisfies the hypothesis of \Cref{linear_continuous_level_medium} with parameters $n,\lambda,p$ and $m^2$ (instead of $m$) and $k$. Indeed, the first two hypotheses hold by the discussion preceding the first claim. The third hypothesis holds by combining the first claim with the inductive hypothesis.  The fourth hypothesis holds by the definition of $\delta_i$ and the assumption that $\delta_i<d$. Thus we can apply \Cref{linear_continuous_level_medium} to conclude.
    \end{proof}

    \begin{clm}
        $\sum_{i<i_0} \int_{K_i \setminus K_{i-1}} \left(\co_p(f)(x)-f(x)\right)\,dx \leq O_{n,\lambda,p,m}(\delta) \int_{\mathbb{R}^n}f(x)\,dx$.
    \end{clm}

    \begin{proof}[Proof of Claim]
    Recall that for $x \in K_i \setminus K_{i-1}$, we have $\co_p(f)(x)=\co_p(f_i)(x)$. Also, $f_i= \min(f, \co_p(f_i))$. Hence,
    $$\int_{K_i \setminus K_{i-1}}\left(\co_p(f)(x)-f(x)\right)\,dx= \int_{K_i \setminus K_{i-1}}\left(\co_p(f_i)(x)-f_i(x)\right)\,dx \leq \int_{\mathbb{R}^n} \left(\co_p(f_i)(x)-f_i(x)\right)\,dx.$$
    
    By the second claim, we have 
    $$\sum_{i<i_0} \int_{K_i \setminus K_{i-1}} \left(\co_p(f)(x)-f(x)\right)\,dx \leq \sum_{i<i_0}O_{n,\lambda,p,m}(\delta_i) \int_{\mathbb{R}^n}f_i(x)\,dx$$

    By the second claim, we also have $$\int_{\mathbb{R}^n} \co_p(f_i)\,dx \leq 2 \int_{\mathbb{R}^n}f_i \,dx.$$

    Recall that for $x \in \text{supp}(\co_p(f_i))$, we have $\co_p(f_i) \in [m^{-i-1},m^{2-i}]$ and $s \geq 2k \geq 4$ and $S_i=\{x \colon f_i(x)\geq 1/m\}$. Hence, $S_i \supset \{ x \colon f_i(x) \geq 4^{-1}\co_p(f_i)\}$. It follows that
   $$\int_{\mathbb{R}^n} f_i\times 1_{\lambda S_i+(1-\lambda) S_i} \,dx \geq 2^{-1} \int_{\mathbb{R}^n}f_i\,dx.$$ 
   
   Combining the last equation and the third last equation, we deduce
       $$\sum_{i<i_0} \int_{K_i \setminus K_{i-1}} \left(\co_p(f)(x)-f(x)\right)\,dx \leq \sum_{i<i_0}O_{n,\lambda,p,m}(\delta_i) \int_{\mathbb{R}^n} f_i\times 1_{\lambda S_i+(1-\lambda) S_i} \,dx.$$

    By the definition of $\delta_i$, we further get
$$\sum_{i<i_0} \int_{K_i \setminus K_{i-1}} \left(\co_p(f)(x)-f(x)\right)\,dx =O_{n,\lambda,p,m}(1) \int_{\mathbb{R}^n}\left( \M^*(f_i \times 1_{S_i},f_i\times 1_{S_i})- f_i\times 1_{\lambda S_i+(1-\lambda) S_i}\right) \,dx.$$

Given that
$$\int_{\mathbb{R}^n}\left(\M^*(f,f)-f\right)\,dx \leq \delta\int_{\mathbb{R}^n}f\,dx,$$
it is enough to argue that
$$ \sum_i\int_{\mathbb{R}^n} \left(\M^*(f_i \times 1_{S_i},f_i\times 1_{S_i})- f_i\times 1_{\lambda S_i+(1-\lambda) S_i}\right) \,dx \leq O_{n,\lambda,p,m}(1)\int_{\mathbb{R}^n}\left(\M^*(f,f)-f\right)\,dx.$$

As $f_i=\min(f, \co_p(f_i))$ is the restriction of $f$ to a $p$-concave function, we have
$$ \int_{\mathbb{R}^n}\left( \M^*(f_i \times 1_{S_i},f_i\times 1_{S_i})- f_i\times 1_{\lambda S_i+(1-\lambda) S_i} \right)\,dx \leq  \int_{\mathbb{R}^n} \max\bigg(0, \M^*(f \times 1_{S_i},f\times 1_{S_i})- f\times 1_{\lambda S_i+(1-\lambda) S_i} \bigg) \,dx.$$

Therefore, it is enough to argue that
$$ \sum_i\int_{\mathbb{R}^n} \max\bigg(0, \M^*(f \times 1_{S_i},f\times 1_{S_i})- f\times 1_{\lambda S_i+(1-\lambda) S_i} \bigg) \,dx\leq  O_{n,\lambda,p,m}(1)\int_{\mathbb{R}^n}\left(\M^*(f,f)-f\right)\,dx.$$

Fix any $x \in \mathbb{R}^n$. It is enough to show that
$$ \sum_i \max\bigg(0, \M^*(f \times 1_{S_i},f\times 1_{S_i})(x)- f\times 1_{\lambda S_i+(1-\lambda) S_i}(x) \bigg) \leq  O_{n,\lambda,p,m}(1)\bigg(\M^*(f,f)(x)-f(x)\bigg).$$

Let $I=\{i \colon x \in \lambda S_i + (1-\lambda)S_i\}$. It is enough to show that
$$ \sum_{i\in I} \max\bigg(0, \M^*(f \times 1_{S_i},f\times 1_{S_i})(x)- f(x) \bigg) \leq  O_{n,\lambda,p,m}(1)\bigg(\M^*(f,f)(x)-f(x)\bigg).$$

Note that for $x \in S_i$, we have $f(x) \in [m^{-2-i}s^{-1},m^{2-i}]$. Indeed, on the one hand, we have $S_i \subset \text{supp}(f_i) \subset K_{i+1}\setminus K_{i-2}$ and for $x \in S_i \subset K_{i-2}^c$, we have $f(x)\leq \co_p(f)(x)\leq m^{2-i}$. Additionally, for $x \in S_i$, we have $f_i(x)\geq m^{-2-i}s^{-1}$. As $f_i=\min(f,\co_p(f_i)) \leq f$,  for $x \in S_i$, we have  $f(x)\geq m^{-2-i}s^{-1}$. 

Therefore, if $i_1,i_2 \in I$ satisfy
$$ \frac{\M^*\left(f \times 1_{S_{i_1}},f\times 1_{S_{i_1}}\right)(x)}{ \M^*\left(f \times 1_{S_{i_2}},f\times 1_{S_{i_2}}\right)(x)} \in [1/2,2],$$
then
$$|i_1-i_2|=O_{n,\lambda,p,m}(1).$$

Moreover, 
$$\sup_i \max\bigg(0, \M^*(f \times 1_{S_i},f\times 1_{S_i})(x)- f(x) \bigg) \leq  \M^*(f,f)(x)-f(x).$$

Combining the last two equations, we conclude that
$$ \sum_{i\in I} \max\bigg(0, \M^*(f \times 1_{S_i},f\times 1_{S_i})(x)- f(x) \bigg) \leq  O_{n,\lambda,p,m}(1)\bigg(\M^*(f,f)(x)-f(x)\bigg).$$
This finishes the proof of the claim
\end{proof}

    \begin{clm}
        If $i_0< \infty$, then $\int_{\mathbb{R}^n}\co_p(f_{i_0}) \,dx= O_{n,\lambda,p,m}(\delta) \int_{\mathbb{R}^n} f\,dx$.
    \end{clm}

\begin{proof}[Proof of Claim]

By construction,  $$\int_{\mathbb{R}^n}\M^*\left(f_{i_0} \times 1_{S_{i_0}},f_{i_0}\times 1_{S_{i_0}}\right) \,dx \geq (1+d) \int_{\mathbb{R}^n} f_{i_0}\times 1_{\lambda S_{i_0}+(1-\lambda) S_{i_0}} \,dx \geq (1+d) \int_{\mathbb{R}^n} f_{i_0}\times 1_{S_{i_0}}\,dx.$$

Hence, $$\int_{\mathbb{R}^n}\left(\M^*(f_{i_0} \times 1_{S_{i_0}},f_{i_0}\times 1_{S_{i_0}})- f_{i_0}\times 1_{S_{i_0}}\right) \,dx \geq  d\int_{\mathbb{R}^n} f_{i_0}\times 1_{S_{i_0}}\,dx. $$

We have $$\int_{\mathbb{R}^n}\left(\M^*(f_{i_0},f_{i_0})- f_{i_0}\right)\, dx\geq \int_{\mathbb{R}^n}\left(\M^*(f_{i_0} \times 1_{S_{i_0}},f_{i_0}\times 1_{S_{i_0}})- f_{i_0}\times 1_{S_{i_0}}\right)\, dx$$

As $f_{i_0}= \min(f, \co_p(f_{i_0}))$ is $f$ capped by a $p$-concave function, we get 
 $$\int_{\mathbb{R}^n}\left(\M^*(f,f)- f\right) \,dx \geq  \int_{\mathbb{R}^n}\left(\M^*(f_{i_0},f_{i_0})- f_{i_0}\right) \,dx . $$

 By hypothesis, we have
 $$\delta \int_{\mathbb{R}^n} f \,dx \geq \int_{\mathbb{R}^n}\left(\M^*(f,f)- f \right)\,dx. $$

 Putting everything together, we get that 
 $$\int_{\mathbb{R}^n} f_{i_0} \times 1_{S_{i_0}}\,dx \leq d^{-1}\delta \int_{\mathbb{R}^n} f \,dx$$

By combining the first and second claims, we also have $$\int_{\mathbb{R}^n}f_{i_0} \,dx \geq k^{-1}\int_{\mathbb{R}^n} \co_p(f_{i_0}) \,dx.$$

 Recall that for $x \in \text{supp}(\co(f_{i_0}))$, we have $\co(f_{i_0})(x) \in [m^{-1-i_0},m^{2-i_0}]$. Recall that $S_{i_0}= \{x \colon f_{i_0}(x) \geq m^{-i_0-2}s^{-1}\}$ and $s \geq 2k$. We deduce that $S_{i_0} \supset \{x \colon f_{i_0}(x)\geq (2k)^{-1}\co_p(f_{i_0})\}$. It follows that

$$\int_{\mathbb{R}^n}f_{i_0} \times 1_{S_{i_0}} \,dx =(2k)^{-1}\int_{\mathbb{R}^n} \co_p(f_{i_0}) \,dx.$$

Therefore, recalling $k=O_{n,p,m}(1)$ and $d=O_{n,p,m,\lambda,k}(1)$, we get

$$\int_{\mathbb{R}^n}\co_p(f_{i_0}) \,dx = O_{n,\lambda,p,m}(\delta) \int_{\mathbb{R}^n} f \,dx$$



\end{proof}

\begin{clm}
If $i_0<\infty$, then $\int_{K_{i_0-1}^c} \co_p(f) \,dx =  O_{n,\lambda,p,m}(\delta) \int_{\mathbb{R}^n} f \,dx$.
\end{clm}

\begin{proof}[Proof of Claim]

By the previous claim, we have 
$$\int_{\mathbb{R}^n}\co_p(f_{i_0}) \,dx = O_{n,\lambda,p,m}(\delta) \int_{\mathbb{R}^n} f \,dx.$$
Recall that for $x \in K_{i_0} \setminus K_{i_0-1}$, we have $\co_p(f_{i_0})(x)=\co_p(f)(x)$; hence, we also have
$$\int_{K_{i_0} \setminus K_{i_0-1}}\co_p(f) \,dx = O_{n,\lambda,p,m}(\delta) \int_{\mathbb{R}^n} f \,dx$$

Now construct the function $g \colon \mathbb{R}^n \rightarrow \mathbb{R}_{\geq 0}$ as follows. For $x \in K_i$, let $g(x)= \co_p(f_i)(x)$ and for $x \in K_{i-1}^c$, let $g(x)=\co_p(f)(x)$. Both $\co_p(f_i)$ and $\co_p(f)$ are $p$-concave. Since being $p$-concave is a local property and $g$ is obtained by gluing together these functions, we also get that $g$ is $p$-concave. 

Recall that for $x \in K^c_{i_0}$, we have $ g(x)=\co_p(f)(x) \leq m^{-i_0}$ and for $x \in K_{i_0-1}$, we have $g(x) \geq m^{-i_0+1}$. By \Cref{p-concave_tail} (assuming wlog $m\geq 2$) together with the previous claim, we deduce $$\int_{K^c_{i_0}} \co_p(f) \,dx=\int_{K_{i_0}^c}g(x) \, dx= O_{n,\lambda,p,m}(1) \int_{K_{i_0}}g(x) \,dx = O_{n,\lambda,p,m}(1) \int_{K_{i_0}}\co_p(f_{i_0})(x) \,dx = O_{n,\lambda,p,m}(\delta) \int_{\mathbb{R}^n} f \,dx.$$

Putting everything together, the conclusion of the claim follows.
\end{proof}

Combining the third and the last claims we conclude that 
$$ \int_{\mathbb{R}^n} \left(\co_p(f)(x)-f(x)\right)\,dx \leq O_{n,\lambda,p,m}(\delta) \int_{\mathbb{R}^n}f(x)\,dx.$$
    This finishes the proof of the proposition.
\end{proof}

\subsection{Reduction from \Cref{linear_continuous_level_medium} to \Cref{linear_continuous_level_advanced}}
\begin{prop}\label{linear_continuous_level_advanced}
        Given $n \in \mathbb{N}$, $\lambda \in (0,1/2]$, $p \in (-1/n, \infty)$ and $ m \in (1, \infty)$ there exists $d=d_{n, \lambda, p,m}>0$  such that the following holds. Let $f\colon \mathbb{R}^n\rightarrow \mathbb{R}_{\geq 0}$ be a function with compact support such that
    \begin{itemize}
        \item $\text{supp}(f)$ is compact and $f$ is continuous on $\text{supp}(f)$,
        \item there exists a compact set $Y$ such that $\co_p(f)=\co_p(1_Y \times f)$ and for $y \in Y$ we have $f(y) \in [1/m,m]$,
        \item $\int_{\mathbb{R}^n} f \,dx>0$, and
        \item $\int_{\mathbb{R}^n}\M^*(f,f) \,dx \leq (1+\delta) \int_{\mathbb{R}^n} f \,dx$ for some $0 \leq \delta \leq d$.
    \end{itemize}
    Then 
        \begin{itemize}
        \item $\int_{\mathbb{R}^n}\left( \co_p(f)-f\right)\, dx=O_{n,\lambda, p,m} (\delta)\int_{\mathbb{R}^n} f\,dx $.
    \end{itemize}
\end{prop}

\begin{proof}[Proof \Cref{linear_continuous_level_advanced} implies \Cref{linear_continuous_level_medium}]
Pick $s\geq 2k$ large enough. As $f$ is continuous on $\overline{\co(\text{supp}(f))}$, it follows that $\text{supp}(f \times 1_S)$ is compact and $f \times 1_S$ is continuous on $\text{supp}(f \times 1_S)$. Moreover, $\int_{\mathbb{R}^n} f \times 1_S \,dx > \frac12\int_{\mathbb{R}^n} f  \,dx >(2k)^{-1} \int_{\mathbb{R}^n} \co_p(f) \,dx $. Furthermore, $f \times 1_S(x) \in [1/ms,m]$ for $x \in \text{supp}(f \times 1_S)$.

As $\co_p(f)(x) \geq m^{-1}$ for $x \in \text{supp}(f)$ and $f(x) \leq 1/ms$ for $x \not \in S$, we get $$\int_{\mathbb{R}^n} f\times 1_{\lambda S+(1-\lambda) S} \,dx \leq \int_{\mathbb{R}^n} f\times 1_S\, dx + s^{-1} \int_{\mathbb{R}^n} \left(\co_p(f) - f\right)  \,dx.$$

As $\co_p(f)=\co_p(1_Y \times f)$ and for $y \in Y$ we have $f(y)=f \times 1_S(y) \in [1/m,m]$, we have that $\co_p(f)= \co_p(f \times 1_S) = \co_p(f \times 1_S \times 1_Y )$.

We deduce that $$\int_{\mathbb{R}^n}\left(\M^*(f \times 1_S,f\times 1_S) - f \times 1_S\right)\,dx \leq \delta   \int_{\mathbb{R}^n} f\times 1_S\, dx + 2s^{-1} \int_{\mathbb{R}^n} \left(\co_p(f ) - f\right)  \,dx.$$

In particular, 
 $$\int_{\mathbb{R}^n}\left(\M^*(f \times 1_S,f\times 1_S) - f \times 1_S\right)\,dx \leq (\delta + 4ks^{-1})   \int_{\mathbb{R}^n} f\times 1_S\, dx.$$

 By taking $\delta>0$ and $s>0$ sufficiently small, we can make $\delta+4ks^{-1}$ arbitrary small. By \Cref{linear_continuous_level_advanced}, we deduce that
 $$\int_{\mathbb{R}^n} \left(\co_p(f \times 1_S)-f\times 1_S \right)\,dx =O_{n,\lambda,p,k,m} (1) \bigg(\delta   \int_{\mathbb{R}^n} f\times 1_S\, dx + 2s^{-1} \int_{\mathbb{R}^n} \left(\co_p(f ) - f\right) \,dx\bigg).$$

 As $\co_p(f \times 1_S)=\co_p(f)$ and $f \times 1_S \leq f$, this implies
 $$\int_{\mathbb{R}^n} \left(\co_p(f)-f\right) \,dx =O_{n,\lambda,p,k,m} (1) \bigg(\delta   \int_{\mathbb{R}^n} f\, dx + 2s^{-1} \int_{\mathbb{R}^n} \left(\co_p(f ) - f \right)\,dx\bigg).$$

 By taking $s$ larger than $O_{n,\lambda,p,k,m} (1)$, we conclude
  $$\int_{\mathbb{R}^n} \left(\co_p(f)-f\right) \,dx =O_{n,\lambda,p,k,m} (\delta) \int_{\mathbb{R}^n} f\, dx .$$

\end{proof}

\subsection{Reduction from \Cref{linear_continuous_level_advanced} to \Cref{linear_continuous_level_strong}}
\begin{prop}\label{linear_continuous_level_strong}
    Given $n \in \mathbb{N}$, $\lambda \in (0,1/2]$, $p \in (-1/n, \infty)$ and $ m \in (1, \infty)$ there exists $d=d_{n, \lambda, p,m}>0$ such that the following holds. Let $f\colon \mathbb{R}^n\rightarrow \mathbb{R}_{\geq 0}$ be a function with compact support such that  
    \begin{itemize}
        \item $\text{supp}(f)$ is compact and $f$ is continuous on $\text{supp}(f)$,
        \item $f(x) \in [m^{-1},m]$ for $x \in \text{supp}(f)$,
        \item $\int_{\mathbb{R}^n} f \,dx>0$, and
        \item $\int_{\mathbb{R}^n}\M^*(f,f) \,dx \leq (1+\delta) \int_{\mathbb{R}^n} f \,dx$ for some $0 \leq \delta \leq d$.
    \end{itemize}
    Then 
        \begin{itemize}
        \item $\int_{\mathbb{R}^n}\left( \co_p(f)-f\right)\, dx=O_{n,\lambda, p,m} (\delta)\int_{\mathbb{R}^n} f\,dx $.
    \end{itemize}
\end{prop}

\begin{proof}[Proof  \Cref{linear_continuous_level_strong} implies  \Cref{linear_continuous_level_advanced}]
Let $A_t \colon =  \{x \colon f(x) \geq t\}$ and $B_t \colon =  \{x \colon \M^*(f,f)(x) \geq t\} $ and $C_t\colon = \{ x \colon \co_p(f)(t) \geq t \}$. The sets $A_t, B_t, C_t$ are nested. For $t \geq 0$, we have 
$$C_t \supset B_t \supset \lambda A_t +(1-\lambda)A_t \supset A_t .$$ Moreover, for $t \geq m$, we have $C_t =\emptyset$ and for $t \leq m^{-1}$, we have $C_t=\co(Y)$ by the second hypothesis.

The third and fourth hypotheses imply $\int_{\mathbb{R}^n}f \,dx \geq 2^{-1} \int_{\mathbb{R}^n}\co_p(f) \,dx $. We have $\int_{\mathbb{R}^n}\co_p(f) \,dx \geq m^{-1}|\co(Y)|$. We also have $\int_{\mathbb{R}^n}f \,dx \leq m|A_{4^{-1}m^{-1}}|+ 4^{-1}m^{-1}|\co(Y)|$. We deduce $|A_{4^{-1}m^{-1}}| \geq 4^{-1}m^{-2}|\co(Y)|$. So for $t \leq 4^{-1}m^{-1}$, we have $|A_t| \geq 4^{-1}m^{-2}|\co(Y)| $. Moreover, by the second hypothesis, $\co(A_t)=\co(Y)$.

By the fourth hypothesis, we get $\int_{\mathbb{R}^n}\left(\M^*(f,f) -f\right) \,dx \leq \delta \int_{\mathbb{R}^n} f \,dx \leq \delta m |\co(Y)|$. Therefore, 
$$\int \left(|\lambda A_t+(1-\lambda)A_t| - |A_t|\right) \,dt \leq \delta m |\co(Y)|.$$
In particular, $\int_{t \leq 4^{-1}m^{-1}} \left(|\lambda A_t+(1-\lambda)A_t| - |A_t|\right) \,dt \leq \delta m |\co(Y)|$. 

As for $t \leq 4^{-1}m^{-1}$, we have $|A_t| \geq 4^{-1}m^{-2}|\co(Y)| $ and $\co(A_t)=\co(Y)$, by \Cref{BMStab}, we deduce that $|\lambda A_t+(1-\lambda)A_t| - |A_t|=\Theta_{n,\lambda,m}(1)|\co(A_t) \setminus A_t|$.

Combining the last two inequalities, we get  $$\int_{t \leq 4^{-1}m^{-1}} |\co(A_t) \setminus A_t| \,dt =  O_{n,\lambda,m}(\delta) |\co(Y)|= O_{n,\lambda,m}(\delta) \int_{\mathbb{R}^n}f \,dx.$$

Let $S= \{ x \colon f(x) \geq 4^{-1}m^{-1}\}$ and let $g \colon \mathbb{R}^n \rightarrow \mathbb{R}_{\geq 0}$ be defined by $g(x)=f(x)\times 1_S$. By the last inequality, we get $$\int_{\mathbb{R}^n}(f-g) \,dx = \int_{t \leq 4^{-1}m^{-1}} |\co(A_t) \setminus A_t| \,dt = O_{n,\lambda,m}(\delta) \int_{\mathbb{R}^n}f \,dx.$$
As $g \leq f$, by the fourth hypothesis, we further deduce that 
$$\int_{\mathbb{R}^n}\M^*(g,g) \,dx \leq (1+\delta) \int_{\mathbb{R}^n} f \,dx \leq (1+O_{n,\lambda,m}(\delta))\int_{\mathbb{R}^n} g \,dx.$$

By the first hypothesis, we deduce that $S=\text{supp}(g)$ is compact and $g$ is continuous on $S$. 

By the second hypothesis and the definition of $S$, we also get that $g(x) \in [ 4^{-1}m^{-1},m]$ for $x \in S=\text{supp}(g)$ and additionally, $\co_p(g)=\co_p(f)$. 

Finally, as $|A_{4^{-1}m^{-1}}|>0$, we infer that $\int_{\mathbb{R}^n}g \,dx>0$.

By \Cref{linear_continuous_level_strong}, we get that $\int_{\mathbb{R}^n} (\co_p(g)-g)\, dx=O_{n,\lambda, p,m} (\delta)\int_{\mathbb{R}^n} g\,dx $. As $\co_p(f)=\co_p(g)$ and additionally $\int_{\mathbb{R}^n} f \,dx \leq (1+O_{n,\lambda,m}(\delta))\int_{\mathbb{R}^n} g \,dx$, we conclude $\int_{\mathbb{R}^n}\left( \co_p(f)-f\right)\, dx=O_{n,\lambda, p,m} (\delta)\int_{\mathbb{R}^n} f\,dx $.
\end{proof}

\subsection{Proof of \Cref{linear_continuous_level_strong}}

\begin{proof}[Proof of \Cref{linear_continuous_level_strong}]
    By the first hypothesis, $f$, $\M^*(f,f)$, and $\co_p(f)$  are integrable. After scaling and replacing $m$ by $m^{1/2}$ we can assume $f(x) \in [2,m]$ for $x \in \text{supp}(f)$. Moreover, $\M^*(f,f)(x) \in [2,m]$ for $x \in \text{supp}(\M^*(f,f))$ and $\co_p(f) \in [2,m]$ for $x \in \text{supp}(\co_p(f))$. 

    Define the function $g \colon \mathbb{R}^n \rightarrow \mathbb{R}_{\geq 0}$ as follows. $$g(x)=\begin{cases}0 &\text{ if } x\not\in \text{supp}(f)\\
\text{sign}(p)f^p(x)+1&\text{ if } p\neq 0 \text{ and }x \in \text{supp}(f)\\
\log(f(x))+1&\text{ if $p=0$} \text{ and }x \in \text{supp}(f)\end{cases}$$

Note that the function $M_{\lambda,1}^*(g,g) \colon \mathbb{R}^n \rightarrow \mathbb{R}_{\geq 0}$ satisfies

$$M^*_{\lambda,1}(g,g)(x)=\begin{cases}0 &\text{ if } x\not\in \text{supp}(\M^*(f,f))\\
\text{sign}(p)(\M^*(f,f)(x))^p +1&\text{ if } p\neq 0 \text{ and }x\in \text{supp}(\M^*(f,f))\\
\log(\M^*(f,f))(x)) +1 &\text{ if $p=0$} \text{ and }x\in \text{supp}(\M^*(f,f))\end{cases}$$

Also note that the function $\co_1(g) \colon \mathbb{R}^n \rightarrow \mathbb{R}_{\geq 0}$ satisfies

$$\co_1(g)(x)=\begin{cases}0 &\text{ if } x\not\in \text{supp}(\co_p(f))\\
\text{sign}(p)\co_p^p(f) +1&\text{ if } p\neq 0 \text{ and }x\in \text{supp}(\co_p(f))\\
\log(\co_p(f)) +1 &\text{ if $p=0$} \text{ and }x\in \text{supp}(\co_p(f))\end{cases}$$

By the above, $\text{supp}(g)=\text{supp}(f)$, $\text{supp}(M_{\lambda,1}^*(g,g))=\text{supp}(M_{\lambda,p}^*(f,f))$, and $\text{supp} (\co_1(g))=\text{supp}( \co_p(f))$. Moreover, for every $x \in \mathbb{R}^n$, we have 
$$ g(x) =  \Theta_{n,,p,m}(1)f(x),\qquad M_{\lambda,1}^*(g,g)(x)= \Theta_{n,p,m}(1)\M^*(f,f)(x), \qquad\co_1(g)(x) = \Theta_{n,p,m}(1) \co_p(f)(x).$$ 
Furthermore, 
$$M^*_{\lambda,1}(g,g)(x)-g(x)= \Theta_{n,p,m}(1)(M^*_{\lambda,p}(f,f)(x)-f(x))\text{ and }\co_1(g)(x)-g(x)= \Theta_{n,p,m}(1)(\co_p(f)(x)-f(x)).$$

As $f$, $\M^*(f,f)$, and $\co_p(f)$ are integrable, so are $g$, $M^*_{\lambda,1}(g,g)$, and $\co_1(g)$. Additionally, as $\int f \, dx>0$, we also get $\int g\, dx>0$.

Putting everything together, we see that
\begin{equation*}
    \begin{split}
        \int \left(M^*_{\lambda,1}(g,g)(x)-g(x)\right) \,dx &= \Theta_{n,m,p}(1) \int \left(M^*_{\lambda,p}(f,f)(x)-f(x)\right) \,dx\\
        &\leq \Theta_{n,m,p}(\delta) \int f(x) \,dx\leq \Theta_{n,m,p}(\delta) \int g(x) \,dx.
    \end{split}
\end{equation*}
By \Cref{BMStab}, we deduce that
$$\int \left(\co_1(g)(x)-g(x)\right)\, dx \leq O_{n,\lambda,p,m}(\delta) \int g(x) \,dx,$$
from which we conclude that
$$\int (\co_p(f)(x)-f(x)) \,dx \leq O_{n,\lambda,p,m}(\delta) \int f(x) \,dx.$$
\end{proof}

\subsection{Proof of  \Cref{linear_shaving_process}}
The purpose of this section is to prove  \Cref{linear_shaving_process}. We will make use of the following remarks. 

\begin{rmk}\label{p-concave_1}
 Let $n \in \mathbb{N}$, $p\in (-1/n, \infty)$ and let  $f \colon \mathbb{R}^n \rightarrow \mathbb{R}_{\geq 0}$ be a $p$-concave function. Then $f$ has  convex support and convex level sets and is continuous on the interior of its support.
\end{rmk}

\begin{rmk}\label{p-concave_2}
 Let $n \in \mathbb{N}$, $p\in (-1/n, \infty)$ and let  $f \colon \mathbb{R}^n \rightarrow \mathbb{R}_{\geq 0}$ be a $p$-concave function. Then $f$ is a measurable function. 
\end{rmk}

\begin{rmk}\label{p-concave_3}
 Let $n \in \mathbb{N}$ and let $p\in (-1/n, \infty)$. Let $I$ be a set of indices. For $i \in I$, let $\ell_i \colon \mathbb{R}^n \rightarrow \mathbb{R}_{\geq 0}$ be a $p$-concave function. Then, the function $g\colon \mathbb{R}^n\rightarrow \mathbb{R}_{\geq 0}$, $g(x)=\inf_i\ell_i(x)$ is $p$-concave.
 
\end{rmk}

\begin{proof}[Proof of \Cref{linear_shaving_process}]

 Let $\mathcal{L}$ be the collection of $p$-concave functions $\ell' \colon \mathbb{R}^n \rightarrow \mathbb{R}_{\geq 0}$ such that the function $f'\colon \mathbb{R}^n \rightarrow \mathbb{R}_{\geq 0}$, $f'=\min(f,\ell')$ satisfies $\int_{\mathbb{R}^n} \left[-c(f-f')+ (\M^*(f,f)- f)-(\M^*(f',f')-f')\right]\,dx \geq 0$.

Put on $\mathcal{L}$ the partial order $\ell_1 \leq \ell_2$ if $\ell_1(x) \leq \ell_2(x)$ for all $x \in \mathbb{R}^n$. To show that $\mathcal{L}$ has a minimal element, by Zorn's lemma it is enough to show that any chain $C \subset \mathcal{L}$ has a lower bound in $\mathcal{L}$.

Let $C$ be a chain $\mathcal{L}$. Let $\ell_0\colon \mathbb{R}^n \rightarrow\mathbb{R}_{\geq 0}$ be defined by $\ell_0(x)=\inf_{\ell \in C} \ell(x)$. By \Cref{p-concave_3}, it follow that $\ell_0 $ is also a $p$-concave function. Clearly, $\ell_0$ is a lower bound for $C$. It remains to show that $\ell_0 \in \mathcal{L}$. 

By \Cref{p-concave_2}, $\ell_0$ is measurable and any $\ell \in C$ is measurable. We will first show that there exist a  sequence $\ell_1 \geq \ell_2 \geq \dots$ in $C$ such that $\int_{\mathbb{R}^n}\ell_i(x)\,dx \rightarrow \int_{\mathbb{R}^n}\ell_0(x)\,dx$ as $i \rightarrow \infty$.

For a function $g \colon \mathbb{R}^n \rightarrow \mathbb{R}_{\geq 0}$ and $c>0$, let $g^c=\{x \in \mathbb{R}^n \colon g(x) \geq c\}$. By \Cref{p-concave_1}, for $c>0$, the family $\{\ell^c \colon \ell \in C\}$ is a chain of convex sets under the subset order. Moreover, $\ell_0^c$ is a convex set. Furthermore, for every $\ell \in C$, $\ell_0^c \subset \ell^c$. Finally, for any $p \not\in \ell_0^c$, there exists $\ell \in C$ such that $p \not\in \ell^c$. 

It follows that for a fixed $c>0$, there exist $\ell_1 \geq \ell_2 \geq \dots $ in $ C$ such that $\ell_1^c \supset \ell_2^c \supset \dots$ and $|\ell_i^c| \rightarrow |\ell_0^c|$ as $i \rightarrow \infty$. It further follows that there exist $\ell_1 \geq \ell_2 \geq \dots $ in $ C$ such that for every $c \in \mathbb{Q}$, $\ell_1^c \supset \ell_2^c \supset \dots$ and $|\ell_i^c| \rightarrow |\ell_0^c|$ as $i \rightarrow \infty$.

Note $\int_{\mathbb{R}^n}\ell_i(x)\,dx= \int_{c\geq 0}|\ell_i^c| \,dc$. Therefore to show $\int_{\mathbb{R}^n}\ell_i(x)dx \rightarrow \int_{\mathbb{R}^n}\ell_0(x)dx$ as $i \rightarrow \infty$, it is enough to show $\int_{c\geq 0}|\ell_i^c|\,dc \rightarrow \int_{c\geq 0}|\ell_0^c|\,dc$ as $i \rightarrow \infty$. 

For each i, the function $|\ell_i^c|$ is left continuous and decreasing in $c$. Moreover, the sequence of functions  $|\ell_i^c|$ is pointwise decreasing and it is lower bounded by  $|\ell_0^c|$. Furthermore, for any $c \in \mathbb{Q}$, we have $|\ell_i^c| \rightarrow |\ell_0^c|$ as $i \rightarrow \infty$. By a triangle inequality, it follows that for any $c \in \mathbb{R}$, we have $|\ell_i^c| \rightarrow |\ell_0^c|$ as $i \rightarrow \infty$. Finally, by pointwise convergence, we get  $\int_{c\geq 0}|\ell_i^c|\, dc \rightarrow \int_{c\geq 0}|\ell_0^c|\, dc$ as $i \rightarrow \infty$.

We infer that there exist a  sequence $\ell_1 \geq \ell_2 \geq \dots$ in $C$ such that $\int_{\mathbb{R}^n}\ell_i(x)\,dx \rightarrow \int_{\mathbb{R}^n}\ell_0(x)\,dx$ as $i \rightarrow \infty$. Let $f_i\colon \mathbb{R}^n \rightarrow \mathbb{R}_{\geq 0}$ be defined by $f_i=\min(\ell_i,f)$. Then it follows that $\int_{\mathbb{R}^n}f_i(x)\,dx \rightarrow \int_{\mathbb{R}^n}f_0(x)\,dx$ as $i \rightarrow \infty$.  We also have $\M^*(f_i,f_i) \geq \M^*(f_0,f_0)$. Recalling the definition of $\mathcal{L}$, from the fact that $\ell_i \in \mathcal{L}$ for all $i$, by taking the limit, we conclude the first bullet point $\ell_0 \in \mathcal{L}$.

 By the previous bullet point, $\M^*(f',f')\geq f'$, and the hypothesis, we conclude that $$\int_{\mathbb{R}^n}(f-f')\, dx \leq c^{-1} \int_{\mathbb{R}^n}\left[(\M^*(f,f)- f)-(\M^*(f',f')-f')\right] \,dx\leq c^{-1} \int_{\mathbb{R}^n}(\M^*(f,f)- f)\,dx\leq c^{-1} \delta\int_{\mathbb{R}^n}f\,dx. $$

By the previous bullet point, we also get that $\int_{\mathbb{R}^n}(\M^*(f',f')-f')  \,dx \leq  \int_{\mathbb{R}^n}(\M^*(f,f)- f)\,dx $. By hypothesis, we deduce that $\int_{\mathbb{R}^n}(\M^*(f',f')-f') \, dx \leq  \delta\int_{\mathbb{R}^n}f\,dx$. By the previous assertion, provided $d \leq c/10$, we conclude that $\int_{\mathbb{R}^n}(\M^*(f',f')-f')  \,dx \leq  3\delta\int_{\mathbb{R}^n}f'\,dx$. 

For the last bullet point, assume $f''\neq f'$. Let $\ell=\min(\ell',\ell'')$. Then $\ell\leq \ell'$ and $\ell\neq \ell'$. Moreover, $f''=\min(f,\ell)$. By the minimality of $\ell'$, we deduce that $$\int_{\mathbb{R}^n}\left[ -c(f-f'')+ (\M^*(f,f)- f)-(\M^*(f'',f'')-f'')\right]\, dx < 0.$$ As 
$$\int_{\mathbb{R}^n} \left[-c(f-f')+ (\M^*(f,f)- f)-(\M^*(f',f')-f')\right]\, dx \geq 0,$$ we reach our desired conclusion that $$\int_{\mathbb{R}^n} \left[-c(f'-f'')+ (\M^*(f',f')- f')-(\M^*(f'',f'')-f'')\right]\, dx < 0.$$
\end{proof}

\subsection{Proof of \Cref{linear_already_shaved}}

The purpose of this subsection is to prove the \Cref{linear_already_shaved}. Before we begin the proof, we need the following lemma.

\begin{lem}\label{carving_convex}
    Let $n \in \mathbb{N}$, $p \in (-1/n,\infty)$, $y \in \mathbb{R}^n$ and $d \in \mathbb{R}$. Let $g \colon \mathbb{R}^n \rightarrow \mathbb{R}_{\geq 0}$ be a continuous $p$-concave function with bounded support. Let $F \subset \mathbb{R}^n$ be a compact convex set such that $F$ is a $p$-face of $g$ with tangent $p$-plane $h_{y,d}\colon \mathbb{R}^n\rightarrow \mathbb{R}_{\geq 0}\cup\{\infty\}$. Let $[o,v]$ be an edge of $F$ and assume $h_{y,d}(o)>h_{y,d}(v)$. Then $\text{sign}(p)<y,v> < 0$ and there exists a sequence of vectors $z_i \rightarrow o$ as $i \rightarrow \infty$  with $<z_i,v>=0$ such that the  $p$-planes $h_{y + z_i,d}\colon \mathbb{R}^n\rightarrow \mathbb{R}_{\geq 0}$ have the following property.  For $x \in [o,v]$, we have $h_{y+  z_i,d}(x)=h_{y,d}(x)=g(x)$. Moreover, the set $S_i=\left\{x \in \overline{\text{supp}(g)} \colon  g(x) \geq h_{y +  z_i,d}(x)\right\}$ satisfies $ S_i \subset [o,v]+B(o, 1/i)$.
\end{lem}

\begin{proof}[Proof of \Cref{carving_convex}]
   Given $z,x \in \mathbb{R}^n$, if $<x,z>> 0$, then $\text{sign}(p) h_{y+ z,d}(x) \geq \text{sign}(p) h_{y,d}(x)$ and if $<x,z>< 0$, then $\text{sign}(p) h_{y+ z,d}(x) \leq \text{sign}(p) h_{y,d}(x)$ and if $<x,z>=0$, then $ h_{y+ z,d}(x) =  h_{y,d}(x)$. Moreover, as $z \rightarrow 0$, $ h_{y+ z,d} \rightarrow  h_{y,d}$ locally uniformly. 

   Recall that $\text{supp}(g)$ is a bounded open convex set and $F \subset \overline{\text{supp}(g)}$. We claim that for $x \in \overline{\text{supp}(g)} \setminus F$, we have $h_{y,d}(x) >g(x)$. Indeed, assume for a contradiction that $h_{y,d}(x) \leq g(x)$. We have $o,x \in \overline{\text{supp}(g)}$ so $[o,x]\in \overline{\text{supp}(g)}$. Moreover, $o \in \text{supp}(g)$ so $[o,x)  \in \text{supp}(g)$. We also have $o \in F$ and $x \not\in F$ so there exists $\lambda\in [0,1)$ such that for $t\in [0,\lambda]$, we have $tx \in F$ and for $t\in (\lambda,1)$, we have $tx\not\in F$. By hypothesis, for $t\in (\lambda,1)$, we have  $tx \in \text{supp}(g) \setminus F$, so $h_{y,d}(tx)>g(x)$. Additionally, $h_{y,d}(\lambda x)=g(\lambda x)$. Given that for $t \in [\lambda,1]$ we have that $g(tx)$ is $p$-concave and continuous and $h_{v,d}(tx)$ is $p$-convex and continuous, we conclude that $h_{v,d}(x)>g(x)$.

   Fix $i \in \mathbb{N}$ sufficiently small and consider the compact set $(F+B(o,1/i))^c$. There exists $\varepsilon_i>0$ such that for $x \in \overline{\text{supp}(g)} \setminus (F+B(o,1/i))$ we have $g(x) \leq h_{y,d}(x)-\varepsilon_i$. Therefore, there exists $\delta_i>0$ sufficiently small such that for $|z|<\delta_i$ and $x \in \overline{\text{supp}(g)} \setminus (F+B(o,1/i)) $, we have $h_{y+z,d}(x)>g(x)$. 

   To conclude, it is enough to show that for any $j$ we can find $i$ and  $z_i$ with $|z_i|<\delta_i$ such that the half spaces $H_i^+= \{x \in \mathbb{R}^n \colon \text{sign}(p)<x,z_i> > 0\}$ satisfy $(H_i^+)^c \cap (F+B(o,1/i)) \subset [0,v]+B(o,1/j)$.

   Indeed, fix $x \in \overline{\text{supp}(g)}$ such that $g(x)\geq h_{y+z_i,d}(x)$. Then $x \in (F+B(o,1/i)) \cap \overline{\text{supp}(g)}$. Clearly $x \in \overline{\text{supp}(h_{y,d})}$. If additionally $x \in H_i^+$, then $g(x)\leq h_{y,d}(x)< h_{y+z_i,d}(x)$, a contradiction. Hence, $x \in (F+B(o,1/i)) \cap (H_i^+)^c$ and we are done provided $(H_i^+)^c \cap (F+B(o,1/i)) \subset [0,v]+B(o,1/j)$.

   For a fixed $j$ and a fixed closed half-space $H^-$, if $i$ is sufficiently large, then $H^- \cap (F+B(o,1/i)) \subset (H^- \cap F) +B(o,1/2j)$. 

   Therefore, it is enough to show that for any fixed $j$, there exists $z \neq o$, such that the closed half-space $H^-=\{x \in \mathbb{R}^n \colon \text{sign}(p)<x,z> \leq 0\}$ satisfies $H^- \cap F \subset [0,v]+B(o,1/2j)$. 

   In other words, it is enough to show that for any fixed $j$, there exists a half-space $H^-$ with the hyperplane $H$ containing $[o,v]$ such that $H^- \cap F \subset [0,v]+B(o,1/2j)$. 

   Assuming for a contradiction this is not the case, we deduce that there exists  $\co(F \setminus ([0,v]+B(o,1/2j))) \cap \text{span}(v) \neq \emptyset $. As $[o,v]$ is an edge of $F$ we obtain the desired contradiction.

   This concludes the proof of the lemma.
\end{proof}

\begin{proof}[Proof of \Cref{linear_already_shaved}]
    Fix $m=m_{n,\lambda,p}>0$ large and $c=c_{n,\lambda,p}>0$ small.  

    As $f$ is continuous with bounded support, so is $\co_p(f)$. Let $F$ be a $p$-face of $\co_p(f)$ and let $[o,v]$ be an edge of $F$ (after translation) such that $\co_p(o)>\co_p(v)$. Let $h_{y,d} \colon \mathbb{R}^n \rightarrow \mathbb{R}_{\geq 0}$ be the tangent $p$-plane corresponding to $F$. 

    Assume for a contradiction that there exist $0 \leq a <b \leq 1$ such that $f(av) >mf(bv)$, $f(av)= \co_p(f)(av)$ and $f(bv)= \co_p(f)(bv)$ and for every $t \in (a,b)$ we have $f(tv)<  \co_p(f)(tv)$. 

    By \Cref{carving_convex},  $\text{sign}(p)<y,v> < 0$ and there exists a sequence of vectors $z_i \rightarrow o$ as $i \rightarrow \infty$  with $<z_i,v>=0$ such that the  $p$-planes $h_{y + z_i,d}\colon \mathbb{R}^n\rightarrow \mathbb{R}_{\geq 0}$ have the following property.  For $x \in [o,v]$, we have $h_{y+  z_i,d}(x)=h_{y,d}(x)=\co_p(x)$. Moreover, the set $S_i=\{x \in \overline{\text{supp}(co_p(f))}\colon \co_p(f)(x) \geq h_{y +  z_i,d}(x)\}$ satisfies $ S_i \subset [o,v]+B(o, 1/i)$.

    Recall that $f$ and $\co_p(f)$ are continuous and $f(tv)<\co_p(f)(tv)$ for $t \in (a,b)$. Hence, there exists a sequence of vectors $z_i \rightarrow o$ as $i \rightarrow \infty$  with $<z_i,v>=0$ such that the  $p$-planes $h_{y + z_i,d}\colon \mathbb{R}^n\rightarrow \mathbb{R}_{\geq 0}$ have the following property.  For $x \in [o,v]$, we have $h_{y+  z_i,d}(x)=h_{y,d}(x)=\co_p(x)$. Moreover, the set $S_i=\{x \in \overline{\text{supp}(f)} \colon f(x) \geq h_{y +  z_i,d}(x)\}$ satisfies $ S_i \subset ([o,av]+B(o, 1/i)) \cup ([bv,v]+B(o, 1/i))$. Furthermore, $f(av)=\co_p(f)(av)=h_{y,d}(av)=h_{y+  z_i,d}(av)$.

    There exists $c\in(a,b)$ sufficiently close to $b$ such that $\co_p(f)(av)>2^{-1}m\co_p(f)(cv)$. 

    Partition the space in three regions $\mathbb{R}^n=H\sqcup H^+ \sqcup H^-$. Here $H=\{x \colon <x,y>= c<v,y>\}$. $H^+=\{x \colon \text{sign}(p)<x,y>> \text{sign}(p)c<v,y> \}$ and $H^-=\{x \colon \text{sign}(p)<x,y>< \text{sign}(p)c<v,y> \}$. Recalling $\text{sign}(p)<v,y><0$, we note $cv \in H$, $[o,av] \in H^+$ and $[bv,v] \in H^-$.  Moreover, for $i$ sufficiently large, $[o,av]+B(o,1/i) \subset H^+$ and $[bv,v]+B(o,1/i) \subset H^-$.

    For $s \in \mathbb{R}_{\geq 0}$, construct the $p$-plane $h_{s}=h_{y+z_i+ sy, d - sc<v,y>}$. Note that for $x \in H$, we have $h_s(x)=h_{y+z_i,d}(x)$; for $x \in H^+$, we have $h_s(x)>h_{y+z_i,d}(x)$ and for $x \in H^-$, we have $h_s(x)\leq h_{y+z_i,d}(x)$ (actually $h_s(x)< h_{y+z_i,d}(x)$ provided $h_{y+z_i,d}>0$). Moreover, for $x \in H^+$, we have $h_s(x)\rightarrow \infty$  increasing as $s \rightarrow \infty$.

    Let $s_i \in \mathbb{R}_{\geq 0}$ and $p_i \in \overline{[o,av]+B(o, 1/i)}$ be such that $h_{s_i}(p_i)= f(p_i)$ and for every $p \in \overline{[o,av]+B(o, 1/i)}$ , $h_{s_i}(p)\geq f(p)$. 

    \begin{rmk}
         $f_i'=\min(f,h_{s_i})$ is continuous and $f' \neq f$
    \end{rmk}
    \begin{proof}[Proof of Remark]
        This follows from the fact that $$\co_p\left(\frac{b+c}{2}v\right)=h_{y+z_i,d}\left(\frac{b+c}{2}v\right)> h_{s_i}\left(\frac{b+c}{2}v\right)$$
        as $\frac{b+c}{2}v \in H^-$ and $h_{y+z_i,d}(\frac{b+c}{2}v)= \co_p(f)(\frac{b+c}{2})>0$.
    \end{proof}

    \begin{clm}
        For $i$ large enough, the $p$-concave continuous functions $h_{s_i} \colon \mathbb{R}^n \rightarrow \mathbb{R} \cup \{\infty\}$ have the following properties
        \begin{itemize}
            \item $h_{s_i}(p_i)=f(p_i) >2^{-1}f(av)>0$, and
            \item  if $f(q)>h_{s_i}(q)$, then $f(p_i)\geq 8^{-1}m \co_p(f)(q) \geq 8^{-1}mf(q)$.
        \end{itemize}
    \end{clm}

    \begin{proof}[Proof of Claim]
        The first part follows from the fact that $f$ is continuous, $f(av)>0$ and $p_i \in \overline{[o,av]+B(o, 1/i)}$. 

        The second part follows in several steps. First, note that $q \in H^-$. Indeed, by construction, if $q \in  \overline{[o,av]+B(o, 1/i)}$, then $f(q) \leq h_{s_i}(q)$, which is a contradiction. If $q \in \overline{[bv,v]+B(o, 1/i)}$, then $q \in H^-$ as noted above. If $q \in (H \cup H^+) \cap ([0,av]+B(o,1/i))^c \cap ([bv,v]+B(o,1/i))^c$, we have $H_{s_i}(q)\geq H_{y+z_i,d}(q)\geq f(q)$, which is a contradiction.  
        
        Second  note that $h_{s_i}(cv)\geq h_{s_i}(q)$ by the definition of $h_{s_i}$ and the fact that $\text{sign}(p)<v,y><0$. 
        
        Third, note that  $h_{s_i}(p_i)\geq 2^{-1}f(av)=2^{-1}\co_p(f)(av)\geq 4^{-1}m \co_p(f)(cv)=4^{-1}h_{y+z_i,d}(cv)= 4^{-1}h_{s_i}(cv)$. Here, the first relation holds by the first part, the second and third relations hold by hypothesis. The last two relations hold by construction.

        Combining the second and third observations, we get $h_{s_i}(p_i) \geq 4^{-1}m h_{s_i}(q)$.

        To conclude the second part, it would be enough to argue that $h_{s_i}(p_i) \geq 4^{-1}m (\co_p(f)(q)-h_{s_i}(q))$. Actually, by the first part, it is enough to show $f(av) \geq 2^{-1}m (\co_p(f)(q)-h_{s_i}(q))$

        The last part follows from the fact that as $i \rightarrow \infty$, we have $z_i \rightarrow 0$ and $[o,av]+B(o,1/i) \rightarrow [o,av]$. These force $s_i \rightarrow 0$. Thus $h_{s_i} \rightarrow h_{y,d}$ and $\co_p(f) \leq h_{y,d}$. This concludes the claim.
    \end{proof}

    \begin{clm}
        If $p \in (-1/n,1]$ and $i$ is sufficiently large, then $$\int_{\mathbb{R}^n} \left(\M^*(f,f)-\M^*(f_i',f_i')\right)\,dx\geq (1+c)\int_{\mathbb{R}^n} (f-f_i')\,dx.$$
    \end{clm}

    \begin{proof}[Proof of Claim]
        After translating, we can assume wlog $p_i=o$. Let $A_i =\{x \colon f(x)>h_{s_i}(x)\}$ be an open set. We have $$\int_{\mathbb{R}^n} (f-f_i')\,dx= \int_{A} (f-h_{s_i})\,dx.$$ 
        
        As $h_{s_i}$ is a $p-$ plane, we also have $$\int_{\mathbb{R}^n} \left(\M^*(f,f)-\M^*(f_i',f_i')\right)\,dx\geq \int_{\mathbb{R}^n} \left(\M^*(f,f)-\M^*(h_{s_i},h_{s_i})\right)\,dx= \int_{\mathbb{R}^n} \left(\M^*(f,f)-h_{s_i}\right)\,dx.$$

        By the previous claim we have $h_{s_i}(o)=f(o)$, so we further get
        
        $$ \int_{(1-\lambda)A} \left(\M^*(f,f)-h_{s_i}\right)\,dx \geq (1-\lambda)^{n} \int_A \left(\M(h_{s_i}(o),f(x))-h_{s_i}((1-\lambda)x)\right) \, dx$$

        It is enough to argue that for $i$ sufficiently large and for every $x \in A$ we have 
        $$\M(h_{s_i}(o),f(x))-h_{s_i}((1-\lambda)x) \geq (1+c)(1-\lambda)^{-n}(f(x)-h_{s_i}(x))$$

        By the previous claim, when $i$ is sufficiently large, we have $h_{s_i}(o)>8^{-1}mf(x)$. Denoting $a=h_{s_i}(o)$, $b=h_{s_i}(x)$ and $\mu=f(x)-h_{s_i}(x)$, we have $a>8^{-1}m(\mu+b)>0$ and the above inequality can be rewritten as

$$(1+c)(1-\lambda)^{-n}\mu \leq \M(a,b+\mu) - \M(a,b).$$

Consider now the case $p \in (-1/n,0)$. In this case both $a,b>0$. If we let $r(x)=  \M(a,x)$, then $r'(x)=(1-\lambda)\bigg(\lambda (a/x)^p+(1-\lambda)\bigg)^{\frac{1-p}{p}}$.  Hence, there exist $0 \leq  \nu \leq \mu$ such that $$\M(a,b+\mu) - \M(a,b) = \mu r'(b+\nu)= \mu (1-\lambda)\bigg(\lambda (a/(b+\nu))^p+(1-\lambda)\bigg)^{\frac{1-p}{p}}.$$

By the conditions on $a,b,\mu$, we have

$$\M(a,b+\mu) - \M(a,b) \geq \mu (1-\lambda)\bigg(\lambda (m/10)^p+(1-\lambda)\bigg)^{\frac{1-p}{p}} \geq (1+c)(1-\lambda)^{-n}\mu.$$

This concludes the proof in the case $p \in (-1/n,0)$. Consider now the case $p=0$. In this case both $a,b>0$. If we let $r(x)=  \M(a,x)$, then $r'(x)=(1-\lambda)a^\lambda x^{-\lambda}$.  Hence, there exist $0 \leq  \nu \leq \mu$ such that $$\M(a,b+\mu) - \M(a,b) = \mu r'(b+\nu)= \mu (1-\lambda)a^{\lambda}(b+\nu)^{-\lambda}.$$

By the conditions on $a,b,\mu$, we have

$$\M(a,b+\mu) - \M(a,b) \geq \mu (1-\lambda)(m/10)^{\lambda} \geq (1+c)(1-\lambda)^{-n}\mu.$$

This concludes the case $p=0$. Consider now the case $p\in (0,1]$. The important observation is that $A \cap (1-\lambda)A= \emptyset$.

Indeed, assume for a contradiction there exists $x \in A \cap (1-\lambda)A$. By the previous claim we have $\co_p(f)(o)>8^{-1}m\co_p(f)(x)>0$ and $\co_p(f)(o)>8^{-1}m\co_p(f)((1-\lambda)x)>0$. However, as $\co_p(f)$ is $p$-concave, we deduce that 
$$\co_p(f)((1-\lambda)x) \geq (\lambda (\co_p(f)(o))^p+(1-\lambda)(\co_p(f)(x))^p)^{1/p} > \lambda^{1/p} \co_p(f)(o).$$ This gives the desired contradiction. We conclude the observation.

As $\M^*(f,f) \geq f$, we have 
$$ \int_{A}\left( \M^*(f,f)-h_{s_i}\right)\,dx \geq \int_{A} \left(f-h_{s_i}\right)\,dx.$$

Combining this last inequality with the first three centered inequalities and the fact that $A \cap (1-\lambda)A =\emptyset$, it is enough to show that for $i$ sufficiently large and for every $x \in A$, we have 

        $$\M(h_{s_i}(o),f(x))-h_{s_i}((1-\lambda)x) \geq c(1-\lambda)^{-n}(f(x)-h_{s_i}(x))$$

        By the previous claim, when $i$ is sufficiently large, we have $h_{s_i}(o)>8^{-1}mf(x)$. Denoting $a=h_{s_i}(o)$, $b=h_{s_i}(x)$ and $\mu=f(x)-h_{s_i}(x)$, we have $a>8^{-1}m(\mu+b)>0$ and the above inequality can be rewritten as

$$c(1-\lambda)^{-n}\mu \leq \M(a,b+\mu) - \M(a,b).$$

We can assume wlog both $a,b>0$. If we let $r(x)=  \M(a,x)$, then $r'(x)=(1-\lambda)\bigg(\lambda (a/x)^p+(1-\lambda)\bigg)^{\frac{1-p}{p}}$.  Hence, there exist $0 \leq  \nu \leq \mu$ such that $$\M(a,b+\mu) - \M(a,b) = \mu r'(b+\nu)= \mu (1-\lambda)\bigg(\lambda (a/(b+\nu))^p+(1-\lambda)\bigg)^{\frac{1-p}{p}}.$$

By the conditions on $a,b,\mu$, we have

$$\M(a,b+\mu) - \M(a,b) \geq \mu (1-\lambda)^p\geq c(1-\lambda)^{-n}\mu.$$

This concludes the proof in the case $p \in (0,1]$.
\end{proof}

Now redefine $c \in (a,b)$ such that $\co_p(cv)>2^{-1}m\co_p(bv)$. For $s \in \mathbb{R}_-$ construct the $p$-plane $h_{s}=h_{y+z_i+sy,d-sc<v,y>}$. Redefine $s_i \in \mathbb{R}_-$ and $p_i \in \overline{[bv,v]+B(o,1/i)}$ such that $h_{s_i}(p_i)=f(p_i)$ and for every  $p \in \overline{[bv,v]+B(o,1/i)}$, we have $h_{s_i}(p)\geq f(p)$. 

By an argument identical to the one in the case $p \in (0,1)$, the following holds.

    \begin{rmk}
         $f_i'=\min(f,h_{s_i})$ is continuous and $f' \neq f$
    \end{rmk}

    \begin{clm}
        If $p >1$, $i$ is sufficiently large, then $\int_{\mathbb{R}^n}\left( \M^*(f,f)-\M^*(f_i',f_i')\right)\,dx\geq (1+c)\int_{\mathbb{R}^n} \left(f-f_i'\right)\,dx$
    \end{clm}

    The last remark and the last claim combine to give the desired contradiction in the case $p>1$. The previous remark and the previous claim combine to give the desired contradiction in the case $p \in (-1/n,1]$. The conclusion of the proposition now follows.

\end{proof}

\section{Proof of Theorem~\ref{thm:LSI} and Corollaries~\ref{cor:gauge_power_signed}-\ref{cor:wulff_signed_sets}} \label{sect:sob}
To prove Theorem~\ref{thm:LSI}, we follow the strategy in \cite{Gentil}. However, in order to obtain a statement valid for every log-concave density, we need first to perform a suitable approximation argument in order to remove the assumptions \cite[Eqs. (2.7)-(2.8)]{Gentil}.

\begin{lem}\label{lem:approx}
Let $\mu_\varphi=e^{-\varphi}$ be a log-concave probability density on $\R^n$.
Fix a nonnegative radial mollifier $\rho\in C_c^\infty(\R^n)$ with $\int\rho=1$ and define
$\rho_\sigma(x):=\sigma^{-n}\rho(x/\sigma)$.
Let
\[
\varepsilon_k:=\frac{1}{k},\qquad \sigma_k:=\varepsilon_k^2=\frac{1}{k^2},
\]
and define the {Moreau--Yosida regularization}
\begin{equation}\label{eq:MYdef}
\varphi^{\varepsilon}(x):=\Bigl(\varphi^*+\frac{\varepsilon}{2}|\cdot|^2\Bigr)^*(x),
\end{equation}
and then
\begin{equation}\label{eq:phikdef}
\varphi_k(x):=\bigl(\varphi^{\varepsilon_k}*\rho_{\sigma_k}\bigr)(x)+\frac{\sigma_k}{2}|x|^2.
\end{equation}
Then the following properties hold:
\begin{enumerate}
\item[(A)] 
There exist constants $A>0$ and $B\in\R$ such that
\begin{equation}\label{eq:linlb}
\varphi(x)\ge A|x|-B\qquad\forall x\in\R^n.
\end{equation}
\item[(B)] 
For every $\varepsilon>0$, it holds
\begin{align}
&\varphi^{\varepsilon}(x)
=\inf_{y\in\R^n}\left\{\varphi(y)+\frac{1}{2\varepsilon}|x-y|^2\right\},
\label{eq:MYinf}\\[0.3em]
&\varphi^{\varepsilon}(x)\uparrow \varphi(x)\quad\text{as }\varepsilon\downarrow 0
\quad\text{for every }x\in\R^n.
\label{eq:MYpointwise}
\end{align}
Also, the function $\varphi^{\varepsilon}$ is finite everywhere, convex, and of class $C^{1,1}$, with
$\Hess \,\varphi^{\varepsilon} \leq \frac1{\varepsilon}$ a.e. 
(All these statements are proved in \cite[Proposition~12.29 and Theorem~12.30]{BauschkeCombettes}.)

\item[(C)] 
Each $\varphi_k\in C^\infty(\R^n)$ is smooth and satisfies the global Hessian bounds
\begin{equation}\label{eq:hessbounds}
\sigma_k {\rm Id} \le \Hess \,\varphi_k(x)\le (\varepsilon_k^{-1}+\sigma_k){\rm Id} \qquad\forall \,x\in\R^n.
\end{equation}

\item[(D)]
Let $A,B$ be in \eqref{eq:linlb}, and set $B'=B+\frac{A^2}{2}$. Then, for all $k$,
\begin{equation}\label{eq:linlbk}
\varphi_k(x)\ge A|x|-B'\qquad\forall x\in\R^n.
\end{equation}
In particular,
\[
e^{-\varphi_k(x)}\le e^{B'}e^{-A|x|}\in L^1(\R^n)\qquad\text{uniformly in }k.
\]

\item[(E)] The following holds:
\begin{equation}\label{eq:L1conv}
\|e^{-\varphi_k}-e^{-\varphi}\|_{L^1(\R^n)}\to 0,
\qquad
\int_{\R^n}e^{-\varphi_k(x)}\,dx\to 1.
\end{equation}

\item[(F)] 
For every $y\in\R^n$,
\begin{equation}\label{eq:conjpoint}
\varphi_k^*(y)\leq  \varphi^*(y)+\frac{\varepsilon_k}{2}|y|^2.
\end{equation}
\end{enumerate}
\end{lem}

\begin{proof}
\emph{Proof of (A).}
We first show that $\varphi$ is coercive. Indeed, if not, there exist $c\in\R$ and a sequence $x_j$ with
$|x_j|\to\infty$ and $\varphi(x_j)\le c$, i.e.\ the sublevel set $\mathcal C_c:=\{x:\varphi(x)\le c\}$ is unbounded.
Since $\mathrm{int}(\dom(\varphi))\neq\emptyset$ (otherwise $e^{-\varphi}$ would be zero a.e.), for $c$ large enough the convex set $C_c$ has nonempty interior.
Now, an unbounded convex set with nonempty interior contains a translated ray and therefore has infinite
Lebesgue measure (see \cite[Section~8]{RockafellarCA} or \cite[Section~2.2]{Schneider}).
This implies that $|\mathcal C_c|=+\infty$, therefore
\[
1\ \ge\ \int_{\mathcal C_c} e^{-\varphi(x)}\,dx\ \ge\ e^{-c}\,|\mathcal C_c|\ =\ +\infty,
\]
a contradiction. This proves that $\varphi$ is coercive.

Since $\varphi$ is convex, coercivity implies the existence of a global affine lower bound
with positive slope
(see for instance \cite[Chapter~I, Section~3.2]{HiriartUrrutyLemarechal1}), which proves \eqref{eq:linlb}.

\smallskip \noindent
\emph{Proof of (B).}
All statements are proved in
\cite[Proposition~12.29 and Theorem~12.30]{BauschkeCombettes}. 

\smallskip \noindent
\emph{Proof of (C).}
Since $\varphi^{\varepsilon_k}\in C^{1,1}$ and $\nabla\varphi^{\varepsilon_k}$ is $\varepsilon_k^{-1}$--Lipschitz,
the convolution $\varphi^{\varepsilon_k}*\rho_{\sigma_k}$ is $C^\infty$, convex, and satisfies
\[
0\leq \Hess(\varphi^{\varepsilon_k}*\rho_{\sigma_k})\le \varepsilon_k^{-1} {\rm Id}.
\]
Since
$\Hess(\varphi_k)=\Hess(\varphi^{\varepsilon_k}*\rho_{\sigma_k})+\sigma_k\mathrm{Id}$, the result follows.

\smallskip \noindent
\emph{Proof of (D).}
From \eqref{eq:linlb} and \eqref{eq:MYinf}, for any $\varepsilon\in(0,1]$,
\[
\varphi^{\varepsilon}(x)
\ge \inf_{y \in \R^n}\Bigl(A|y|-B+\frac{1}{2\varepsilon}|x-y|^2\Bigr)
\ge A|x|-\Bigl(B+\frac{\varepsilon A^2}{2}\Bigr)
\ge A|x|-\Bigl(B+\frac{A^2}{2}\Bigr),
\]
where we used $|y|\ge |x|-|x-y|$ and we minimized in $t:=|x-y|$.
Since $\varphi^{\varepsilon}*\rho_{\sigma}\ge \varphi^{\varepsilon}$ by Jensen's inequality and because $\sigma_k|x|^2\geq 0$, we obtain \eqref{eq:linlbk} with $B'=B+\frac{A^2}{2}$.

\smallskip \noindent
\emph{Proof of (E).}
By \eqref{eq:MYpointwise}, $\varphi^{\varepsilon_k}(x)\uparrow\varphi(x)$ pointwise, hence
$e^{-\varphi^{\varepsilon_k}(x)}\downarrow e^{-\varphi(x)}$ pointwise. 
Also, since $0\leq {\rm Hess}\,\varphi^{\varepsilon_k}\leq \varepsilon_k^{-1}{\rm Id}$, 
for all $x,y$ we have
\[
\big|\varphi^{\varepsilon_k}(x+y)-\varphi^{\varepsilon_k}(x)-\nabla\varphi^{\varepsilon_k}(x)\cdot y\big|
\le \frac{1}{2\varepsilon_k}|y|^2.
\]
Integrating against $\rho_{\sigma_k}(y)\,dy$ and using $\int y\,\rho_{\sigma_k}(y)\,dy=0$ (recall that $\sigma$ is radial), this yields
\[
\big|(\varphi^{\varepsilon_k}*\rho_{\sigma_k})(x)-\varphi^{\varepsilon_k}(x)\big|
\le \frac{1}{2\varepsilon_k}\int_{\R^n}|y|^2\rho_{\sigma_k}(y)\,dy
= C_\rho\,\varepsilon_k^{-1}\sigma_k^2 \xrightarrow[k\to\infty]{}0.
\]
Since also $\frac{\sigma_k}{2}|x|^2\to0$ for every fixed $x$, it follows that
$\varphi_k(x)\to\varphi(x)$ pointwise and therefore $e^{-\varphi_k(x)}\to e^{-\varphi(x)}$ pointwise.
Moreover, by (D), $e^{-\varphi_k(x)}\le e^{B'}e^{-A|x|}\in L^1$, so
dominated convergence yields
\[
\|e^{-\varphi_k}-e^{-\varphi}\|_{L^1}\to 0.
\]
This implies, in particular, that $\int e^{-\varphi_k}\to \int e^{-\varphi}=1$.

\smallskip \noindent
\emph{Proof of (F).}
Since $\varphi_k\ge \varphi^{\varepsilon_k}$ by Jensen, we immediately have $\varphi_k^*(y)\le (\varphi^{\varepsilon_k})^*(y)=\varphi^*(y)+\frac{\varepsilon_k}{2}|y|^2$ (recall \eqref{eq:MYdef}).
\end{proof}

We can now prove Theorem~\ref{thm:LSI}.

\begin{proof}[Proof of Theorem~\ref{thm:LSI}]
Consider first $\phi\in C_c^\infty(\R^n)$,
let $\varphi_k$ and $\mu_{\varphi_k}$ be given by Lemma~\ref{lem:approx},
and apply \Cref{sharplambda} to the probability densities $\frac{e^{-\varphi_k+\phi}}{Z_{e^{-\varphi_k+\phi}}}$ and $\frac{e^{-\varphi_k}}{Z_{e^{-\varphi_k}}}$ (note that the latter is log-concave) to get
\begin{multline}
\label{eq:PL lambda}
 \frac{1}{Z_{e^{-\varphi_k+\phi}}^{1-\lambda}Z_{e^{-\varphi_k}}^\lambda}\int \sup_{x=(1-\lambda)y+\lambda z}e^{(1-\lambda)\phi(y)-(1-\lambda)\varphi_k(y)-\lambda\varphi_k(z)}\, dx-1\\
 \geq c_n \lambda \inf_{x_0 \in \R^n} \left\|\frac{e^{-\varphi_k+\phi}}{Z_{e^{-\varphi_k+\phi}}}-\frac{e^{-\varphi_k(\cdot-x_0)}}{Z_{e^{-\varphi_k}}}\right\|_{L^1(\R^n)}^2.
\end{multline}
Applying \cite[Lemma 2.3]{Gentil} with $(1-\lambda)\phi$ in place of $g$ and $\varphi_k$ in place of $\varphi$, thanks to \eqref{eq:hessbounds} we obtain
\begin{multline*}
\sup_{x=(1-\lambda)y+\lambda z}e^{(1-\lambda)\phi(y)-(1-\lambda)\varphi_k(y)-\lambda\varphi_k(z)}\\
=e^{\phi(x)-\varphi_k(x)}\left[1+
\lambda \left[-\phi(x)+x\cdot \nabla\phi(x)-\varphi_k^*(\nabla \varphi_k(x))+\varphi_k^*(\nabla\varphi_k(x)-\nabla\phi(x))\right]+\lambda^2 O_{\varphi_k,\phi}(1+|x|^2)    \right]
\end{multline*}
Also,
$$
Z_{e^{-\varphi_k+\phi}}^{1-\lambda}Z_{e^{-\varphi_k}}^\lambda=Z_{e^{-\varphi_k+\phi}}\left[1+\lambda \left(\log Z_{e^{-\varphi_k}}-\log Z_{e^{-\varphi_k+\phi}}\right)+O_{\varphi_k,\phi}(\lambda^2) \right].
$$
Inserting these expansions in \eqref{eq:PL lambda}, dividing both the left-hand side and the right-hand side by $\lambda$, and letting $\lambda\to 0$,
we deduce that
\begin{multline*}
\frac{1}{Z_{e^{-\varphi_k+\phi}}}\int \left[-\phi(x)+x\cdot \nabla\phi(x)-\varphi_k^*(\nabla \varphi_k(x))+\varphi_k^*(\nabla\varphi_k(x)-\nabla\phi(x))\right]e^{\phi(x)-\varphi_k(x)}\, dx\\ - \left(\log Z_{e^{-\varphi_k}}-\log Z_{e^{-\varphi_k+\phi}}\right)
\geq c_n \inf_{x_0 \in \R^n} \left\|\frac{e^{-\varphi_k+\phi}}{Z_{e^{-\varphi_k+\phi}}}-\frac{e^{-\varphi_k(\cdot-x_0)}}{Z_{e^{-\varphi_k}}}\right\|_{L^1(\R^n)}^2
\end{multline*}
(note that the quadratic term in $\lambda$ vanishes as $\lambda \to 0$ thanks to the uniform decay $e^{\phi(x)-\varphi_k(x)}\leq C_{\phi}\,e^{-A|x|}$, see \eqref{eq:linlbk}).
Defining the new function $\psi:=\varphi_k-\phi$ and noticing that 
$$
\varphi_k^*(\nabla \varphi_k(x))-x\cdot \nabla \varphi_k=-\varphi_k(x),\qquad \int_{\R^n}x\cdot \nabla \psi(x)e^{-\psi(x)}\,dx=n\int_{\R^n}e^{-\psi(x)}\,dx=nZ_{e^{-\psi}},
$$
rearranging terms we obtain
\begin{multline}
\label{eq:LSI k}
\int_{\R^n}\varphi_k^*(\nabla\psi(x))e^{-\psi(x)}\,dx-nZ_{e^{-\psi}}-\Ent(e^{-\psi})\\
\geq c_n \,Z_{e^{-\psi}} \inf_{x_0 \in \R^n} \left\|\frac{e^{-\psi}}{Z_{e^{-\psi}}}-\frac{e^{-\varphi_k(\cdot-x_0)}}{Z_{e^{-\varphi_k}}}\right\|_{L^1(\R^n)}^2+ Z_{e^{-\psi}}\log Z_{e^{-\varphi_k}}
\end{multline}
(cp. \cite[Proof of Theorem 3.6]{Gentil}).
Note that, at the moment, this equation is valid for all functions $\psi$ such that $\psi-\varphi_k \in C^\infty_c(\R^n)$.

Now, given $\psi:\R^n\to \R$ Lipschitz such that $e^{-\psi}\in L^1(\R^n)$, 
consider a smooth cutoff $\eta_R\in C_c^\infty(\R^n)$ with $\eta_R\equiv 1$ on $B_R$ and $\eta_R\equiv 0$ on $B_{2R}^c$, and set
\[
\psi_{R,\sigma}:=\varphi_k+\bigl((\psi-\varphi_k)\eta_R\bigr)*\rho_\sigma.
\]
Then $\psi_{R,\sigma}-\varphi_k\in C_c^\infty(\R^n)$, so \eqref{eq:LSI k} applies to $\psi_{R,\sigma}$.
Letting first $\sigma\downarrow 0$ (with $R$ fixed) and then $R\to\infty$ yields the validity of \eqref{eq:LSI k} for $\psi$.
This proves that, by approximation, \eqref{eq:LSI k} holds for every Lipschitz function $\psi$ such that $e^{-\psi}\in L^1(\R^n)$.

Finally, recalling \eqref{eq:conjpoint}, this implies that
\begin{multline*}
\int_{\R^n}\varphi^*(\nabla\psi(x))e^{-\psi(x)}\,dx-nZ_{e^{-\psi}}-\Ent(e^{-\psi}) +\frac{\varepsilon_k}{2}\int_{\R^n}|\nabla\psi(x))|^2e^{-\psi(x)}\,dx\\
\geq c_n \,Z_{e^{-\psi}} \inf_{x_0 \in \R^n} \left\|\frac{e^{-\psi}}{Z_{e^{-\psi}}}-\frac{e^{-\varphi_k(\cdot-x_0)}}{Z_{e^{-\varphi_k}}}\right\|_{L^1(\R^n)}^2+Z_{e^{-\psi}}\log Z_{e^{-\varphi_k}},
\end{multline*}
so, thanks to \eqref{eq:L1conv}, by letting $k \to \infty$ we deduce that 
\eqref{eq:29} holds for all functions Lipschitz functions $\psi$ such that $e^{-\psi}\in L^1(\R^n)$. Again by approximation, \eqref{eq:29} extends to all $\psi$ such that $e^{-\psi}\in L^1(\R^n)\cap W^{1,1}_{\rm loc}(\R^n)$, concluding the proof.

\end{proof}

\begin{proof}[Proof of Corollary~\ref{cor:gauge_power_signed}]
Assume first that $u>0$ is smooth, and set
\[
\varphi_q(x):=\frac{\|x\|^q}{q}+\log L_q,
\qquad
\psi:=-p\log u.
\]
Then
\[
e^{-\psi}=u^p,
\qquad
\varphi_q^*(\xi)=\frac1p\|\xi\|_*^p-\log L_q.
\]
For every $\lambda>0$, define
\[
\psi_\lambda(x):=\psi(\lambda x)-n\log\lambda.
\]
Then
\[
e^{-\psi_\lambda(x)}=\lambda^n u(\lambda x)^p,
\qquad
\nabla\psi_\lambda(x)=\lambda\nabla\psi(\lambda x),
\qquad
Z_{e^{-\psi_\lambda}}=\|u\|_p^p=:Z,
\]
and
\[
\Ent(e^{-\psi_\lambda})=\Ent(u^p)+nZ\log\lambda.
\]
Applying Theorem~\ref{thm:LSI} to $\psi_\lambda$ with reference density $e^{-\varphi_q}$ gives
\[
p^{p-1}\lambda^p\int_{\R^n}\|\nabla u\|_*^p\,dx
-nZ-\Ent(u^p)-Z\log L_q-nZ\log\lambda
\ge c_n Z\,D(\lambda)^2,
\]
where
\[
D(\lambda):=
\inf_{x_0\in\R^n}
\left\|
\lambda^n\frac{u(\lambda\cdot)^p}{Z}
-
\frac{1}{L_q}\exp\!\Big(-\frac{\|\,\cdot-x_0\|^q}{q}\Big)
\right\|_{L^1}.
\]
Since this holds for every $\lambda>0$, we may take the infimum over $\lambda$. A direct minimization yields
\[
\inf_{\lambda>0}\left\{
p^{p-1}\lambda^p\int \|\nabla u\|_*^p
-nZ\log\lambda
\right\}
=
\frac{n}{p}Z
\log\!\left(
e\,\frac{p^p\int \|\nabla u\|_*^p}{nZ}
\right).
\]
On the other hand, by the change of variables $y=\lambda x$ and the $1$--homogeneity of $\|\cdot\|$,
\[
\inf_{\lambda>0}D(\lambda)
=
\inf_{\rho\in\mathcal E_q}
\left\|
\frac{u^p}{Z}-\rho
\right\|_{L^1}.
\]
Thus
\begin{equation}
\label{eq:p Sob+}
\frac{n}{p}\,Z\,
\log\!\left(
\frac{p^p}{n e^{p-1}L_q^{p/n}}\,
\frac{\int \|\nabla u\|_*^p}{Z}
\right)
-\Ent(u^p)
\ \ge\
c_n Z\inf_{\rho\in\mathcal E_q}\left\|\frac{u^p}{Z}-\rho\right\|_1^2.
\end{equation}
By approximation, the same estimate holds for every nonnegative $u\in W^{1,p}(\R^n)$.

We now pass to general signed $u$. Set
\[
u_+:=\max\{u,0\},\qquad u_-:=\max\{-u,0\},
\qquad
m_\pm:=\|u_\pm\|_p^p,\qquad M:=\|u\|_p^p,
\qquad
\theta:=\frac{m_-}{M}.
\]
Apply the nonnegative inequality to $u_+$ with the gauge $\|\cdot\|_*$, and to $u_-$ with the reflected gauge
$\|\cdot\|_*\hspace{-.1cm}\check{}$ (i.e., $\|\xi\|_*\hspace{-.1cm}\check{}:=\|-\xi\|_*$). Since
\[
\nabla u_+=\mathbf 1_{\{u>0\}}\nabla u,\qquad
\nabla u_-=-\mathbf 1_{\{u<0\}}\nabla u,
\]
the gradient terms recombine into $\int \|\nabla u\|_*^p$. Summing the two inequalities and using the
log-sum inequality gives
\[
\frac{n}{p}\,M
\log\!\left(
\frac{p^p}{n e^{p-1}L_q^{p/n}}
\frac{\int \|\nabla u\|_*^p}{M}
\right)
-\Ent(u_+^p)-\Ent(u_-^p)
\ge c_n\bigl(m_+\alpha_+^2+m_-\alpha_-^2\bigr),
\]
where
\[
\alpha_+:=\inf_{\rho\in\mathcal E_q}\left\|\frac{u_+^p}{m_+}-\rho\right\|_1,
\qquad
\alpha_-:=\inf_{\rho\in\check{\mathcal E}_q}\left\|\frac{u_-^p}{m_-}-\rho\right\|_1
\]
(with the convention $\alpha_\pm=0$ if $m_\pm=0$).
Since $u_+^p$ and $u_-^p$ have disjoint supports,
\[
\Ent(|u|^p)=\Ent(u_+^p)+\Ent(u_-^p)-M h(\theta),
\qquad
h(\theta):=-\theta\log\theta-(1-\theta)\log(1-\theta).
\]
Thus
\[
\frac{n}{p}\,\|u\|_p^p\,
\log\!\left(
\frac{p^p}{n e^{p-1} L_q^{p/n}}\,
\frac{\int_{\R^n}\|\nabla u(x)\|_*^p\,dx}{\|u\|_p^p}
\right)
-\Ent(|u|^p)\geq 
c_n\bigl(m_+\alpha_+^2+m_-\alpha_-^2\bigr)+M h(\theta).
\]
Define now
\[
v:=\frac{|u|^{p-1}u}{M}=\frac{u_+^p-u_-^p}{M}.
\]
If $\theta\le \frac12$, then $m_+\ge \frac{M}{2}$ and
\[
v-\frac{u_+^p}{m_+}
=
-\theta \frac{u_+^p}{m_+}-\theta \frac{u_-^p}{m_-}.
\]
Since $\frac{u_+^p}{m_+}$ and $\frac{u_-^p}{m_-}$ are nonnegative, have disjoint supports,
and both have integral $1$, it follows that
\[
\left\|v-\frac{u_+^p}{m_+}\right\|_1=2\theta,
\]
hence
$\inf_{\sigma\in\mathcal S_q}\|v-\sigma\|_1
\le 2\theta+\alpha_+.$
Moreover, on $[0,\frac12]$ one has
$h(\theta)\ge 2\log 2\,\theta\geq \log 2(2\theta)^2$,
therefore
\begin{multline*}
\frac{n}{p}\,\|u\|_p^p\,
\log\!\left(
\frac{p^p}{n e^{p-1} L_q^{p/n}}\,
\frac{\int_{\R^n}\|\nabla u(x)\|_*^p\,dx}{\|u\|_p^p}
\right)
-\Ent(|u|^p)
\ge \frac{c_n}{2}M\,\alpha_+^2+(\log 2)\,M\,(2\theta)^2\\
\geq \kappa_n M\,(2\theta+\alpha_+)^2
\geq \kappa_n M\,
\inf_{\sigma\in\mathcal S_q}\|v-\sigma\|_1^2,
\end{multline*}
with  $\kappa_n:=\frac14\min\{c_n,2\log 2\} \geq \frac{c_n}4.$

The case $\theta\ge \frac12$ is identical (just using 
$u_-$ in place of $u_+$). This concludes the proof of \eqref{eq:gauge_power_signed}.
\end{proof}

\begin{proof}
We first prove \eqref{eq:wulff_signed} for $u\in C^\infty_c(\R^n)$.
For every $q>1$, let $p=\frac{q}{q-1}$. Applying Corollary~\ref{cor:gauge_power_signed} to $u$ gives
\[
\frac{n}{p}\,\|u\|_p^p\,
\log\!\left(
\frac{p^p}{n e^{p-1}L_q^{p/n}}\,
\frac{\int \|\nabla u\|_*^p}{\|u\|_p^p}
\right)
-\Ent(|u|^p)
\ \ge\
\frac{c_n}4\,\|u\|_p^p\,
\inf_{\rho\in\mathcal S_q}
\left\|\frac{|u|^p}{\|u\|_p^{p-1}u}-\rho\right\|_1^2.
\]
Letting $q\to\infty$ (equivalently $p\to 1^+$), one has by dominated convergence
\[
L_q\to |K|,
\qquad
\|u\|_p^p\to \|u\|_1,
\qquad
\int \|\nabla u\|_*^p\to \int \|\nabla u\|_*,
\qquad
\Ent(|u|^p)\to \Ent(|u|).
\]
It remains to take the limit on the right-hand side. Writing
$a=r^{-q}$ gives
\[
\rho_{a,x_0}(x)=\frac{r^{-n}}{L_q}\exp\!\Big(-\frac{\|(x-x_0)/r\|^q}{q}\Big),
\qquad 
\check\rho_{a,x_0}(x)=\frac{r^{-n}}{L_q}\exp\!\Big(-\frac{\|(x_0-x)/r\|^q}{q}\Big).
\]
Taking the limit as $q \to \infty$ we see that the families
$\mathcal E_q$ and $\check{\mathcal E}_q$ converge in $L^1$, uniformly in the parameters,
to $\mathcal W$ and $\check{\mathcal W}$ respectively. Therefore
\[
\inf_{\rho\in\mathcal S_q}
\left\|\frac{|u|^{p-1}u}{\|u\|_p^p}-\rho\right\|_1
\longrightarrow
\inf_{\rho\in\mathcal W\cup\check{\mathcal W}}
\left\|\frac{u}{\|u\|_1}-\rho\right\|_1\qquad \text{as }q \to \infty.
\]
This proves \eqref{eq:wulff_signed} for $u\in C^\infty_c(\R^n)$.
The extension to general $u\in W^{1,1}(\R^n)$ with $\|u\|_1>0$ and $\Ent(|u|)<\infty$
follows by standard approximation.

Finally, given a set of finite perimeter $E$, let $u_j$ be a standard sequence of nonnegative smooth $BV$ approximations of $\chi_E$ such that
\[
u_j\to \chi_E \quad\text{in }L^1(\R^n),
\qquad
\int_{\R^n}\|\nabla u_j\|_*\,dx\to P_{\|\cdot\|}(E),
\qquad
\|u_j\|_1\to |E|.
\]
Applying \eqref{eq:wulff_signed} to $u_j$ and passing to the limit  gives
\[
n|E|\log\!\left(
\frac{P_{\|\cdot\|}(E)}{n|K|^{1/n}|E|}
\right)-\Ent(\chi_E)
\ge
\frac{c_n}4\,|E|\,
\inf_{\rho\in\mathcal W}
\left\|\frac{\chi_E}{|E|}-\rho\right\|_1^2.
\]
Since $\Ent(\chi_E)=-|E|\log|E|,$
this is exactly \eqref{eq:wulff_sets}.
\end{proof}

\section{\Cref{sharplambda}: structure of the proof}\label{sharplambdaoverviewsec}

In this overview all dependence on $n$ and $p$ is absorbed into
$O_{n,p}$ and $\Omega_{n,p}$.

We pick levels
    $(a_i)_{i\in\mathbb Z}$ and set
    \[
    \ell_i:=\min\{a_{i+2},\ell\cdot\textbf{1}_{L_{a_i}}\}.
    \]
    \Cref{FindCorrespondingg_i} constructs nested convex sets $K_i$ and caps
    $b'_i$ such that, for
    $g_i:=\min\{b'_i,g\cdot\textbf{1}_{K_i}\}$,
    \[
    \int\ell_i\,dx=\int g_i\,dx,
    \qquad
    \int\M^*(\ell_i,g_i)\,dx
    \leq\int\ell_i\,dx+\delta\int\ell\,dx.
    \]
    Thus the effective
    deficit for $\ell_i$ and $g_i$ is $\delta_i:=\delta\int\ell/\int\ell_i$. We first prove (steps 1 through 6, \Cref{intervaltranslateprop})
$$\int|\ell_i(x)-g_i(x+v_i)|\,dx\leq O_{n,p}\left(\sqrt{\frac{\delta_i}{\lambda}}\right)
    \int\ell_i\,dx,$$
    using that $|L_{a_i}|\leq 1.5|L_{a_{i+2}}|$. We then show in the remaining steps that you can stitch together the $\ell_i$'s, using that $\ell_i$ and $\ell_{i+1}$ overlap on the values $[a_{i+1},a_{i+2}]$.

\begin{enumerate}
    \item Consider level sets $L^i_t:=\{x\in\mathbb{R}^n: \ell_i(x)>t\}$ and $G^i_t:=\{x\in\mathbb{R}^n: g_i(x)>t\}$ of $\ell_i$ and $g_i$. Define $T$ by
    \[
    \int_0^t|L^i_s|\,ds=\int_0^{T(t)}|G^i_s|\,ds.
    \]
    Using the inclusion
    $H_{\M(t,T(t))}\supset\lambda L^i_t+(1-\lambda)G^i_{T(t)}$, we find that roughly expressed
    $$ \int |\lambda L^i_t+(1-\lambda)G^i_{T(t)}|- M_{\lambda,1/n}(|L^i_t|,|G^i_{T(t)}|) dt + \int M_{\lambda,1/n}(|L^i_t|,|G^i_{T(t)}|)-M_{\lambda,-p}(|L^i_t|,|G^i_{T(t)}|)dt \leq \delta_i \int \ell_i dx. $$

    \item \Cref{quantitativepowermeangaplem} suggests that $\int M_{\lambda,1/n}(|L^i_t|,|G^i_{T(t)}|)-M_{\lambda,-p}(|L^i_t|,|G^i_{T(t)}|)dt\leq \delta \int \ell dx$ being small, implies that for a typical $t$, we have $|L^i_t|\approx|G^i_{T(t)}|$. This in turn suggests $T'(t)=\frac{|L^i_t|}{|G^i_{T(t)}|}\approx 1$, so $T(t)\approx t$.

    \item By the stability of the Brunn--Minkowski inequality (cf. \Cref{BarchiesiJulin}) if $|L^i_t|=|G^i_{T(t)}|$, we have for some $v_t\in\mathbb{R}^n$
    $$|\lambda L^i_t+(1-\lambda)G^i_{T(t)}|- M_{\lambda,1/n}(|L^i_t|,|G^i_{T(t)}|)\geq \Omega_{n,p}\left(\lambda\frac{|(L^i_t+v_t)\triangle G^i_{T(t)}|^2}{|L^i_t|}\right)$$
 
    \item \Cref{symdiflevelsetslem} combines these two bounds and integrates over $t$, to give 
    $$\int |(L^i_t+v_t)\triangle G^i_t|dt\leq O_{n,p}\left(\sqrt{\frac{\delta_i}{\lambda}}\right)\int \ell_i dx$$

    \item Aiming to replace the
    vectors $v_t$ by one vector, let $C_0,\dots,C_n$ be the regular-simplex
    cones.  By \Cref{parallelconeslem}, there is a vector $v_{i}$ such that $\int_{C_j}\ell_i\,dx
    =\int_{C_j+v_i}g_i\,dx.$
    Apply \Cref{symdiflevelsetslem} in each of the cones, to find 
    $$\int |((C_j\cap L^i_t)+v_{j,t})\triangle ((C_j+v_i)\cap G^i_t)|dt\leq O_{n,p}\left(\sqrt{\frac{\delta_i}{\lambda}}\right)\int \ell_i dx.$$

    \item \Cref{translationsincones} says that the global information and information in the cones combines to give
    $$|(L^i_t+v_i)\triangle G^i_t|\leq O_{n,p}\left(|(L^i_t+v_t)\triangle G^i_t|+\sum_j |((C_j\cap L^i_t)+v_{j,t})\triangle ((C_j+v_i)\cap G^i_t)| \right)$$
    Integrating allows us to conclude \Cref{intervaltranslateprop}
    \[
    \int |\ell_i(x)-g_i(x+v_i)|\, dx=
    \int |(L^i_t+v_i)\triangle G^i_t|\,dt
    \leq O_{n,p}\left(
    \sqrt{\frac{\delta_i}{\lambda}}\int\ell_i\,dx\right).
    \]

    \item Set the levels $a_i$ for $i\in\mathbb{Z}$ as follows:
    \[
    a_0=\frac12\sup\ell,
    \qquad
    a_i=a_0\exp(c_0i/n)\quad(\text{for }i\leq0), \quad a_i \text{ such that } |L_{a_i}|=1.1^{-i}|L_{a_0}| \quad (\text{for }i\geq0)
    \]
    \Cref{levelgeometry}(\ref{leveltwostepconclusion}) gives that for every $i\in\mathbb{Z}$, $|L_{a_i}|\leq\frac32|L_{a_{i+2}}|$.

    \item We show that the mass of $\ell_i$ decreases exponentially both as $i\to \infty$ and as $i\to-\infty$. Hence, we may apply \Cref{intervaltranslateprop} to all $\ell_i$ and still have by
    \Cref{intexponentialdecay} and \Cref{lowlevelhardy} that
    \[
    \sum_{i=-\infty}^{\infty} \sqrt{\frac{\delta_i}{\lambda}} \int \ell_i=
    \sum_{i=-\infty}^{\infty}\sqrt{\frac{\int\ell_i\,dx}{\int\ell\,dx}}\sqrt{\frac{\delta}{\lambda}}\int\ell\,dx=O_{n,p}\left(\sqrt{\frac{\delta}{\lambda}}\right)\int\ell\,dx
    .
    \]

    \item Choose $k$ at the last level for which
    \[
    \int_0^{a_{k+2}}|L_t|\,dt
    \leq O_{n,p}\left(\sqrt{\frac{\delta}{\lambda}}\right)\int\ell\,dx,
    \]
    so that the levels below $a_{k+2}$ don't contribute much mass.
    
    \item We next compare $v_i$ and $v_{i+1}$.  For 
    $t\in[a_{i+1},a_{i+2}]$, we have $L^i_t=L^{i+1}_t$, while the nested supports give $G^{i+1}_t\subset G^i_t$. Hence, if both  $|(L^i_t+v_i)\triangle G^{i}_t|$ and $|(L^{i+1}_t+v_{i+1})\triangle G^{i+1}_t|$ are small, then $v_i$ must be similar to $v_{i+1}$. In particular, \Cref{adjacenttranslatesarecloseprop} shows
    $$\int_{a_{i+1}}^{a_{i+2}} |L_t \triangle (L_t+(v_i-v_{i+1}))|dt
\leq 4\left(
\int\left|\ell_i(x)-g_i(x+v_i)\right|\,dx
+\int\left|\ell_{i+1}(x)-g_{i+1}(x+v_{i+1})\right|\,dx\right).$$

    \item If $|L_t \triangle (L_t+(v_i-v_{i+1}))|$ is small in terms of $|L_t|$, then $v_i-v_{i+1}$ is small in terms of $L_t$, in particular $v_i-v_{i+1}\in \alpha (L_t-L_t)$ for some small $\alpha$. 
    \item For $k<i\leq0$, the sets $L_{a_i}$ grow as $i$ decreases. Hence, \Cref{integratedlowtranslationlem} shows
    $$\int_0^{a_{i+2}}|L_t\triangle (L_t+(v_i-v_{i+1}))|dt\leq O_n\left(\alpha a_{i+2} |L_{a_{i+2}}|\right)\leq O_n\left(\int_{a_{i+1}}^{a_{i+2}}|L_t\triangle (L_t+(v_i-v_{i+1}))|dt\right).$$
    \item For $i\geq 0$, \Cref{positiveproportionofremainingheight} gives that $a_{i+1}-a_{i}\geq \Omega_n(\sup \ell - a_i)$ and \Cref{translationestimateslem} gives that $|K\triangle (K+v)|\leq |L\triangle (L+v)|$ for convex $K\subset L$. Combining these gives
    \[
    \int_{a_{i+1}}^{\sup\ell}|L_t\triangle(L_t+(v_i-v_{i+1}))|\,dt
    \leq O_{n,p}\left(\int_{a_{i+1}}^{a_{i+2}}|L_t\triangle(L_t+(v_i-v_{i+1}))|\,dt\right).
    \]

    \item Telescoping, we find $|G_t\setminus (L_t+ v_0)|\leq  |G_t\setminus (L_t+ v_i)|+\sum_{j=1}^i |L_t\triangle (L_t+ (v_j-v_{j-1}))|$. Combining these two steps, we can use $v_0$ for every block to conclude 
    $$\int |\ell(x)-g(x+v_0)|dx =2\int |(L_t+v_0)\setminus G_t| dt= O_{n,p}\left(\sqrt{\frac{\delta}{\lambda}}\right)\int \ell dx.$$
\end{enumerate}

\section{Proof of \Cref{sharplambda}}\label{sharplambdaproofsec}

\subsection{Preliminary results for the proof of \Cref{sharplambda}}

First we show that if we restrict $\ell$ between particular level sets, we can find a suitable part of $g$ so that the doubling doesn't go up and moreover, we can do this in a consistent way.
\begin{lem}\label{FindCorrespondingg_i}
    Let $n,\lambda,p,\ell,g,h$ as in \Cref{sharplambda}. Let $0\leq a_0<a_1<\dots\leq \sup \ell$ and $0\leq b_0<b_1<\dots\leq \sup \ell$ so that $a_i<b_i$ for all $i=0,1,\dots$. Let $\ell_i:=\min\{b_i, \ell\cdot \textbf{1}_{L_{a_i}}\}$. Then there exists some $0\leq b'_0,b'_1,\dots\leq \sup g$ and nested convex sets $K_0\supset K_1\supset \dots$ so that if we let $g_i:=\min\{b'_i, g\cdot \textbf{1}_{K_i}\}$, then $\int \ell_i dx=\int g_i dx$ and if we let $h_i=\M^*(\ell_i,g_i)$, then
    $$\int h_i dx\leq \int \ell_i dx +\delta\int \ell dx.$$
\end{lem}

\subsubsection{Proof of Lemmas}

\begin{proof}[Proof of \Cref{FindCorrespondingg_i}]
We first construct the supports $K_i$ and only afterwards choose the heights $b'_i$. Set $L_{a_{-1}}:=\mathbb{R}^n$ and $K_{-1}:=\mathbb{R}^n$. We argue inductively. Suppose that $K_{i-1}$ has already been constructed so that
$$
\int_{L_{a_{i-1}}}\ell dx=\int_{K_{i-1}}gdx
$$
and, writing $\bar\ell_{i-1}:=\ell\cdot\textbf{1}_{L_{a_{i-1}}}$, $\bar g_{i-1}:=g\cdot\textbf{1}_{K_{i-1}}$, and $\bar h_{i-1}:=\M^*(\bar\ell_{i-1},\bar g_{i-1})$, we have
$$
\int \bar h_{i-1}dx\leq \int \bar\ell_{i-1}dx+\delta\int \ell dx.
$$
For $i=0$ this is exactly the hypothesis of \Cref{sharplambda}.

Since $L_{a_i}\subset L_{a_{i-1}}$ is convex, we may write it, up to a null set, as the intersection of $L_{a_{i-1}}$ with a countable family $\mathcal{H}$ of half-spaces. Let $X_0:=L_{a_{i-1}}$ and $Y_0:=K_{i-1}$. We choose a sequence of half-spaces $H_1,\dots\in \mathcal{H}$ so that $\left|\left(L_{a_{i-1}}\cap \bigcap_{m=1}^M H_m\right)\triangle L_{a_i}\right|\to 0$ as $M\to\infty$, and adjust $X_0$ and $Y_0$ plane by plane to find $K_i$. Suppose that $X_m$ and $Y_m$ have been constructed with
$$
\int_{X_m}\ell dx=\int_{Y_m}gdx,
$$
then for the next half-space $H_{m+1}$ there is a parallel translate $H_{m+1}+v_{m+1}$ such that
$$
\int_{X_m\cap H_{m+1}}\ell dx=\int_{Y_m\cap (H_{m+1}+v_{m+1})}gdx.
$$
To see this, translate $H_{m+1}$ continuously in its normal direction. The mass of $g$ in $Y_m$ cut by the translated half-space varies continuously from $0$ to $\int_{Y_m}gdx$, so the intermediate value theorem gives the desired translate. If $g$ is not compactly supported, first apply this argument to large compact truncations of $g$ and then pass to the limit. Define
$$
X_{m+1}:=X_m\cap H_{m+1},\qquad Y_{m+1}:=Y_m\cap (H_{m+1}+v_{m+1}).
$$
If $\bar\ell_m:=\ell\cdot\textbf{1}_{X_m}$, $\bar g_m:=g\cdot\textbf{1}_{Y_m}$, and $\bar h_m:=\M^*(\bar\ell_m,\bar g_m)$, then
$$
\bar h_{m+1}\leq \bar h_m\cdot\textbf{1}_{H_{m+1}+(1-\lambda)v_{m+1}}.
$$
By Borell--Brascamp--Lieb on the complementary half-spaces and the choice of $v_{m+1}$,
$$
\int_{H_{m+1}^c+(1-\lambda)v_{m+1}}\bar h_mdx\geq \int_{H_{m+1}^c}\bar\ell_mdx.
$$
Consequently
$$
\int \bar h_{m+1}dx\leq \int \bar\ell_{m+1}dx+\left(\int \bar h_mdx-\int \bar\ell_mdx\right),
$$
so the deficit does not increase during the shaving. Passing to the limit over the half-spaces gives a convex set 
$K_i:=K_{i-1}\cap \bigcap_{m=1}^\infty (H_m+v_m)\subset K_{i-1}$ such that, for $\bar\ell_i:=\ell\cdot\textbf{1}_{L_{a_i}}$, $\bar g_i:=g\cdot\textbf{1}_{K_i}$, and $\bar h_i:=\M^*(\bar\ell_i,\bar g_i)$,
$$
\int \bar\ell_i dx=\int \bar g_i dx,
\qquad
\int \bar h_i dx\leq \int \bar\ell_i dx+\delta\int \ell dx.
$$
This constructs the nested convex sets $K_0\supset K_1\supset\dots\supset K_k$.

It remains to choose the caps. For each $i$, consider the continuous nondecreasing function
$$
\Phi_i(c):=\int \min\{c,\bar g_i\}dx,\qquad 0\leq c\leq \sup \bar g_i.
$$
It satisfies $\Phi_i(0)=0$ and $\Phi_i(\sup \bar g_i)=\int \bar g_i dx=\int \bar\ell_i dx$. Since
$$
0\leq \int \min\{b_i,\bar\ell_i\}dx\leq \int \bar\ell_i dx,
$$
there is a $b'_i\in[0,\sup \bar g_i]$ such that
$$
\int \min\{b_i,\bar\ell_i\}dx=\int \min\{b'_i,\bar g_i\}dx.
$$
Choosing the generalized inverse of $\Phi_i$ and perturbing the levels harmlessly if a flat part occurs, we may take $b'_0,\dots,b'_k$; the estimates below are stable under the limiting procedure that removes this perturbation. Now set
$$
\ell_i:=\min\{b_i,\bar\ell_i\},\qquad g_i:=\min\{b'_i,\bar g_i\},\qquad h_i:=\M^*(\ell_i,g_i).
$$
We claim that the deficit bound survives this capping. Indeed, pointwise
$$
 h_i\leq \min\{\bar h_i,\M(b_i,b'_i)\}.
$$
Thus
$$
\int h_i dx\leq \int \bar h_i dx-\int_{\M(b_i,b'_i)}^{\sup\bar h_i}|\bar H_{i,t}|dt,
$$
where $\bar H_{i,t}$ denotes the $t$-level set of $\bar h_i$. We now prove the needed tail estimate directly. Let $T_i\colon[0,\sup\bar\ell_i]\to[0,\sup\bar g_i]$ be the monotone transport determined by
$$
\int_0^s |\bar L_{i,t}|dt=\int_0^{T_i(s)} |\bar G_{i,t}|dt.
$$
By the choice of $b'_i$, we have $T_i(b_i)=b'_i$. For every $s\geq b_i$, the defining inequality for $\bar h_i$ gives
$$
\bar H_{i,\M(s,T_i(s))}\supset \lambda \bar L_{i,s}+(1-\lambda)\bar G_{i,T_i(s)}.
$$
Differentiating the transport identity gives, for a.e. $s$,
$$
\frac{d}{ds}T_i(s)=\frac{|\bar L_{i,s}|}{|\bar G_{i,T_i(s)}|},\qquad \text{therefore}\qquad 
\frac{d}{ds}\M(s,T_i(s))\geq\frac{1}{M_{\lambda,-p}\left(1,\frac{|\bar G_{i,T_i(s)}|}{|\bar L_{i,s}|}\right)}.
$$
Using this identity and Brunn--Minkowski, we get
\begin{align*}
|\bar H_{i,\M(s,T_i(s))}|\frac{d}{ds}\M(s,T_i(s))
\geq |\lambda \bar L_{i,s}+(1-\lambda)\bar G_{i,T_i(s)}|\frac{d}{ds}\M(s,T_i(s))\geq \frac{M_{\lambda,1/n}(|\bar L_{i,s}|,|\bar G_{i,T_i(s)}|)}{M_{\lambda,-p}\left(1,\frac{|\bar G_{i,T_i(s)}|}{|\bar L_{i,s}|}\right)}
\geq |\bar L_{i,s}|,
\end{align*}
so integrating in $s$ gives
$$
\int_{\M(b_i,b'_i)}^{\sup\bar h_i}|\bar H_{i,t}|dt
\geq \int_{b_i}^{\sup\bar\ell_i}|\bar L_{i,t}|dt
=\int \bar\ell_i dx-\int \ell_i dx.
$$
Combining this with the full-support deficit estimate for $\bar h_i$ yields
$$
\int h_i dx\leq \int \bar h_i dx-\left(\int \bar\ell_i dx-\int \ell_i dx\right)
\leq \int \ell_i dx+\delta\int \ell dx,
$$
as required.
\end{proof}

\subsection{\Cref{sharplambda} if $\ell$ takes values in $\{0\}\cup[s_0,s_1]$.}

Recall that $\zeta_{n,p}=\min\{1,1+np\}$ and 
$\mathcal N_n=(Cn)^{Cn}$; where $C$ may increase from line to line after enlarging
the universal constant $C$. 

\begin{prop}\label{intervaltranslateprop}
Let $n \in \mathbb{N}$, $\lambda \in (0,1/2]$, and $p \in (-1/n, \infty)$. Let $\ell,g, h\colon \mathbb{R}^n\rightarrow \mathbb{R}_{\geq0}$ be integrable functions such that
    \begin{itemize}
        \item $\ell$ is $p$-concave
        \item $\int_{\mathbb{R}^n} \ell\,dx=  \int_{\mathbb{R}^n} g\,dx$,
        \item $h(\lambda x +(1-\lambda)y) \geq M_{\lambda,p}(\ell(x),g(y))$ for all $x,y \in \mathbb{R}^n$, and
        \item $\int_{\mathbb{R}^n}h\,dx = (1+\delta)\int_{\mathbb{R}^n}\ell\,dx$ for some $\delta \geq 0$.
    \end{itemize} 

 Additionally let $0\leq s_0<s_1\leq \sup \ell$ such that 
 \begin{itemize}
 \item $\ell:\mathbb{R}^n\to\{0\}\cup [s_0,s_1]$, and
 \item $|L_{s_0}|\leq 1.5 |L_{s_1}|$
 \end{itemize} 
 Then there exists a translate $v_{s_0,s_1}\in\mathbb{R}^n$ such that
 \[
 \int \left|\ell(x)-g(x+v_{s_0,s_1})\right|dx
 \leq \mathcal N_n\zeta_{n,p}^{-1/2}
 \sqrt{\frac{\delta}{\lambda}}\int \ell dx.
 \]
\end{prop}

\subsubsection{Preliminary lemmas for \Cref{intervaltranslateprop}}

We prove the following lemma that shows that if we restrict $\ell$ between level sets, we can find a translate for every level set so that the symmetric difference with the level set of $g$ is small.
\begin{lem}\label{symdiflevelsetslem}
  Let $n,\lambda,p,\ell,g,h$ as in \Cref{sharplambda}. Assume that there is $m>0$ such that
  $m\leq|L_t|\leq(2n)^{n+2}m$ for all $t\leq\max\ell$.  Then there exist $(v_t)_{0\leq t\leq\sup\ell}$ in $\mathbb R^n$ such that
  \[
  \int_0^{\sup\ell}|L_t\triangle(G_t+v_t)|\,dt
  \leq \mathcal N_n\zeta_{n,p}^{-1/2}
  \sqrt{\frac{\delta}{\lambda}}\int\ell\,dx.
  \]
\end{lem}

To prove this lemma, we use the following, which is Corollary 1 from \cite{Barchiesi}.
\begin{thm}\label{BarchiesiJulin}
    Let  $K,B\subset\mathbb{R}^n$ with $K$ convex, $|K|=|B|$ and $\lambda \in (0,1/2]$, and assume that
    \[
    |\lambda K+(1-\lambda)B|\leq (1+\delta)|K|
    \]
    for some $\delta>0$. Then there exists a translation $v\in\mathbb{R}^n$ so that 
$$|(K+v)\triangle B|\leq \mathcal N_n\sqrt{\frac{\delta}{\lambda}}|K|.$$
The stated bound for $\mathcal N_n$ follows by tracking the constant
in \cite[Remark~2]{Barchiesi}: the deficit coefficient there is bounded below
by a multiple of $9^{-n}$.

\end{thm}

The crucial tool to go from \Cref{symdiflevelsetslem} to \Cref{intervaltranslateprop} is the following lemma that says that if the symmetric difference between sets is simultaneously small in all cones and globally, then the translate is controlled.

\begin{lem}\label{translationsincones}
Let $C_0,\dots,C_n$ be the
regular-simplex cones in \Cref{Cnidefn}.  Consider measurable sets
$K,X\subset\mathbb R^n$, with $K$ convex, and vectors
$v,v_0,\dots,v_n\in\mathbb R^n$.  Suppose that
$B_2^n\subset K$, $|K|\leq2n^n|B_2^n|$, and, for some $\alpha\geq0$,
\[
|K\triangle(v+X)|\leq\alpha|K|,
\]
and that, for every $i=0,\dots,n$,
\[
|(K\cap C_i)\triangle(v_i+(X\cap C_i))|
\leq\alpha|K|.
\]
Then
\[
|K\triangle X|\leq\mathcal N_n\alpha|K|.
\]
\end{lem}

\subsubsection{Proof of \Cref{intervaltranslateprop}}

\begin{proof}[Proof of \Cref{intervaltranslateprop}]
By monotonicity of $\M$ in $p$ , we may assume $-1/n<p\leq-1/(2n)$.
The proposition is affine invariant, so we may assume
\[
B_2^n\subset L_{s_1}\subset nB_2^n.
\]
For every nonempty level $L_t$ we have
$B_2^n\subset L_{s_1}\subset L_t\subset L_{s_0},$
 hence
\[
|B_2^n|\leq |L_t|\leq |L_{s_0}|\leq\frac32|L_{s_1}|
\leq\frac32n^n|B_2^n|.
\]
In particular,
\[
\mathcal N_n^{-1}\leq |L_t|\leq\mathcal N_n
\qquad\text{whenever }L_t\neq\emptyset.
\]
Thus \Cref{symdiflevelsetslem} applies to $\ell$ with $m=|B_2^n|$.
Using \Cref{symdiflevelsetslem}, find translates $v_t\in \mathbb{R}^n$ so that
$$\int_0^{\sup\ell} |L_t\triangle (G_t+v_t)|dt
\leq \mathcal N_n\zeta_{n,p}^{-1/2}
\sqrt{\frac{\delta}{\lambda}}\int\ell\,dx.$$

Let $C_0,\dots,C_n$ be the regular-simplex cones from \Cref{Cnidefn} with common vertex at the
origin.  Apply \Cref{parallelconeslem} to $g$
with proportions $\int_{C_i}\ell\,dx/\int\ell\,dx$.  This gives a translate
$v_{s_0,s_1}\in\mathbb R^n$ such that, for all $i=0,\dots,n$,
$$\int_{C_i}\ell dx=\int_{C_i+v_{s_0,s_1}}gdx.$$
At every nonempty level $t$, we have
\[
\frac{|B_2^n|}{n+1}\leq |C_i\cap L_t|
\leq\frac32n^n|B_2^n|
\leq(2n)^{n+2}\frac{|B_2^n|}{n+1}.
\]
Consequently,
\[
\mathcal N_n^{-1}\leq |C_i\cap L_t|\leq\mathcal N_n,
\qquad
\mathcal N_n^{-1}\int\ell\,dx
\leq\int_{C_i}\ell\,dx\leq\int\ell\,dx.
\]
Define $\ell_i:=\ell\cdot\textbf{1}_{C_i}$, $g_i:=g\cdot\textbf{1}_{C_i+v_{s_0,s_1}}$
and $h_i:=\M(\ell_i,g_i)$. By Borell-Brascamp-Lieb, we have
$$\int h_i dx\geq \int \ell_i dx,$$
for every $i=0,\dots,n$. Clearly, $ h_i\leq h \cdot \textbf{1}_{C_i+(1-\lambda)v_{s_0,s_1}}$, so that 
$$\int h_i dx \leq \int h dx -\sum_{j\neq i}\int h_jdx\leq\delta \int \ell dx+ \int \ell dx-\sum_{j\neq i} \int \ell_j dx\leq \int \ell_i dx+\delta \int \ell dx\leq(1+\mathcal N_n\delta)\int\ell_i\,dx,$$
so applying \Cref{symdiflevelsetslem} on each cone gives families of vectors
$v_{i,t}\in\mathbb R^n$ such that
\[
\int_0^{\sup\ell}|L_{i,t}\triangle(G_{i,t}+v_{i,t})|\,dt
\leq\mathcal N_n\zeta_{n,p}^{-1/2}
\sqrt{\frac{\delta}{\lambda}}\int\ell\,dx.
\]

Using that $L_{i,t}=L_t\cap C_i$ and $G_{i,t}=G_t\cap (C_i+v_{s_0,s_1})$, this translates to 
\begin{align*}
\int_0^{\sup\ell} |(L_t\cap C_i)\triangle (((G_t-v_{s_0,s_1})\cap C_i)+v_{i,t}+v_{s_0,s_1})|dt&=\int_0^{\sup\ell} |(L_t\cap C_i)\triangle ((G_t\cap (C_i+v_{s_0,s_1}))+v_{i,t})|dt\\
&\leq\mathcal N_n\zeta_{n,p}^{-1/2}
\sqrt{\frac{\delta}{\lambda}}\int\ell\,dx.
\end{align*}
Recall that 
\begin{align*}
\int_0^{\sup\ell}
|L_t\triangle((G_t-v_{s_0,s_1})+v_t+v_{s_0,s_1})|\,dt&=\int_0^{\sup\ell}|L_t\triangle(G_t+v_t)|\,dt
\leq\mathcal N_n\zeta_{n,p}^{-1/2}
\sqrt{\frac{\delta}{\lambda}}\int\ell\,dx.
\end{align*}
For every nonempty level, put
\[
\alpha_t:=\frac{1}{|L_t|}\bigg(
|L_t\triangle(G_t+v_t)|
+\sum_{i=0}^n|(L_t\cap C_i)\triangle
((G_t\cap(C_i+v_{s_0,s_1}))+v_{i,t})|\bigg).
\]
Apply \Cref{translationsincones} with
$K=L_t,$ $X=G_t-v_{s_0,s_1}$, and 
 \[
\alpha_t:=\frac{1}{|L_t|}\bigg(
|L_t\triangle(G_t+v_t)|
+\sum_{i=0}^n|(L_t\cap C_i)\triangle
((G_t\cap(C_i+v_{s_0,s_1}))+v_{i,t})|\bigg),
\]
to conclude
\[
\begin{aligned}
|L_t\triangle(G_t-v_{s_0,s_1})|
&\leq\mathcal N_n\alpha_t|L_t|=\mathcal N_n\bigg(|L_t\triangle(G_t+v_t)|
+\sum_{i=0}^n|(L_t\cap C_i)\triangle
((G_t\cap(C_i+v_{s_0,s_1}))+v_{i,t})|\bigg).
\end{aligned}
\]

Integrating the last inequality and using the preceding estimates gives
\[
\begin{aligned}
\int_0^{\sup\ell}|L_t\triangle(G_t-v_{s_0,s_1})|\,dt
&\leq\mathcal N_n
\int_0^{\sup\ell}|L_t\triangle(G_t+v_t)|\,dt+\mathcal N_n\sum_{i=0}^n
\int_0^{\sup\ell}|(L_t\cap C_i)\triangle
((G_t\cap(C_i+v_{s_0,s_1}))+v_{i,t})|\,dt\\
&\quad\leq\mathcal N_n\zeta_{n,p}^{-1/2}
\sqrt{\frac{\delta}{\lambda}}\int\ell\,dx.
\end{aligned}
\]
Since $\int \ell dx=\int gdx$, we find
$$\int \left|\ell(x)-g(x+v_{s_0,s_1})\right|dx
\leq \mathcal N_n\zeta_{n,p}^{-1/2}
\sqrt{\frac{\delta}{\lambda}}\int\ell\,dx.$$
 This proves the
proposition.

\end{proof}

\subsubsection{Proof of lemmas}

\begin{lem}\label{quantitativepowermeangaplem}
Let $n\in\mathbb N$, $\lambda\in(0,1/2]$,
$p\in(-1/n,0)$ and $r>0$.  Then
$$
\frac{M_{\lambda,1/n}(1,r)}{M_{\lambda,-p}(1,r)}-1
\geq \Omega\!\left(\lambda(1/n+p)\min\{(\log r)^2,1\}\right).
$$
\end{lem}

\begin{proof}
First write
\[
\frac{M_{\lambda,1/n}(1,r)}{M_{\lambda,-p}(1,r)}-1\geq \Omega\left(\log\frac{M_{\lambda,1/n}(1,r)}{M_{\lambda,-p}(1,r)}\right)
=\Omega\left(\int_{-p}^{1/n}\frac{d}{ds}\log M_{\lambda,s}(1,r)\,ds\right)
\]
so that it suffices to lower bound $\frac{d}{ds}\log M_{\lambda,s}(1,r)$ for
$s\in[-p,1/n]$.  Differentiating
$\log M_{\lambda,s}(1,r)
=\frac1s\log\bigl(\lambda+(1-\lambda)r^s\bigr),$ we find
\[
\frac{d}{ds}\log M_{\lambda,s}(1,r)
=-\frac1{s^2}\log\bigl(\lambda+(1-\lambda)r^s\bigr)
+\frac1s\frac{(1-\lambda)r^s\log r}{\lambda+(1-\lambda)r^s}=\frac{1}{s^2}\left(
\frac{s(1-\lambda)r^s\log r}{\lambda+(1-\lambda)r^s}
-\log\bigl(\lambda+(1-\lambda)r^s\bigr)\right).
\]
Writing $v=s\log r$, the expression in parentheses is
$\frac{v(1-\lambda)e^v}{\lambda+(1-\lambda)e^v}
-\log\bigl(\lambda+(1-\lambda)e^v\bigr).$
This expression vanishes at $v=0$, and its derivative with respect to $v$ is
$v\frac{\lambda(1-\lambda)e^v}
{\bigl(\lambda+(1-\lambda)e^v\bigr)^2}.$
Hence, we can write
\[
\frac{d}{ds}\log M_{\lambda,s}(1,r)
=\frac{1}{s^2}\int_0^{s\log r}
u\frac{\lambda(1-\lambda)e^u}
{\bigl(\lambda+(1-\lambda)e^u\bigr)^2}\,du.
\]
For $|u|\leq1$ and $\lambda\leq1/2$, the fraction $\frac{\lambda(1-\lambda)e^u}
{\bigl(\lambda+(1-\lambda)e^u\bigr)^2}$
is $\Omega(\lambda)$.  Restricting the integral to the interval between
$0$ and
$\operatorname{sgn}(\log r)\min\{s|\log r|,1\}$ therefore gives
\[
\frac{d}{ds}\log M_{\lambda,s}(1,r)\geq \frac{1}{s^2}\int_0^{\operatorname{sgn}(\log r)\min\{s|\log r|,1\}}
u\frac{\lambda(1-\lambda)e^u}
{\bigl(\lambda+(1-\lambda)e^u\bigr)^2}\,du
\geq \Omega(\lambda)\frac{\min\{s^2(\log r)^2,1\}}{s^2}
\geq \Omega(\lambda)\min\{(\log r)^2,1\}.
\]
Consequently,
\[
\log\frac{M_{\lambda,1/n}(1,r)}{M_{\lambda,-p}(1,r)}
\geq\int_{-p}^{1/n}\frac{d}{ds}\log M_{\lambda,s}(1,r)\,ds\geq\Omega\!\left(\lambda(1/n+p)\min\{(\log r)^2,1\}\right),
\]
which proves the lemma.
\end{proof}

\begin{proof}[Proof of \Cref{symdiflevelsetslem}]
By monotonicity of $\M$ in $p$, we may assume
$-1/n<p\leq-1/(2n)$.  If
$\mathcal N_n\zeta_{n,p}^{-1/2}\sqrt{\frac{\delta}{\lambda}}\geq2$, the
claimed $L^1$ bound follows from the trivial estimate $2\int\ell$.
Hence assume
$\sqrt{\delta/\lambda}<\mathcal N_n^{-1}\zeta_{n,p}^{1/2}$.
Under this assumption, every term bounded by
$\mathcal N_n\zeta_{n,p}^{-1}(\delta/\lambda)\int\ell$ is absorbed by
$\mathcal N_n\zeta_{n,p}^{-1/2}\sqrt{\delta/\lambda}\int\ell$.
After an affine normalisation, take $m=1$, so
$1\leq|L_t|\leq(2n)^{n+2}$.

We follow a similar strategy to the proof of \Cref{almostconvexlevelsets}. Leveraging the $p$-concavity of  $\ell$ to gain more precise control. As ever, consider $T:[0,\sup \ell]\to [0,\sup g]$, defined by $\int_0^t |L_s|ds=\int_0^{T(t)}|G_s|ds$ and use $H_{\M(t,T(t))}\supset \lambda L_t+(1-\lambda) G_{T(t)}$ to prove Borell-Brascamp-Lieb. Indeed, we find
\begin{equation}\label{intbblproofeq}
\begin{aligned}
  \int h\,dx&=\int_0^{\sup h} |H_s|\,ds\geq \int_0^{\sup h} \left|H_{\M(t,T(t))}\right|\,d\M(t,T(t))\geq \int_0^{\sup \ell} |\lambda L_t+(1-\lambda) G_{T(t)}|\frac{dt}{M_{\lambda,-p}\left(1,\frac{1}{\frac{dT}{dt}(t)}\right)}\\
  &\geq \int_0^{\sup \ell} \frac{|\lambda L_t+(1-\lambda) G_{T(t)}|-M_{\lambda,1/n}(|L_t|,|G_{T(t)}|)}{M_{\lambda,-p}\left(1,\frac{1}{\frac{dT}{dt}(t)}\right)}\,dt + \int_0^{\sup \ell} M_{\lambda,1/n}\left(|L_t|,|G_{T(t)}|\right)\frac{dt}{M_{\lambda,-p}\left(1,\frac{1}{\frac{dT}{dt}(t)}\right)},
\end{aligned}
\end{equation}
where we used \Cref{Holderderivative} to say that $\frac{d\M(t,T(t))}{dt}\geq\frac{1}{M_{\lambda,-p}\left(1,\frac{1}{\frac{dT}{dt}(t)}\right)}$.
For the first term, set
\[
\lambda_t:=\frac{\min\{\lambda|L_t|^{1/n},(1-\lambda)|G_{T(t)}|^{1/n}\}}
{\lambda|L_t|^{1/n}+(1-\lambda)|G_{T(t)}|^{1/n}},
\]
the effective weighing parameter for these level-sets.
Let $C_n\geq1$ be chosen later.  It will remain fixed, and we will ensure
that it has the size allowed for $\mathcal N_n$.  Define the interval on
which the two level-set volumes are comparable by
$$I:=\left\{t\in[0,\sup\ell]:
|G_{T(t)}|\in[C_n^{-1}|L_0|,C_n|L_0|]\right\}.$$

Applying \Cref{BarchiesiJulin} after dilating $L_t$ to have the same
volume as $G_{T(t)}$ gives, for a suitable $v_t\in\mathbb R^n$, the
estimate
\[
|\lambda L_t+(1-\lambda)G_{T(t)}|
\geq M_{\lambda,1/n}(|L_t|,|G_{T(t)}|)
\Bigg(1+\mathcal N_n^{-1}\lambda_t
\frac{\left|G_{T(t)}\triangle
\left(v_t+\left(\frac{|G_{T(t)}|}{|L_t|}\right)^{1/n}L_t\right)\right|^2}
{|G_{T(t)}|^2}\Bigg).
\]
For $t\in I$, the bounds on the level-set volumes give $C_n^{-1}\leq\frac{|G_{T(t)}|}{|L_t|}
\leq C_n(2n)^{n+2}$, and thus $\lambda_t\geq\mathcal N_n^{-1}\lambda$.

For the second term, use
$\frac{dT}{dt}(t)=\frac{|L_t|}{|G_{T(t)}|}$.  Since for all $r>0$, we have
$\min\{(\log r)^2,1\}\geq
\Omega\!\left(\min\{(r^{-1}-1)^2,1\}\right)$, \Cref{quantitativepowermeangaplem}
gives
$$\frac{M_{\lambda,1/n}\left(|L_t|,|G_{T(t)}|\right)}{M_{\lambda,-p}\left(1,\frac{1}{\frac{dT}{dt}(t)}\right)}
\geq |L_t|\left[1+\Omega\!\left(\lambda\frac{\zeta_{n,p}}{n}\right)
\min\left\{1,\left(\frac{|L_t|}{|G_{T(t)}|}-1\right)^2\right\}\right].$$

Putting these estimates into \eqref{intbblproofeq}, we find 
\begin{equation*}
\begin{aligned}
\delta\int\ell\,dx
&=\int h\,dx-\int\ell\,dx\\
&\geq \mathcal N_n^{-1}\lambda\zeta_{n,p}
\int_0^{\sup\ell}|L_t|
\min\left\{1,\left(\frac{|L_t|}{|G_{T(t)}|}-1\right)^2\right\}dt+\mathcal N_n^{-1}\lambda
\int_I |L_t|
\frac{\left|G_{T(t)}\triangle
\left(v_t+\left(\frac{|G_{T(t)}|}{|L_t|}\right)^{1/n}L_t\right)\right|^2}
{|G_{T(t)}|^2}\,dt,\\[2mm]
\end{aligned}
\end{equation*}
Hence,

\begin{equation}\label{intsymdifcontroleq}
\begin{aligned}
\quad\int_0^{\sup\ell}|L_t|
\min\left\{1,\left(\frac{|L_t|}{|G_{T(t)}|}-1\right)^2\right\}dt
&\leq \mathcal N_n\zeta_{n,p}^{-1}\frac{\delta}{\lambda}\int\ell\,dx,\\
\int_I |L_t|
\frac{\left|G_{T(t)}\triangle
\left(v_t+\left(\frac{|G_{T(t)}|}{|L_t|}\right)^{1/n}L_t\right)\right|^2}
{|G_{T(t)}|^2}\,dt
&\leq \mathcal N_n\frac{\delta}{\lambda}\int\ell\,dx.
\end{aligned}
\end{equation}
Since the $G_t$ are nested and $T$ is monotone, $I$ is an interval.
We now choose
\[
C_n=2(2n)^{n+2},
\]
which has the size allowed for $\mathcal N_n$.  If $t\notin I$, then either $\frac{|L_t|}{|G_{T(t)}|}\leq C_n^{-1}$ or $\frac{|L_t|}{|G_{T(t)}|}\geq\frac{C_n}{(2n)^{n+2}}=2.$ Hence, the minimum in the first bound of \eqref{intsymdifcontroleq} is at
least $1/4$ outside $I$, and that bound gives
$$\int_{\mathbb{R}_{\geq0}\setminus I}|L_t|dt
\leq \mathcal N_n\zeta_{n,p}^{-1}
\frac{\delta}{\lambda}\int\ell\,dx,$$
Since
$[0,\min I)\subset\mathbb R_{\geq0}\setminus I$ and
$1\leq|L_t|\leq(2n)^{n+2}$,
\[
\min I\leq
\mathcal N_n\zeta_{n,p}^{-1}\frac{\delta}{\lambda}
\frac{\int\ell\,dx}{|L_0|}.
\]
The definitions of $T$ and $I$ give
\[
\begin{aligned}
T(\min I)C_n^{-1}|L_0|
&\leq T(\min I)|G_{T(\min I)}|\leq\int_0^{T(\min I)}|G_s|\,ds
=\int_0^{\min I}|L_s|\,ds\leq(\min I)|L_0|.
\end{aligned}
\]
Hence $T(\min I)\leq C_n\min I$, and therefore
\[
\max\{\min I,T(\min I)\}
\leq \mathcal N_n\zeta_{n,p}^{-1}\frac{\delta}{\lambda}
\frac{\int\ell\,dx}{|L_0|}
\leq \mathcal N_n\zeta_{n,p}^{-1/2}\sqrt{\frac{\delta}{\lambda}}
\frac{\int\ell\,dx}{|L_0|}.
\]

Inside $I$, \eqref{intsymdifcontroleq} reduces to the simpler
$$\int_I |L_t| \left( \left(\frac{|L_t|}{|G_{T(t)}|}-1\right)^2+\frac{\left|G_{T(t)}\triangle \left(v_t+\left(\frac{|G_{T(t)}|}{|L_t|}\right)^{1/n} L_t\right)\right|^2}{|G_{T(t)}|^2}\right) dt
\leq \mathcal N_n\zeta_{n,p}^{-1}
\frac{\delta}{\lambda}\int\ell\,dx.$$
For each $t\in I$, choose $v'_t$ so that the two homothetic
copies $v'_t+L_t$ and $v_t+\left(\frac{|G_{T(t)}|}{|L_t|}\right)^{1/n}L_t$
are nested.  Their symmetric difference then has measure $|L_t|\left|\frac{|G_{T(t)}|}{|L_t|}-1\right|.$ Moreover, on $I$ we have $C_n^{-1}\leq\frac{|G_{T(t)}|}{|L_t|}\leq C_n(2n)^{n+2}.$
The triangle inequality and \eqref{intsymdifcontroleq} therefore
give
\[
\int_I |L_t|\left[
\left(\frac{|L_t|}{|G_{T(t)}|}-1\right)^2
+\frac{|G_{T(t)}\triangle(v'_t+L_t)|^2}{|L_t|^2}
\right]dt
\leq\mathcal N_n\zeta_{n,p}^{-1}\frac{\delta}{\lambda}\int\ell\,dx.
\]
Cauchy--Schwarz, $\frac{d}{dt}T(t)=\frac{|L_t|}{|G_{T(t)}|}$,   and
$\int_I|L_t|\,dt\leq\int\ell\,dx$ give
\begin{equation}\label{TprimeLonecontrol}
\int_I\left(|L_{t}|\left|\frac{d}{dt}T(t)-1\right|
+\left|G_{T(t)}\triangle \left(v'_t+L_t\right)\right|\right)dt
\leq \mathcal N_n\zeta_{n,p}^{-1/2}
\sqrt{\frac{\delta}{\lambda}}\int\ell\,dx.
\end{equation}
Since $|L_t|\geq1$ and $|L_0|\leq(2n)^{n+2}$, the first term in
\eqref{TprimeLonecontrol} and the bound for $\min I$ imply
\begin{equation}\label{Tminusidentity}
\begin{aligned}
\sup_{t\in I}|T(t)-t|
&\leq |T(\min I)-\min I|
+\int_I\left|\frac{d}{ds}T(s)-1\right|\,ds\leq\max\{\min I,T(\min I)\}
+\int_I|L_s|\left|\frac{d}{ds}T(s)-1\right|\,ds\\
&\leq\mathcal N_n\zeta_{n,p}^{-1/2}\sqrt{\frac{\delta}{\lambda}}
\left(\frac{\int\ell\,dx}{|L_0|}+\int\ell\,dx\right)\leq\mathcal N_n\zeta_{n,p}^{-1/2}\sqrt{\frac{\delta}{\lambda}}\frac{\int\ell\,dx}{|L_0|}.
\end{aligned}
\end{equation}
We next replace $G_{T(t)}$ by $G_t$ without losing another power of
$\zeta_{n,p}$.  Put $I':=I\cap T(I)$.  By \eqref{Tminusidentity}, the
corresponding endpoints of $I$ and $T(I)$ are at distance at most
$\mathcal N_n\zeta_{n,p}^{-1/2}\sqrt{\frac{\delta}{\lambda}}
\frac{\int\ell\,dx}{|L_0|}$.  On
$I\setminus I'$, the sets $G_{T(t)}$ have measure at most
$C_n|L_0|$, so
$$\int_{I\setminus I'} |G_{T(t)}|dt\leq C_n|L_0|\cdot \mathcal N_n\zeta_{n,p}^{-1/2}\sqrt{\frac{\delta}{\lambda}}\frac{\int\ell\,dx}{|L_0|}\leq \mathcal N_n\zeta_{n,p}^{-1/2}\sqrt{\frac{\delta}{\lambda}}\int\ell\,dx.$$
On the other hand, by the definition of $T$,
\[
\int_{I\setminus I'}|G_t|\,dt
\leq\int_{\mathbb R_{\geq0}\setminus T(I)}|G_t|\,dt
=\int_{\mathbb R_{\geq0}\setminus I}|L_t|\,dt\leq\mathcal N_n\zeta_{n,p}^{-1}
\frac{\delta}{\lambda}\int\ell\,dx.
\]
Combining these, we find
\[
\int_{I\setminus I'}|G_{T(t)}\triangle G_t|\,dt\leq \int_{I\setminus I'}|G_{T(t)}|+| G_t|\,dt
\leq\mathcal N_n\zeta_{n,p}^{-1/2}\sqrt{\frac{\delta}{\lambda}}\int\ell\,dx.
\]
For $s\in T(I)$, \eqref{Tminusidentity} gives
$|T^{-1}(s)-s|\leq\mathcal N_n\zeta_{n,p}^{-1/2}
\sqrt{\frac{\delta}{\lambda}}\frac{\int\ell\,dx}{|L_0|}$.
Since $G_{T(t)}$ is the $t$-level set of
$T^{-1}\circ g$, layer cake on $g^{-1}(T(I))$, together with
$|G_{\min T(I)}|=|G_{T(\min I)}|\leq C_n|L_0|$, yields
\[
\begin{aligned}
\int_{I'}|G_{T(t)}\triangle G_t|\,dt
&\leq\int_{g^{-1}(T(I))}|T^{-1}(g(x))-g(x)|\,dx\leq\sup_{s\in T(I)}|T^{-1}(s)-s|\,|G_{\min T(I)}|\\
&\leq\mathcal N_n\zeta_{n,p}^{-1/2}\sqrt{\frac{\delta}{\lambda}}
\frac{\int\ell\,dx}{|L_0|}\,C_n|L_0|\leq\mathcal N_n\zeta_{n,p}^{-1/2}\sqrt{\frac{\delta}{\lambda}}\int\ell\,dx.
\end{aligned}
\]
Combining these two estimates with the second term in
\eqref{TprimeLonecontrol}, we obtain
\[
\int_I|G_t\triangle(v'_t+L_t)|\,dt\leq \int_{I}|(v'_t+L_t)\triangle G_{T(t)}|\,dt+ \int_{I}|G_{T(t)}\triangle G_t|\,dt
\leq\mathcal N_n\zeta_{n,p}^{-1/2}\sqrt{\frac{\delta}{\lambda}}\int\ell\,dx,
\]
and hence
\[
\left|\int_I|G_t|\,dt-\int_I|L_t|\,dt\right|
\leq\mathcal N_n\zeta_{n,p}^{-1/2}\sqrt{\frac{\delta}{\lambda}}\int\ell\,dx.
\]
The mass outside $I$ is therefore bounded by
\begin{align*}\int_{\mathbb{R}_{\geq0}\setminus I}\left|G_{t}\triangle L_t\right|dt
&\leq\int_{\mathbb{R}_{\geq0}\setminus I}\left|G_{t}\right|+\left| L_t\right|dt\leq 2\int_{\mathbb{R}_{\geq0}\setminus I}| L_t|dt+\left|\int_I\left|G_{t}\right|dt -\int_I \left|L_t\right|dt\right|\\
&\leq\mathcal N_n\zeta_{n,p}^{-1}\frac{\delta}{\lambda}\int\ell\,dx
+\mathcal N_n\zeta_{n,p}^{-1/2}\sqrt{\frac{\delta}{\lambda}}\int\ell\,dx\leq\mathcal N_n\zeta_{n,p}^{-1/2}\sqrt{\frac{\delta}{\lambda}}\int\ell\,dx.
\end{align*}
Setting $v_t'=0$ for $t\notin I$, we can then conclude
$$\int_0^\infty\left|G_{t}\triangle \left(v'_t+L_t\right)\right|dt
\leq \mathcal N_n\zeta_{n,p}^{-1/2}
\sqrt{\frac{\delta}{\lambda}}\int\ell\,dx.$$
\end{proof}

\begin{proof}[Proof of \Cref{translationsincones}]
    Note that in this setup $\mathcal{N}_n^{-1}\leq |K|\leq \mathcal{N}_n$.
We first prove a geometric claim.
\begin{clm} \label{quantitativeconedisplacement}
If $-w\in C_i$,
then
\[
|K\triangle(K-w)|
\leq\mathcal N_n
|K\cap((C_i+w)\setminus C_i)|.
\]
\end{clm}

\begin{proof}[Proof of the claim]
    Since $B_2^n\subset K$ and the $C_i$ are regularly placed we have
\[
|K\cap((C_i+w)\setminus C_i)|\geq |B_2^n\cap((C_i+w)\setminus C_i)|
\geq\mathcal N_n^{-1}\min\{|w|,1\}.
\]

For $|w|\leq1$, slicing $K$ along lines parallel to $w$ we find
\[
|K\triangle(K-w)|\leq2|w||\pi_{w^\perp}K|
\leq \mathcal{N}_n |w|.
\]
  For $|w|>1$, use
$|K\triangle(K-w)|\leq2|K|$.  Combining these bounds proves the claim.
\end{proof}

The small symmetric difference in the cones implies near equal mass in the cones;
\[
\big||K\cap C_i|-|X\cap C_i|\big|\leq\big|(K\cap C_i)\triangle(v_i+(X\cap C_i))\big|\leq\alpha|K|.
\]
On the other hand, the global small symmetric difference implies near equal mass in the cones after translation;
\[
\big||(K-v)\cap C_i|-|X\cap C_i|\big|\leq \big|((K-v)\triangle X)\cap C_i|\big|\leq \big|(K-v)\triangle X|\big|\leq\alpha|K|.
\]
Hence,
\begin{equation*}
\big||(K-v)\cap C_i|-|K\cap C_i|\big|
\leq2\alpha|K|.
\end{equation*}
Choose $i$ so that $-v\in C_i$.  Since $C_i\subset C_i+v$,
\[
|K\cap((C_i+v)\setminus C_i)|
=|(K-v)\cap C_i|-|K\cap C_i|
\leq2\alpha|K|,
\]
and \Cref{quantitativeconedisplacement} gives
$|K\triangle(K-v)|\leq\mathcal N_n\alpha|K|.$
Finally,
\[
|K\triangle X|
\leq|K\triangle(K-v)|+|(K-v)\triangle X|
=|K\triangle(K-v)|+|K\triangle(v+X)|
\leq\mathcal N_n\alpha|K|.
\]
This proves the lemma.
\end{proof}

\subsection{Adjacent blocks have similar translates: Proof of \Cref{adjacenttranslatesarecloseprop}}
We prove a lemma that says that if two blocks of level sets overlap sufficiently, then the translates that generate small symmetric difference must be close together.
\begin{prop}\label{adjacenttranslatesarecloseprop}
Let $\ell,g\colon \mathbb{R}^n\rightarrow \mathbb{R}_{\geq0}$ be integrable functions with $\ell$ $p$-concave, $0< a_0< a_1<b_0<b_1\leq \sup \ell$, $b'_0,b'_1>0$, and $K_1\subset K_0\subset\mathbb{R}^n$ convex sets.
Let 
$$\ell_0:=\min\{b_0, \ell\cdot \textbf{1}_{L_{a_0}}\}, \qquad g_0:=\min\{b'_0, g\cdot \textbf{1}_{K_0}\},$$
and
$$\ell_1:=\min\{b_1, \ell\cdot \textbf{1}_{L_{a_1}}\}, \qquad g_1:=\min\{b'_1, g\cdot \textbf{1}_{K_1}\}.$$
Assume that $|L_{a_i}|\leq 1.5 |L_{b_i}|$ for $i=0,1$.
Let $v_0,v_1\in\mathbb{R}^n$.
Then
$$\int_{a_1}^{b_0} |L_t \triangle (L_t+(v_0-v_1))|dt
\leq 4\left(
\int\left|\ell_0(x)-g_0(x+v_0)\right|\,dx
+\int\left|\ell_1(x)-g_1(x+v_1)\right|\,dx\right).$$
Moreover, if
$$\int\left|\ell_0(x)-g_0(x+v_0)\right|\,dx
+\int\left|\ell_1(x)-g_1(x+v_1)\right|\,dx
\leq \frac14\int_{a_1}^{b_0}|L_t|\,dt,$$
then
\[
v_0-v_1\in 3
\frac{\displaystyle
\int\left|\ell_0(x)-g_0(x+v_0)\right|\,dx
+\int\left|\ell_1(x)-g_1(x+v_1)\right|\,dx}
{\displaystyle\int_{a_1}^{b_0}|L_t|\,dt}
\bigl(L_{b_0}-L_{b_0}\bigr).
\]

\end{prop}

\subsubsection{Lemma for \Cref{adjacenttranslatesarecloseprop}}

We first record translation estimates in the difference-body gauge.
\begin{defn}
Given a convex body $B$, write
\[
\|w\|_{B-B}:=\inf\{r>0:w\in r(B-B)\}.
\]
\end{defn}
\begin{lem}\label{translationestimateslem}
Let $A\subset B\subset\mathbb{R}^n$ be convex bodies.  Then, for every
$w\in\mathbb{R}^n$,
\[
2\min\{1,\|w\|_{B-B}\}|A|
\leq |A\triangle(A+w)|
\leq2n\min\{1,\|w\|_{B-B}\}|B|.
\]
Moreover,
$$|A\triangle(A+w)|\leq |B\triangle(B+w)|.$$
\end{lem}

\begin{proof}[Proof of \Cref{translationestimateslem}]
The result is immediate when $w=0$, so write $w=su$ with
$u\in S^{n-1}$ and $s>0$.  For $y\in u^\perp$, let
$\ell_A(y)$ and $\ell_B(y)$ be the lengths of the corresponding sections
of $A$ and $B$.  Translating such an interval by $s$ changes it on a set
of length $2\min\{s,\ell_A(y)\}$.  Since $A\subset B$, we have
$\ell_A(y)\leq\ell_B(y)$, and integration gives
\[
|A\triangle(A+w)|\leq|B\triangle(B+w)|.
\]

Put $M:=\max_{y\in u^\perp}\ell_B(y)$.  A segment of length $r$ in
$B$ parallel to $u$ gives $ru\in B-B$, and conversely.  Hence
\[
\|w\|_{B-B}=\frac{s}{M}.
\]
Since $\ell_A(y)\leq M$,
\[
\min\{s,\ell_A(y)\}
\geq\min\left\{1,\frac{s}{M}\right\}\ell_A(y).
\]
Integrating proves the lower bound.

For the upper bound, let $D:=\pi_{u^\perp}B$.  The function $\ell_B$ is
concave on $D$.  If $\ell_B(y_0)=M$, then for $0\leq r\leq M$,
\[
\big|\{y\in D:\ell_B(y)\geq r\}\big|
\geq\left(1-\frac rM\right)^{n-1}|D|.
\]
Layer cake therefore gives
\[
|B|=\int_D\ell_B(y)\,dy
\geq\frac{M}{n}|D|.
\]
If $s\leq M$, it follows that
\[
|A\triangle(A+w)|
\leq2s|D|
\leq2n\frac{s}{M}|B|.
\]
If $s>M$, the trivial bound
$|A\triangle(A+w)|\leq2|A|\leq2|B|$ proves the result.
\end{proof}

\begin{lem}\label{integratedlowtranslationlem}
Let $-1/n<p\leq-1/(2n)$, let $\ell$ be an integrable $p$-concave
function.  Suppose that $0<s\leq(\sup\ell)/2$ and that $L_s$ is a
convex body.  If $w\in\alpha(L_s-L_s)$ for some $\alpha\geq0$, then
\[
\int_0^s|L_t\triangle(L_t+w)|\,dt
\leq \mathcal N_n\alpha s|L_s|.
\]
\end{lem}

\begin{proof}
After rescaling the values of $\ell$ and translating its domain, assume
$\sup\ell=\ell(0)=1$.  By John's theorem, a linear change of coordinates
gives, for some $x_s\in\mathbb R^n$,
\[
x_s+B_2^n\subset L_s\subset x_s+nB_2^n.
\]
Let
$\pi_{w^\perp}\colon\mathbb R^n\to w^\perp$ be the orthogonal
projection.  The John inclusions give
\[
\mathcal N_n^{-1}\leq|L_s|\leq\mathcal N_n,\qquad
L_s-L_s\subset2nB_2^n,\qquad
|\pi_{w^\perp}L_s|\leq\mathcal N_n.
\]
In particular, $|w|\leq\mathcal N_n\alpha$.  Slicing the convex set
$L_t$ by lines parallel to $w$ gives $|L_t\triangle(L_t+w)|\leq2|w|\,|\pi_{w^\perp}L_t|.$

For $0<t\leq s$, $p$-concavity gives
$\frac{s^p-1}{t^p-1}L_t\subset L_s$.  Since $s\leq1/2$,
\[
\begin{aligned}
|\pi_{w^\perp}L_t|
&\leq \left(\frac{t^p-1}{s^p-1}\right)^{n-1}
|\pi_{w^\perp}L_s|
\leq \mathcal N_n\left(\frac ts\right)^{(n-1)p}.
\end{aligned}
\]
Indeed,
\[
\frac{t^p-1}{s^p-1}
=\left(\frac ts\right)^p\frac{1-t^{-p}}{1-s^{-p}}
\leq\left(\frac ts\right)^p\frac{1}{1-2^p}
\leq O(n)\left(\frac ts\right)^p,
\]
where we used $s\leq1/2$ and
$1-2^p\geq1-2^{-1/(2n)}=\Omega(1/n)$.
Finally,
\[
\int_0^s\left(\frac ts\right)^{(n-1)p}dt
=\frac{s}{1+(n-1)p}\leq ns.
\]
Combining these estimates gives
\[
\begin{aligned}
\int_0^s|L_t\triangle(L_t+w)|\,dt
&\leq2|w|\int_0^s|\pi_{w^\perp}L_t|\,dt\leq\mathcal N_n\alpha
\int_0^s\left(\frac ts\right)^{(n-1)p}dt\leq\mathcal N_n\alpha s.
\end{aligned}
\]
\end{proof}

\subsubsection{Proof of \Cref{adjacenttranslatesarecloseprop}}
With this lemma in place, we're ready to prove the proposition.
\begin{proof}[Proof of \Cref{adjacenttranslatesarecloseprop}]
For $i=0,1$, denote by $L_{i,t}$ and $G_{i,t}$ the $t$-level sets of $\ell_i$ and $g_i$, so that we have
\begin{equation*}
\int_0^\infty |L_{i,t}\triangle (G_{i,t}-v_i)|dt
=\int\left|\ell_i(x)-g_i(x+v_i)\right|\,dx.
\end{equation*}
For $t\in [a_1,b_0]$, the two capped functions have the same $\ell$-level set:
$$L_{0,t}=L_{1,t}=L_t.$$

We first establish that $v_1-v_0$ is a small translate with respect to a typical $L_t$ with $t\in [a_1,b_0]$ which is the first conclusion of the proposition.

\begin{clm}\label{overlapcontrolsdisplacementeq} We have that
\begin{equation*}
    \int_{a_1}^{b_0} |L_t\triangle (L_t+v_1-v_0)|dt
    \leq 4\left(
    \int\left|\ell_0(x)-g_0(x+v_0)\right|\,dx
    +\int\left|\ell_1(x)-g_1(x+v_1)\right|\,dx\right).
\end{equation*}
\end{clm}
\begin{proof}[Proof of Claim]
Suppose first that $t\in[a_1,b_0]$ and $t<b'_0$.  If
$t<b'_1$, then
$G_{1,t}=G_t\cap K_1\subset G_t\cap K_0=G_{0,t}$; if $t\geq b'_1$,
then $G_{1,t}=\emptyset$, so the same inclusion holds.  Consequently,
\begin{align*}
|L_t\triangle (L_t+v_1-v_0)|
&= |(L_t+v_0)\triangle (L_t+v_1)|\\
&\leq |(L_t+v_0)\triangle G_{0,t}|
+ |G_{0,t}\triangle G_{1,t}|
+|G_{1,t}\triangle (L_t+v_1)|\\
&=|(L_t+v_0)\triangle G_{0,t}|
+ |G_{0,t}|-|G_{1,t}|
+|(L_t+v_1)\triangle G_{1,t}|.
\end{align*}
Moreover,
\[
\begin{aligned}
|G_{0,t}|-|G_{1,t}|
&\leq\big||G_{0,t}|-|L_t|\big|
+\big||L_t|-|G_{1,t}|\big|\leq |(L_t+v_0)\triangle G_{0,t}|
+|(L_t+v_1)\triangle G_{1,t}|,
\end{aligned}
\]
so that
\[
|L_t\triangle (L_t+v_1-v_0)|
\leq2|(L_t+v_0)\triangle G_{0,t}|
+2|(L_t+v_1)\triangle G_{1,t}|.
\]
Since $L_{i,t}=L_t$ on $[a_1,b_0]$, we have
\[
\int_{[a_1,b_0]\cap(0,b'_0)}
|L_t\triangle(L_t+v_1-v_0)|\,dt\quad\leq
2\int\left|\ell_0(x)-g_0(x+v_0)\right|\,dx
+2\int\left|\ell_1(x)-g_1(x+v_1)\right|\,dx.
\]
For $t\in[a_1,b_0]$ with $t\geq b'_0$, we have $G_{0,t}=\emptyset$, so that 
\[
\begin{aligned}
|L_t\triangle(L_t+v_1-v_0)|
&\leq2|L_t|=2|L_t\triangle(G_{0,t}-v_0)|.
\end{aligned}
\]
and thus
\[
\int_{[a_1,b_0]\cap[b'_0,\infty)}
|L_t\triangle(L_t+v_1-v_0)|\,dt
\leq2\int\left|\ell_0(x)-g_0(x+v_0)\right|\,dx.
\]
Adding the last two bounds gives
\[
\begin{aligned}
\int_{a_1}^{b_0}|L_t\triangle(L_t+v_1-v_0)|\,dt
&\leq4\int\left|\ell_0(x)-g_0(x+v_0)\right|\,dx+2\int\left|\ell_1(x)-g_1(x+v_1)\right|\,dx\\
\end{aligned}
\]
which proves the claim.
\end{proof}

For the second conclusion, on the interval $t\in[b_0,a_1]$, we have
\[
L_{b_0}\subset L_t\subset L_{a_1},
\qquad \text{ and }\qquad
|L_{a_1}|\leq |L_{a_0}|\leq1.5|L_{b_0}|.
\]
Consequently, by 
\Cref{translationestimateslem},
\[
\begin{aligned}
\int_{a_1}^{b_0}|L_t\triangle(L_t+v_1-v_0)|\,dt
&\geq (b_0-a_1)
|L_{b_0}\triangle(L_{b_0}+v_1-v_0)|\\
&\geq 2(b_0-a_1)
\min\{1,\|v_1-v_0\|_{L_{b_0}-L_{b_0}}\}
|L_{b_0}|\\
&\geq\frac43
\min\{1,\|v_1-v_0\|_{L_{b_0}-L_{b_0}}\}
\int_{a_1}^{b_0}|L_t|\,dt.
\end{aligned}
\]
Here the last inequality follows from $\int_{a_1}^{b_0}|L_t|\,dt
\leq(b_0-a_1)|L_{a_1}|
\leq\frac32(b_0-a_1)|L_{b_0}|.$ 
Comgining this bound with \Cref{overlapcontrolsdisplacementeq} gives
\[
\begin{aligned}
\frac43
\min\{1,\|v_1-v_0\|_{L_{b_0}-L_{b_0}}\}
\int_{a_1}^{b_0}|L_t|\,dt&\leq
\int_{a_1}^{b_0}|L_t\triangle(L_t+v_1-v_0)|\,dt\\
&\leq4\left(
\int\left|\ell_0(x)-g_0(x+v_0)\right|\,dx
+\int\left|\ell_1(x)-g_1(x+v_1)\right|\,dx\right)\leq\int_{a_1}^{b_0}|L_t|\,dt.
\end{aligned}
\]
where the last inequality is the  hypothesis in the proposition.
Thus
$\min\{1,\|v_1-v_0\|_{L_{b_0}-L_{b_0}}\}\leq3/4<1$.
The minimum therefore equals
$\|v_1-v_0\|_{L_{b_0}-L_{b_0}}$, and the first two inequalities give
\[
v_0-v_1\in 3
\frac{\displaystyle
\int\left|\ell_0(x)-g_0(x+v_0)\right|\,dx
+\int\left|\ell_1(x)-g_1(x+v_1)\right|\,dx}
{\displaystyle\int_{a_1}^{b_0}|L_t|\,dt}
\bigl(L_{b_0}-L_{b_0}\bigr),
\]
which proves the result.

\end{proof}

\subsection{Conclusion: Proof of \Cref{sharplambda}}

To prove \Cref{sharplambda}, we first prove two short lemmas about
$p$-concave functions.  They show that the restriction of $\ell$ between
level sets $L_{s_0}$ and $L_{s_1}$ satisfying
$|L_{s_0}|/|L_{s_1}|=1.1$ has rapidly decreasing mass both as $s_0\to0$
and as $s_0\to\sup \ell$.

\begin{lem}\label{positiveproportionofremainingheight}
Let $-1/n<p\leq-1/(2n)$, and let
$\ell\colon \mathbb{R}^n\to\mathbb{R}_{\geq 0}$ be an integrable
$p$-concave function with positive integral.
If
$$0.1\sup\ell\leq t<s\leq \sup \ell
\qquad\text{and}\qquad
|L_t|\geq 1.1|L_s|,$$
then
\[
 s-t\geq \mathcal N_n^{-1}(\sup\ell-t).
\]
\end{lem}

\begin{proof}[Proof of \Cref{positiveproportionofremainingheight}]
By scaling $\ell$, we may assume that $\sup\ell=1$, and choose
$z\in\mathbb{R}^n$ such that $\ell(z)=1$. The lemma is immediate if
$s=1$, so suppose that $s<1$ and set
\[
\theta:=\frac{s^p-1}{t^p-1}\in(0,1),
\]
so that $s^p=\theta t^p+(1-\theta)$.
By $p$-concavity, we have $\theta L_t+(1-\theta)z\subseteq L_s$.
As $|L_s|>0$, it follows that
\[
|L_s|\geq\theta^n|L_t|
\geq1.1\theta^n|L_s|,
\]
and hence $\theta\leq1.1^{-1/n}$ and thus $1-\theta\geq1-1.1^{-1/n}=\Omega(1/n)$. 
Since $-1<p<0$, for every $x\in[1/10,1]$ we have
$|p|\leq |px^{p-1}|\leq100|p|.$ Hence, we can write
\[
100|p|(s-t)
\geq \int_s^t |p|x^{p-1}dx=t^p-s^p
=(1-\theta)(t^p-1)=(1-\theta)\int_t^1 |p|x^{p-1}dx
\geq\Omega(1/n)|p|(1-t).
\]
Cancelling $|p|$ gives
$s-t\geq\Omega(1/n)(1-t)\geq\mathcal N_n^{-1}(1-t)$.
\end{proof}

\begin{lem}\label{positiveproportionmassdownward}
Let $-1/n<p\leq-1/(2n)$, and let
$\ell\colon\mathbb{R}^n\to\mathbb{R}_{\geq 0}$ be an integrable
$p$-concave function with positive integral.
For every $0<s\leq0.9\sup\ell$,
\[
\int_0^s |L_r|\,dr
\leq\mathcal N_n\zeta_{n,p}^{-1}s|L_s|.
\]
Moreover, if
\[
0\leq t<s\leq 0.9\sup\ell
\qquad\text{and}\qquad
|L_t|\geq 1.1|L_s|,
\]
then
\[
\int_t^s |L_r|\,dr
\geq\mathcal N_n^{-1}\zeta_{n,p}\int_0^s|L_r|\,dr.
\]
\end{lem}

\begin{proof}[Proof of \Cref{positiveproportionmassdownward}]
Renormalising and translating, we may assume $\sup\ell=\ell(0)=1$.
For $0<r\leq s$, $p$-concavity gives
$\frac{s^p-1}{r^p-1}L_r
\subseteq L_s.$
Consequently, we have 
\[|L_r|
\leq
\left(\frac{r^p-1}{s^p-1}\right)^n|L_s|
\leq
\frac{1}{(1-0.9^{-p})^n}
\left(\frac{r}{s}\right)^{np}|L_s|.
\]
Since $-p\geq1/(2n)$, we have $
1-0.9^{-p}\geq1-0.9^{1/(2n)}=\Omega(1/n)$, and thus
$\frac{1}{(1-0.9^{-p})^n}\leq\mathcal N_n.$
Since $np>-1$, it follows that
\[
\int_0^s |L_r|\,dr\leq \mathcal{N}_n |L_s| \int_0^s \left(\frac{r}{s}\right)^{np}\,dr 
\leq\mathcal N_n|L_s|\frac{s}{1+np}
=\mathcal N_n\zeta_{n,p}^{-1}s|L_s|.
\]

This proves the first conclusion.  For the second conclusion, the
result is immediate if $t=0$, so we may assume that $t>0$ and $|L_s|>0$.
Applying the same inclusion with $r=t$ and using the hypothesis gives
\[
1\geq
1.1\left(\frac{s^p-1}{t^p-1}\right)^n,
\]
and hence
\[
\left(\frac{s}{t}\right)^{-p}
=\left(\frac{t}{s}\right)^p
\geq1+\bigl(1.1^{1/n}-1\bigr)(1-s^{-p})
\geq1+\Omega(1/n^2).
\]
Indeed, the two factors satisfy
$1-s^{-p}\geq1-0.9^{1/(2n)}=\Omega(1/n)$ and
$1.1^{1/n}-1=\Omega(1/n)$.  Since $-p\leq1/n$,
$s/t\geq(1+\Omega(1/n^2))^{-p}\geq 1+\Omega(1/n)$ and
$s-t\geq \Omega(1/n)s$. Therefore, by
monotonicity of the level sets,
\[
\int_t^s|L_r|\,dr
\geq \Omega(1/n)s|L_s|
\geq\mathcal N_n^{-1}\zeta_{n,p}\int_0^s|L_r|\,dr.
\]
\end{proof}

\subsubsection{Proof of \Cref{sharplambda}}
With all the propositions in place, we are ready to prove \Cref{sharplambda}. The plan is to implement the following strategy.
We choose levels $a_i\in [0,\sup\ell]$ for $i\in\mathbb{Z}$ and apply \Cref{intervaltranslateprop} to
overlapping truncations of $\ell$ defined as $\ell_i:=\min\{a_{i+2}, \ell\cdot\textbf{1}_{L_{a_i}}\}$, obtaining a translation $v_i$ for each
truncation.  We discard a sufficiently low tail, whose mass is already small, and compare consecutive translations using
\Cref{adjacenttranslatesarecloseprop}. Set $a_0=\frac12 \sup \ell$. Below $a_0$, the accumulated cost of
these comparisons is controlled by \Cref{integratedlowtranslationlem}; above
$a_0$, it is summable because the level-set volumes decay geometrically.
This allows us to replace every $v_i$ by $v_0$ and sum the local estimates.

\begin{proof}[Proof of \Cref{sharplambda}]
    Let $C_n\geq1$ be chosen later to make \Cref{Cndefeq} true.  It remains fixed throughout
    the proof, and we will ensure that it has the size allowed for
    $\mathcal N_n$. 
    By monotonicity of $\M(a,b)$ in $p$, we may assume
\[
-\frac1n<p\leq-\frac1{2n}.
\]
If
$C_n\zeta_{n,p}^{-1}\sqrt{\delta/\lambda}\geq1$, the trivial
bound $2\int\ell\,dx$ gives the desired conclusion, so we may
assume
\begin{equation}\label{nontrivialregime}
\sqrt{\frac{\delta}{\lambda}}<C_n^{-1}\zeta_{n,p}.
\end{equation}
In particular,
$\delta/\lambda\leq C_n^{-2}\zeta_{n,p}^2$.
Removing a negligible amount of $\ell$, we may assume $|L_{\sup\ell}|=0$.
Choose a sufficiently small universal constant $c_0>0$ to be chosen later and set for $i\leq 0$;
\[
a_0=\frac12\sup\ell,
\qquad \text{ and }\qquad a_i=a_0\exp(c_0i/n).
\]
Since $\ell^p$ is convex and $|L_{\sup\ell}|=0$, the function
$t\mapsto|L_t|$ is continuous on $[a_0,\sup\ell]$.  For $i\geq1$, choose
$a_i$ so that
\[
|L_{a_i}|=1.1^{-i}|L_{a_0}|.
\]
Then $a_i$ tends to $\sup \ell$ as $i\to\infty$.

For every $i$, set
\[
\ell_i:=\min\{a_{i+2},\ell\cdot\textbf{1}_{L_{a_i}}\}.
\]
We have $\int_0^{a_i}|L_t|\,dt\to 0$ as
$i\to-\infty$ and $\int_0^{a_i}|L_t|\,dt\to\int\ell\,dx$ as $i\to\infty$.
Hence, let $k$ be maximal so that
\begin{equation}\label{quantitativecutoff}
\int_0^{a_{k+2}}|L_t|\,dt
\leq C_n\zeta_{n,p}^{-1}
\sqrt{\frac{\delta}{\lambda}}\int\ell\,dx,
\end{equation}
which reflects the small values of $\ell$ that contain little mass.

By \Cref{FindCorrespondingg_i}, for all $i\geq\min\{k,-1\}$ find the
$K_i$ and $b'_i$ and set
$g_i:=\min\{b'_i,g\cdot\textbf{1}_{K_i}\}$.
Then
\[
\int\ell_i\,dx=\int g_i\,dx,
\qquad
\int\M^*(\ell_i,g_i)\,dx
\leq\left(1+\delta\frac{\int\ell\,dx}{\int\ell_i\,dx}\right)
\int\ell_i\,dx.
\]

Because $\ell$ is $p$-concave, the level sets $L_t$ are well behaved, in the sense of the following claim.

\begin{clm}\label{levelgeometry}
The following estimates hold.
\begin{enumerate}[(i)]
\item\label{leveltwostepconclusion} For every $i$,
\begin{equation*}
\frac{|L_{a_i}|}{|L_{a_{i+2}}|}\leq\frac32.
\end{equation*}
\item\label{levelbandpointwiseconclusion} For every $i\leq0$,
\begin{equation*}
\mathcal N_n^{-1}a_i|L_{a_i}|
\leq \int_{a_i}^{a_{i+1}}|L_t|\,dt
\leq\mathcal N_na_i|L_{a_i}|.
\end{equation*}
\item\label{levelbandlowerconclusion} For every $i\leq0$,
\begin{equation*}
\mathcal N_n^{-1}\int_{a_{i+1}}^{a_{i+2}}|L_t|\,dt
\leq\int_{a_i}^{a_{i+1}}|L_t|\,dt.
\end{equation*}
\item\label{levelbandupperconclusion} For every $i\leq-1$,
\begin{equation*}
\int_{a_i}^{a_{i+1}}|L_t|\,dt
\leq\mathcal N_n\int_{a_{i+1}}^{a_{i+2}}|L_t|\,dt.
\end{equation*}
\item\label{levelbandcumulativeconclusion} For every $i\leq0$,
\begin{equation*}
\int_{a_i}^{a_{i+1}}|L_t|\,dt
\geq\mathcal N_n^{-1}\zeta_{n,p}
\int_0^{a_{i+1}}|L_t|\,dt.
\end{equation*}
\item\label{adjacentlevelmassconclusion} For every $j\leq-1$,
\begin{equation*}
\int\ell_j\,dx+\int\ell_{j+1}\,dx
\leq\mathcal N_n\int_{a_{j+1}}^{a_{j+2}}|L_t|\,dt.
\end{equation*}
\end{enumerate}
\end{clm}

\begin{proof}[Proof of the claim]
The function $\ell^p$, set to $+\infty$ off the support of $\ell$,
is convex.  Thus, for $u_0,u_1\geq(\sup\ell)^p$ and
$\theta\in[0,1]$,
\[
(1-\theta)\{x:\ell(x)^p<u_0\}
+\theta\{x:\ell(x)^p<u_1\}
\subseteq
\{x:\ell(x)^p<(1-\theta)u_0+\theta u_1\}.
\]
Brunn--Minkowski therefore implies that $u\mapsto|\{x:\ell(x)^p<u\}|^{1/n}$ is concave.  Now let $0<t<s<\sup\ell$ and let $\theta\in[0,1]$ be so that $s^p=(1-\theta)(\sup\ell)^p+\theta t^p,$ i.e. $\theta = \frac{s^p-(\sup\ell)^p}{t^p-(\sup\ell)^p}
$. Then the concavity gives
\[
|L_s|^{1/n}\geq(1-\theta)|L_{\sup\ell}|^{1/n}+\theta|L_t|^{1/n}
=\theta|L_t|^{1/n}.
\]
Equivalently,
\begin{equation}\label{levelvolumeratio}
\frac{|L_t|}{|L_s|}
\leq\left(\frac{t^p-(\sup\ell)^p}
{s^p-(\sup\ell)^p}\right)^n.
\end{equation}
If $i\leq-2$, take $s=a_{i+2}\leq(\sup\ell)/2$ and
$t=a_i=s\exp(-2c_0/n)$.  Then
\[
\frac{t^p-(\sup\ell)^p}{s^p-(\sup\ell)^p}
=1+\frac{\exp(-2c_0p/n)-1}
{1-(s/\sup\ell)^{-p}}
\leq1+O(c_0/n),
\]
because $\exp(-2c_0p/n)-1\leq O(c_0/n^2)$ and
$1-2^p\geq \Omega(1/n)$.  Choosing $c_0$ small makes the $n$th power at most
$3/2$, and hence \eqref{levelvolumeratio} gives
(\ref{leveltwostepconclusion})
for $i\leq-2$:
\[
\frac{|L_{a_i}|}{|L_{a_{i+2}}|}
\leq\left(1+O(c_0/n)\right)^n\leq\frac32.
\]
The same calculation with $t=a_{-1}$ and $s=a_0$ gives, after shrinking
$c_0$ if necessary,
$|L_{a_{-1}}|/|L_{a_0}|\leq(3/2)/1.1$.  Since
$|L_{a_0}|/|L_{a_1}|=1.1$, this proves the desired ratio for $i=-1$.
For $i\geq0$, the ratio equals $1.1^2<3/2$ by definition.

For $i\leq-1$, monotonicity of the level-set volumes and
$a_{i+1}=a_i\exp(c_0/n)$ give
\[
(\exp(c_0/n)-1)a_i|L_{a_{i+1}}|
\leq\int_{a_i}^{a_{i+1}}|L_t|\,dt
\leq(\exp(c_0/n)-1)a_i|L_{a_i}|.
\]
Together with \eqref{levelvolumeratio} and
$(\exp(c_0/n)-1)^{-1}=O(n)$, these inequalities give, for $i\leq-1$,
\[
\mathcal N_n^{-1}a_i|L_{a_i}|\leq(\exp(c_0/n)-1)a_i|L_{a_{i+1}}|\leq\int_{a_i}^{a_{i+1}}|L_t|\,dt\leq(\exp(c_0/n)-1)a_i|L_{a_i}|\leq\mathcal N_na_i|L_{a_i}|
\]
proving (\ref{levelbandpointwiseconclusion}).  For $i\leq-2$, we have
\[
\mathcal N_n^{-1}\int_{a_{i+1}}^{a_{i+2}}|L_t|\,dt\leq\mathcal N_n^{-1}(\exp(c_0/n)-1)a_{i+1}|L_{a_{i+1}}|\leq(\exp(c_0/n)-1)a_i|L_{a_{i+1}}|\leq\int_{a_i}^{a_{i+1}}|L_t|\,dt
\]
and
\[
\int_{a_i}^{a_{i+1}}|L_t|\,dt\leq(\exp(c_0/n)-1)a_i|L_{a_i}|\leq\mathcal N_n(\exp(c_0/n)-1)a_{i+1}|L_{a_{i+2}}|\leq\mathcal N_n\int_{a_{i+1}}^{a_{i+2}}|L_t|\,dt.
\]
proving (\ref{levelbandlowerconclusion}) and
(\ref{levelbandupperconclusion}), respectively.  

\Cref{positiveproportionmassdownward} gives
\[
\int_0^s|L_t|\,dt
\leq\mathcal N_n\zeta_{n,p}^{-1}s|L_s|,
\qquad 0<s\leq\frac12\sup\ell.
\]
For $i\leq-1$, we have $a_{i+1}\leq a_0=(\sup\ell)/2$, so the preceding
estimate with $s=a_{i+1}$ gives
\[
\begin{aligned}
\int_0^{a_{i+1}}|L_t|\,dt
&\leq\mathcal N_n\zeta_{n,p}^{-1}a_{i+1}|L_{a_{i+1}}|\\
&=\mathcal N_n\zeta_{n,p}^{-1}
\frac{\exp(c_0/n)}{\exp(c_0/n)-1}
(\exp(c_0/n)-1)a_i|L_{a_{i+1}}|\\
&\leq\mathcal N_n\zeta_{n,p}^{-1}
\int_{a_i}^{a_{i+1}}|L_t|\,dt.
\end{aligned}
\]
Here the equality uses $a_{i+1}=\exp(c_0/n)a_i$, while the final inequality
uses the lower bound above and absorbs
$\exp(c_0/n)/(\exp(c_0/n)-1)=O(n)$ into $\mathcal N_n$.  Rearranging proves
(\ref{levelbandcumulativeconclusion}) for $i\leq-1$.

It remains to verify the estimates at $a_0$.  Since
$a_{-1}=a_0\exp(-c_0/n)$, we have
\[
a_0-a_{-1}=(1-\exp(-c_0/n))a_0
\geq\mathcal N_n^{-1}a_0.
\]
Moreover, \eqref{levelvolumeratio} with $t=a_{-1}$ and $s=a_0$ gives $\frac{|L_{a_{-1}}|}{|L_{a_0}|}
\leq\left(1+\frac{\exp(-c_0p/n)-1}{1-2^p}\right)^n
\leq\mathcal N_n.$
Monotonicity of $t\mapsto|L_t|$ therefore gives
\[
\begin{aligned}
\mathcal N_n^{-1}a_0|L_{a_0}|
&\leq(a_0-a_{-1})|L_{a_0}|\leq\int_{a_{-1}}^{a_0}|L_t|\,dt\leq(a_0-a_{-1})|L_{a_{-1}}|=(1-\exp(-c_0/n))a_0|L_{a_{-1}}|\leq\mathcal N_na_0|L_{a_0}|.
\end{aligned}
\]

On the other hand, the definition of $a_1$ gives $|L_{a_1}|=1.1^{-1}|L_{a_0}|$ and \Cref{positiveproportionofremainingheight}, applied to
$a_0<a_1$ gives 
\[
a_1-a_0\geq\mathcal N_n^{-1}(\sup\ell-a_0)
=\mathcal N_n^{-1}a_0.
\]
Since $a_1\leq\sup\ell=2a_0$, monotonicity of $t\mapsto|L_t|$ now gives
\[
\mathcal N_n^{-1}a_0|L_{a_0}|
\leq(a_1-a_0)|L_{a_1}|
\leq\int_{a_0}^{a_1}|L_t|\,dt
\leq(a_1-a_0)|L_{a_0}|
\leq a_0|L_{a_0}|
\leq\mathcal N_na_0|L_{a_0}|
\]
proving (\ref{levelbandpointwiseconclusion}) for $i=0$ and
(\ref{levelbandlowerconclusion}) and (\ref{levelbandupperconclusion}) for
$i=-1$.  Moreover,
\[
\int_{a_1}^{a_2}|L_t|\,dt
\leq(\sup\ell)|L_{a_1}|
\leq\mathcal N_n\int_{a_0}^{a_1}|L_t|\,dt,
\]
which proves (\ref{levelbandlowerconclusion}) for $i=0$.
Finally, this lower bound for $\int_{a_0}^{a_1}|L_t|\,dt$  combined with the first conclusion of \Cref{positiveproportionmassdownward},
applied with $s=a_0$;
\[
\int_0^{a_0}|L_t|\,dt
\leq\mathcal N_n\zeta_{n,p}^{-1}
a_0|L_{a_0}|.
\]
gives
\[
\int_{a_0}^{a_1}|L_t|\,dt
\geq\mathcal N_n^{-1}\zeta_{n,p}
\int_0^{a_1}|L_t|\,dt,
\]
which proves (\ref{levelbandcumulativeconclusion}) for $i=0$.

For every $i\leq0$, the definition of $\ell_i$ and
(\ref{levelbandpointwiseconclusion}) and
(\ref{levelbandlowerconclusion}) give
\[
\int\ell_i\,dx
=a_i|L_{a_i}|+\int_{a_i}^{a_{i+1}}|L_t|\,dt
+\int_{a_{i+1}}^{a_{i+2}}|L_t|\,dt
\leq\mathcal N_n\int_{a_i}^{a_{i+1}}|L_t|\,dt.
\]
Applying this with $i=j,j+1$ and using
(\ref{levelbandupperconclusion}) proves
(\ref{adjacentlevelmassconclusion}) and completes the proof of the claim.
\end{proof}

By the construction of $\ell_i$ and $g_i$, the effective $\delta$ for a given $\ell_i$ is $\delta\frac{\int\ell\,dx}{\int\ell_i\,dx}$, which we expect to contribute an error of the order $\sqrt{\delta\frac{\int\ell\,dx}{\int\ell_i\,dx}}$, so it will be useful the have the following estimate for $i\geq 0$.
\begin{clm}\label{intexponentialdecay}
We have
\[
\sum_{i=0}^{\infty}
\sqrt{\frac{\int\ell_i\,dx}{\int\ell\,dx}}
\leq\mathcal N_n.
\]
\end{clm}
\begin{proof}[Proof of claim]
By \Cref{positiveproportionofremainingheight},
$a_i\geq0.99\sup\ell$ for every $i\geq\mathcal N_n$, with a dimensional
constant independent of $p$.  For such $i$,
\[
\int\ell_{i+1}\,dx\leq \int\ell\cdot \textbf{1}_{L_{a+1}}\,dx \leq |L_{a_{i+1}}|\sup\ell,
\qquad
\int\ell_i\,dx\geq \int a_{i}\cdot \textbf{1}_{L_{a}}\,dx\geq0.99|L_{a_i}|\sup\ell,
\]
and hence
\[
\frac{\int\ell_{i+1}\,dx}{\int\ell_i\,dx}
\leq\frac{|L_{a_{i+1}}|\sup\ell}{|L_{a_{i}}|0.99\sup\ell}\leq\frac{1}{1.1\cdot0.99}\leq\frac1{1.05}.
\]
Each of the first $\mathcal N_n$ terms in $\sum_{i=0}^{\infty}
\sqrt{\frac{\int\ell_i\,dx}{\int\ell\,dx}}$ is at most one, and
the remaining terms decrease geometrically, so the claim follows.
\end{proof}

We obtain a similar estimate for $i\leq 0$.
\begin{clm}\label{lowlevelhardy}
We have
\begin{equation*}
\sum_{i=-\infty}^{0}\sqrt{\frac{\int\ell_i\,dx}{\int\ell\,dx}}
\leq\mathcal N_n\zeta_{n,p}^{-1/2}
.
\end{equation*}
\end{clm}

\begin{proof}[Proof of the claim]
\Cref{levelgeometry}(\ref{levelbandcumulativeconclusion}) gives, for every
$i\leq0$,
\[
\int_0^{a_{i+1}}|L_t|\,dt
\leq\mathcal N_n\zeta_{n,p}^{-1}
\int_{a_i}^{a_{i+1}}|L_t|\,dt.
\]
Since $\int_0^{a_i}|L_t|\,dt\leq\int_0^{a_{i+1}}|L_t|\,dt$, it follows that
\begin{align*}
&\sqrt{\int_0^{a_{i+1}}|L_t|\,dt}
+\sqrt{\int_0^{a_i}|L_t|\,dt}\leq2\sqrt{\int_0^{a_{i+1}}|L_t|\,dt}
\leq\mathcal N_n\zeta_{n,p}^{-1/2}
\sqrt{\int_{a_i}^{a_{i+1}}|L_t|\,dt}.
\end{align*}
On the other hand,
\[
\int_{a_i}^{a_{i+1}}|L_t|\,dt
=\left(\sqrt{\int_0^{a_{i+1}}|L_t|\,dt}
-\sqrt{\int_0^{a_i}|L_t|\,dt}\right)
\left(\sqrt{\int_0^{a_{i+1}}|L_t|\,dt}
+\sqrt{\int_0^{a_i}|L_t|\,dt}\right),
\]
so combining these bounds gives
\[
\sqrt{\int_{a_i}^{a_{i+1}}|L_t|\,dt}
\leq\mathcal N_n\zeta_{n,p}^{-1/2}
\left(\sqrt{\int_0^{a_{i+1}}|L_t|\,dt}
-\sqrt{\int_0^{a_i}|L_t|\,dt}\right).
\]
Telescoping these bounds and using
\Cref{levelgeometry}(\ref{adjacentlevelmassconclusion}), which gives $\int\ell_i\,dx
\leq\mathcal N_n\int_{a_i}^{a_{i+1}}|L_t|\,dt$, we find
\[
\sum_{i=-\infty}^{0}\sqrt{\int\ell_i\,dx}\leq \mathcal N_n\sum_{i=-\infty}^{0}\sqrt{\int_{a_i}^{a_{i+1}}|L_t|}
\leq\mathcal N_n\zeta_{n,p}^{-1/2}
\sqrt{\int_0^{a_1}|L_t|\,dt}
\leq\mathcal N_n\zeta_{n,p}^{-1/2}
\sqrt{\int\ell\,dx},
\]
whic proves the claim.
\end{proof}

For $i\leq-1$, \Cref{levelgeometry}(\ref{leveltwostepconclusion}) gives
$|L_{a_i}|/|L_{a_{i+2}}|\leq3/2$; for $i\geq0$ this ratio equals
$1.1^2<1.5$.  Hence \Cref{intervaltranslateprop} gives, for every
$i>\min\{k,-1\}$, a vector $v_i\in\mathbb{R}^n$ such that
\begin{equation}\label{ligisymdifeq}
\int|\ell_i(x)-g_i(x+v_i)|\,dx
\leq\mathcal N_n\zeta_{n,p}^{-1/2}
\sqrt{\frac{\delta}{\lambda}
\left(\int\ell\,dx\right)\left(\int\ell_i\,dx\right)}.
\end{equation}

Our aim will be to show that $v_0$ is actually suitable as a global translate.
We can express the target as follows, using $\int \ell dx=\int g dx$;
$$\int \left|\ell(x)-g(x+v_0)\right|dx\leq 2\left(\sum_{i=-\infty}^0 \int_{a_i}^{a_{i+1}}|(L_t+v_0)\setminus G_t|dt+\sum_{i=0}^\infty \int_{a_i}^{a_{i+1}}|(L_t+v_0)\setminus G_t|dt\right).$$
For $t\in[a_i,a_{i+1}]$, we have $L_{i,t}=L_t$ and
$G_{i,t}\subset G_t$, so that
\[
\int_{a_i}^{a_{i+1}}|(L_t+v_i)\setminus G_t|\,dt
\leq\int|\ell_i(x)-g_i(x+v_i)|\,dx.
\]
It remains to compare $v_i$ with $v_0$.
We evaluate the contributions for $i\leq0$ and $i\geq0$ separately, starting
with $i\leq0$.

\begin{clm}
\[
\sum_{i=-\infty}^0 \int_{a_i}^{a_{i+1}}|(L_t+v_0)\setminus G_t|\,dt
\leq\mathcal N_n\zeta_{n,p}^{-1}
\sqrt{\frac{\delta}{\lambda}}\int\ell\,dx.
\]
\end{clm}

\begin{proof}[Proof of claim]
Recall $k$ is the cutoff defined in \eqref{quantitativecutoff}. If $k\geq-1$, then
\[
\int_0^{a_1}|L_t|\,dt
\leq\int_0^{a_{k+2}}|L_t|\,dt
\leq C_n\zeta_{n,p}^{-1}
\sqrt{\frac{\delta}{\lambda}}\int\ell\,dx,
\]
which proves the claim.  Henceforth assume that $k\leq-2$. 
The main part of proving the claim is the following subclaim, which captures the application of \Cref{adjacenttranslatesarecloseprop} that $v_j$ and $v_{j+1}$ are similar in terms of $L_{a_{j+1}}$.
\begin{subclm}\label{adjtrasubclm}
$$\int_0^{a_{j+1}}
|L_t\triangle(L_t+v_j-v_{j+1})|\,dt
\leq\mathcal N_n\left(
\int|\ell_j(x)-g_j(x+v_j)|\,dx
+\int|\ell_{j+1}(x)-g_{j+1}(x+v_{j+1})|\,dx\right).$$
\end{subclm}
\begin{proof}[Proof of subclaim]
For every $k<j<0$,
\Cref{levelgeometry}(\ref{levelbandcumulativeconclusion}) and the maximality
of $k$ give
\begin{equation}\label{overlaplowerbound}
\int_{a_{j+1}}^{a_{j+2}}|L_t|\,dt
\geq\mathcal N_n^{-1}\zeta_{n,p}
\int_0^{a_{j+2}}|L_t|\,dt
>C_n^{1/2}\sqrt{\frac{\delta}{\lambda}}
\int\ell\,dx.
\end{equation}
Equation \eqref{ligisymdifeq} and
\Cref{levelgeometry}(\ref{adjacentlevelmassconclusion}) give
\begin{equation*}
\begin{aligned}
\int|\ell_j(x)-g_j(x+v_j)|\,dx
+\int|\ell_{j+1}(x)-g_{j+1}(x+v_{j+1})|\,dx
&\leq\mathcal N_n\zeta_{n,p}^{-1/2}
\sqrt{\frac{\delta}{\lambda}\int\ell\,dx}
\left(\sqrt{\int\ell_j\,dx}
+\sqrt{\int\ell_{j+1}\,dx}\right)\\
&\leq\mathcal N_n\zeta_{n,p}^{-1/2}
\sqrt{\frac{\delta}{\lambda}
\left(\int\ell\,dx\right)
\left(\int_{a_{j+1}}^{a_{j+2}}|L_t|\,dt\right)}.
\end{aligned}
\end{equation*}
Dividing both sides by $\int_{a_{j+1}}^{a_{j+2}}|L_t|\,dt$ and using \eqref{overlaplowerbound} and the smallness of $\delta/\lambda$ \eqref{nontrivialregime}, we find
\begin{equation}\label{Cndefeq}
\frac{
\int|\ell_j(x)-g_j(x+v_j)|\,dx
+\int|\ell_{j+1}(x)-g_{j+1}(x+v_{j+1})|\,dx}
{\int_{a_{j+1}}^{a_{j+2}}|L_t|\,dt}
\leq
\mathcal N_n\sqrt{
\frac{(\delta/\lambda)\int\ell\,dx}
{\zeta_{n,p}\int_{a_{j+1}}^{a_{j+2}}|L_t|\,dt}
}
<\mathcal N_n C_n^{-3/4}\leq\frac14,
\end{equation}
where the final inequality is ensured by the choice of $C_n$.  Hence
\[
\int|\ell_j(x)-g_j(x+v_j)|\,dx
+\int|\ell_{j+1}(x)-g_{j+1}(x+v_{j+1})|\,dx
\leq\frac14
\int_{a_{j+1}}^{a_{j+2}}|L_t|\,dt.
\]
Thus the second conclusion of \Cref{adjacenttranslatesarecloseprop}
applies to each pair $j,j+1$ to give
\[
\begin{aligned}
v_j-v_{j+1}
&\in3\frac{
\displaystyle\int|\ell_j(x)-g_j(x+v_j)|\,dx
+\displaystyle\int|\ell_{j+1}(x)-g_{j+1}(x+v_{j+1})|\,dx}
{\displaystyle\int_{a_{j+1}}^{a_{j+2}}|L_t|\,dt}\cdot\bigl(L_{a_{j+2}}-L_{a_{j+2}}\bigr).
\end{aligned}
\]
Hence \Cref{integratedlowtranslationlem} applied with $s=a_{j+2}$ gives
\begin{equation}\label{integratededgeexposure}
\begin{aligned}
\int_0^{a_{j+1}}
|L_t\triangle(L_t+v_j-v_{j+1})|\,dt
&\leq\mathcal{N}_n\frac{
\displaystyle\int|\ell_j(x)-g_j(x+v_j)|\,dx
+\displaystyle\int|\ell_{j+1}(x)-g_{j+1}(x+v_{j+1})|\,dx}
{\displaystyle\int_{a_{j+1}}^{a_{j+2}}|L_t|\,dt} a_{j+2}|L_{a_{j+2}}|\\
&\leq\mathcal N_n\left(
\int|\ell_j(x)-g_j(x+v_j)|\,dx
+\int|\ell_{j+1}(x)-g_{j+1}(x+v_{j+1})|\,dx\right),
\end{aligned}
\end{equation}
 where for the second inequality we used that for $k<j\leq-2$,
$a_{j+1}=a_{j+2}\exp(-c_0/n)$ gives
\[
\int_{a_{j+1}}^{a_{j+2}}|L_t|\,dt
\geq(a_{j+2}-a_{j+1})|L_{a_{j+2}}|
=(1-\exp(-c_0/n))a_{j+2}|L_{a_{j+2}}|
\geq\mathcal N_n^{-1}a_{j+2}|L_{a_{j+2}}|.
\]
For $j=-1$, use that
$L_{a_1}-L_{a_1}\subset L_{a_0}-L_{a_0}$ and 
$\int_{a_0}^{a_1}|L_t|\,dt
\geq\mathcal N_n^{-1}a_0|L_{a_0}|$
by \Cref{levelgeometry}(\ref{levelbandpointwiseconclusion}) to apply
\Cref{integratedlowtranslationlem} with
$s=a_0=(\sup\ell)/2$. This gives
the bound as in \eqref{integratededgeexposure} for $j=-1$.
\end{proof}

With this subclaim in place, we just need to do the proper accounting.
For $k+2\leq i\leq0$ and $t\in[a_i,a_{i+1}]$, the triangle inequality
for symmetric differences gives
\begin{align*}
|(L_t+v_0)\setminus G_t|
&\leq |(L_t+v_i)\setminus G_t|
   +|(L_t+v_0)\triangle(L_t+v_i)|\leq |(L_t+v_i)\setminus G_t|
   +\sum_{j=i}^{-1}|(L_t+v_j)\triangle(L_t+v_j)|,
\end{align*}
where the sum is empty when $i=0$.  Hence, using that $|(L_t+v_j)\triangle(L_t+v_j)|=|L_t\triangle(L_t+v_j-v_{j+1})|$, we can expand
\begin{align*}
\sum_{i=-\infty}^{0}
\int_{a_i}^{a_{i+1}}|(L_t+v_0)\setminus G_t|\,dt\leq \int_0^{a_{k+2}}|L_t|\,dt
&+\sum_{i=k+2}^{0}\int_{a_i}^{a_{i+1}}|(L_t+v_i)\setminus G_t|\,dt\\
&+
\sum_{i=k+2}^{-1}\sum_{j=i}^{-1}
\int_{a_i}^{a_{i+1}}
|L_t\triangle(L_t+v_j-v_{j+1})|\,dt.
\end{align*}

For the first term, by the definition of $k$, we have
\begin{align*}
\int_0^{a_{k+2}}|L_t|\,dt
&\leq\mathcal N_n\zeta_{n,p}^{-1}
\sqrt{\frac{\delta}{\lambda}}\int\ell\,dx,
\end{align*}

For the second term, we have 
\begin{align*}
\sum_{i=k+2}^{0}\int_{a_i}^{a_{i+1}}|(L_t+v_i)\setminus G_t|\,dt&\leq \sum_{i=k+2}^{0} |\ell_i(x)-g(x+v_i)|dx\leq \sum_{i=k+2}^{0} \mathcal N_n\zeta_{n,p}^{-1/2} \sqrt{\frac{\delta}{\lambda}\frac{\int\ell_i\,dx}{\int\ell\,dx}}\int\ell\,dx\leq\mathcal N_n\zeta_{n,p}^{-1}
\sqrt{\frac{\delta}{\lambda}}\int\ell\,dx
\end{align*}
where the final inequality uses \Cref{lowlevelhardy}.

For the final term, interchanging the two finite sums and applying
\Cref{adjtrasubclm} gives
\begin{align*}
\sum_{i=k+2}^{-1}\sum_{j=i}^{-1}
\int_{a_i}^{a_{i+1}}
|L_t\triangle(L_t+v_j-v_{j+1})|\,dt&=
\sum_{j=k+2}^{-1}\int_{a_{k+2}}^{a_{j+1}}
|L_t\triangle(L_t+v_j-v_{j+1})|\,dt\\
&\leq
\sum_{j=k+2}^{-1}\int_0^{a_{j+1}}
|L_t\triangle(L_t+v_j-v_{j+1})|\,dt\\
&\leq\mathcal N_n\sum_{j=k+2}^{-1}\left(
\int|\ell_j(x)-g_j(x+v_j)|\,dx
+\int|\ell_{j+1}(x)-g_{j+1}(x+v_{j+1})|\,dx\right)\\
&\leq\mathcal N_n
\sum_{i=k+2}^{0}\int|\ell_i(x)-g_i(x+v_i)|\,dx\leq\mathcal N_n\zeta_{n,p}^{-1}
\sqrt{\frac{\delta}{\lambda}}\int\ell\,dx.
\end{align*}
This concludes the proof of the claim.
\end{proof}

We prove the same bound for the $i\geq 0$.

\begin{clm}
\[
\sum_{i=0}^\infty \int_{a_i}^{a_{i+1}}|(L_t+v_0)\setminus G_t|dt
\leq\mathcal N_n\zeta_{n,p}^{-1/2}
\sqrt{\frac{\delta}{\lambda}}\int\ell\,dx.
\]
\end{clm}

\begin{proof}[Proof of claim]

We expand using the triangle inequality for symmetric differences
\begin{align*}
    \sum_{i=0}^\infty \int_{a_i}^{a_{i+1}}|(L_t+v_0)\setminus G_t|dt&\leq \sum_{i=0}^\infty \int_{a_i}^{a_{i+1}}\left(|(L_t+v_i)\setminus G_t|+\sum_{j=0}^{i-1}|(L_t+v_j)\triangle (L_t+v_{j+1})|\right)dt\\
    &=\sum_{i=0}^\infty \int_{a_i}^{a_{i+1}}|(L_t+v_i)\setminus G_t| dt+\sum_{j=0}^{\infty}\sum_{i=j+1}^{\infty}\int_{a_i}^{a_{i+1}}|(L_t+v_j)\triangle (L_t+v_{j+1})|dt\\
    &=\sum_{i=0}^\infty \int_{a_i}^{a_{i+1}}|(L_t+v_i)\setminus G_t| dt+\sum_{j=0}^{\infty}\int_{a_{j+1}}^{\infty}|(L_t+v_j)\triangle (L_t+v_{j+1})|dt.
\end{align*}
For the first term \Cref{intexponentialdecay} gives
\begin{multline*}
\sum_{i=0}^\infty \int_{a_i}^{a_{i+1}}|(L_t+v_i)\setminus G_t|\,dt
\leq \sum_{i=0}^\infty
\int |\ell_i(x)-g_i(x+v_i)|\,dx\\
\leq\mathcal N_n\zeta_{n,p}^{-1/2}
\sqrt{\frac{\delta}{\lambda}\int\ell\,dx}
\sum_{i=0}^\infty\sqrt{\int\ell_i\,dx}\leq\mathcal N_n\zeta_{n,p}^{-1/2}
\sqrt{\frac{\delta}{\lambda}}\int\ell\,dx.
\end{multline*}
For the second term,
\Cref{positiveproportionofremainingheight} gives $\frac{\sup\ell-a_{j+1}}{a_{j+2}-a_{j+1}}
\leq\mathcal N_n$ and \Cref{adjacenttranslatesarecloseprop}, gives that adjacent $v_i$'s are close together:
\begin{equation*}
\begin{aligned}
    \int_{a_{i+1}}^{a_{i+2}} |L_t \triangle (L_t+(v_i-v_{i+1}))|dt
    \leq4\left(
    \int|\ell_i(x)-g_i(x+v_i)|\,dx
    +\int|\ell_{i+1}(x)-g_{i+1}(x+v_{i+1})|\,dx\right).
\end{aligned}
\end{equation*}
  Hence, \Cref{translationestimateslem} allows us to conclude
\begin{align*}
\int_{a_{j+1}}^{\infty}
|L_t\triangle(L_t+v_j-v_{j+1})|\,dt&\leq\mathcal N_n
\int_{a_{j+1}}^{a_{j+2}}
|L_t\triangle(L_t+v_j-v_{j+1})|\,dt\\
&\leq\mathcal N_n\left(
\int|\ell_j(x)-g_j(x+v_j)|\,dx
+\int|\ell_{j+1}(x)-g_{j+1}(x+v_{j+1})|\,dx\right).
\end{align*}
Summing over $j\geq0$, and then using
\Cref{intexponentialdecay} as before gives
\[
\sum_{j=0}^{\infty}\int_{a_{j+1}}^{\infty}
|L_t\triangle(L_t+v_j-v_{j+1})|\,dt
\leq \mathcal{N}_n\sum_{j=0}^{\infty}\int|\ell_j(x)-g_j(x+v_j)|\,dx\leq\mathcal N_n\zeta_{n,p}^{-1/2}
\sqrt{\frac{\delta}{\lambda}}\int\ell\,dx.
\]
This concludes the proof of the claim.
\end{proof}

Combining the two claims, we conclude
\[
\inf_{v\in\mathbb R^n}\int|\ell(x)-g(x+v)|\,dx\leq \int|\ell(x)-g(x+v_0)|\,dx
\leq\mathcal N_n\zeta_{n,p}^{-1}
\sqrt{\frac{\delta}{\lambda}}\int\ell\,dx.
\]

This concludes the proof of the theorem.
\end{proof}

\bigskip
\noindent
{\it Acknowledgments:} AF is grateful to the Marvin V. and Beverly J. Mielke Fund for supporting his stay at the Institute for Advanced Study in Princeton, where part of this work was carried out. He is also partially supported by the Lagrange Mathematics and Computation Research Center. PvH thanks the Institute for Advanced Study in Princeton, and in particular the Director’s Discretionary Fund, for supporting his stay during the 2024–2025 academic year. MT gratefully acknowledges the hospitality of FIM at ETH Z\"urich, where part of this work was completed.

\bibliographystyle{alpha}
\bibliography{references}

\newcommand{\etalchar}[1]{$^{#1}$}
\begin{thebibliography}{{\VAN{Hintum}{v}}HST23b}

\bibitem[BB10]{ball2010stability}
Keith~M Ball and K{\'a}roly~J B{\"o}r{\"o}czky.
\newblock Stability of the pr{\'e}kopa--leindler inequality.
\newblock {\em Mathematika}, 56(2):339--356, 2010.

\bibitem[BB11]{ball2011stability}
Keith~M Ball and K{\'a}roly~J B{\"o}r{\"o}czky.
\newblock Stability of some versions of the {Pr{\'e}kopa--Leindler} inequality.
\newblock {\em Monatshefte f{\"u}r Mathematik}, 163(1):1--14, 2011.

\bibitem[BC17]{BauschkeCombettes}
Heinz~H. Bauschke and Patrick~L. Combettes.
\newblock {\em Convex Analysis and Monotone Operator Theory in Hilbert Spaces}.
\newblock Springer, 2 edition, 2017.

\bibitem[BD21]{boroczky2021stability}
K{\'a}roly~J B{\"o}r{\"o}czky and Apratim De.
\newblock Stability of the {Pr{\'e}kopa-Leindler} inequality for log-concave functions.
\newblock {\em Advances in Mathematics}, 386:107810, 2021.

\bibitem[BDS24]{BrigatiDolbeaultSimonov}
Gabriele Brigati, Jean Dolbeault, and Nikolay Simonov.
\newblock Stability for the logarithmic {S}obolev inequality.
\newblock {\em J. Funct. Anal.}, 287(8):110562, 2024.

\bibitem[BDS26]{BDSreview}
Gabriele Brigati, Jean Dolbeault, and Nikolay Simonov.
\newblock Logarithmic {S}obolev inequalities: a review on stability and instability results.
\newblock {\em La Matematica}, 5:Article 5, 2026.

\bibitem[BF14]{bucur2014lower}
Dorin Bucur and Ilaria Fragal{\`a}.
\newblock Lower bounds for the {Pr{\'e}kopa-Leindler} deficit by some distances modulo translations.
\newblock {\em J. Convex Anal}, 21(1):289--305, 2014.

\bibitem[BFR23]{boroczky2023quantitative}
K{\'a}roly~J B{\"o}r{\"o}czky, Alessio Figalli, and Jo{\~a}o~PG Ramos.
\newblock A quantitative stability result for the {Pr{\'e}kopa--Leindler} inequality for arbitrary measurable functions.
\newblock {\em Annales de l'Institut Henri Poincar{\'e} C}, 41(3):565--614, 2023.

\bibitem[BGRS14]{BobkovGozlanRobertoSamson}
S.~G. Bobkov, N.~Gozlan, C.~Roberto, and P.-M. Samson.
\newblock Bounds on the deficit in the logarithmic {S}obolev inequality.
\newblock {\em J. Funct. Anal.}, 267(11):4110--4138, 2014.

\bibitem[BJ17]{Barchiesi}
Marco Barchiesi and Vesa Julin.
\newblock Robustness of the {G}aussian concentration inequality and the {B}runn-{M}inkowski inequality.
\newblock {\em Calc. Var. Partial Differential Equations}, 56, 05 2017.

\bibitem[BK18]{BaloghKristaly2018}
Zoltán~M. Balogh and Alexandru Kristály.
\newblock Equality in {Borell–Brascamp–Lieb} inequalities on curved spaces.
\newblock {\em Advances in Mathematics}, 339:453--494, 2018.

\bibitem[BL76]{brascamp-lieb}
Jan Brascamp and Elliott Lieb.
\newblock On extensions of the {Brunn-Minkowski} and {Prékopa-Leindler} theorems, including inequalities for log concave functions, and with an application to the diffusion equation.
\newblock {\em Journal of Functional Analysis}, 22:366--389, 1976.

\bibitem[Bor75]{Borell}
C.~Borell.
\newblock Convex set functions in d-space.
\newblock {\em Period Math Hung.}, 6:111--136, 1975.

\bibitem[Car91]{Carlen91}
Eric~A. Carlen.
\newblock Superadditivity of fisher's information and logarithmic {S}obolev inequalities.
\newblock {\em J. Funct. Anal.}, 101(1):194--211, 1991.

\bibitem[Car17]{CarlenDuality}
Eric~A. Carlen.
\newblock Duality and stability for functional inequalities.
\newblock {\em Ann. Fac. Sci. Toulouse Math.}, 26(2):319--350, 2017.

\bibitem[CFMP09]{CianchiFuscoMaggiPratelli}
Andrea Cianchi, Nicola Fusco, Francesco Maggi, and Aldo Pratelli.
\newblock The sharp {S}obolev inequality in quantitative form.
\newblock {\em Journal of the European Mathematical Society}, 11(5):1105--1139, 2009.

\bibitem[Chr12a]{christ2012near}
Michael Christ.
\newblock Near equality in the {Brunn--Minkowski} inequality.
\newblock {\em arXiv preprint arXiv:1207.5062}, 2012.

\bibitem[Chr12b]{christ2012planar}
Michael Christ.
\newblock Near equality in the two-dimensional {Brunn--Minkowski} inequality.
\newblock {\em arXiv preprint arXiv:1206.1965}, 2012.

\bibitem[CM17]{carlen2017stability}
Eric Carlen and Francesco Maggi.
\newblock Stability for the {Brunn-Minkowski} and {Riesz} rearrangement inequalities, with applications to {G}aussian concentration and finite range non-local isoperimetry.
\newblock {\em Canad. J. Math.}, 69(5):1036--1063, 2017.

\bibitem[DEF{\etalchar{+}}25]{DEFFL}
Jean Dolbeault, Maria~J. Esteban, Alessio Figalli, Rupert~L. Frank, and Michael Loss.
\newblock Sharp stability for {S}obolev and log-{S}obolev inequalities, with optimal dimensional dependence.
\newblock {\em Camb. J. Math.}, 13(2):359--430, 2025.

\bibitem[DT16]{DolbeaultToscani}
Jean Dolbeault and Giuseppe Toscani.
\newblock Stability results for logarithmic sobolev and gagliardo--nirenberg inequalities.
\newblock {\em International Mathematics Research Notices}, (2):473--498, 2016.

\bibitem[Dub77]{dubuc1977criteres}
Serge Dubuc.
\newblock Crit{\`e}res de convexit{\'e} et in{\'e}galit{\'e}s int{\'e}grales.
\newblock In {\em Annales de l'institut Fourier}, volume~27, pages 135--165, 1977.

\bibitem[EK14]{eldan2014dimensionality}
Ronen Eldan and Bo'az Klartag.
\newblock Dimensionality and the stability of the {Brunn--Minkowski} inequality.
\newblock {\em Annali della Scuola Normale Superiore di Pisa. Classe di scienze}, 13(4):975--1007, 2014.

\bibitem[ELS20]{EldanLehecShenfeld}
Ronen Eldan, Joseph Lehec, and Yair Shenfeld.
\newblock Stability of the logarithmic {S}obolev inequality via the {F}{\"o}llmer process.
\newblock {\em Ann. Inst. Henri Poincar\'e Probab. Stat.}, 56(3):2253--2269, 2020.

\bibitem[Fig15]{figalli2015stability}
Alessio Figalli.
\newblock Stability results for the {Brunn--Minkowski} inequality.
\newblock In {\em Colloquium De Giorgi 2013 and 2014}, pages 119--127. Springer, 2015.

\bibitem[FIL16]{FathiIndreiLedoux}
Max Fathi, Emanuel Indrei, and Michel Ledoux.
\newblock Quantitative logarithmic {S}obolev inequalities and stability estimates.
\newblock {\em Discrete Contin. Dyn. Syst.}, 36(12):6835--6853, 2016.

\bibitem[FJ15]{figalli2015quantitative}
Alessio Figalli and David Jerison.
\newblock Quantitative stability for sumsets in $\mathbb{R}^n$.
\newblock {\em J. Eur. Math. Soc. (JEMS)}, 17(5):1079--1106, 2015.

\bibitem[FJ17]{figalli2017quantitative}
Alessio Figalli and David Jerison.
\newblock Quantitative stability for the {Brunn--Minkowski} inequality.
\newblock {\em Adv. Math.}, 314:1--47, 2017.

\bibitem[FJ21]{figalli2021quantitative}
Alessio Figalli and David Jerison.
\newblock A sharp {Freiman} type estimate for semisums in two and three dimensional {Euclidean} spaces.
\newblock {\em Ann. Sci. \'Ec. Norm. Sup\'er.}, 54(4):235--257, 2021.

\bibitem[FMM18]{Euclidean}
Alessio Figalli, Francesco Maggi, and Connor Mooney.
\newblock The sharp quantitative {E}uclidean concentration inequality.
\newblock {\em Camb. J. Math.}, 6:59--87, 3 2018.

\bibitem[FMP08]{FMP08}
Nicolo Fusco, Francesco Maggi, and Aldo Pratelli.
\newblock The sharp quantitative isoperimetric inequality.
\newblock {\em Ann. of Math. (2)}, 168(3):941--980, 2008.

\bibitem[FMP09]{Figalli09}
Alessio Figalli, Francesco Maggi, and Aldo Pratelli.
\newblock A refined {B}runn-{M}inkowski inequality for convex sets.
\newblock {\em Ann. Inst. H. Poincar\'e Anal. Non Lin\'eaire,}, 26:2511--2519, 11 2009.

\bibitem[FMP10a]{Figalli10amass}
Alessio Figalli, Francesco Maggi, and Aldo Pratelli.
\newblock A mass transportation approach to quantitative isoperimetric inequalities.
\newblock {\em Invent. Math}, pages 167--211, 2010.

\bibitem[FMP10b]{figalli2010mass}
Alessio Figalli, Francesco Maggi, and Aldo Pratelli.
\newblock A mass transportation approach to quantitative isoperimetric inequalities.
\newblock {\em Invent. Math.}, 182(1):167--211, 2010.

\bibitem[FMP13]{FMP}
Alessio Figalli, Francesco Maggi, and Aldo Pratelli.
\newblock Sharp stability theorems for the anisotropic {S}obolev and log-{S}obolev inequalities on functions of bounded variation.
\newblock {\em Adv. Math.}, 242:80--101, 2013.

\bibitem[FPR17]{FeoPosteraroRoberto}
Francesca Feo, Maria~R. Posteraro, and Cyril Roberto.
\newblock Some remarks on the stability of the log-{S}obolev inequality for the {G}aussian measure.
\newblock {\em Potential Anal.}, 47:37--52, 2017.

\bibitem[FR24]{figalli_ramos_PL24}
Alessio Figalli and Joao P.~G. Ramos.
\newblock Improved stability versions of the {P}rékopa–{L}eindler inequality.
\newblock {\em arXiv preprint arXiv:2410.01122, to appear on Journal of Convex Analysis}, 2024.

\bibitem[FvHT23]{BMStab}
Alessio Figalli, Peter van Hintum, and Marius Tiba.
\newblock Sharp quantitative stability of the {Brunn-Minkowski} inequality.
\newblock {\em arXiv preprint arXiv:2310.20643}, 2023.

\bibitem[FvHT24]{OTBMStab}
Alessio Figalli, Peter van Hintum, and Marius Tiba.
\newblock Sharp stability of the {Brunn-Minkowski} inequality via optimal mass transportation.
\newblock {\em arXiv preprint arXiv:2407.10932}, 2024.

\bibitem[FZ22]{FigalliZhang}
Alessio Figalli and Yi~Ru-Ya Zhang.
\newblock Sharp gradient stability for the {S}obolev inequality.
\newblock {\em Duke Mathematical Journal}, 171(12):2407--2459, 2022.

\bibitem[Gar02]{gardner2002brunn}
Richard Gardner.
\newblock The {Brunn-Minkowski} inequality.
\newblock {\em Bulletin of the American mathematical society}, 39(3):355--405, 2002.

\bibitem[Gen08]{Gentil}
Ivan Gentil.
\newblock From the {P}r{\'e}kopa--{L}eindler inequality to modified logarithmic {S}obolev inequality.
\newblock {\em Ann. Fac. Sci. Toulouse Math.}, 17(2):291--308, 2008.

\bibitem[GS17]{ghilli2017quantitative}
Daria Ghilli and Paolo Salani.
\newblock Quantitative {Borell-Brascamp-Lieb} inequalities for power concave functions.
\newblock {\em Journal of Convex Analysis}, 24(3):857--888, 2017.

\bibitem[{\VAN{Hintum}{v}}HK23]{van2023locality}
Peter {\VAN{Hintum}{v}}an~Hintum and Peter Keevash.
\newblock Locality in sumsets.
\newblock {\em arXiv preprint arXiv:2304.01189}, 2023.

\bibitem[{\VAN{Hintum}{v}}HK24]{SharpDelta}
Peter {\VAN{Hintum}{v}}an~Hintum and Peter Keevash.
\newblock The sharp doubling threshold for approximate convexity.
\newblock {\em Bulletin of the London Mathematical Society}, 56(10):3229--3239, 2024.

\bibitem[{\VAN{Hintum}{v}}HST22]{van2021sharp}
Peter {\VAN{Hintum}{v}}an~Hintum, Hunter Spink, and Marius Tiba.
\newblock Sharp stability of {Brunn--Minkowski} for homothetic regions.
\newblock {\em J. Eur. Math. Soc. (JEMS)}, 24(12):4207--4223, 2022.

\bibitem[{\VAN{Hintum}{v}}HST23a]{van2020sharp}
Peter {\VAN{Hintum}{v}}an~Hintum, Hunter Spink, and Marius Tiba.
\newblock Sharp {$L^1$} {I}nequalities for {S}up-{C}onvolution.
\newblock {\em Discrete Anal.}, 7:16pp, 2023.

\bibitem[{\VAN{Hintum}{v}}HST23b]{planarBM}
Peter {\VAN{Hintum}{v}}an~Hintum, Hunter Spink, and Marius Tiba.
\newblock Sharp quantitative stability of the planar {Brunn--Minkowski} inequality.
\newblock {\em J. Eur. Math. Soc. (JEMS)}, 26(2):695--730, 2023.

\bibitem[HUL93]{HiriartUrrutyLemarechal1}
Jean-Baptiste Hiriart-Urruty and Claude Lemar{\'e}chal.
\newblock {\em Convex Analysis and Minimization Algorithms {I}}.
\newblock Springer, 1993.

\bibitem[IK21]{IndreiKim}
Emanuel Indrei and Doo~Sung Kim.
\newblock Deficit estimates for the logarithmic {S}obolev inequality.
\newblock {\em Differential Integral Equations}, 34:437--466, 2021.

\bibitem[IM14]{IndreiMarcon}
Emanuel Indrei and Davide Marcon.
\newblock A quantitative log-{S}obolev inequality for a two parameter family of functions.
\newblock {\em Int. Math. Res. Not. IMRN}, (20):5563--5580, 2014.

\bibitem[KKM29]{knaster1929beweis}
Bronis{\l}aw Knaster, Kazimierz Kuratowski, and Stefan Mazurkiewicz.
\newblock Ein beweis des fixpunktsatzes f{\"u}r $n$-dimensionale simplexe (in {G}erman).
\newblock {\em Fund. Math.}, 14(1):132--137, 1929.

\bibitem[Roc70]{RockafellarCA}
R.~Tyrrell Rockafellar.
\newblock {\em Convex Analysis}.
\newblock Princeton University Press, 1970.

\bibitem[RS17]{rossi2017stability}
Andrea Rossi and Paolo Salani.
\newblock Stability for {Borell-Brascamp-Lieb} inequalities.
\newblock In {\em Geometric Aspects of Functional Analysis: Israel Seminar (GAFA) 2014--2016}, pages 339--363. Springer, 2017.

\bibitem[Ruz91]{ruzsa1991diameter}
Imre~Z Ruzsa.
\newblock Diameter of sets and measure of sumsets.
\newblock {\em Monatshefte f{\"u}r Mathematik}, 112(4):323--328, 1991.

\bibitem[Ruz97]{ruzsa1997brunn}
Imre~Z Ruzsa.
\newblock The {Brunn--Minkowski} inequality and nonconvex sets.
\newblock {\em Geometriae Dedicata}, 67:337--348, 1997.

\bibitem[Ruz06]{ruzsa2006additive}
Imre~Z Ruzsa.
\newblock Additive combinatorics and geometry of numbers.
\newblock In {\em Proceedings of the International Congress of Mathematicians}, volume~3, pages 911--930. Citeseer, 2006.

\bibitem[Sch13]{schneider2013convex}
Rolf Schneider.
\newblock {\em Convex bodies: the Brunn--Minkowski theory}, volume 151.
\newblock Cambridge university press, 2013.

\bibitem[Sch14]{Schneider}
Rolf Schneider.
\newblock {\em Convex Bodies: The Brunn--Minkowski Theory}, volume 151 of {\em Encyclopedia of Mathematics and its Applications}.
\newblock Cambridge University Press, Cambridge, 2 edition, 2014.

\end{thebibliography}

\end{document}